\documentclass[10pt,a4paper]{article}
\usepackage[utf8]{inputenc}
\usepackage[T1]{fontenc}
\usepackage{geometry}
\usepackage{amsthm}
\usepackage[french,english]{babel}
\usepackage{amssymb}
\usepackage{lmodern}
\usepackage{amsmath,amsfonts}
\usepackage{mathrsfs} 
\usepackage{dsfont}
\usepackage{stmaryrd}
\DeclareMathOperator{\dive}{div} 
\usepackage{graphicx}
\usepackage{stackengine}
\usepackage{fullpage}
\usepackage[nottoc,notlot,notlof]{tocbibind}
\usepackage[colorlinks=true,urlcolor=black,linkcolor=black,citecolor=black]{hyperref}
\numberwithin{equation}{section}
\usepackage{mathtools,array}
\usepackage{verbatimbox}
\usepackage{color}
\usepackage{bm}
\usepackage{tikz-cd}
\usepackage{pst-node}
\usetikzlibrary{matrix}
\usepackage{mathtools}
\usepackage{verbatimbox}
\usepackage{diagbox}
\usepackage{array}
\newcolumntype{C}{>{$\displaystyle} c <{$}}
\usepackage{makecell}
\usepackage{float}
\usepackage{tabu}
\usepackage{cancel}
\usepackage{lscape}
\usepackage[babel=true]{csquotes}
\usepackage{bm}
\usepackage{xcolor}
\makeatletter
\def\env@dmatrix{\hskip -\arraycolsep
	\let\@ifnextchar\new@ifnextchar
	\def\arraystretch{2}%
	\array{*{\c@MaxMatrixCols}{>{\displaystyle}c}}%
}

\makeatother

\DeclareFontShape{OMX}{cmex}{m}{n}{
	<-7.5> cmex7
	<7.5-8.5> cmex8
	<8.5-9.5> cmex9
	<9.5-> cmex10
}{}
\SetSymbolFont{largesymbols}{normal}{OMX}{cmex}{m}{n}
\SetSymbolFont{largesymbols}{bold}  {OMX}{cmex}{m}{n}

\begin{document}

	\title{Quantization of the Willmore Energy in Riemannian Manifolds}
	
	\author{Alexis Michelat\footnote{Mathematical Institute, University of Oxford (UK)\hspace{.5em}\href{mailto:michelat@maths.ox.ac.uk}{michelat@maths.ox.ac.uk}/ \href{mailto:alexis.michelat@normalesup.org}{alexis.michelat@normalesup.org}}
		 $\,$ and Andrea Mondino\footnote{Mathematical Institute, University of Oxford (UK)\hspace{.5em}\href{Andrea.Mondino@maths.ox.ac.uk}{Andrea.Mondino@maths.ox.ac.uk}}}    
	\date{\today}
	
	\maketitle
	
	\vspace{-0.5em}
	
	\begin{abstract}
		We show that the quantization of energy for Willmore spheres into closed 
		Riemannian manifolds holds provided that the Willmore energy 		and the area are uniformly bounded. The analogous energy quantization result holds for Willmore surfaces of arbitrary genus, under the additional assumptions that the immersion maps weakly converge to a limiting (possibly branched, weak immersion) map from the same surface, and that the conformal structures stay in a compact domain of the moduli space.  
			\end{abstract}
	
	\tableofcontents
	\vspace{0cm}
	\begin{center}
	\end{center}

	\theoremstyle{plain}
	\newtheorem*{theorem*}{Theorem}
	\newtheorem{theorem}{Theorem}[section]
	\newenvironment{theorembis}[1]
	{\renewcommand{\thetheorem}{\ref{#1}$'$}%
		\addtocounter{theorem}{-1}%
		\begin{theorem}}
		{\end{theorem}}
	\renewcommand*{\thetheorem}{\Alph{theorem}}
	\newtheorem{lemme}[theorem]{Lemma}
	\newtheorem*{lemme*}{Lemma}
	\newtheorem{propdef}[theorem]{Definition-Proposition}
	\newtheorem*{propdef*}{Definition-Proposition}
	\newtheorem{prop}[theorem]{Proposition}
	\newtheorem{cor}[theorem]{Corollary}
	\theoremstyle{definition}
	\newtheorem*{definition}{Definition}
	\newtheorem{defi}[theorem]{Definition}
	\newtheorem*{definition*}{Definition}
	\newtheorem{assumption}[theorem]{Assumption}
	\newtheorem{rem}[theorem]{Remark}
	\newtheorem*{rem*}{Remark}
	\newtheorem{rems}[theorem]{Remarks}
	\newtheorem{remimp}[theorem]{Important Remark}
	\newtheorem{exemple}[theorem]{Example}
	\newtheorem{defi2}{Definition}
	\newtheorem{propdef2}[defi2]{Proposition-Definition}
	\newtheorem{remintro}[defi2]{Remark}
	\newtheorem{remsintro}[defi2]{Remarks}
	\newtheorem{conj}{Conjecture}
	\newtheorem{question}{Open Question}
	\renewcommand\hat[1]{%
		\savestack{\tmpbox}{\stretchto{%
				\scaleto{%
					\scalerel*[\widthof{\ensuremath{#1}}]{\kern-.6pt\bigwedge\kern-.6pt}%
					{\rule[-\textheight/2]{1ex}{\textheight}}
				}{\textheight}%
			}{0.5ex}}%
		\stackon[1pt]{#1}{\tmpbox}
	}
	\parskip 1ex
	\newcommand{\totimes}{\ensuremath{\,\dot{\otimes}\,}}
	\newcommand{\vc}[3]{\overset{#2}{\underset{#3}{#1}}}
	\newcommand{\conv}[1]{\ensuremath{\underset{#1}{\longrightarrow}}}
	\newcommand{\A}{\ensuremath{\vec{A}}}
	\newcommand{\B}{\ensuremath{\vec{B}}}
	\newcommand{\C}{\ensuremath{\mathbb{C}}}
	\newcommand{\Sp}{\ensuremath{\mathbb{S}}}
	\newcommand{\D}{\ensuremath{\nabla}}
	\newcommand{\Disk}{\ensuremath{\mathbb{D}}}
	\newcommand{\E}{\ensuremath{\vec{E}}}
	\newcommand{\I}{\ensuremath{\mathbb{I}}}
	\newcommand{\Q}{\ensuremath{\vec{Q}}}
	\newcommand{\loc}{\ensuremath{\mathrm{loc}}}
	\newcommand{\z}{\ensuremath{\bar{z}}}
	\newcommand{\hh}{\ensuremath{\mathscr{H}}}
	\newcommand{\h}{\ensuremath{\vec{h}}}
	\newcommand{\vol}{\ensuremath{\mathrm{vol}}}
	\newcommand{\hs}[3]{\ensuremath{\left\Vert #1\right\Vert_{\mathrm{H}^{#2}(#3)}}}
	\newcommand{\hsdot}[3]{\ensuremath{\left\Vert #1\right\Vert_{\mathrm{\dot{H}}^{#2}(#3)}}}
	\newcommand{\R}{\ensuremath{\mathbb{R}}}
	\renewcommand{\P}{\ensuremath{\mathbb{P}}}
	\newcommand{\N}{\ensuremath{\mathbb{N}}}
	\newcommand{\Z}{\ensuremath{\mathbb{Z}}}
	\newcommand{\p}[1]{\ensuremath{\partial_{#1}}}
	\newcommand{\Res}{\ensuremath{\mathrm{Res}}}
	\newcommand{\lp}[2]{\ensuremath{\mathrm{L}^{#1}(#2)}}
	\renewcommand{\wp}[3]{\ensuremath{\left\Vert #1\right\Vert_{\mathrm{W}^{#2}(#3)}}}
	\newcommand{\wpn}[3]{\ensuremath{\Vert #1\Vert_{\mathrm{W}^{#2}(#3)}}}
	\newcommand{\np}[3]{\ensuremath{\left\Vert #1\right\Vert_{\mathrm{L}^{#2}(#3)}}}
	\newcommand{\znp}[4]{\ensuremath{\left\Vert #1\right\Vert_{\mathrm{L}_{#3}^{#2}(#4)}}}
	\newcommand{\mixdotHL}[4]{\ensuremath{\left\Vert #1\right\Vert_{\dot{\mathrm{H}}^{#2}+\mathrm{L}^{#3}(#4)}}}
	\newcommand{\hp}[3]{\ensuremath{\left\Vert #1\right\Vert_{\mathrm{H}^{#2}(#3)}}}
	\newcommand{\ck}[3]{\ensuremath{\left\Vert #1\right\Vert_{\mathrm{C}^{#2}(#3)}}}
	\newcommand{\hardy}[2]{\ensuremath{\left\Vert #1\right\Vert_{\mathscr{H}^{1}(#2)}}}
	\newcommand{\lnp}[3]{\ensuremath{\left| #1\right|_{\mathrm{L}^{#2}(#3)}}}
	\newcommand{\npn}[3]{\ensuremath{\Vert #1\Vert_{\mathrm{L}^{#2}(#3)}}}
	\newcommand{\nc}[3]{\ensuremath{\left\Vert #1\right\Vert_{C^{#2}(#3)}}}
	\renewcommand{\Re}{\ensuremath{\mathrm{Re}\,}}
	\renewcommand{\Im}{\ensuremath{\mathrm{Im}\,}}
	\newcommand{\dist}{\ensuremath{\mathrm{dist}}}
	\newcommand{\diam}{\ensuremath{\mathrm{diam}\,}}
	\newcommand{\leb}{\ensuremath{\mathscr{L}}}
	\newcommand{\supp}{\ensuremath{\mathrm{supp}\,}}
	\renewcommand{\phi}{\ensuremath{\vec{\Phi}}}
	\renewcommand{\H}{\ensuremath{\vec{H}}}
	\renewcommand{\L}{\ensuremath{\vec{L}}}
	\renewcommand{\lg}{\ensuremath{\mathscr{L}_g}}
	\renewcommand{\ker}{\ensuremath{\mathrm{Ker}}}
	\renewcommand{\epsilon}{\ensuremath{\varepsilon}}
	\renewcommand{\bar}{\ensuremath{\overline}}
	\newcommand{\s}[2]{\ensuremath{\langle #1,#2\rangle}}
	\newcommand{\pwedge}[2]{\ensuremath{\,#1\wedge#2\,}}
	\newcommand{\bs}[2]{\ensuremath{\left\langle #1,#2\right\rangle}}
	\newcommand{\scal}[2]{\ensuremath{\langle #1,#2\rangle}}
	\newcommand{\sg}[2]{\ensuremath{\left\langle #1,#2\right\rangle_{\mkern-3mu g}}}
	\newcommand{\n}{\ensuremath{\vec{n}}}
	\newcommand{\ens}[1]{\ensuremath{\left\{ #1\right\}}}
	\newcommand{\lie}[2]{\ensuremath{\left[#1,#2\right]}}
	\newcommand{\g}{\ensuremath{g}}
	\newcommand{\dzeta}{\ensuremath{\det\hphantom{}_{\kern-0.5mm\zeta}}}
	\newcommand{\e}{\ensuremath{\vec{e}}}
	\newcommand{\f}{\ensuremath{\vec{f}}}
	\newcommand{\ig}{\ensuremath{|\vec{\mathbb{I}}_{\phi}|}}
	\newcommand{\ik}{\ensuremath{\left|\mathbb{I}_{\phi_k}\right|}}
	\newcommand{\w}{\ensuremath{\vec{w}}}
	\newcommand{\hooklongrightarrow}{\lhook\joinrel\longrightarrow}
	\renewcommand{\tilde}{\ensuremath{\widetilde}}
	\newcommand{\vg}{\ensuremath{\mathrm{vol}_g}}
	\newcommand{\im}{\ensuremath{\mathrm{W}^{2,2}_{\iota}(\Sigma,N^n)}}
	\newcommand{\imm}{\ensuremath{\mathrm{W}^{2,2}_{\iota}(\Sigma,\R^3)}}
	\newcommand{\timm}[1]{\ensuremath{\mathrm{W}^{2,2}_{#1}(\Sigma,T\R^3)}}
	\newcommand{\tim}[1]{\ensuremath{\mathrm{W}^{2,2}_{#1}(\Sigma,TN^n)}}
	\renewcommand{\d}[1]{\ensuremath{\partial_{x_{#1}}}}
	\newcommand{\dg}{\ensuremath{\mathrm{div}_{g}}}
	\renewcommand{\Res}{\ensuremath{\mathrm{Res}}}
	\newcommand{\un}[2]{\ensuremath{\bigcup\limits_{#1}^{#2}}}
	\newcommand{\res}{\mathbin{\vrule height 1.6ex depth 0pt width
			0.13ex\vrule height 0.13ex depth 0pt width 1.3ex}}
	\newcommand{\antires}{\mathbin{\vrule height 0.13ex depth 0pt width 1.3ex\vrule height 1.6ex depth 0pt width
			0.13ex}}
	\newcommand{\ala}[5]{\ensuremath{e^{-6\lambda}\left(e^{2\lambda_{#1}}\alpha_{#2}^{#3}-\mu\alpha_{#2}^{#1}\right)\left\langle \nabla_{\vec{e}_{#4}}\vec{w},\vec{\mathbb{I}}_{#5}\right\rangle}}
	\setlength\boxtopsep{1pt}
	\setlength\boxbottomsep{1pt}

	\allowdisplaybreaks
	\newcommand*\mcup{\mathbin{\mathpalette\mcapinn\relax}}
	\newcommand*\mcapinn[2]{\vcenter{\hbox{$\mathsurround=0pt
				\ifx\displaystyle#1\textstyle\else#1\fi\bigcup$}}}
	\def\Xint#1{\mathchoice
		{\XXint\displaystyle\textstyle{#1}}%
		{\XXint\textstyle\scriptstyle{#1}}%
		{\XXint\scriptstyle\scriptscriptstyle{#1}}%
		{\XXint\scriptscriptstyle\scriptscriptstyle{#1}}%
		\!\int}
	\def\XXint#1#2#3{{\setbox0=\hbox{$#1{#2#3}{\int}$ }
			\vcenter{\hbox{$#2#3$ }}\kern-.58\wd0}}
	\def\ddashint{\Xint=}
	\newcommand{\dashint}[1]{\ensuremath{{\Xint-}_{\mkern-10mu #1}}}
	\newcommand\ccancel[1]{\renewcommand\CancelColor{\color{red}}\cancel{#1}}
	\newcommand\colorcancel[2]{\renewcommand\CancelColor{\color{#2}}\cancel{#1}}
	\newcommand{\abs}[1]{\left\lvert #1 \right \rvert}
	\newcommand{\norm}[1]{\ensuremath{\left\Vert #1 \right\Vert}}
	\newcommand{\genorm}[2]{\ensuremath{\left\Vert #1 \right\Vert_{#2}}}
	
	\renewcommand{\thetheorem}{\thesection.\arabic{theorem}}

	\section{Introduction}

		Let $m\geq 3$ and $(M^m,h)$ be a compact Riemannian manifold. For every smooth immersion $\phi:\Sigma\rightarrow M^m$,  the (conformal) Willmore energy is defined by 
	\begin{align*}
		W(\phi)=\int_{\Sigma}\left(|\H|^2+K_h(\phi_{\ast}T\Sigma)\right)d\vg,
	\end{align*}
    where $g=\phi^{\ast}h$ is the induced metric by $\phi$ on $\Sigma$, $K_h(\phi_{\ast}T\Sigma)$ is the sectional curvature of the two-plan induced by $\phi$, and $\H$ is the mean curvature 
    vector, defined by 
    \begin{align*}
    	\H=\frac{1}{2}\sum_{i,j=1}^{2}g^{i,j}\, \vec{\I}_{i,j}\, ,
    \end{align*}
    where $\vec{\I}_{i,j}$ is the second fundamental form. This functional, first introduced in the Euclidean space by Poisson in 1814 (\cite{poisson}, see also the work of Sophie Germain \cite{germain3}) in the context of non-linear elasticity, was rediscovered by Blaschke and Thomsen \cite{blaschke} in the 1920's  in the 
    framework of conformal geometry and by Willmore \cite{willmore1} in $1965$.    
    
    A key property of such a Lagrangian is the conformal invariance.
    Furthermore, it has been recently proved in \cite{mondinonguyen} that the Willmore functional is the {unique} (up to linear combinations with topological terms) \emph{conformally invariant} integral curvature energy for surfaces (in $\R^3$, it was already known that $(H^2-K_g)d\vg$ is up to scaling the only \emph{pointwise} conformally invariant $2$-form). 

   In addition to the aforementioned strong connection with conformal geometry, the Willmore functional in curved ambient spaces has remarkable links with other topics in mathematics and physics. For instance, the Willmore energy is the main term of the Hawking mass \cite{hawking} in the framework of general relativity (see for instance \cite{lammmetzgerschulze, eichmairkorber, brownemondino}), moreover it corresponds to the 
main term of the Nambu-Got\={o} action
 in string theory \cite{polyakov} 
 and the renormalised area functional in the AdS/CFT correspondence 
  \cite{alexakismazzeo1, alexakismazzeo2}.

From the point of view of calculus of variations, as well as motivated by the aforementioned connections to physics, it is natural to investigate the existence and the properties of the \emph{Willmore immersions} that are by definition the critical points of the Willmore energy. The standard first variation formulae show that critical points of $W$ satisfy the Euler-Lagrange equation
    \begin{align}\label{el1}
    	\Delta_g^{\perp}\H-2|\H|^2\H+\mathscr{A}(\H)+\mathscr{R}_1^{\perp}(\H)-2\,\tilde{K}_h\,\H+2\,\mathscr{R}_2(d\phi)+(DR)(d\phi)=0\, ,
    \end{align}
    where $\mathscr{A}$ is the Simons operator, and the other terms are curvature functionals defined by 
    \begin{align*}
    	\mathscr{A}(\w)&=-\frac{1}{2}\sum_{i,j=1}^2\left(\s{\e_i}{\D_{\e_j}\w}+\s{\e_j}{\D_{\e_i}\w}\right)\vec{\I}(\e_i,\e_j)\\
    	\mathscr{R}_1^{\perp}(\w)&=\left(\sum_{i=1}^2R(\w,\e_i)\e_i\right)^{\perp}\\
    	\tilde{K}_h&=K_h(\phi_{\ast}T\Sigma)\\
    	\mathscr{R}_2(d\phi)&=\sum_{i=1}^{2}\left(\s{R(\vec{\I}(\e_i,\e_1),\e_2)\e_2}{\e_1}+\s{R(\e_1,\vec{\I}(\e_i,\e_2))\e_2}{\e_1}\right)\\
    	(DR)(d\phi)&=\sum_{i=1}^{m}\s{(\D_{\vec{v}_j}R)(\e_1,\e_2)\e_2}{\e_1}\vec{v}_j \, ,
    \end{align*}
    where $(\e_1,\e_2)$ is a orthonormal moving frame of $\phi_{\ast}(T\Sigma)$,  $(\vec{v}_1,\cdots,\vec{v}_m)$ is a local orthonormal frame of $TM^m$, and $R$ is the Riemann curvature tensor 
    of the ambient space $(M^{m}, h)$.
    
        The literature about Willmore immersions in Riemannian manifolds (other than $\R^{n}$ or, equivalently by conformal invariance, $S^{n}$) is relatively recent 
    and in expansion. The first existence results for Willmore spheres have been obtained in perturbative settings  by the second author  \cite{mon1, mon2}.  Under the area-constraint condition, the existence  and the geometric properties of Willmore-type spheres have been investigated by Lamm-Metzger-Schulze \cite{lammmetzgerschulze}, Lamm-Metzger \cite{lammmetzger, lammmetzger2}, Laurain and the second author
     \cite{laurainmondino} and Eichmair-Körber \cite{eichmairkorber}.  Area-constrained Willmore tori of small area have been recently constructed by Ikoma, Malchiodi and the second author \cite{IMM1, IMM2}. All the aforementioned results are perturbative in nature, \emph{i.e.} either the surfaces have sufficiently small area, or the ambient Riemannian metric is sufficiently close to either the Euclidean or the spherical metrics.

The global problem to study the existence of smooth immersed spheres
minimising quadratic curvature functionals in Riemannian manifolds
was  studied by the second author in collaboration with
Kuwert and Schygulla \cite{KMS}  adapting Simon's ambient approach \cite{simon}, and by the second author with Rivière \cite{mondinoriviere, mondinoriviereACV} via a parametric approach, proving 
the existence of area‐constrained Willmore spheres in homotopy
classes as well as the existence of Willmore spheres under various assumptions and constraints. Also, Chen-Li \cite{chenli} proved the existence of stratified weak branched immersions of arbitrary genus minimising quadratic curvature functionals under various constraints (for weak immersions, refer to \cite{muller, toro} and to \cite{rivierecrelle, kuwertli} for works more in relationship with Willmore surfaces).

The main goal of this article is to generalise the quantization result of Bernard-Rivière \cite{quanta} to the case for Willmore immersions in Riemannian manifolds. This result should be seen as the first step to generalise Rivière's min-max theory for Willmore spheres \cite{eversion} to immersions with values into closed Riemannian manifolds. This extension is natural since, by conformal invariance of the Willmore energy, the quantization result in Euclidean spaces is equivalent to the energy quantization in  the sphere $S^n$ (for $n\geq 3$) equipped with its standard round metric.

    \renewcommand{\thetheorem}{\Alph{theorem}}
    \begin{theorem}\label{thm:MainThm}
    	Let $(M^m,h)$ be a smooth compact Riemannian manifold 
    	of dimension $m\geq 3$, and let $\{\phi_k\}_{k\in\N}\subset \mathrm{Imm}(S^{2},M^m)$ be a sequence of Willmore immersions.  
    	Assume that
    	\begin{align}\label{hyp0}
    		\left\{\begin{alignedat}{1}
    		&\limsup\limits_{k\rightarrow \infty}W(\phi_k)<\infty\\
    		& \limsup_{k\rightarrow \infty}\mathrm{Area}(\phi_k)<\infty.
    		\end{alignedat}\right.
    	\end{align}
       Then, up to a subsequence, the following energy identity holds
        \begin{align}\label{eq:QuantWSp2}
        	\lim\limits_{k\rightarrow \infty}W_{(M^{m},h)}(\phi_k)&=W_{(M^{m},h)}(\phi_{\infty})+\sum_{j=1}^{u} W_{(M^{m},h)}(\vec{\Psi}_{j})  +\sum_{s=1}^{p}W_{\R^{m}}(\vec{\eta}_s)+\sum_{t=1}^{q}\left(W_{\R^{m}}(\vec{\zeta}_t)-4\pi\theta_{0,t}\right)\, ,
        \end{align}
        where: 
      \begin{itemize}
      \item[\rm{(1)}]  The map $\phi_{\infty}$ is a smooth Willmore immersion of $S^{2}$ into $(M^{m},h)$, possibly branched at finitely many points $a_1,\cdots,a_N\in S^{2}$.
      \item[\rm{(2)}] 
      For any $j=1,\ldots, u\in \N,$ the maps  $\vec{\Psi}_{j}$ are smooth, possibly branched, Willmore immersions of $S^{2}$ into $(M^{m},h)$.
      \item[\rm{(3)}]  
      For any $s=1,\ldots, p\in \N, \, t=1,\ldots, q\in \N,$    the maps $\vec{\eta}_{s}: S^2\rightarrow \R^m$ and $\vec{\zeta}_t: S^2\rightarrow \R^m$ are smooth, possibly branched, Willmore immersions in $\R^{m}$ and $\theta_{0,t}=\theta_0(\vec{\zeta}_t,x_t)\in \N$ is the multiplicity of $\vec{\zeta}_t$ at some point $x_t\in \R^m$.
      \item[\rm{(4)}] 
      The map $\phi_{\infty}:S^{2}\to M^{m}$ is obtained as follows: there exist a sequence of diffeomorphisms $\{f_{k}\}_{k\in\N}$ of $S^{2}$ such that $\phi_{k}\circ f_{k}$ is conformal and 
        \begin{equation*}
        \phi_k\circ f_k\conv{k\rightarrow \infty} \phi_{\infty}\quad \text{ in }C^l_{\mathrm{loc}}(S^{2}\setminus\ens{a_1,\cdots,a_N}), \quad \forall  l\in \N.
        \end{equation*}
        Furthermore, it holds
         \begin{equation*}
       \lim_{k\to \infty}W_{(M^{m},h)}(\phi_{k})= W_{(M^{m},h)}(\phi_{\infty})\;  \Longleftrightarrow  \; \phi_k\circ f_k\conv{k\rightarrow \infty} \phi_{\infty} \text{ in }C^l(\Sigma),  \quad \forall l\in \N.
        \end{equation*}
        
          \end{itemize}
       Moreover, if  $\lim_{k\rightarrow \infty}\mathrm{Area}(\phi_k)=0$, then the first two terms in the right hand side of \eqref{eq:QuantWSp2} are not present, \emph{i.e.} using the same notation as above for $\vec{\eta}_s, \, \vec{\zeta}_t$ and $\theta_{0,t}$, it holds
        \begin{align*}
        	\lim\limits_{k\rightarrow \infty}W_{(M^{m},h)}(\phi_k)&=\sum_{s=1}^{p}W_{\R^{m}}(\vec{\eta}_s)+\sum_{t=1}^{q}\left(W_{\R^{m}}(\vec{\zeta}_t)-4\pi\theta_{0,t}\right)\, .
        \end{align*}

     \end{theorem}
     
     \renewcommand{\thetheorem}{\thesection.\arabic{theorem}}
     \begin{rems}  
     Let us denote  $\vec{\xi}_k=\phi_k\circ f_k$, where $f_k$ is given by (4) in Theorem \ref{thm:MainThm}.
       \begin{itemize}
       \item      The \emph{Riemannian Willmore bubbles} $\vec{\Psi}_j:S^{2}\to M^{m}$ are obtained as follows:  for any $j\in \{1,\ldots,u\}$, there exist a sequence of positive Möbius transformations $\psi^{j}_{k}$ of $S^{2}$  concentrating at one of $\ens{a_1,\cdots,a_N}$ such that:
        \begin{equation*}
       \vec{\xi}_k \circ \psi^{j}_{k} \conv{k\rightarrow \infty} \vec{\Psi}_{j} \quad \text{ in }C^l_{\mathrm{loc}}(S^{2} \setminus \big\{a_1^{j},\cdots,a^{j}_{N_{j}}\big\}), \quad \forall  l\in \N,
        \end{equation*}    
      where   $\big\{a_1^{j},\cdots,a^{j}_{N_{j}}\big\}$ is a finite set of points in $S^{2}$.
   
\item   The \emph{Euclidean Willmore bubbles} $\vec{\eta}_s, \vec{\zeta}_t: S^{2}\to \R^{m}$ are obtained by the following blow up procedure: for any $s\in \{1,\ldots, p\}$ (resp. for any $t\in \{1,\ldots, q\}$), there exists a point $\bar{x}^{s}\in M$ (resp.  $\bar{x}^{t}\in M$),  there exist a sequence of positive Möbius transformations $\psi^{s}_{k}$ (resp.  $\psi^{t}_{k}$) of $S^{2}$  concentrating at one of $\ens{a_1,\cdots,a_N}$, a sequence of rescalings $\lambda^{s}_{k}\to +\infty$ (resp. $\lambda^{t}_{k}\to +\infty$) and inversions $\Xi^{t}_{k}$ of $\R^{m}$ such that:
     \begin{equation*}
       \lambda^{s}_{k} \cdot {\rm Exp}_{\bar{x}^{s}}^{-1}\circ \vec{\xi}_k \circ \psi^{s}_{k} \conv{k\rightarrow \infty} \vec{\eta}_{s} \quad \text{ in }C^l_{\mathrm{loc}}(S^{2} \setminus \big\{a_1^{s},\cdots,a^{s}_{N_{s}}\big\}), \quad \forall  l\in \N,
        \end{equation*}    
        and, respectively,
          \begin{equation*}
  \Xi^{t}_{k}\circ     \lambda^{t}_{k} \cdot {\rm Exp}_{\bar{x}^{t}}^{-1}\circ \vec{\xi}_k \circ \psi^{t}_{k} \conv{k\rightarrow \infty} \vec{\zeta}_{t} \quad \text{ in }C^l_{\mathrm{loc}}(S^{2} \setminus \ens{a_1^{t},\cdots,a^{t}_{N_{t}}}), \quad \forall  l\in \N,
        \end{equation*}   
        where   $\big\{a_1^{s},\cdots,a^{s}_{N_{s}}\big\}, \big\{a_1^{t},\cdots,a^{t}_{N_{t}}\big\}\subset S^{2}$ are finite sets of points.
         \end{itemize}
     \end{rems}

      Arguing along the lines of the proof of Theorem \ref{thm:MainThm}, one can prove the energy quantization for surfaces of arbitrary genus, under the assumption of $\mathrm{W}^{2,2}$ weak convergence to a limit map and a bound on the conformal structures; the reader is referred to Theorem \ref{thm:MainThm2} for the precise statement.

    \begin{rems}\label{hypothesis}
    	\begin{enumerate}
	  \item[(1)] Since the Gauss curvature is quantized as well 
	  (see \eqref{final_quanta2}, and refer to \cite[Lemma V.$1$]{quanta}), the quantization of energy stated in  Theorem \ref{thm:MainThm} and Theorem \ref{thm:MainThm2} also holds for a general quadratic curvature functional of the form
    	    \begin{align*}
    	    	F_{\lambda_1, \lambda_2}(\phi)=\lambda_1\int_{\Sigma}|\H|^2\, d\vg+\lambda_2\int_{\Sigma}|\h_0|^2_{WP\, }d\vg \, ,
    	    \end{align*}
    	    for some $\lambda_1,\lambda_2\in\R$, where $\h_0$ is the Weingarten tensor (see for example the form introduced by Calabi in \cite{calabi}) and $|\,\cdot\,|_{WP}$ the Weil-Petersson metric. Explicitly, we have in a conformal local chart $\h_0=2\,\pi_{\n}(\p{z}^2\phi)dz^2=2e^{2\lambda}\,\p{z}\left(e^{-2\lambda}\p{z}\phi\right)dz^2$, and for two $2$-forms $\alpha=\varphi(z)dz^2$ and $\beta=\psi(z)dz^2$, we have
    	    \begin{align*}
    	    \s{\alpha}{\beta}_{WP}=e^{-4\lambda}\varphi(z)\bar{\psi(z)}=g^{-2}\otimes \alpha\otimes \bar{\beta}.
    	    \end{align*}
             Indeed, the quantity 
            \begin{align*}
            	\int_{\Sigma}K_{g_k}d\mathrm{vol}_{g_k}
            \end{align*}
          is  equal to  $2\pi \chi(\Sigma)$ in the limit 
             by Gauss-Bonnet Theorem and the smooth convergence of the conformal structures.

    		\item[(2)] In \cite{quanta} the  boundedness of the area is not explicitly assumed  because the authors work in $\R^n$; but making a stereographic projection and using the conformal invariance of the Willmore energy, their result is equivalent to the quantization of energy for $S^n$-valued maps, where $S^n$ is equipped with its standard round metric. Indeed, for all immersion $\phi:\Sigma\rightarrow S^n$, the Willmore energy in the sphere is defined by  
    	    		\begin{align*}
    			W_{S^n}(\phi)=\int_{\Sigma}\left(1+|\H|^2\right)d\vg,
    		\end{align*}
    		which shows in particular that a uniform bound on the Willmore energy implies a uniform bound on the area.  		
				
    		\item[(3)] Consider the following assumption. 
    		\begin{assumption}\label{Assum:SecCurv}
		   The sectional curvature $K_h$ of the ambient manifold $(M,h)$ is bounded below by  $\kappa_0>0$, \emph{i.e.}
    			there exists $\kappa_0>0$ such that $K_h(P)\geq \kappa_0$ for any  $2$-dimensional non-isotropic tangent plane $P\in \mathscr{G}_2(TM)$, where 
    			\begin{align*}
    				K_h(P)=\frac{\s{R(\vec{v},\vec{w})\vec{w}}{\vec{v}}}{|\vec{v}|^2|\vec{w}|^2-\s{\vec{v}}{\vec{w}}^2},
    			\end{align*}
    			where $P=\vec{v}\wedge\vec{w}$. 
    		\end{assumption}
    		The assumption \ref{Assum:SecCurv} implies that 
    		 for any immersion $\phi:\Sigma\rightarrow (M^m,h)$, it holds
    		\begin{align*}
    			\mathrm{Area}(\phi)\leq \frac{1}{\kappa_0}\int_{\Sigma}K_{h}(\phi_{\ast}T\Sigma)\, d\vg\leq \frac{1}{\kappa_0}\int_{\Sigma}\left(|\H|^2+K_h(\phi_{\ast}T\Sigma)\right)d\vg\leq \frac{1}{\kappa_0}W(\phi),
    		\end{align*}
    		which implies in particular that the bound on the area follows
            once a Willmore energy bound is in place.            In particular, the theorem will hold for any small enough (in the $C^2$ topology) perturbation of the round metric on $S^n$ \emph{without the assumption on the uniform boundedness of the area}.                	    \item[(4)] 
    	    Given a closed manifold $M^m$,  for a {generic} Riemannian metric $h$ (see \cite{random}), one can control the area of an immersion $\phi:\Sigma\rightarrow (M^m,h)$  by its $\mathrm{L}^2$-curvature energy  (\cite{bangert_kuwert}; see also \cite{mondiminus}) 
	    \begin{equation*}
	    F_{1,1}(\phi)=  \int_{\Sigma}|\vec{\mathbb{I}}|^{2}\, d\vg.
	    \end{equation*}
	      Therefore,  the area bound in the assumption \eqref{hyp0} can be dropped for a generic metric on a closed ambient manifold, when considering the quantization of energy for a functional $F_{\lambda_1, \lambda_2}$, with $\lambda_1, \lambda_2>0$. 
	        	\end{enumerate} 	  
    \end{rems}

    \subsection*{Some ideas of the proofs}
    \subsubsection*{Related literature on energy quantization and general strategy}

    The main results  Theorem \ref{thm:MainThm} and Theorem \ref{thm:MainThm2} shall be read in the context of other bubble-neck decomposition and energy quantization results, previously obtained in the literature. One can  mention  \cite{sacks, Struwe85, Jost91,  dingtian, parker, linriv, RivICM}
     in the setting of harmonic maps and other conformally invariant variational problems.
    A fundamental difference between the aforementioned energies and the Willmore functional is that, in the former, the corresponding Euler Lagrange equations are of second order, while the latter is a fourth-order problem.

 As already mentioned, the first quantization result for the Willmore energy was obtained by  Bernard-Rivière \cite{quanta} for Willmore surfaces  with bounded conformal structures and immersed in  Euclidean ambient spaces. A first generalisation of  \cite{quanta} was established by Laurain-Rivière \cite{quantamoduli} in the case of Willmore immersions with degenerating conformal classes, still with values into $\R^n$. The crux of the proof of the quantization of energy is to obtain the \emph{no-neck energy} property, once a suitable decomposition of the domain is performed. 	The idea is not restricted to the Willmore energy and applies to any quadratic energy (the Dirichlet energy \cite{linriv}, 	the Ginzburg-Landau energy \cite{linriv_GL1, linriv_GL2}, the Euclidean Willmore energy \cite{quanta, quantamoduli}, horizontal $1/2$-harmonic maps \cite{half_harmonic}, etc). To fix ideas, consider an immersion $\vec{u}:(\Sigma,g)\rightarrow (M^m,h)\subset \R^N$ between Riemannian manifolds (where we assume without loss of generality that $(M^m,h)$ is isometrically embedded into $\R^N$)  and  its Dirichlet energy given by 
	\begin{align*}
		E(\vec{u})=\frac{1}{2}\int_{\Sigma}|d\vec{u}|_g^2\, d\mathrm{vol}_g\, .
	\end{align*}
	 A neck region is conformally equivalent to an annulus $\Omega=B_R\setminus\bar{B}_r(0)$.We say that the \emph{no-neck energy} property holds provided that for any neck-region $\Omega_k(1)=B_{R_k}\setminus\bar{B}_{r_k}(0)\subset \C$, we have 
	\begin{align}\label{quanta0}
		\lim\limits_{\alpha\rightarrow 0}\limsup_{k\rightarrow \infty}\int_{\Omega_k(\alpha)}|\D\vec{u}_k|^2dx=0\, ,
	\end{align}
    where, for all $0<\alpha\leq 1$, $\Omega_k(\alpha)=B_{\alpha R_k}\setminus\bar{B}_{\alpha^{-1}r_k}(0)$. 
    The idea of the proof is to use Lorentz spaces (see the Appendix \ref{appendix} for more details) and the duality between $\mathrm{L}^{2,1}$ and $\mathrm{L}^{2,\infty}$, where $\mathrm{L}^{2,\infty}$ is the weak $\mathrm{L}^2$ space and $\mathrm{L}^{2,1}$ its pre-dual which can be explicitly characterised. The duality implies in particular that for any measured space $(X,\mu)$ and any measurable  maps $\vec{u},\vec{v}:X\rightarrow \R^N$, it holds
    \begin{align}\label{duality}
    	\left|\int_{\Omega}\s{\vec{u}}{\vec{v}}\, d\mu\right|\leq \np{\vec{u}}{2,1}{\Omega}\np{\vec{v}}{2,\infty}{\Omega}.
    \end{align}
    The idea of the quantization is to first show that, for some $\alpha_0>0$ independent of $k\in \N$, a uniform bound
    \begin{align}\label{duality1}
    	\np{d\vec{u}_k}{2,1}{\Omega_k(\alpha_0)}\leq C
    \end{align}
    holds. 
    Then, if one can prove a \emph{weak} quantization of energy, \emph{i.e.} that 
    \begin{align}\label{duality2}
    	\lim\limits_{\alpha\rightarrow 0}\limsup_{k\rightarrow \infty}\np{\D\vec{u}_k}{2,\infty}{\Omega_k(\alpha)}=0,
    \end{align}
    then the duality inequality \eqref{duality}, \eqref{duality1} and \eqref{duality2} show that the energy quantization in \eqref{quanta0} holds. Indeed: 
    \begin{align*}
    	\int_{\Omega_k(\alpha)}|\D\vec{u}_k|^2dx\leq \np{\D\vec{u}_k}{2,1}{\Omega_k(\alpha)}\np{\D\vec{u}_k}{2,\infty}{\Omega_k(\alpha)}\leq C\np{\D\vec{u}_k}{2,\infty}{\Omega_k(\alpha)}. 
    \end{align*}
    This approach was first developed by Lin-Rivière in the context of the Ginzburg-Landau functional 
    \cite{linriv_GL1, linriv_GL2} and for harmonic maps \cite{linriv}. It was used more recently for Willmore immersions with values into $\R^n$ by Bernard-Rivière \cite{quanta} and Laurain-Rivière \cite{quantamoduli}.

 \subsubsection*{Difficulties and novelties of this paper}

 Although we build on Bernard-Rivière's work (\cite{quanta}), new technical difficulties arise in the Riemannian setting. One key point is the need to introduce new Lorentz spaces that were not previously used in this context to our knowledge.

A fundamental ingredient in the proof of the energy quantization for Willmore immersions in Euclidean spaces \cite{quanta} is the introduction of conservation laws by  Rivière \cite{riviere1}: these  permit to rewrite the fourth-order Willmore equation into a system of second-order Jacobian-type equations which can be handled with tools from integrability by compensation.

    The first technical difficulty compared to \cite{quanta} is that we do not get the existence of 
    such an exact system of conservation laws for any 
    critical immersion. 
     Indeed, it is necessary to assume that the area in small enough to get the existence 
    of a perturbed system of conservation laws  (see  \cite[Lemma A.1 and Lemma A.2]{mondinoriviere}), and Lemmas \ref{A1} and \ref{A2} in this paper. 
    The perturbation is caused by the ambient curvature, and makes the system to be complex-valued of non-pure Jacobians (rather than real-valued of pure Jacobians as in the Euclidean setting). This is why,  in order to obtain the existence of the perturbed system of conservation laws in the neck region, we assume that the area is bounded and we prove that the area in neck-regions (and bubble regions too) is quantized, and therefore arbitrarily small. 
        
    Another technical point is to obtain a suitable $\epsilon$-regularity result for Willmore immersions with values into 
    curved ambient spaces. We prove in Theorem \ref{eps_reg} that there exists $\epsilon_0=\epsilon_0(M^m,h)>0$ with the following property: provided that $\phi:B(0,1)\rightarrow M^m$ is a weak Willmore immersion, the estimate
    \begin{align*}
    	\mathrm{Area}(\phi(B(0,1)))+\int_{B(0,1)}|\D\n|^2dx\leq \epsilon_0
    \end{align*}
    implies that $\phi\in C^{\infty}(B(0,1))$ and, for all $k\in \N$, there exists $C_k<\infty$ such that
    \begin{align*}
    	\np{\D^k\n}{\infty}{B(0,\frac{1}{2})}\leq C_k\np{\D\n}{2}{B(0,1)}.
    \end{align*}
    Notice that by the previous Remark \ref{hypothesis} ($3$)., the bound on area is superfluous for a generic metric on $M^m$, 
    or in case of ambient metric with positive sectional curvature.

    As mentioned above, a key technical difficulty in this paper  is that the conservation laws are complex-valued and the systems are non-pure Jacobians. This makes all the estimates more complicated: the additional terms coming from the non-vanishing Christoffel symbols force us to obtain \emph{a priori} estimates on the functions appearing in the system, contrary to the Euclidean case where one can proceed directly from the existence and obtain the $\mathrm{L}^{2,\infty}$ estimate \emph{a posteriori}. 
    
    The main technical difficulty 
    is to obtain a pointwise $\mathrm{L}^{2,\infty}$ bound 
    (see Theorem \ref{thm:L2infLW}).  In order to prove it, we will introduce a generalised Lorentz (or Orlicz-Lorentz) 
    space modelled on $\mathrm{L}^{2,\infty}$ and named $\mathrm{L}^{2,\infty}_{\log^{\beta}}$ (where $0\leq \beta\leq 1$ and $\mathrm{L}^{2,\infty}_{\log^0}=\mathrm{L}^{2,\infty}$)  in the analysis.  
    The reader is referred to Appendix \ref{appendix} for more details on these Banach function spaces.
    
     Another important step  in the proof is to show that for holomorphic maps (the same proof works more generally for harmonic maps), the standard $\epsilon$-regularity and scaling considerations giving that a $\mathrm{L}^{2,\infty}$ bound implies locally a $\mathrm{W}^{1,1}\cap \mathrm{L}^{2,1}$ estimate, hold 
    more generally when one has a $\mathrm{L}^{2,\infty}_{\log^{\beta}}$ bound (see Lemma \ref{lemme_holomorphe1}).    
    Since the more classical improvement 
    from $\mathrm{L}^{2,\infty}$  to $\mathrm{W}^{1,1}\cap \mathrm{L}^{2,1}$ for harmonic maps had several applications, it is natural to expect that the aforementioned sharpened improvement obtained in  Lemma \ref{lemme_holomorphe1} will be useful also in other settings.

  \textbf{Acknowledgments}.   The first   author is supported by the Early Postdoc Mobility \emph{Variational Methods in Geometric Analysis} P$2$EZP$2$\_$191893$.
   {The second  author is  supported by the European Research Council (ERC), under the European Union Horizon 2020 research and innovation programme, via the ERC Starting Grant  “CURVATURE”, grant agreement No. 802689}.

\section{Notation and preliminaries}
Throughout the paper, $(M^m, h)$ is a compact Riemannian manifold without boundary and $\Sigma$ is a closed Riemann surface.
Given a smooth immersion $\phi:\Sigma\to (M^m,h)$, we endow $\Sigma$ with the pull-back metric $g=\phi^*h$. An important role will be played by the conformal structure associated to the metric $g$ (for conformal structures on compact Riemann surfaces see for instance \cite{JostRS}). In particular we will assume that, given a sequence of immersions $\phi_{k}:\Sigma\to (M^m,h)$, the conformal structures associated to the pull-back metrics $g_k=\phi_k^*h$ are contained in a compact region of the moduli space. This assumption prevents the degeneration of the Riemann surface in the domain  (see \cite{quanta, quantamoduli}).

Without loss of generality, we can assume that the smooth immersion $\phi:\Sigma\to (M^m,h)$ is a conformal parametrisation, \emph{i.e.} we can choose local coordinates $(x_{1}, x_{2})$ on $\Sigma$ such that $g_{i,j}=e^{2\lambda} \delta_{i,j}$. The real valued function $\lambda$ will be called \emph{conformal factor}. Sometimes, taking advantage of the complex structure of a Riemann surface, 
it will be useful to switch the complex notation $z=x_{1}+i\, x_{2}$.

    Thanks to the Nash isometric embedding theorem, we can assume without loss of generality that $M^m\subset \R^n$, and that $h=\iota^{\ast}g_{\R^n}$, where $\iota:M^m\hookrightarrow \R^n$ is the 
    Nash embedding. Since $M^m$ is a compact manifold, we have in particular
    \begin{align*}
    	\np{\vec{\I}_{M^m}}{\infty}{M^m}\leq C_0<\infty,
    \end{align*}
    where $\vec{\I}_{M^m}$ is the second fundamental form of $\iota$. 
    Given a smooth immersion $\phi:\Sigma\rightarrow (M^m,h)$   of the $2$-dimensional surface $\Sigma$, we define the generalised Gauss map (see Hoffman-Osserman \cite{hoffman_gauss}) $\n_{\phi}:\Sigma \rightarrow \Lambda^{m-2}TM^m$ by
    \begin{align*}
    	\n_{\phi}=\star_h\frac{\p{x_1}\phi\wedge \p{x_2}\phi}{|\p{x_1}\phi\wedge \p{x_2}\phi|}
    \end{align*}
    where  $z=x_1+i\,x_2$ are arbitrary local coordinates on $\Sigma$, and $\star_h:\Lambda^2TM^m\rightarrow \Lambda^{m-2}TN^m$ is the linear Hodge operator associated to the metric $h$. We claim that the following formula holds:
    \begin{align}\label{second_form1}
    	\n_{\iota\circ\phi}=\iota_{\ast}(\n_{\phi})\wedge (\n_{\iota})\circ\phi.
    \end{align}
    Indeed, we have locally
    $
    	\n_{\iota}=\vec{v}_1\wedge\cdots\wedge \vec{v}_{n-m},
    $
    and
    $ 
        \n_{\phi}=\n_1\wedge\cdots\wedge\n_{m-2}
    $. If $\e_i=e^{-\lambda}\p{x_i}\phi$ in a conformal chart, we deduce by definition that $(\e_1,\e_2,\n_1,\cdots\n_{m-2})$ is an orthonormal basis of $TM^m$, which implies that $(\iota_{\ast}(\e_1),\iota_{\ast}(\e_2),\iota_{\ast}(\n_1),\cdots,\iota_{\ast}(\n_{m-2}),\vec{v}_1,\cdots,\vec{v}_{n-m})$ is an orthonormal basis of $\R^n$ and we deduce the claim \eqref{second_form1}.

 Even if the main Theorem \ref{thm:MainThm} concerns smooth immersions, some of the intermediate results that we will  establish will hold more generally for weak conformal immersions. A weak conformal immersion of the unit ball $B(0,1)\subset \R^{2}$ is map $\phi\in \mathrm{W}^{1,\infty}\cap \mathrm{W}^{2,2}(B(0,1), M^{m})$ such that the a.e. well defined pullback metric  $g=\phi^*h$ is conformal to the Euclidean metric on $B(0,1)$, \emph{i.e.} $g_{i,j}=e^{2\lambda} \delta_{i,j}$ for some a.e. well defined function $\lambda$. 
 Observe that the space of weak (conformal) immersions corresponds to the energy space for the Willmore functional $W$, thus it provides a natural functional analytic framework for the analysis \& calculus of variations of such an energy functional. Indeed the space of weak immersions with bounded area and Willmore energy satisfy useful pre-compactness properties. Let us recall the following pre-compactness result from \cite{mondinoriviere}, after \cite{rivierecrelle, KuwertLi}.

 \begin{theorem}\label{thm:compS2}
 Let $\phi_k:S^2 \hookrightarrow (M^m,h)$ be a sequence of weak immersions of $S^2$ into the closed $m$-dimensional Riemannian manifold $(M^m,h)$ and assume that the uniform area and Willmore bounds \eqref{hyp0} hold.
 Then, up to pre-composing with suitable bi-Lipschitz diffeomorphisms of $S^{2}$, one can assume that $\phi_{k}$ are conformally parametrised. Moreover,
 \begin{itemize}
 \item[\rm{(1)}] \textbf{Either}, $\diam (\phi_k(S^2))\to 0$ and thus there exists a point $\bar{x}\in M$ such that, up to a subsequence,  $\phi_k(S^2)\to \bar{x}$ in Hausdorff distance sense;
 \item[\rm{(2)}] \textbf{Or}, for every $k\in \N$, there exists a positive Möbius transformation $f_{k}$ of $S^2$ such that, called
  \begin{equation*}
\vec{\xi}_{k}= \phi_{k} \circ f_{k}
  \end{equation*}
 the reparametrised immersion and 
  \begin{equation*}
 \tilde{\lambda}_{k}= \log| \partial_{x_{1}} \vec{\xi}_{k} |= \log| \partial_{x_{2}} \vec{\xi}_{k} |
  \end{equation*}
 the new conformal factor, the following holds {\rm (}up to a subsequence{\rm)}: 
 \begin{itemize}
 \item[\rm{(i)}] There exists a finite set of points $\{a_{1}, \ldots, a_{N}\}$ such that for any compact subset $K\subset S^{2}\setminus \{a_{1}, \ldots, a_{N}\}$ 
   \begin{equation*}
 \sup_{k\in \N}\np{\tilde{\lambda}_{k}}{\infty}{K}<\infty\,.  
  \end{equation*}
 \item[\rm{(ii)}] There exists a conformal weak immersion $\vec{\xi}_{\infty}:S^2 \hookrightarrow (M^m,h)$, possibly branched at  $\{a_{1}, \ldots, a_{N}\}$, such that 
 \begin{equation}\label{eq:weakConvCompactThm}
  \vec{\xi}_{k}\rightharpoonup  \vec{\xi}_{\infty}\,  \text{ weakly in } \mathrm{W}^{2,2}_{\mathrm{loc}} (S^{2}\setminus \{a_{1}, \ldots, a_{N}\})\,.
 \end{equation}
Moreover, 
 \begin{equation*}
    W(\vec{\xi}_{\infty})\leq \liminf_{k\to \infty} W(\vec{\xi}_{k})\, .
 \end{equation*}
 \item[\rm{(iii)}] Furthermore, $W(\vec{\xi}_{\infty})= \lim_{k\to \infty} W(\vec{\xi}_{k})$ if and only if one can choose $ \{a_{1}, \ldots, a_{N}\}=\emptyset$ in the above claims.
  \end{itemize}

 \end{itemize}
 \end{theorem}
   
\begin{rem}
Thanks to Simon's monotonicity formula \cite{simon}, the first case in Theorem \ref{thm:compS2} is equivalent to $\mathrm{Area} (\phi_k(S^2))\to 0$. 
\end{rem}

As a consequence of the $\varepsilon$-regularity Theorem \ref{eps_reg} that we will prove later in the paper, if $\phi_{k}$ are Willmore spheres, then \eqref{eq:weakConvCompactThm} can be improved to a $C^{l}_{\mathrm{loc}} (S^{2}\setminus \{a_{1}, \ldots, a_{N}\})$ convergence for every $l\in \N$. 
The goal of the paper is to perform a fine analysis of $\vec{\xi}_{k}$, including the conformal factor, area, and Willmore energy, around the points $a_{1}, \ldots, a_{N}$, under the assumption that $\vec{\xi}_{k}$ are Willmore immersions. In order to simplify the notation, throughout the paper we will assume that the $\phi_{k}$'s are already conformally parametrised and in the good gauge satisfying  $\displaystyle\sup_{k\in\N} \np{\lambda_{k}}{\infty}{K}<\infty$, for every compact subset  $K\subset S^{2}\setminus \ens{a_{1}, \cdots, a_{N}}$.

    \section{$\mathrm{L}^{2,1}$ estimates on the mean curvature in the neck region}

    The first part of the proof of the main theorem is to establish $\mathrm{L}^{2,1}$ estimates on the mean curvature in the neck region (see Theorem \ref{L21_necks}).  To this aim, in Subsection 
     \ref{SS:AreaQuantHarn} we prove:  a no-neck area property   (in \eqref{eq:NoNeckArea}; see \cite{mondinoriviereACV}), a Harnack-type inequality for the conformal factors (in \eqref{new_harnarck}; see \cite{quanta}),   an $\mathrm{L}^{p}$ quantization result for the conformal parameters for some $p>2$ (see \eqref{neck_lambdap}) and, finally,  uniform $\mathrm{L}^{p}$ estimates for the conformal parameters   for some $p>2$ (see \eqref{lambdak_p}). 
    
       These results are of independent interest as they hold in general for weak immersions, regardless whether they satisfy the Willmore equation or not.

    The heart of the proof of Theorem \ref{L21_necks} will be to establish refined estimates on the approximate conservation laws \cite{mondinoriviere} satisfied by Willmore immersions. This will be achieved in Subsection  \ref{SS:estConsLaw}.

    The crux will be to remove a $\log |z|$ term in one of the key estimates (more precisely, in \eqref{vk}).  This will take almost all of Subsection \ref{SSS:logz} and will be the most innovative and technical part of the paper, requiring the introduction of apparently new Lorentz-type function spaces.   Let us stress that such a $\log |z|$ term is due to the curved ambient space and therefore was not present in the proof of the energy quantization for Willmore surfaces in Euclidean spaces \cite{quanta}.
    \\We next pass to some preliminary considerations.

    Thanks to the hypothesis \eqref{hyp0} of the theorem, we have
    \begin{align}\label{assumption}
    	\Lambda=\sup_{k\in \N}\left(\mathrm{Area}(\phi_k(\Sigma))+W(\phi_k)\right)<\infty\,.
    \end{align}

    Combining  \eqref{second_form1} with the triangle inequality, 
    the Gauss equations and the conformal invariance of the Dirichlet energy, 
    we get:
    \begin{align}\label{embedding}
    	&\int_{\Sigma}|d\n_{\iota\circ\phi}|^2_gd\vg\leq 2\int_{\Sigma}|d\n_{\phi}|^2_gd\vg+2\int_{\Sigma}|d(\n_{\iota}\circ \phi)|_g^2d\vg
    	= 2\int_{\Sigma}|d\n_{\phi}|^2_gd\vg+2\int_{\phi(\Sigma)}|\vec{\I}_{M^m}|^2_hd\mathrm{vol}_h\nonumber\\
    	&\leq 2\int_{\Sigma}|d\n_{\phi}|^2_gd\vg+2C_0^2\,\mathrm{Area}(\phi)
    	=2\int_{\Sigma}\left(4|\H|^2-2K_g+2\,K_h(\phi_{\ast}T\Sigma)\right)d\vg+2C_0^2\,\mathrm{Area}(\phi)\nonumber\\
    	&\leq 8\,W(\phi)+2C_0^2\mathrm{Area}(\phi)+4\np{K_h}{\infty}{M^m}-8\pi\,\chi(\Sigma)\, . 
    \end{align}
    Since $M^m$ is a closed manifold, we deduce that the sectional curvature $K_h$ of the smooth metric $h$ on $M^m$ is bounded. 
    Therefore, 
    the combination of  \eqref{assumption} and \eqref{embedding} yields: 
    \begin{align}\label{hyp}
    	\sup_{k\in\N}\int_{\Sigma}|d\n_{\iota\circ\phi_k}|^2_gd\vg\leq (8+2C_0^2)\Lambda+4\np{K_h}{\infty}{M^m}-8\pi\,\chi(\Sigma)=\Lambda(h)<\infty\, .
    \end{align}
    This allows us to apply the bubble-neck decomposition of Bernard-Rivière \cite[Proposition III.$1$]{quanta},  and the other theorems of \cite[Sections III and IV]{quanta}, since they do not use the Euler-Lagrange equation of Willmore surfaces in $\R^n$ and work for any sequence of smooth immersions of bounded Willmore energy. 
    
    \begin{lemme}[Bernard-Rivière, Lemma V.$1$ of \cite{quanta}]\label{gauss_quant}
    	There exist constants $\epsilon_0(n),C_0(n)>0$ with the following property. If $0<\epsilon<\epsilon_0(n)$, $0<4r<R<\infty$, $\Omega=B_R\setminus\bar{B}_r(0)$, and $\phi:\Omega\rightarrow \R^n$ is a conformal weak immersion satisfying the conditions
    	\begin{align}\label{hyp_gauss}
    		\left\{\begin{alignedat}{1}
    			\np{\D\n}{2,\infty}{\Omega}&\leq \epsilon\\
    			\int_{\partial B_r(0)}|\D\n|d\mathscr{H}^1&\leq \epsilon\\
    			\int_{B_R\setminus\bar{B}_{\frac{R}{2}}(0)}|\D\n|^2dx+\int_{B_{2r}\setminus\bar{B}_r(0)}|\D\n|^2dx&\leq \epsilon,
    		\end{alignedat}\right.
    	\end{align}
        then 
        \begin{align*}
        	\left|\int_{\Omega}K_gd\vg\right|\leq C_0(n)\,\epsilon,
        \end{align*}
        where $g=\phi^{\ast}g_{\R^n}$. 
    \end{lemme}

    By the Theorema Egregium of Gauss, since $\iota: (M^m,h)\rightarrow (\iota(M^m),g_{\R^n})$ is an isometry, we deduce that we can apply this result to $\phi_k$. Hence, if $\phi_k$ is parametrising a neck-region  $\Omega=B_R\setminus\bar{B}_r(0)$ and satisfies the hypothesis of Lemma \ref{gauss_quant}, we deduce that
    \begin{align*}
    	\left|\int_{\Omega}K_{g_k}d\mathrm{vol}_{g_k}\right|\leq C_0(n)\epsilon.
    \end{align*}
    Then, the other results of \cite[Section V]{quanta} where no Euler-Lagrange equation is used can be applied identically to $\{\iota \circ\phi_k\}_{k\in\N}$. The Liouville equation
    \begin{align*}
    	-\Delta \lambda_k=e^{2\lambda_k}K_{g_k}
    \end{align*}
    implies by the Adams-Morrey embedding (see \cite{rivnotes}) that 
    \begin{equation}\label{eq:estL2inftlak}
    	\sup_{k\in \N}\np{\D\lambda_k}{2,\infty}{B(0,1)}\leq C. 
    \end{equation}

    Now, we will describe the differences from  \cite[Section VI]{quanta} onwards.

   In \cite[Section VI]{quanta}, the conservative form of the Willmore equation for immersions in $\R^{n}$ discovered in \cite{riviere1} plays a fundamental role. In order to extend such analysis to the curved ambient setting, we will use the Euler-Lagrange equation in conservative form obtained in \cite{mondinoriviere} for immersions into Riemannian manifolds. However, this requires to prove that the area in neck regions is small enough (in fact, we need a slightly stronger statement which will follow from this bound thanks to a Harnack inequality; see  Lemma \ref{A1}).

   The next theorem is the main result of this section.

     \begin{theorem}\label{L21_necks}
    	Let $\ens{r_k}_{k\in\N},\ens{R_k}_{k\in\N}\subset (0,\infty)$ be such that $r_k\conv{k\rightarrow\infty}0$ and $R_k\conv{k\rightarrow \infty}R\in (0,\infty)$.  Let  $(M^m,h)\subset \R^n$ be a closed Riemannian manifold that we assume isometrically embedded in $\R^n$.  For any $0<\alpha\leq 1$, define the subset $\Omega_k(\alpha)=B_{\alpha R_k}\setminus\bar{B}_{\alpha^{-1}r_k}(0)\subset B(0,R_k)$. There exists constants $\epsilon_0=\epsilon_0(n,h),\alpha_0=\alpha_0(n,h)>0$ with the following property. Let $\{\phi_k\}_{k\in\N}\subset C^{\infty}(B(0,R_k),M^m)$ be a sequence of Willmore disks satisfying:
    	\begin{align*}
    		&\Lambda =\sup_{k\in\N}\left(\np{\D\lambda_k}{2,\infty}{B(0,R_k)}+\mathrm{Area}(\phi_k(B(0,R_k)))+\int_{B(0,R_k)}|\D\n_k|^2dx\right)<\infty \\
        	&\sup_{s\in \left[r_k,\frac{R_k}{2}\right]}\int_{B_{2s}\setminus\bar{B}_s(0)}|\D\n_k|^2dx\leq \epsilon_0 \, .
        \end{align*}
        Then we have
        \begin{align*}
        	\np{e^{\lambda_k}\H_k}{2,1}{\Omega_k(\alpha_0)}\leq C_0(n,h,\Lambda). 
        \end{align*}
    \end{theorem}

    The rest of the section will be devoted to the proof of  Theorem \ref{L21_necks}. This will require to estabish several  results of independent interest.  
           \subsection{ $\mathrm{L}^{p}$ bounds and quantization for the conformal parameters}\label{SS:AreaQuantHarn}
           
           The goal  of this section is four-fold: we prove a  no-neck area property  (see \eqref{eq:NoNeckArea}),  establish a  Harnack-type inequality for the conformal factors (see \eqref{new_harnarck}), we prove an $\mathrm{L}^{p}$ quantization result for the conformal parameters for some $p>2$ (see \eqref{neck_lambdap}), and  establish   uniform $\mathrm{L}^{p}$ estimates for the conformal parameters   for some $p>2$ (see \eqref{lambdak_p}).

    Let us first recall the following lemma from \cite{pointwise},  slightly generalising  \cite[Lemma IV.$1$]{quanta}.
     \begin{theorem}\label{neckfine}
    	There exists a positive real number $\epsilon_1=\epsilon_1(n)$ with the following property. Let $0<2^6r<R<\infty$ be fixed radii and $\phi:\Omega=B_R\setminus\bar{B}_r(0)\rightarrow \R^n$ be a weak immersion of finite total curvature such that 
    	\begin{align}\label{I0}
    		\np{\D\n}{2,\infty}{\Omega}\leq \epsilon_1(n).
    	\end{align}
    	For all $\left(\dfrac{r}{R}\right)^{\frac{1}{2}}<\alpha<1$, define $\Omega({\alpha})=B_{\alpha R} \setminus  \bar{B}_{\alpha^{-1}r}(0)$. Then there exists a universal constant $C_1=C_1(n)$ and $d\in \R$ {\rm (}depending on $r,R,\phi$ but not on $\alpha${\rm )} such that for all $\left(\dfrac{r}{R}\right)^{\frac{1}{3}}<\alpha<\dfrac{1}{4}$, we have
    	\begin{align}\label{I1}
    		\np{\D(\lambda-d\log|z|)}{2,1}{\Omega(\alpha)}\leq C_1\left(\sqrt{\alpha}\np{\D\lambda}{2,\infty}{\Omega}+\int_{\Omega}|\D\n|^2dx\right)
    	\end{align} 
    	and for all $r\leq \rho<R$, we have
    	\begin{align}\label{I2}
    		\left|d-\frac{1}{2\pi}\int_{\partial B_{\rho}}\partial_{\nu}\lambda \,d\mathscr{H}^1\right|\leq C_1\left(\int_{B_{\max\ens{\rho,2r}}\setminus \bar{B}_r(0)}|\D\n|^2dx+\frac{1}{\log\left(\frac{R}{\rho}\right)}\int_{\Omega}|\D\n|^2dx\right) \, .
    	\end{align}
    	In particular, there exists a universal constant $C_1'=C'_1(n)$ with the following property: for all $\left(\dfrac{r}{R}\right)^{\frac{1}{3}}<\alpha<\dfrac{1}{4}$, there exists $A_{\alpha}\in \R$ such that
    	\begin{align}\label{I3}
    		\np{\lambda-d\log|z|-A_{\alpha}}{\infty}{\Omega({\alpha})}\leq C_1'\left(\sqrt{\alpha}\np{\D\lambda}{2,\infty}{\Omega}+\int_{\Omega}|\D\n|^2dx\right).
    	\end{align}    	
    \end{theorem}

    Applying Theorem \ref{neckfine} to $\{\phi_k\}_{k\in\N}$ (that we see from now as a map $\phi_k:\Sigma\rightarrow \R^n$ such that $\phi_k(\Sigma)\subset M^m$ for all $k\in\N$), we deduce that in a neck region $\Omega_k(\alpha)=B(0,\alpha R_k)\setminus\bar{B}(0,\alpha^{-1}r_k)$ (where $\displaystyle\limsup_{k\rightarrow \infty}R_k<\infty$), there exists large enough $d_k\in \R$ and $A_k\in \R$ such that 
    \begin{align*}
    	\np{\lambda_k-d_k\,\log|z|-A_k}{\infty}{\Omega_k(\alpha)}\leq C_1'\left(\sqrt{\alpha}\np{\D\lambda_k}{2,\infty}{\Omega_k(\alpha)}+\int_{\Omega_k(1)}|\D\n_k|^2dx\right)\leq C_1''<\infty
	    \end{align*}
    thanks to \eqref{hyp} and \eqref{eq:estL2inftlak}. We deduce that:
    \begin{align}\label{harnack1}
    	e^{-2C_1''}e^{2A_k}|z|^{2d_k}\leq e^{2\lambda_k}\leq e^{2C_1''}e^{2A_k}|z|^{2d_k}, \quad \text{for all $z\in \Omega_k(\alpha)$}.
    \end{align}
    Thanks to \eqref{hyp}, we deduce that 
    \begin{align}\label{hyp4}
    	\sup_{k\in\N}\int_{\Omega_k(\alpha)}e^{2A_k}|z|^{2d_k}|dz|^2\leq \sup_{k\in\N}e^{2C_1''}\mathrm{Area}(\phi_k)\leq e^{2C_1''}\Lambda<\infty. 
    \end{align}
    Now, by the $\epsilon$-regularity Theorem \ref{eps_reg} to be proven below, 
     we deduce that $\phi_k\conv{k\rightarrow \infty}\phi_{\infty}$ in $C^{l}_{\mathrm{loc}}(B(0,1)\setminus\ens{0})$ for all $l\in\N$, where $\phi_{\infty}:B(0,1)\rightarrow \R^n$ is a branched immersion having at most a branch point at $0$. Therefore, by \cite{riviere1} (see also \cite{quanta}), there exist an integer $\theta_0\geq 1$ and $\vec{A}_0\in\C^n\setminus\ens{0}$ such that 
    \begin{align*}
    	\p{z}\phi_{\infty}=\vec{A}_0z^{\theta_0-1}+o(|z|^{\theta_0-1})\, .
    \end{align*}
    We also let $\beta_0>0$ such that 
    \begin{align*}
    	e^{2\lambda_{\infty}}=2|\p{z}\phi_{\infty}|^2=2|\vec{A}_0|^2|z|^{2\theta_0-2}(1+o(1))=\beta_0^2|z|^{2\theta_0-2}(1+o(1)). 
    \end{align*}
    Now, applying \eqref{I2} to $d_k$ and $\rho=\alpha\frac{R}{2}$, we deduce that 
    \begin{align*}
    	\limsup_{k\rightarrow \infty}\left|d_k-\frac{1}{2\pi}\int_{\partial B_{\alpha \frac{R}{2}}}\partial_{\nu}\lambda_k\,d\mathscr{H}^1\right|\leq \Gamma_2<\infty.
    \end{align*}
    By the strong convergence, it follows that 
    \begin{align*}
    	\frac{1}{2\pi}\int_{\partial B_{\alpha \frac{R}{2}}}\partial_{\nu}\lambda_k\,d\mathscr{H}^1\conv{k\rightarrow \infty}\frac{1}{2\pi}\int_{\partial B_{\alpha \frac{R}{2}}}\partial_{\nu}\lambda_{\infty}d\mathscr{H}^1=\theta_0-1+O(\alpha R)
    \end{align*}
    and we deduce that $\ens{d_k}_{k\in\N}\subset \R$ is a bounded sequence. Therefore, we can assume up to a subsequence that $d_k\conv{k\rightarrow \infty}d\in\R$.

    \begin{lemme}\label{lem:d>-1}
     $d>-1$.
    \end{lemme}
    \begin{proof}
    Define
    \begin{align*}
    	\underline{A}&=\liminf_{k\rightarrow \infty}A_k\in\R\cup\ens{-\infty,\infty}\\
    	\bar{A}&=\limsup_{k\rightarrow \infty}A_k\in\R\cup \ens{-\infty,\infty}.
    \end{align*}
    By the strong convergence, assuming without loss of generality that $R_k\conv{k\rightarrow \infty}R$ for some $R>0$, we deduce that for all $|z|=\alpha R$, we have
    \begin{align*}
    	e^{-2C_1''}e^{2\underline{A}}|z|^{2d}\leq \beta_0^2|z|^{2\theta_0-2}(1+o(1))\leq e^{2C_1''}e^{{A}}|z|^{2d}.
    \end{align*}
    It follows that 
    \begin{align*}
     \left\{\begin{alignedat}{1}
    	-\infty&<\underline{A}\leq \bar{A}<\infty\\
    	 e^{2\bar{A}}&\geq \beta_0^2e^{-2C_1''}(\alpha R)^{2\theta_0-2(d+1)}(1+o(1))\geq \frac{2}{3}\beta_0^2e^{-2C_1''}(\alpha R)^{2\theta_0-2(d+1)}\\
    	 e^{2\bar{A}}&\leq \frac{3}{2}\beta_0^2e^{2C_1''}(\alpha R)^{2\theta_0-2(d+1)}.
         \end{alignedat}\right.
    \end{align*}
    Now, if $d_k\neq -1$, we have
    \begin{align}\label{control_conf}
    	&\int_{\Omega_k(\alpha)}e^{2A_k}|z|^{2d_k}|dz|^2=2\pi e^{2A_k}\int_{\alpha^{-1}r_k}^{\alpha R_k}r^{2d_k+1}dz= \frac{\pi}{d_k+1} e^{2A_k}\left((\alpha R_k)^{2d_k+2}-(\alpha^{-1}r_k)^{2d_k+2}\right)\\
    	&\geq \frac{\pi}{2(d_k+1)}\beta_0^2e^{-2C_1''}(\alpha R)^{2\theta_0
    	}\left(\left(\frac{R}{R_k}\right)^{-2(d+1)}(\alpha R_k)^{2(d_k-d)} -(\alpha R)^{-2(d_k+1)}(\alpha^{-1}r_k)^{2d_k+2}\right).\nonumber
    \end{align}
    If $d<-1$,  we deduce that 
    \begin{align*}
    	\frac{\pi}{2(d_k+1)}\beta_0^2e^{-2C_1''}(\alpha R)^{2\theta_0
    	}\left(\left(\frac{R}{R_k}\right)^{-2(d+1)}(\alpha R_k)^{2(d_k-d)} -(\alpha R)^{-2(d_k+1)}(\alpha^{-1}r_k)^{2d_k+2}\right)\conv{k\rightarrow \infty}\infty,
    \end{align*}
    which contradicts \eqref{hyp4}. Therefore, we have $d\geq -1$. Now, notice that provided that $d_k=-1$, we have
    \begin{align*}
    	\int_{\Omega_k(\alpha)}e^{2A_k}|z|^{2d_k}|dz|^2=2\pi e^{2A_k}\log\left(\frac{\alpha^2R_k}{r_k}\right)\geq \pi\beta_0^2e^{-2C_1''}(\alpha R)^{2\theta_0-2(d+1)}\log\left(\frac{\alpha^2R_k}{r_k}\right)\conv{k\rightarrow \infty}\infty. 
    \end{align*}
    Therefore, in the limiting case $d=-1$, we can assume that $d_k\neq -1$ for $k$ large enough. Now, notice that since $R_k\conv{k\rightarrow \infty}R$, we have
    \begin{align*}
    	&\left(\frac{R}{R_k}\right)^{-2(d+1)}(\alpha R_k)^{2(d_k-d)}=1+o(1)\\
    	&(\alpha R)^{-2(d_k+1)}(\alpha^{-1}r_k)^{2d_k+2}=(1+o(1))r_k^{2d_k+2},
    \end{align*}
    which implies that 
    \begin{align}\label{bounded}
    	\limsup_{k\rightarrow \infty}\left(\frac{1}{d_k+1}\left(1-r_k^{2d_k+2}\right)\right)<\infty.
    \end{align}
    First assume that $d_k<-1$ for $k\in\N$ large enough. Writing $d_k=-1-\epsilon_k$, we get than there exists $C>0$ such that for all $k\in\N$ large enough
    \begin{align*}
    	\frac{1}{d_k+1}\left(1-r_k^{2d_k+2}\right)=\frac{1}{-\epsilon_k}\left(1-\frac{1}{r_k^{2\epsilon_k}}\right)=\frac{1}{\epsilon_k}\left(\frac{1}{r_k^{2k}}-1\right)\leq C.
    \end{align*}
    Therefore, this implies that 
    \begin{align*}
    	\frac{1}{r_k^{2\epsilon_k}}\leq 1+C\epsilon_k\, .
    \end{align*}
    Taking the logarithm, 
    \begin{align*}
    	2\epsilon_k\log\left(\frac{1}{r_k}\right)\leq \log\left(1+C\epsilon_k\right)\leq C\epsilon_k
    \end{align*}
    and finally
    \begin{align*}
    	\log\left(\frac{1}{r_k}\right)\leq \frac{C}{2}
    \end{align*}
    which is absurd since $r_k\conv{k\rightarrow\infty}0$. Coming back to the 
     identity \eqref{control_conf}, 
     we deduce that for $k$ large enough if $\delta>0$ is such that $d>-1+\delta$
    \begin{align*}
    	\int_{\Omega_k(\alpha)}e^{2\lambda_k}|dz|^2\leq \frac{4\pi}{\delta}\beta_0^2e^{2C_1''}(\alpha R)^{2\theta_0}\conv{\alpha\rightarrow 0}0. 
    \end{align*}
    Then, provided that $d_k=-1+\epsilon_k$, \eqref{bounded} implies that there  exists some $C<\infty$ such that
    \begin{align*}
    	\frac{1}{\epsilon_k}\left(1-r_k^{2\epsilon_k}\right)=\frac{1}{d_k+1}\left(1-r_k^{2d_k+2}\right)\leq C,
    \end{align*}
    which yields
    \begin{align*}
    	\log(r_k)\geq \frac{\log(1-C\epsilon_k)}{2\epsilon_k}\geq -C,
    \end{align*}
    for $k$ large enough, using the inequality $\log(1-x)\geq -2x$ for $0\leq x\leq \dfrac{1}{2}$. Since $r_k\conv{k\rightarrow\infty}0$, we have $\log(r_k)\conv{k\rightarrow \infty}-\infty$ and we obtain a contradiction. We conclude that $d>-1$.
    \end{proof}

    \begin{rem}
    	Once the quantization of energy is established,  it will imply \emph{a posteriori} that  (see \cite{pointwise})
    	\begin{equation}\label{eq:dkInt}
    		d_k\conv{k\rightarrow \infty}\theta_0-1\geq 0\, .
    	\end{equation}
        In fact, 
        \cite[Theorem A]{pointwise} 
        holds for an arbitrary immersion with values into $\R^n$ and yields that $d_k=\theta_0-1\geq 0$. However, 
        in the present setting of {curved} ambient space, we could not  get \eqref{eq:dkInt}  \emph{a priori} (as it happens in the case of a {flat} ambient space). The origin of such a difficulty lies in the non-vanishing 	curvature that perturbs the system of conservation laws associated to the Willmore equation \cite{mondinoriviere}.
  
  Let us also observe that the fact that a no-neck energy property implies the  asymptotic integrality of $d_{k}$ suggests  that an hypothesis ensuring that $d_k>-1+\epsilon$ (for $k$ large enough) is necessary. 
    \end{rem}

    Now, since by \cite[Proposition III.$1$]{quanta} neck-regions are disjoint and finite unions of such annuli, if $\Omega_k(\alpha)$ is the whole neck-region associated to a concentration point $a_i$ (where $1\leq i\leq N$, and $N$ is the number of concentration points), we have 
    (see also \cite{mondinoriviereACV}):
    \begin{align}\label{no_area1}
    	\lim_{\alpha\rightarrow 0}\limsup_{k\rightarrow \infty}\int_{\Omega_k(\alpha)}e^{2\lambda_k}|dz|^2=0\, .
    \end{align}
   
  Consider a typical  bubble  region
   \begin{equation}\label{eq:defBijak}
    	B(i,j,\alpha,k)=B_{\alpha^{-1}r_k^{i,j}}(x_k^{i,j})\setminus\bigcup_{j'\in I^{i,j}}B_{\alpha r_k^{i,j}}(x_k^{i,j'})
    \end{equation}
    from the bubble-neck decomposition  \cite[Proposition III.$1$]{quanta}. We refer to \cite{quanta} for the precise statement and relevant definitions. For our purpose here it is sufficient to recall that
    \begin{itemize}
    \item[{(1)}] $B_{r_k^{i,j}}(x_k^{i,j})$  corresponds to a bubble for $\phi_k$;
    \item[{(2)}] the set of indices $I^{i,j}$ corresponds to the bubbles contained in $B_{r_k^{i,j}}(x_k^{i,j})$; in this sense, the role of \eqref{eq:defBijak} is to isolate a single bubble.
    \item[{(3)}] the total number of bubbles is bounded:  $\displaystyle\sup_{k\in \N} \mathrm{card} \bigcup_{i,j} I^{i,j}<\infty$; 
    \vspace{-0.7em}
    \item[{(4)}] $\displaystyle\lim_{k\to \infty} x^{i,j}_{k}\to a_{i}$, for every $i, j$.
    \item[{(5)}]  $\displaystyle \lim_{k\to \infty} r^{i,j}_{k}\to 0$, for every $i, j$.
    \end{itemize}
    From \cite[(VIII.$10$)]{quanta}, a uniform Harnack inequality holds
    : for all $0<\alpha<1$, there exists $C_{\alpha}>1$ such that 
    \begin{equation}\label{eq:HarnackBubble}
    	\sup_{B(i,j,\alpha,k)}e^{2\lambda_k}\leq C_{\alpha}\,\inf_{B(i,j,\alpha,k)}e^{2\lambda_k}, \quad \text{ for all $k\in\N$ large enough}.
    \end{equation}
      Therefore, the estimate \eqref{harnack1} implies that there exists $C_{\alpha}'$ such that
    \begin{align}\label{harnack0}
    	e^{2\lambda_k(z)}\leq C_{\alpha}'(\alpha^{-1}r_{k}^{i,j})^{2d_k}\leq C_{\alpha}''\left(r_k^{i,j}\right)^{2d_k}\, , \quad\text{ for all $z\in B(i,j,\alpha,k)$}\, .
    \end{align}
      Since $d_k\conv{k\rightarrow }d>-1$, we deduce that 
    \begin{align}\label{no_area2}
    	\int_{B(i,j,\alpha,k)}e^{2\lambda_k}|dz|^2\leq \pi C_{\alpha}''\left(r_k^{i,j}\right)^{2d_k+2}\conv{k\rightarrow \infty}0\,, \quad \text{ for every fixed $\alpha\in (0,1)$}.
    \end{align}
    \begin{equation}\label{eq:NoNeckArea}
    	\lim_{\alpha\rightarrow 0}\limsup_{k\rightarrow \infty}\int_{B(0,\alpha R_k)}e^{2\lambda_k}|dz|^2=0\,.
    \end{equation}
    We also deduce that there exists $A\in \R$ such that, for all $k\in \N$ large enough, it holds:
    \begin{align}\label{new_harnarck}
    	e^{-A}|z|^{d_k}\leq e^{\lambda_k(z)}\leq e^{A}|z|^{d_k}, \quad \text{ for all $z\in \Omega_k(1/2)$}.
    \end{align}

For the next developments, we need to sharpen the above estimates to  an $\mathrm{L}^{2,1}$ bound for the conformal parameter and an $\mathrm{L}^{p}$ quantization result for it.
    
    Let $d\in \R$ and $f:B(0,R)\rightarrow \R\cup\ens{\infty}$ be such that, for all $z\in B(0,R)$, it holds $f(z)=|z|^{d}$. If $d\leq -1$, since $\np{\,\cdot\,}{2,1}{X}\geq 2\sqrt{2}\np{\,\cdot\,
    }{2}{X}$, we have:   
    \begin{align*}
    	\np{|z|^d}{2,1}{B(0,R)}=\infty.
    \end{align*}
    Recalling that for all measured space $(X,\mu)$, for all $1<p<\infty$, it holds (see for example \cite[Appedix 3.7.1]{pointwise})
    \begin{align*}
    	\np{f}{p,1}{X}=\frac{p^2}{p-1}\int_{0}^{\infty}\mu\left(X\cap\ens{x: |f(x)|>t}\right)^{\frac{1}{p}}dt,
    \end{align*}
    we have:
        \begin{align*}
    	\np{|z|^d}{2,1}{B(0,R)}&=4\int_{0}^{1}\mu(B(0,R))^{\frac{1}{2}}dt=4\sqrt{\pi}R\, , \quad \text{for $d=0$},\\
    	\np{|z|^d}{2,1}{B(0,R)}&=4\int_{0}^{{R^d}}\mu(B(0,R))^{\frac{1}{2}}dt+4\int_{R^d}^{\infty}\mu(B(0,t^{\frac{1}{d}}))^{\frac{1}{2}}dt\\
    	&=4\sqrt{\pi}R^{1+d}+4\sqrt{\pi}\frac{-d}{1+d}R^{1+d}
    	=\frac{4\sqrt{\pi}}{1+d}R^{1+d}\, , \quad \text{for $-1<d<0$} , \\
    	\np{|z|^d}{2,1}{B(0,R)}&=4\int_{0}^{R^d}\mu(B(0,R)\setminus\bar{B}(0,t^{\frac{1}{d}}))^{\frac{1}{2}}dt=4\sqrt{\pi}R^{1+d}\int_{0}^1\sqrt{1-s^{\frac{2}{d}}}ds\\
	&\leq 4\sqrt{\pi}R^{1+d}\, , \quad \text{for $d>0$}.
    \end{align*}
    Combining the last estimates with \eqref{new_harnarck}, we deduce that, for all $k\in \N$ large enough, it holds
    \begin{align}\label{neck_lambda21_1}
    	\np{e^{\lambda_k}}{2,1}{\Omega_k(\alpha)}\leq  4\sqrt{\pi} e^{A} \max\ens{1,\frac{1}{1+d_k}}(\alpha R_k)^{1+d_k}\conv{k\rightarrow \infty}4\sqrt{\pi} e^{A}\max\ens{1,\frac{1}{1+d}}(\alpha R)^{1+d}.
    \end{align}
    By \eqref{harnack0}, we have
    \begin{align}\label{neck_lambda21_2}
    	\np{e^{\lambda_k}}{2,1}{B(i,j ,\alpha,k)}\leq 
    	4\sqrt{\pi}\sqrt{C_{\alpha}''}\left(r_k^{i,j}\right)^{1+d_k}\conv{k\rightarrow \infty}0.
    \end{align}
    The combination of  \eqref{neck_lambda21_1} and \eqref{neck_lambda21_2} yields
    \begin{align}\label{neck_lambda21}
    	\lim\limits_{\alpha\rightarrow 0}\limsup_{k\rightarrow \infty}\np{e^{\lambda_k}}{2,1}{B(0,\alpha R_k)}=0.
    \end{align}

   Later on, we will need the following  improvement of the quantization \eqref{neck_lambda21}. Since $d_k\conv{k\rightarrow\infty}d>-1$, we deduce that there exists $0<\epsilon<1$ and $N\in \N$ such that for all $k\geq N$, we have $d_k\geq -1+\epsilon$. In particular, this implies that 
    \begin{align*}
    	\int_{\Omega_k(\alpha)}e^{p\lambda_k(z)}|dz|^2\leq e^{pA}\int_{B(0,\alpha R_k)}|z|^{pd_k}|dz|^2=2\pi e^{pA}\frac{(\alpha R_k)^{2+pd_k}}{2+pd_k}\, ,\quad\text{ for all $p<\dfrac{2}{1-\epsilon}$}.
    \end{align*} 
    Using the Harnack inequality in bubble domains \eqref{eq:HarnackBubble}, we deduce that 
    \begin{align}\label{neck_lambdap}
    	\lim\limits_{\alpha\rightarrow 0}\limsup_{k\rightarrow \infty}\np{e^{\lambda_k}}{p}{B(0,\alpha R_k)}=0\,, \quad \text{for all $p<\frac{2}{1-\epsilon}$} .
    \end{align}
    Indeed, we have $2+pd_k>2+p(-1+\epsilon)>0$ if and only if $p<\dfrac{2}{1-\epsilon}$. In particular, we deduce that 
    \begin{align}\label{lambdak_p}
    	e^{\lambda_k}\quad\text{is bounded in }\;\, \mathrm{L}^p\left(B\left(0,R_k/2\right)\right)\;\,\text{for all}\;\, p<\dfrac{2}{1-\epsilon}. 
    \end{align}

    \subsection{Refined estimates on the approximate conservation laws}\label{SS:estConsLaw}
    
    From now on, $\{\phi_k\}_{k \in \N}$ will be a sequence of smooth Willmore immersions satisfying the assumptions of  the main Theorem \ref{thm:MainThm}.
In the next lemma, we generalise \cite[Lemma A.1]{mondinoriviere}, by relaxing the $\mathrm{L}^{\infty}$ control to an $\mathrm{L}^{2,1}$ control. 
    \begin{lemme}\label{A1}
    	There exists constants $\epsilon_2(m),C_2(m)>0$ with the following property. 
       	For all $j,k\in \ens{1,\cdots, m}$, let $\gamma_{j}^k\in \mathrm{W}^{1,2}\cap C^0(\C)$ be such that $\supp(\gamma_{j}^k)\subset B(0,2)$ and $\np{\gamma_{j}^k}{2,1}{\C}\leq \epsilon_0$. For all $\vec{U}\in \mathrm{L}^{1}_{\mathrm{loc}}(\C)$, define
       	\begin{align*}
       		\left(\D_z \vec{U}\right)_j&=\p{z}\vec{U}_j+\sum_{k=1}^{m}\gamma_{j}^k\vec{U}_k\qquad\text{in}\;\,\mathscr{D}'(\C),\;\, \text{where}\;\, 1\leq j\leq m. 
       	\end{align*}
        Then for all $\vec{Y}\in \left(\dot{\mathrm{H}}^{-1}+\mathrm{L}^1\right)(\C)$
        , there exists a unique $\vec{U}\in \mathrm{L}^{2,\infty}(B(0,1))$ satisfying 
        \begin{align*}
        	\left\{\begin{alignedat}{2}
        		\D_z\vec{U}&=\vec{Y}\qquad&&\text{in}\;\,\mathscr{D}'(B(0,1))\\
        		\Im(\vec{U})&=0\qquad&&\text{on}\;\,\partial B(0,1). 
        	\end{alignedat}
        	 \right.
        \end{align*}
        Furthermore, we have the estimate
        \begin{align*}
        	\np{\vec{U}}{2,\infty}{B(0,1)}
        	\leq C_2
        	\mixdotHL{\vec{Y}}{-1}{1}{B(0,1)}.
        \end{align*}
    \end{lemme}
    \begin{proof}
    	As in \cite{mondinoriviere}, we use a fix point argument.  For all $\vec{U}\in \mathrm{L}^{2,\infty}(\C)$, define
    	\begin{align*}
    		T(\vec{U})=\left(-\frac{1}{\pi\z}\ast \left(\vec{Y}_j-\sum_{k=1}^{m}\gamma_j^k\vec{U}_k\right)\right)_{1\leq j\leq m}.
    	\end{align*}
        By the Young inequality         for weak $\mathrm{L}^p$ spaces (that follows from the classical Young inequality by interpolation) and the $\mathrm{L}^{2,1}/\mathrm{L}^{2,\infty}$ duality, we have for some $\Gamma_0<\infty$
        \begin{align}\label{cont1}
        	\np{T(\vec{U})+\frac{1}{\pi\z}\ast \vec{Y}}{2,\infty}{\C}&\leq \sum_{j,k=1}^{m}\Gamma_0\np{\frac{1}{\pi\z}}{2,\infty}{\C}\np{\gamma_j^k\vec{U}_k}{1}{\C}
        	\leq \frac{\Gamma_0}{\sqrt{\pi}}\sum_{j,k=1}^{m}\np{\gamma_j^k}{2,1}{\C}\np{\vec{U}_k}{2,\infty}{\C}\nonumber\\
        	&\leq \frac{\Gamma_0 m}{\sqrt{\pi}}\epsilon_0\sum_{k=1}^{m}\np{\vec{U}_k}{2,\infty}{\C}\leq \frac{\Gamma_0 m^2}{\sqrt{\pi}}\epsilon_0\np{\vec{U}}{2,\infty}{\C}.
        \end{align}
        Choose 
        \begin{align*}
        	\epsilon_0=\frac{\sqrt{\pi}}{2\Gamma_0 m^2}.
        \end{align*}
        Now, exactly as in \cite{mondinoriviere}, we get the estimate
        \begin{align}\label{cont2}
        	\np{\frac{1}{\pi\z}\ast\vec{Y}}{2,\infty}{\C}\leq \Gamma_1(m)\mixdotHL{\vec{Y}}{-1}{1}{\C},
        \end{align}
        where for all $u:\R^m\rightarrow \R^m$, we have
        \begin{align*}
        	\mixdotHL{u}{-1}{1}{\R^m}=\inf\left\{\hsdot{u_1}{-1}{\R^m}+\np{u_2}{1}{\R^m}: u=u_1+u_2\right\}.
        \end{align*}
        Therefore, \eqref{cont1} and \eqref{cont2} implies that for all $\vec{U}\in \mathrm{L}^{2,\infty}(\C)$, we have
        \begin{align}\label{cont3}
        	\np{T(\vec{U})}{2,\infty}{\C}\leq \frac{1}{2}\np{\vec{U}}{2,\infty}{\C}+\Gamma_1(n)\mixdotHL{\vec{Y}}{-1}{1}{\C}\, .
        \end{align}
        By \eqref{cont1}, for all $\vec{U}_1,\vec{U}_2\in \mathrm{L}^{2,\infty}(\C)$, we have
        \begin{align}\label{cont4}
        	\np{T(\vec{U}_1)-T(\vec{U}_2)}{2,\infty}{\C}\leq \frac{1}{2}\np{\vec{U}_1-\vec{U}_2}{2,\infty}{\C}.
	        \end{align}
        Therefore, \eqref{cont3} and \eqref{cont4}
        prove that $T:\mathrm{L}^{2,\infty}(\C)\rightarrow \mathrm{L}^{2,\infty}(\C)$ is a contraction, and therefore admits a fixed point by Banach contraction mapping Theorem. The estimate follows from \eqref{cont2}. 
        
    \end{proof}
    
  \noindent
    Now, let $f_k:\C^m\rightarrow \C^m$ be the linear map such that for all $X\in C^{\infty}(B(0,R_k),\C^m)$, we have 
        \begin{align*}
    	\D_z\vec{X}=\p{z}\vec{X}+f_k(\vec{X})=\p{z}\vec{X}+\left(\sum_{l=1}^m 
    	\gamma_{j,k}^l\vec{X}_l\right)_{1\leq j\leq m},
    \end{align*}
    where, 
    denoting with $(\phi_k^1, \ldots, \phi_k^m)$ the components of $\phi_k$ in the local coordinates of $M^m$, we set:
    \begin{align*}
  	\gamma_{j,k}^l=\sum_{q=1}^{m}\Gamma_{j,q}^l\p{z}\phi_k^q,
    \end{align*}
   where $\Gamma_{j,q}^l$ are the Christoffel symbols of the ambient Riemannian manifold $(M^m,h)$. We now fix some $\epsilon_3(m)
    \leq \epsilon_2(m)$ to be determined later. By the estimate \eqref{neck_lambda21}, we deduce that there exists $\alpha_0>0$ and $N\in \N$ such that
    \begin{align}\label{smallness_end}
    	\sup_{1\leq j,l\leq m}\np{\gamma_{j,k}^l}{2,1}{B(0,\alpha_0 R_k)}\leq \epsilon_3(m)\leq \epsilon_2(m)\, , \quad \text{ for all $k\geq N$}.   
    \end{align}
   Recalling the ${\mathrm L}^{2,1}/{\mathrm L}^{2,\infty}$ duality, we have that for all $X\in \mathrm{L}^{2,\infty}(B(0,R_k),\C^m)$ and for all $r\leq \alpha_0R_k$ the following estimate holds:
    \begin{align}\label{fk}
    	\np{f_k(\vec{X})}{1}{B(0,r)}\leq \sum_{j,l=1}^{m}\np{\gamma_j^l}{2,1}{B(0,r)}\np{\vec{X}_l}{2,\infty}{B(0,r)}\leq \epsilon_3(m) m^2\np{\vec{X}}{2,\infty}{B(0,r)}. 
    \end{align}
    We also have the pointwise estimate 
    \begin{align}\label{new_pointwise}
    	|f_k(\vec{X})(z)|\leq C_1(h)\,e^{\lambda_k(z)}|\vec{X}(z)|\, ,
    \end{align}
    where $C_1(h)>0$ depends only on $h$. 
     Let ${\vec{Y}_k}:\Omega_k(\alpha_0)\rightarrow \C^m$ be defined by 
    \begin{equation}\label{eq:defYk}
    	\vec{Y}_k=i\left(\D_z\H_k-3\,\D^{\perp}_z\H_k-i\,\star_h\left(\D_z\n_k\wedge \H_k\right)\right)\, .
    \end{equation}
 Notice that ${\vec{Y}_k}$ is smooth, so in particular it is an element of $\mathrm{H}^{-1}+\mathrm{L}^1$. Thus, the extension  ${\vec{Y}_k}$ to the whole $\C$ by setting   ${\vec{Y}_k} \equiv 0$ on $\C\setminus \Omega_k(\alpha_0)$ is an element of $\dot{\mathrm{H}}^{-1}+\mathrm{L}^1 (\C)$. Making use of the $\mathrm{L}^{2,1}$ estimate \eqref{smallness_end}, we are in position to apply  Lemma \ref{A1} and deduce that there exists $\vec{L}_k\in {\mathrm L}^{2,\infty}(B(0,\alpha_0R_k),\C)$ such that
    \begin{align}\label{systemL}
    	\left\{\begin{alignedat}{2}
    		\D_z\vec{L}_k&=\vec{Y}_k\qquad&&\text{in}\;\,B(0,\alpha_0 R_k)\\
    		\Im(\vec{L}_k)&=0\qquad&&\text{on}\;\,\partial B(0,\alpha_0 R_k).
    	\end{alignedat} \right.
    \end{align}
    
\begin{rem}
    The boundary condition for $\Im(\vec{L}_k)$ in \eqref{systemL} is obtained in the exact same way as in Lemma A.$2$ in \cite{mondinoriviere} since $\phi_k$ is smooth (see Lemma \ref{A2} for more details). However, the $\mathrm{L}^{2,\infty}$ estimate obtained here by simply applying Lemma \ref{A1}  will depend on $k$ and for technical reasons we need to obtain a function $\vec{L}_k$ controlled in $\mathrm{L}^{2,\infty}$ independently of $k$. Indeed, without \emph{a priori} estimates, since the boundary condition of $\Re(\vec{L}_k)$ cannot be prescribed, we would not be able to get a $\mathrm{L}^{2,\infty}$ control on $\Re(\vec{L}_k)$ (in \cite{quanta}, at a crucial step, the authors use the fact that the equation in $\D\vec{L}$ holds up to a constant, which allows to assume that some mean of $\vec{L}_k$ vanishes; however, if the Christoffel symbols do not vanish, the equation is not invariant by translation). This is due to the fact that in general, one cannot prescribe the full boundary condition in a $\bar{\partial}$ equation. 
   \end{rem} 
    From the $\epsilon$-regularity Theorem \ref{eps_reg} 
    we know that there exists a constant $C>0$ independent of $k$ and $\alpha_0$ such that
    \begin{align}\label{control_Y}
    	e^{\lambda_k(z)}|\vec{Y}_k(z)|\leq Ce^{-\lambda_k(|z|)}\frac{\delta(|z|)}{|z|}\leq \frac{C_0}{|z|^{2+d_k}}. 
    \end{align}
    \noindent
    For technical reasons, we will have to perform a disjunction of cases depending on the value of $d_k$.  
    The analysis in the first case, $d_k\leq 0$, will take several pages; the other cases will be  discussed after \eqref{final_V_2a}.
     
      \subsection*{Analysis of case 1: $d_k\leq 0$} 
    If  $d_k\leq 0$, then we have \emph{a fortiori}  
    \begin{align}\label{control_Y0}
    	|\vec{Y}_k(z)|\leq \frac{C_0}{|z|^2}. 
    \end{align}
    
    First define $\vec{U}_k:\C\to \C^{m}$ by 
    \begin{align}\label{u1}
    	\vec{U}_k(z)=\frac{1}{\z^2}\left(-\frac{1}{\pi \bar{\zeta}}\ast \left(\bar{\zeta}^2\vec{Y}_k(\zeta)\right)\right)(z)=-\frac{1}{\pi\z^2}\int_{\C}\frac{\bar{\zeta}^2\vec{Y}_k(\zeta)}{\bar{z-\zeta}}|d\zeta|^2=\frac{1}{2\pi i\z^2}\int_{\C}\frac{\bar{\zeta}^2\vec{Y}_k(\zeta)}{\bar{\zeta-z}}d\bar{\zeta}\wedge d\zeta. 
    \end{align}
   
   \begin{lemme}\label{Lem:UKlog}
There exists constants $C_{0},C_{1}>0$ and $a_{k,0}\in \C^{m}$, for $k\in \N$,  with $ \displaystyle\sup_{k\in \N} |a_{k,0}|<\infty$ such that
  \begin{align}\label{u14Lemma}
    	\left|\vec{U}_k(z)-\frac{a_{k,0}}{\z^2}\right|\leq C_1\frac{\log\left(\frac{R}{|z|}\right)}{|z|}+\frac{8C_0}{|z|}, \quad \text{for all } z\in \C\, . 
    \end{align}
   \end{lemme} 
    
    \begin{proof}
  Since $\z^{-2}$ is anti-holomorphic, we have: 
    \begin{align}\label{u2}
    	\p{z}\vec{U}_k(z)=\frac{1}{\z^2}\left(\delta_z\ast \left(\bar{\zeta}^2\vec{Y}_k(\zeta)\right)\right)=\frac{1}{\z^2}\cdot \z^2\vec{Y}_k(z)=\vec{Y}_k(z),
    \end{align}
    where $\delta_z$ is the Dirac mass in $z\in\C$. 
    Write for simplicity $r=2r_k$ and $R=\frac{R_k}{2}$.

    Fix some $z\in B(0,R)$. First, if $0<2|z|\leq r$, we have 
    \begin{align}\label{u3}
    	\vec{U}_k(z)=\frac{1}{2\pi i\z^2}\int_{B(0,R)\setminus \bar{B}(0,2|z|)}\frac{\bar{\zeta}^2\vec{Y}_k(\zeta)}{\bar{\zeta-z}}d\bar{\zeta}\wedge d\zeta
    \end{align}
    and 
    \begin{align}\label{u4}
    	&\frac{1}{2\pi i}\int_{B(0,R)\setminus \bar{B}(0,2|z|)}\frac{\bar{\zeta}^2\vec{Y}_k(\zeta)}{\bar{\zeta-z}}d\bar{\zeta}\wedge d\zeta=\sum_{l=0}^{\infty}\left(\frac{1}{\pi}\int_{B(0,R)\setminus \bar{B}(0,2|z|)}\frac{\bar{\zeta}^2\vec{Y}_k(\zeta)}{\zeta^{l+1}}|d\zeta|^2\right)z^l=\sum_{l=0}^{\infty}a_{k,l}z^l\, .
       \end{align}
    For $l=0$, we have 
    \begin{align}\label{u5}
    	\left|\int_{B(0,R)\setminus\bar{B(0,2|z|)}}\frac{\bar{\zeta}^2\vec{Y}_k(\zeta)}{\zeta^{l+1}}\frac{|d\zeta|^2}{|\zeta|}\right|\leq C_0\int_{B(0,R)}\frac{|d\zeta|^2}{|\zeta|}=2\pi C_0R,
    \end{align}
    and for $l=1$, we have 
    \begin{align}\label{u6}
    	\left|\int_{B(0,R)\setminus\bar{B}(0,2|z|)}\frac{\bar{\zeta^2\vec{Y}_k(\zeta)}}{\bar{\zeta-z}}|d\zeta|^2\right|\leq 2\pi \int_{2|z|}^{R}\frac{dt}{t}=2\pi\log\left(\frac{R}{2|z|}\right)\, .
    \end{align}
    Therefore, by \eqref{u3}, \eqref{u4}, \eqref{u5} and \eqref{u6}, we deduce that there exists a universal constant $C_1$ such that 
    \begin{align}\label{u7}
    	\left|\z^2\vec{U}_k(z)-a_{k,0}-\sum_{l=2}^{\infty}a_{k,l}z^l\right|\leq C_1|z|\log\left(\frac{R}{|z|}\right),
    \end{align}
    and $a_{k,0}$ is uniformly bounded for $k\geq N$. 

    Furthermore, we have for all $l\geq 2$
    \begin{align}\label{u8}
    	\left|\int_{B(0,R)\setminus \bar{B}(0,2|z|)}\frac{\bar{\zeta}^2\vec{Y}_k(\zeta)}{\zeta^{l+1}}|d\zeta|^2\right|\leq C_0\int_{B(0,R)\setminus \bar{B}(0,2|z|)}\frac{|d\zeta|^2}{|\zeta|^{l-1}}\leq 2\pi C_0\int_{2|z|}^{R}\frac{dt}{t^{l}}=\frac{2\pi C_0}{l-1}\left(\frac{1}{(2|z|)^{l-1}}-\frac{1}{R^{l-1}}\right),
    \end{align}
    which  implies that 
    \begin{align}\label{u9}
    	\left|\sum_{l=2}^{\infty}a_{k,l}z^l\right|\leq 2C_0\sum_{l=2}^{\infty}\frac{1}{(2|z|)^{l-1}}\times |z|^l=2C_0|z|. 
    \end{align}   
    Therefore, by \eqref{u8} and \eqref{u9} we get
    \begin{align}\label{u10}
    	\left|\vec{U}_k(z)-\frac{a_{k,0}}{\z^2}\right|\leq C_1\frac{\log\left(\frac{R}{|z|}\right)}{|z|}+\frac{2C_0}{|z|}. 
    \end{align}

    Now, assume that $2|z|>r$. Then we have 
    \begin{align}\label{u11}
    	\z^2\vec{U}_k(z)=\frac{1}{2\pi i}\int_{B(0,2|z|)\setminus\bar{B}(0,r)}\frac{\bar{\zeta}^2\vec{Y}_k(\zeta)}{\bar{\zeta-z}}d\bar{\zeta}\wedge d\zeta+\frac{1}{2\pi i}\int_{B(0,R)\setminus\bar{B}(0,2|z|)}\frac{\bar{\zeta}^2\vec{Y}_k(\zeta)}{\bar{\zeta-z}}d\bar{\zeta}\wedge d\zeta=\vec{u}_1(z)+\vec{u}_2(z). 
    \end{align}
    We first easily estimate 
    \begin{align}\label{u12}
    	\left|\frac{1}{2\pi i}\int_{B(0,2|z|)\setminus\bar{B(0,r)}}\frac{\bar{\zeta^2}\vec{Y}_k(\zeta)}{\bar{\zeta-z}}d\bar{\zeta}\wedge d\zeta\right|\leq \frac{C_0}{\pi}\int_{B(0,2|z|)}\frac{|d\zeta|^2}{|\zeta-z|}\leq \frac{C_0}{\pi}\int_{B(z,3|z|)}\frac{|d\zeta|^2}{|\zeta-z|}=6 C_0R|z|
    \end{align}
    and the previous argument shows that (notice that this is the same constant $a_{k,0}$)
    \begin{align}\label{u13}
    	\left|\vec{u}_2(z)-a_{k,0}\right|\leq C_1|z|{\log\left(\frac{R}{|z|}\right)}+2C_0|z|.
    \end{align}
    Finally, by \eqref{u11}, \eqref{u12} and \eqref{u13}, we deduce that 
    \begin{align}\label{u14}
    	\left|\vec{U}_k(z)-\frac{a_{k,0}}{\z^2}\right|\leq C_1\frac{\log\left(\frac{R}{|z|}\right)}{|z|}+\frac{8C_0}{|z|},   
    \end{align}
    which concludes the proof of the lemma.
    \end{proof}
  
Next, define $\vec{V}_{k}:\C\to \C^{m}$ by:
    \begin{align}\label{def_Vk}
    	\vec{V}_k(z)=\vec{U}_k(z)-\frac{a_{k,0}}{\z^2}\, , \quad \text{for all } z\in \C.
    \end{align}
    
    \begin{lemme}\label{lem:EstVk}
    Let $\vec{V}_{k}:\C\to \C^{m}$ be defined in  \eqref{def_Vk}.  Then there exists $C_{2}>0$ such that 
     \begin{align}\label{vk}
    	\left|\vec{V}_k(z)\right|\leq \frac{C_2}{|z|}\left({\log\left(\frac{R}{|z|}\right)}+1\right).
    \end{align}
    Moreover,  $\D\Im(\vec{V}_k)\in \mathrm{L}^{2,\infty}(\C)$ which shows in particular that $\Im(\vec{V}_k)\in \mathrm{W}^{1,(2,\infty)}(B(0,R_k))$. Furthermore, by the Sobolev embedding, it holds: 
    \begin{align*}
    	\Im(\vec{V}_k)\in \bigcap_{p<\infty}\mathrm{L}^p(B(0,R_k)). 
    \end{align*}
    \end{lemme}

    \begin{proof}
From Lemma \ref{Lem:UKlog},   we get that $\p{z}\vec{V}_k=\vec{Y}_k$ on $\C$ and that there exists $C_2>0$ such that:
    \begin{align*}
    	\left|\vec{V}_k(z)\right|\leq \frac{C_2}{|z|}\left({\log\left(\frac{R}{|z|}\right)}+1\right), \quad \text{for all $z\in \C$}.
    \end{align*}
    Furthermore, we have 
    \begin{align*}
    	\Delta \Im\left(\vec{V}_k\right)=4\,\Im\left(\p{\z}\vec{Y}_k\right)=4\,\Im\left(\D_{\z}\vec{Y}_k(z)\right)-4\,\Im\left(\bar{f_k}(\vec{Y}_k)\right)\quad \text{ on $\Omega_k(1/2)$ },
    \end{align*}
    while $\Im(\vec{V}_k)$ is harmonic on $B(0,\alpha_0^{-1}r_k)$. 
    Recall that by \cite[Lemma 3.2 and Theorem 3.1]{mondinoriviere}, the following identities hold:
    \begin{align*}
    	&\vec{Y}_k=i\left(\D_z\H_k-3\,\D_z^{\perp}\H_k-i\star_h(\D_z\n_k\wedge \H_k)\right)=-2i\left(\D_z^{\perp}\H_k+\s{\H_k}{\H_{k,0}}\p{\z}\phi_{k}\right)\\
    	&4e^{-2\lambda}\Re\left(\D_{\z}\left(\D_z^{\perp}\H_k+\s{\H_k}{\H_{0,k}}\p{\z}\phi_k\right)\right)=\Delta_g^{\perp}\H_k-2|\H_k|^2\H+\mathscr{A}(\H_k)+8\,\Re\left(\s{R(\e_{\z},\e_z)\e_z}{\H_k}\e_{\z}\right). 
    \end{align*}
    Therefore:
    \begin{align*}
    	\Im(\D_{\z}\vec{Y}_k)&=\Im\left(-2i\D_{\z}\left(\D_z^{\perp}\H_k+\s{\H_k}{\H_{k,0}}\p{\z}\phi_k\right)\right)=-2\,\Re\left(\D_{\z}\left(\D_z^{\perp}\H_k+\s{\H_k}{\H_{k,0}}\p{\z}\phi_k\right)\right)\\
    	&=-\frac{1}{2}e^{2\lambda_k}\left(\Delta_{g_k}^{\perp}\H_k-2|\H_k|^2\H+\mathscr{A}(\H_k)\right)-4e^{2\lambda_k}\,\Re\left(\s{R(\e_{\z},\e_z)\e_z}{\H_k}\e_{\z}\right).
    \end{align*}
    Using that $\phi_{k}$ is Willmore 
    and using \eqref{el1}, we deduce that
    \begin{align}\label{identity_y0}
    	\Im(\D_{\z}\vec{Y}_k)=\frac{1}{2}e^{2\lambda_k}\left(\mathscr{R}_1^{\perp}(\H_k)-2\,\tilde{K}_h\,\H_k+2\,\mathscr{R}_2(d\phi_k)+(DR)(d\phi_k)-8\,\Re\left(\s{R(\e_{\z},\e_z)\e_z}{\H_k}\e_{\z}\right)\right).
    \end{align}
    Therefore, by \eqref{lambdak_p} and Hölder's inequality, we get:
    \begin{align}\label{elliptic_V1}
    	|\Im(\D_{\z}\vec{Y}_k)|\leq Ce^{2\lambda_k(z)}\left(1+|\H_k|\right)\in \bigcap_{p<\frac{1}{1-\epsilon}}\mathrm{L}^{p}(B(0,R_k)).
	    \end{align} 
 However, from \eqref{control_Y0} we have 
    \begin{align*}
    	\left|\bar{f_k}(\vec{Y}_k)\right|\leq Ce^{\lambda_k}|\vec{Y}_k|\leq C'\frac{\delta(|z|)}{|z|}\leq \frac{C''}{|z|^2}, 
    \end{align*}
    which does not suffice as one cannot obtain elliptic estimates from a $\mathrm{L}^{1,\infty
    }$ bound on the Laplacian.
   In order to circumvent this problem, we will argue differently. Recall that for any vector-field $\vec{X}$, we defined
    \begin{align*}
    \D_z\vec{X}=\p{z}\vec{X}+f_k(\vec{X})=\p{z}\vec{X}+\left(\sum_{l=1}^m\gamma_{j,k}^l\vec{X}_l\right)_{1\leq j\leq m},
    \end{align*}
     where 
    \begin{align*}
        \gamma_{j,k}^l=\sum_{q=1}^{m}\Gamma_{j,q}^l\p{z}\phi_{k}^{q}.
    \end{align*}
    Recalling that $\vec{Y}_k=i\left(\D_z\H_k-3\,\D^{\perp}_z\H_k-i\,\star_h\left(\D_z\n_k\wedge \H_k\right)\right)$, we get
    \begin{align*}
    	\vec{Y}_k&=i\left(\p{z}\H_k-3\pi_{\n_k}(\p{z}\H_k)-2f_k(\H_k)-i\star_h (\p{z}\n_k\wedge \H_k)-i\star_h\left(G_k(\n_k)\wedge\H_k\right)\right)\\
    	&=i\left(-2\,\p{z}\H_k+3\left(\p{z}\pi_{\n_k}\right)\H_k-2f_k(\H_k)-i\star_h (\p{z}\n_k\wedge \H_k)-i\star_h\left(G_k(\n_k)\wedge\H_k\right)\right).
    \end{align*}
   Since $e^{\lambda_k}$, and $e^{\lambda_k}\H_k$ and $|\D\n_k|$ are bounded in $\mathrm{L}^2(B(0,R_k))$ by hypothesis, we deduce that there exists a constant $C>0$ such that:
    \begin{align}\label{elliptic_V2}
    	\left|\bar{f_k}(\vec{Y}_k+2i\,\p{z}\H_k)\right|\leq C\left(|\D\n_k|e^{\lambda_k}|\H_k|+e^{\lambda_k}|\H_k|+|\D\n_k|e^{\lambda_k}|\H_k|+e^{2\lambda_k}|\H_k|\right)\in \mathrm{L}^1(B(0,R_k)).
    \end{align}
   It follows that
    \begin{align*}
    	&\Im\left(\bar{f_k}(-2i\,\p{z}\H_k)\right)_{j}=\Im\left(-2i\sum_{q=1}^{m}\Gamma_{j,q}^l\p{\z}\phi_{k}^{q}\p{z}\H_{k,l}\right)=-2\,\sum_{q=1}^m\Gamma_{j,q}^l\Re\left(\p{\z}\phi_{k}^{q}\p{z}\H_{k,l}\right)\\
    	&=-\frac{1}{2}\sum_{q=1}^{m}\Gamma_{j,q}^l\s{\D\phi_{k}^{q}}{\D\H_{k,l}}\\
    	&=-\frac{1}{2}\dive\left(\sum_{q=1}^{m}\Gamma_{j,q}^l\D\phi_{k}^{q}\,\H_{k,l}\right)+\frac{1}{2}\sum_{q=1}^{m} \s{\D\Gamma_{j,q}^l}{\D\phi_{k}^{q}}\H_{k,l}+\frac{1}{2}\sum_{q=1}^m\Gamma_{j,q}^l\Delta\phi_{k}^{q}\,\H_{k,l}\\
    	&=-\frac{1}{2}\dive\left(\sum_{q=1}^{m}\Gamma_{j,q}^l\D\phi_{k}^{q}\,\H_{k,l}\right)+\frac{1}{2}\sum_{q=1}^{m} \s{\D\Gamma_{j,q}^l}{\D\phi_{k}^{q}}\H_{k,l}+\sum_{q=1}^{m}\Gamma_{j,q}^le^{2\lambda_k}\H_{k,q}\H_{k,l}. 
    \end{align*}
    Since $e^{\lambda_k}\H_k\in \mathrm{L}^2(B(0,R_k))$, we deduce that 
    \begin{align}\label{elliptic_V3}
    	\Im\left(\bar{f_k}(-2i\,\p{z}\H_k)\right)\in \mathrm{H}^{-1}+\mathrm{L}^1(B(0,R_k)).
    \end{align}
    Therefore, by \eqref{elliptic_V1}, \eqref{elliptic_V2} and \eqref{elliptic_V3}, we finally deduce that  
    \begin{align*}
    	\Delta \Im(\vec{V_k})\in \mathrm{H}^{-1}+\mathrm{L}^1(\C).
    \end{align*}
    Standard elliptic estimates imply that $\Im(\vec{V}_k)\in \mathrm{W}^{1,(2,\infty)}(\C)$. Indeed,  by the preceding estimates, we have a decomposition 
    \begin{align*}
    	\Delta \Im(\vec{V}_k)=\dive(\vec{X}_{k})+\vec{Z}_{k},
    \end{align*}
    where $\vec{X}_k\in \mathrm{L}^2(\C)$ and $\vec{Z}_k\in \mathrm{L}^1(\C)$. Therefore, making the decomposition $\Im(\vec{V}_k)=\vec{u}_k+\vec{v}_k$, where 
    \begin{align*}
    	\Delta\vec{u}_k=\dive(\vec{X}_k) \quad \text{and}\quad   	\Delta \vec{v}_k=\vec{Z}_k, 
    \end{align*}
    we obtain: 
    \begin{align*}
    	\vec{v}_k(z)=\frac{1}{2\pi}\int_{\C}\Delta \vec{v}_k(\zeta)\log|z-\zeta|\,|d\zeta|^2=\frac{1}{2\pi}\int_{\C}\vec{Z}_k(\zeta)\log|z-\zeta|\,|d\zeta|^2.
    \end{align*}
    Clearly, it holds
    \begin{align*}
    	\p{z}\vec{v}_k(z)=\frac{1}{4\pi}\int_{\C}\frac{\vec{Z}_k(\zeta)}{z-\zeta}|d\zeta|^2,
    \end{align*}
    which, by  Young inequality, yields:
    \begin{align*}
    	\np{\D\vec{v}_k}{2,\infty}{\C}\leq C_Y\np{\frac{1}{z}}{2,\infty}{\C}\np{\vec{Z}_k}{1}{\C}\leq C.
    \end{align*}
    Likewise, we have in the distributional sense 
    \begin{align*}
    	\p{z}\vec{u}_k(z)=\frac{1}{4\pi}\int_{\C}\frac{\dive(\vec{X}_k)(\zeta)}{z-\zeta}|\zeta|^2.
    \end{align*}
    Now, since the Fourier transform of $\dfrac{1}{z}$ is $\dfrac{c_0}{\bar{\xi}}$ for some constant $c_0\in \C$, we have 
    \begin{align*}
    	\mathscr{F}(\p{x_j}\vec{X})(\xi)=-i\xi_j\mathscr{F}(\vec{X})(\xi)\, , \quad \text{ for all $j\in\ens{1,2}$}.
    \end{align*}
    Therefore, by the Parseval identity, we have 
    \begin{align*}
    	\int_{\C}|\D\vec{u}_k|^2dx&=4\int_{\C}\left|\p{z}\vec{u}_k\right|^2|dz|^2=\frac{1}{4\pi^2}\int_{\C}\left|\frac{1}{z}\ast \dive(\vec{X}_k)\right|^2|dz|^2=\frac{c_0^2}{16\pi^4}\int_{\C}\frac{1}{|\xi|^2}\left|-i\xi_1\vec{X}_{k,1}-i\xi_2\vec{X}_{k,2}\right|^2|d\xi|^2\\
    	&\leq \frac{c_0^2}{8\pi^4}\int_{\C}|\vec{X}_k|^2dx.
    \end{align*}
    Finally, we deduce that $\D\Im(\vec{V}_k)\in \mathrm{L}^{2,\infty}(\C)$ which shows in particular that $\Im(\vec{V}_k)\in \mathrm{W}^{1,(2,\infty)}(B(0,R_k))$. Furthermore, by the Sobolev embedding, we have 
    \begin{align*}
    	\Im(\vec{V}_k)\in \bigcap_{p<\infty}\mathrm{L}^p(B(0,R_k)),
    \end{align*}
    which concludes the proof of the lemma. 
    \end{proof}
    
    The next step will be to remove the $\log|z|$ term in the estimate \eqref{vk}.

   \subsection{Removal of the logarithmic singularity  in the estimate of  $\vec{V}_k$}\label{SSS:logz}
    Removing the  $\log|z|$ term in the inequality \eqref{vk} is the most technical part of the proof of Theorem \ref{L21_necks}. To this aim, we will introduce and use some (apparently new) Lorentz-type functional space. Let us stress out that this difficulty is due to the curved ambient space and therefore was not present in the proof of the energy quantization for Willmore surfaces in Euclidean spaces \cite{quanta}.

   Let us  first make the decomposition $\Im(\vec{V}_k)=\vec{u}_k+\vec{v}_k+\vec{w}_k$, where 
    \begin{align*}
    	\left\{\begin{alignedat}{2}
        	\Delta \vec{u}_k&=4\,\Im\left(\D_{\z}\vec{Y}_k\right)\qquad&&\text{in}\;\, B(0,{R_k}/{2})\\
        	\vec{u}_k&=0\qquad&&\text{on}\;\,\partial B(0,{R_k}/{2})\, ,
    	\end{alignedat}\right.
    \end{align*}
    \begin{align*}
    	\left\{\begin{alignedat}{2}
    		\Delta \vec{v}_k&=-4\,\Im\left(\bar{f_k}(\vec{Y}_k)\right)\qquad&&\text{in}\;\, B(0,{R_k}/{2})\\
    		\vec{v}_k&=0\qquad&&\text{on}\;\,\partial B(0,{R_k}/{2})\, ,
    	\end{alignedat}\right.
    \end{align*}
    and 
    \begin{align*}
    	\left\{\begin{alignedat}{2}
    		\Delta \vec{w}_k&=0\qquad&&\text{in}\;\, B(0,{R_k}/{2})\\
    		\vec{w}_k&=\Im(\vec{V}_k)\qquad&&\text{on}\;\,\partial B(0,{R_k}/{2}).
    	\end{alignedat}\right.
    \end{align*}
       By the bound \eqref{elliptic_V1}, we have $\Delta \vec{u}_k\in \mathrm{L}^p(\C)$ for all $p<\dfrac{1}{1-\epsilon}$ which implies by Calder\'{o}n-Zygmund estimates that $\vec{u}_k\in \mathrm{W}^{2,p}(B(0,R_k/2))$ for all $p<\dfrac{1}{1-\epsilon}$. By standard elliptic regularity (\cite{helein}, Chapter $3$, $3.3$), we also get $\vec{w}_k\in \mathrm{W}^{1,(2,\infty)}(B(0,R_k/2))$.  
   Regarding $\vec{v}_{k}$, it holds
    \begin{align*}
    	\p{\z}\left(z^2\p{z}\vec{v}_k-\frac{1}{\pi}\int_{\C}\frac{\zeta^2\Im(\bar{f_k}(\vec{Y}_k)(\zeta))}{\zeta-z}|d\zeta|^2\right)=0,
    \end{align*}
    and $\vec{v}_k\in \mathrm{L}^{2,\infty}(\C)$, which implies that there exists a holomorphic function $h_{k}$ such that 
    \begin{align*}
    	\p{z}\vec{v}_k=\frac{h_{k}(z)}{z^2}+\frac{1}{\pi z^2}\int_{\C}\frac{\zeta^2\Im(\bar{f_k}(\vec{Y}_k)(\zeta)}{\zeta-z}|d\zeta|^2.
    \end{align*}
    By the estimates in the proof of Lemmas \ref{Lem:UKlog} and \ref{lem:EstVk} 
    , we deduce that there exists a constant $C>0$ such that 
    \begin{align}\label{est1a}
    	\left|\frac{1}{\pi z^2}\int_{\C}\frac{\zeta^2\Im(\bar{f_k}(\vec{Y}_k)(\zeta)}{\zeta-z}|d\zeta|^2-\frac{a_{0,k}}{z^2}\right|\leq C\frac{\log\left(\frac{R_k}{|z|}\right)}{|z|}.
    \end{align}
    Taking the expansion $h_k(z)=b_{0,k}+O(|z|)$, from $\D\vec{v}_k\in \mathrm{L}^{2,\infty}(\C)$ we deduce that 
    \begin{align}
    	b_{0,k}=-a_{0,k},
    \end{align}
    and that 
    \begin{equation}\label{eq:deffk}
   \psi_{k}(z)=\dfrac{h_{k}(z)-b_{0,k}}{z^2}\in \mathrm{L}^{p}_{\mathrm{loc}}(\C)\,, \quad \text{for all $p<2$}.
    \end{equation}    
       More precisely, the estimate \eqref{est1a} shows that there exists a function $\psi_{1,k}$ and a constant $C_{k}>0$ such that 
    \begin{align}\label{est1}
    	|\psi_{k}(z)-\psi_{1,k}(z)| \in \mathrm{L}^{2,\infty}\left(B\left(0,\frac{R_k}{2}\right)\right)
    \end{align}
    and 
    \begin{align}\label{est2}
    	|\psi_{1,k}(z)|\leq \frac{C_{k}}{|z|}\log\left(\frac{R_k}{|z|}\right).
    \end{align}
    Therefore, we deduce that $\psi_{k}$ admits the Laurent expansion
    \begin{align}\label{est3}
    \psi_{k}(z)=\sum_{n\geq -1}^{\infty}a_{k,n}z^n.
    \end{align}
    Next, we will show that 
    \begin{equation}\label{eq:psikL21}
    \psi_{k} \in \mathrm{L}^{2,1}(B(0,\alpha R_k))\, .
    \end{equation}
   The proof is quite involved and will make use  of some (apparently new) Lorentz-type function spaces.
    For a more systematic discussion of generalised Lorentz spaces, the reader is refereed to the Appendix \ref{appendix}.
    \\ First, write for simplicity $R=\dfrac{R_k}{2}$, and for all $\alpha>0$, let $\varphi_{\alpha}:(0,R)\rightarrow \R$ be defined by
    \begin{align*}
    	\varphi_{\alpha}(t)=\frac{t}{\log^{\alpha}\left(\frac{R}{t}\right)}.
    \end{align*}
    Let $W:\R_+\rightarrow\R_+$ be the Lambert function,  which satisfies for all $x\geq 0$
    \begin{align}\label{lambert}
    	W(x)e^{W(x)}=x 
    \end{align}
    One easily checks that
      \begin{align*}
    	\psi_\alpha(t)=Re^{-\alpha W\left(\frac{1}{\alpha}\left(\frac{R}{t}\right)^{\frac{1}{\alpha}}\right)} 
    \end{align*}
    is the inverse of $\varphi_{\alpha}$.
        Notice that for $\alpha=1$, the identiy \eqref{lambert} gives
    \begin{align*}
    	\psi_1(t)=tW\left(\frac{R}{t}\right). 
    \end{align*}
    Explicit computations give
    \begin{align*}
    	\varphi_\alpha'(t)=\frac{1}{\log^{\alpha}\left(\frac{R}{t}\right)}+\frac{\alpha}{\log^{\alpha+1}\left(\frac{R}{t}\right)}>0\,, \quad 
    	\varphi_\alpha''(t)=\frac{\alpha}{t\log^{\alpha+1}\left(\frac{R}{t}\right)}+\frac{\alpha(\alpha+1)}{t\log^{\alpha+2}(t)}>0\, ,
    \end{align*}  
    which show that $\varphi_\alpha$ is convex and strictly increasing. We deduce that  for all $t>0$, it holds 
    \begin{align*}
    	&\leb^2\left(B(0,R)\cap\ens{x:\frac{\log^{\alpha}\left(\frac{R}{|x|}\right)}{|x|}>t}\right)=\leb^2\left(B(0,R)\cap\ens{x:\varphi_{\alpha}(|x|)<\dfrac{1}{t}}\right)\\
    	&=\leb^2\left(B(0,R)\cap\ens{x:|x|<\psi_{\alpha}\left(\frac{1}{t}\right)}\right)=\pi \min\ens{R^2,\psi_{\alpha}^2\left(\frac{1}{t}\right)}=\pi \min\ens{R^2,R^2e^{-2\alpha W\left(\alpha^{-1}(Rt)^{\frac{1}{\alpha}}\right)}}\\
    	&=\pi R^2 e^{-2\alpha W\left(\alpha^{-1}(Rt)^{\frac{1}{\alpha}}\right)}\, . 
    \end{align*}
    The following asymptotic expansions hold: 
    \begin{align*}
    	\left\{\begin{alignedat}{2}
    		W(t)&=t+O(t^2)\qquad&&\text{when}\;\, t\rightarrow 0\\
    		W(t)&=\log(t)-\log\log(t)+o(1)\qquad&&\text{when}\;\,t\rightarrow \infty\, . 
    	\end{alignedat}\right.
    \end{align*}
    Therefore, when $t\rightarrow 0$, we deduce that
    \begin{align*}
    	e^{-2\alpha W\left(\alpha^{-1}(Rt)^{\frac{1}{\alpha}}\right)}=e^{-2(Rt)^{\frac{1}{\alpha}}+O(t^{\frac{2}{\alpha}})}\conv{t\rightarrow 0}1\, ,
    \end{align*}
    so that
    \begin{align*}
    	\lim\limits_{t\rightarrow 0+}t^2\leb^2\left(B(0,R)\cap\ens{x:\frac{\log^{\alpha}|x|}{|x|}>t}\right)=0\, . 
    \end{align*}
    However, when $t\rightarrow \infty$, we have
    \begin{align*}
    	&e^{-2\alpha\,W\left(\alpha^{-1}(Rt)^{\frac{1}{\alpha}}\right)}=e^{-2\alpha\left(\log\left(\alpha^{-1}(Rt)^{\frac{1}{\alpha}}\right)-\log\log\left(\alpha^{-1}(Rt)^{\frac{1}{\alpha}}\right)+o(1)\right)}\\
    	&=\frac{e^{2(\alpha\log(\alpha)-\log(R))+o(1)}}{t^2}\log^{2\alpha}\left(\alpha^{-1}\left(Rt\right)^{\frac{1}{\alpha}}\right)=\frac{e^{A+o(1)}}{t^2}\left(\log^{2\alpha}(t)+O(1)\right).
    \end{align*}
    for some $A\in \R$. This implies that
    \begin{align*}
    	t^2\leb^2\left(B(0,R)\cap\ens{x:\frac{\log^{\alpha}|x|}{|x|}>t}\right)=\frac{e^{A+o(1)}}{t^2}\left(\log^{2\alpha}(t)+O(1)\right)\conv{t\rightarrow \infty}\infty. 
    \end{align*}
    Note that
    \begin{align*}
    	\lim\limits_{t\rightarrow \infty}\left(\frac{t}{\log^{\alpha}(t)}\right)^{2}\leb^2\left(\R^2\cap\ens{x:\frac{\log^{\alpha}|x|}{|x|}>t}\right)\leq C<\infty.
    \end{align*}
    This suggests that $u=u_{\alpha}=1/\varphi_{\alpha}(|\,\cdot\,|)$ belongs to a Lorentz space. To determine its weight, we first compute the function $u_{\ast}:\R_+\rightarrow \R_+$ defined by 
    \begin{align*}
    	u_{\ast}(t)=\inf\ens{s>0:\leb^2\left(B(0,R)\cap\ens{x:|u(x)|>s}\right)\leq t}.
    \end{align*}
    For $t\leq \pi R^2$, it holds
    \begin{align*}
    	\leb^2\left(B(0,R)\cap\ens{x:|u(x)|>s}\right)=\pi \psi^2\left(\frac{1}{s}\right)\leq t \Longleftrightarrow \frac{1}{s}\leq \varphi\left(\sqrt{\frac{t}{\pi}}\right)\Longleftrightarrow s\geq u\left(\sqrt{\frac{t}{\pi}}\right)
	    \end{align*}
    while  $u_{\ast}(t)= 0$,  for $t\geq \pi R^2$.
    Therefore, 
    \begin{align*}
    	u_{\ast}(t)=u\left(\sqrt{\frac{t}{\pi}}\right)=\sqrt{\frac{\pi}{t}}{\log^{\alpha}\left(R\sqrt{\frac{\pi}{t}}\right)},\, \quad\text{for $t\leq \pi R^2$}
    \end{align*}
    and
    \begin{align*}
    	u_{\ast}(t)=\left\{\begin{alignedat}{2}
    		&\sqrt{\frac{\pi}{t}}\log^{\alpha}\left(R\sqrt{\frac{\pi}{t}}\right)\qquad&&\text{for}\;\, t<\pi R^2\\
    		&0\qquad&&\text{for}\;\, t\geq \pi R^2
    	\end{alignedat}\right.
    \end{align*}
    so that 
    \begin{align*}
    	\int_{0}^{t}u_{\ast}(s)ds&=\sqrt{\pi}\left[2\sqrt{s}\log^{\alpha}\left(R\sqrt{\frac{\pi}{s}}\right)\right]_0^t+{2\pi \alpha R}\int_{0}^t\frac{\log^{\alpha-1}\left(R\sqrt{\frac{\pi}{s}}\right)}{\sqrt{s}}ds\\
    	&=2\sqrt{\pi t}\log^{\alpha}\left(R\sqrt{\frac{\pi}{t}}\right)+2\pi\alpha R\int_{0}^{t}\log^{\alpha-1}\left(R\sqrt{\frac{\pi}{s}}\right)\frac{ds}{\sqrt{s}},
    \end{align*}
    which, for $\alpha=1$, shows that 
    \begin{align*}
    	\int_{0}^tu_{\ast}(s)ds=2\sqrt{\pi t}\log\left(R\sqrt{\frac{\pi}{t}}\right)+4\pi R\sqrt{t}.
    \end{align*}
    First assume that $\alpha\in\N$ and notice that,  for $s=\pi R^2v$,  it holds:
    \begin{align*}
    	\int_{0}^tu_{\ast}(s)ds=\pi R^2\int_{0}^{\frac{t}{\pi R^2}}\frac{1}{R\sqrt{v}}\log^{\alpha}\left(\frac{1}{v}\right)du=\pi R\int_{0}^{\frac{t}{\pi R^2}}\frac{1}{\sqrt{v}}\log^{\alpha}\left(\frac{1}{v}\right)dv.
    \end{align*}
   We introduce the integrals $I(\alpha,r)$ for $r\leq 1$ defined by 
    \begin{align*}
    	I(\alpha,r)=\int_{0}^r\log^{\alpha}\left(\frac{1}{t}\right)\frac{dt}{\sqrt{t}}.
    \end{align*}
   It holds
    \begin{align*}
    	I(\alpha,r)&=2\sqrt{r}\log^{\alpha}\left(\frac{1}{r}\right)+2\alpha\int_{0}^r\log^{\alpha}\left(\frac{1}{t}\right)\frac{dt}{\sqrt{t}}=2\sqrt{r}\log^{\alpha}\left(\frac{1}{r}\right)+2\alpha\, I(\alpha-1,r)\\
    	&=2\sqrt{r}\log^{\alpha} \left(\frac{1}{r}\right)+4\alpha\sqrt{r}\log^{\alpha-1}\left(\frac{1}{r}\right)+4\alpha(\alpha-1)I(\alpha-2,r).
    \end{align*}
    Therefore, we deduce that provided that $\alpha=n+\beta$ with $n\in\N$ and $0\leq \beta<1$, we have
    \begin{align*}
    	I(\alpha,r)=\sum_{k=0}^{n-1}2^{k+1}\prod_{l=0}^{k-1}(n-l+\beta)\sqrt{r}\log^{n-l+\beta}\left(\frac{1}{r}\right)+2^{n}\prod_{k=0}^{n}(n-k+\beta)\int_{0}^r\log^{\beta}\left(\frac{1}{t}\right)\frac{dt}{\sqrt{t}}.
    \end{align*}
    Now, we will estimate the last integral by the method of stationary phase. A change of variable $x=\log^{\beta}\left(\frac{1}{t}\right)$ shows that, for  $p=\frac{1}{\beta}>1$, it holds
    \begin{align*}
    	\int_{0}^r\log^{\beta}\left(\frac{1}{t}\right)\frac{dt}{\sqrt{t}}=\int_{\log^{\beta}\left(\frac{1}{r}\right)}^{\infty}px^{p}e^{-\frac{x^p}{2}}dx&=\left[-2x\,e^{-\frac{x^p}{2}}\right]_{\log^{\beta}\left(\frac{1}{r}\right)}^{\infty}+2\int_{\log^{\beta}\left(\frac{1}{r}\right)}^{\infty}e^{-\frac{x^p}{2}}dx\\
    	&=2\sqrt{r}\log^{\beta}\left(\frac{1}{r}\right)+2\log^{\beta}\left(\frac{1}{r}\right)\int_{1}^{\infty}e^{-\log\left(\frac{1}{r}\right)\frac{x^p}{2}}dx,
    \end{align*}
    where we performed a linear change of variable in the last integral. Now, if $F:(0,\infty)\rightarrow \R$ is defined by 
    \begin{align}
    	F(t)=\int_{1}^{\infty}e^{-t\frac{x^p}{2}}dx,
    \end{align}
    then we can directly apply the method of stationary phase since $p>1$ (it will be clear that it fails for $p=1$ since in this case, we have $F(t)=2/t$). Indeed, if $\varphi:[1,\infty)\rightarrow \R$ is defined by $\varphi(x)=-\frac{x^p}{2}$, then $\varphi$ is strictly decreasing and $\varphi''(1)=-\frac{p(p-1)}{2}<0$, so the method of stationary phase (or rather Laplace's method) implies that 
    \begin{align*}
    	F(t)\underset{t\rightarrow\infty }{\sim}\sqrt{\frac{2\pi}{-\varphi''(1)}}\frac{e^{-\frac{t}{2}}}{\sqrt{t}}=\frac{4\pi}{p(p-1)}\frac{e^{-\frac{t}{2}}}{\sqrt{t}}.
    \end{align*}
    Applying it to $t=\log\left(\frac{1}{r}\right)$, we deduce that 
    \begin{align*}
    	2\log^{\beta}\left(\frac{1}{r}\right)\int_{1}^{\infty}e^{-\log\left(\frac{1}{r}\right)\frac{x^p}{2}}dx\underset{r\rightarrow 0}{\sim}2\beta\log^{\beta}\left(\frac{1}{r}\right)\sqrt{\frac{4\pi}{1-\beta}}\sqrt{\frac{r}{\log\left(\frac{1}{r}\right)}}=2\beta\sqrt{\frac{4\pi}{1-\beta}}\sqrt{r}\log^{\beta-\frac{1}{2}}\left(\frac{1}{r}\right). 
    \end{align*}
    Therefore, we deduce that for $\alpha>0$, we have 
    \begin{align*}
    	I(\alpha,r)=2\sqrt{r}\log^{\alpha}\left(\frac{1}{r}\right)\left(1+\frac{1}{\sqrt{\log\left(\frac{1}{r}\right)}}\right).
    \end{align*}
    In particular, choosing the function 
    \begin{align}\label{lambda}
    	\Lambda_{\alpha}(t)=\sqrt{t}\left(1+\log_+^{\alpha}\left(R\sqrt{\frac{\pi}{t}}\right)\right),
    \end{align}
    we deduce that $u_{\alpha}\in M(\Lambda_{\alpha})=\mathrm{L}^{2,\infty}_{\log^{\alpha}}(B(0,R))$, where 
    \begin{align*}
    	M(\Lambda_{\alpha})=\mathrm{L}^{1}_{\mathrm{loc}}(B(0,R))\cap\ens{f:\Vert f\Vert_{M(\Lambda_{\alpha})}=\sup_{t>0}\left(\frac{1}{\Lambda_{\alpha}(t)}\int_{0}^tf_{\ast}(s)ds\right)<\infty}. 
    \end{align*}
    Likewise, we define the space $N(\Lambda_{\alpha})=\mathrm{L}^{2,1}_{\log^{\alpha}}(B(0,R))$ by 
    \begin{align*}
    	N(\Lambda_{\alpha})=\mathrm{L}^{1}_{\mathrm{loc}}(B(0,R))\cap\ens{f:\Vert f\Vert_{N(\Lambda_{\alpha})}=\int_{0}^{\infty}\Lambda_{\alpha}\left(\leb^2\left(B(0,R)\cap\ens{x:|f(x)|>t}
    		\right)\right)dt<\infty
    	}. 
    \end{align*}
    Furthermore, notice that $\Lambda_{\beta}:(0,\infty)\rightarrow (0,\infty)$ is non-zero, concave (for $\beta\leq 1$), and that $\Lambda_{\beta}$ extends continuously to $0$ and that $\Lambda_{\beta}(0)=0$. Therefore, we can apply the classical results of Steigerwalt and White to deduce 
    (\cite[Theorem $4.1$ and Theorem $4.4$]{lorentz_general}, see also  Theorem \ref{thm:dualityNM} in the appendix) that $N(\Lambda_{\beta})^{\ast}=M(\Lambda_{\beta})$, and that for all $(f,g)\in N(\Lambda_{\beta})\times M(\Lambda_{\beta})$, the  product $fg\in \mathrm{L}^1(B(0,R))$ with
    \begin{align*}
    	\left|\int_{B(0,R)}fg \, d\leb^2\right|\leq \Vert f\Vert_{N(\Lambda_{\beta})}\Vert g\Vert_{M(\Lambda_{\beta})}.
    \end{align*}
    To see that $\Lambda_{\beta}$ is concave, we can assume that $t\leq \pi R^2$ without loss of generality, so that up to a scaling $\Lambda_{\beta}$ is concave if and only if 
    \begin{align*}
    	\psi(t)=\sqrt{t}\log^{\beta}\left(\frac{1}{t}\right)
    \end{align*}
    is concave on $(0,1)$. We compute
    \begin{align*}
    	\psi'(t)&=\frac{1}{2\sqrt{t}}\log^{\beta}\left(\frac{1}{t}\right)-\frac{\beta}{\sqrt{t}}\log^{\beta-1}\left(\frac{1}{t}\right)\\
    	\psi''(t)&=-\frac{1}{4t^{\frac{3}{2}}}\log^{\beta}\left(\frac{1}{t}\right)-\frac{\beta}{2t^{\frac{3}{2}}}\log^{\beta-1}\left(\frac{1}{t}\right)+\frac{\beta}{2t^{\frac{3}{2}}}\log^{\beta-1}\left(\frac{1}{t}\right)+\frac{\beta(\beta-1)}{t^{\frac{3}{2}}}\log^{\beta-2}\left(\frac{1}{t}\right)\\
    	&=-\frac{1}{4t^{\frac{3}{2}}}\log^{\beta}\left(\frac{1}{t}\right)-\frac{\beta(1-\beta)}{t^{\frac{3}{2}}}\log^{\beta-2}\left(\frac{1}{t}\right)<0
    \end{align*}
    for all $0<t<1$ and $0\leq \beta\leq 1$. Now, fix $0<\beta\leq 1$, let $x\in B(0,\frac{R}{2})$ and $0<r<\frac{R}{4}$. By the co-area formula, we have
    \begin{align*}
    	\int_{B_{2r}\setminus\bar{B}_r(x)}|u(x)|dx=\int_{r}^{2r}\left(\rho\int_{\partial B_{\rho}(x)}|u|d\mathscr{H}^1\right)\frac{d\rho}{\rho}\leq \log(2)\inf_{r<\rho<2r}\left(\rho\int_{\partial B_{\rho}(x)}|u|d\mathscr{H}^1\right).
    \end{align*}
    Therefore, there exists $\rho\in (r,2r)$ such that 
    \begin{align*}
    	\rho\int_{\partial B_{\rho}(x)}|u|d\mathscr{H}^1&\leq \frac{1}{\log(2)}\int_{B_{2r}\setminus\bar{B}_r(x)}|u(x)|dx.
    \end{align*}  
    Now, we will show that if $u$ is a holomorphic function on $B(0,R)$, then a $\mathrm{L}^{2,\infty}_{\log^{\beta}}$ control implies a $\mathrm{W}^{1,1}$ control on $B(0,\alpha R)$ for all $\alpha<1$. Since it seems not standard to us, we give a full proof of this claim (in fact, we are not aware of a previous study of such spaces in the past literature).
    Using the $\mathrm{L}^{2,1}_{\log^{\beta}}/\mathrm{L}^{2,\infty}_{\log^{\beta}}$ duality  
    (see Theorem \ref{thm:dualityNM} in the appendix, or \cite[Theorem $4.4$]{lorentz_general}), we get
    \begin{align*}
    	\int_{B_{2r}\setminus \bar{B}_{r}(x)}|u(x)|dx\leq \Vert 1\Vert_{\mathrm{L}^{2,1}_{\mathrm{log}^{\beta}}(B_{2r}\setminus \bar{B}_r(0))}\Vert u\Vert_{\mathrm{L}^{2,\infty}_{\mathrm{log}^{\beta}}(B_{2r}\setminus\bar{B}_r(0))}.
    \end{align*}
    Notice that 
    \begin{align*}
    	\lambda_{1}(t)=\leb^2\left(B_{2r}\setminus\bar{B}_r(x)\cap\ens{x:1>t}\right)=\left\{\begin{alignedat}{2}
    		&3\pi r^2\qquad&&\text{if}\;\,t<1\\
    		&0\qquad&&\text{if}\;\,t\geq 1.
    	\end{alignedat}\right.
    \end{align*}
    We have by definition
    \begin{align*}
         \Vert 1\Vert_{\mathrm{L}^{2,1}_{\mathrm{log}}(B_{2r}\setminus \bar{B}_r(x))}=\int_{0}^{\infty
         }\left(\lambda_1(t)\right)^{\frac{1}{2}}\left(1+\log_+^{\beta}\left(R\sqrt{\frac{\pi}{\lambda_1(t)}}\right)\right)dt=\sqrt{3\pi }r\left(1+\log_+^{\beta}\left(\frac{1}{\sqrt{3}}\frac{R}{r}\right)\right).
    \end{align*}
    Finally, we deduce that 
    \begin{align*}
    	\rho\int_{\partial B_{\rho}(x)}|u|d\mathscr{H}^1\leq \frac{\sqrt{3\pi}}{\log(2)}r\left(1+\log_+^{\beta}\left(\frac{1}{\sqrt{3}}\frac{R}{r}\right)\right)\znp{u}{2,\infty}{\log^{\beta}}{B_{2r}\setminus\bar{B}_r(x)},
    \end{align*}
    and that,  \emph{a fortiori},  for all $x\in B(0,\frac{R}{2})$ and $0<r\leq \frac{R}{2}$ such that $B(x,2r)\subset B(0,R)$, there exists $\rho\in[r,2r]$ such that 
    \begin{align}\label{estimate_lebesgue}
    	\int_{\partial B_{\rho}(x)}|u|d\mathscr{H}^1\leq \frac{2\sqrt{3\pi}}{\log(2)}\left(1+\log_+^{\beta}\left(\frac{1}{\sqrt{3}}\frac{R}{r}\right)\right)\znp{u}{2,\infty}{\log^{\beta}}{B_{2r}\setminus\bar{B}_r(x)}.
    \end{align}
    Thanks to this result, we will now be able to show a variant of a lemma appearing in \cite{quantamoduli}. 
    \begin{lemme}\label{lemme_holomorphe1}
    	Let $u:B(0,R)\rightarrow \C$ be a holomorphic function and fix some $0\leq \alpha<1$, $0\leq \beta\leq 1$. Assume that $u\in \mathrm{L}^{2,\infty}_{\log^{\beta}}(B(0,R))$. Then $u\in \mathrm{L}^2(B(0,\alpha R))$ and there exists a universal constant $\Gamma_0$ {\rm (}independent of $\alpha$ and $\beta${\rm )} such that 
    	\begin{align*}
    		\np{u}{2}{B(0,\alpha R)}\leq \Gamma_0\frac{\alpha}{(1-\alpha)}\left(1+\log^{\beta}\left(\frac{1}{1-\alpha}\right)\right)\znp{u}{2,\infty}{\log^{\beta}}{B(0,R)}. 
    	\end{align*}
    \end{lemme}
    \begin{proof}
    Write 
    \begin{align*}
    	u(z)=\sum_{n=0}^{\infty}a_nz^n.
    \end{align*}
      Notice that $u=\p{z}v$, where 
    \begin{align}
    	v(z)=\sum_{n=1}^{\infty}b_nz^n=\sum_{n=1}^{\infty}\frac{a_{n-1}}{n}z^n.
    \end{align}
    First, using the estimate \eqref{estimate_lebesgue} applied to $\D v$ at a point $z\in \partial B(0,\alpha R)$ with $r=\frac{1}{2}(1-\alpha)R$, we deduce by the mean-value formula that for some $\rho\in[\frac{1}{2}(1-\alpha)R,(1-\alpha)R]$, we have 
    \begin{align}\label{neue_id0}
    	&|\D v(z)|=|2\,\p{z}v(z)|=\left|\frac{1}{\pi \rho}\int_{\partial B(z,\rho)}\p{\zeta}v(\zeta)d\zeta\right|\nonumber\\
    	&\leq \frac{4}{\log(2)}\sqrt{\frac{3}{\pi}}\frac{1}{(1-\alpha)R}\left(1+\log^{\beta}\left(\frac{2}{\sqrt{3}}\frac{1}{1-\alpha}\right)\right)\znp{\D v}{2,\infty}{\log^{\beta}}{B(0,R)}\nonumber\\
    	&= \frac{8}{\log(2)}\sqrt{\frac{3}{\pi}}\frac{1}{(1-\alpha)R}\left(1+\log^{\beta}\left(\frac{2}{\sqrt{3}}\frac{1}{1-\alpha}\right)\right)\znp{u}{2,\infty}{\log^{\beta}}{B(0,R)},
    \end{align}
    where we used $|\D v|^2=4|\p{z}v|^2$ by the holomorphy of $v$. Therefore, via integration by parts and using that both $u$ and $\bar{u}$ are harmonic, we deduce that 
    \begin{align}\label{neue_id1}
    	\int_{B(0,\alpha R)}|u(z)|^2|dz|^2&\leq \frac{1}{2}\int_{B(0, \alpha R)}|\D v|^2dx=\frac{1}{2}\int_{B(0, \alpha R)}\dive\left(\bar{v}\D v\right)=\frac{1}{2}\int_{\partial B(0, \alpha R)}\bar{v}\,{\partial_{\nu}{v}}\,d\mathscr{H}^1\nonumber\\
    	&=\frac{1}{2}\int_{\partial B(0,\alpha R)}\bar{(v-v_{\alpha R})}\partial_{\nu} v\,d\mathscr{H}^1,
    \end{align}
    where for all $0<\rho<R$, we set
    \begin{align*}
    	v_{\rho}=\dashint{\partial B_{\rho}(0)}v\,d\mathscr{H}^1.
    \end{align*}
    Thanks to the $\mathrm{L}^{\infty}$ bound \eqref{neue_id0} and the Sobolev embedding $\mathrm{H}^{\frac{1}{2}}(S^1)\hookrightarrow \mathrm{L}^1(B(0,1))$, there exists a uniform constant $C_0>0$ such that 
    \begin{align*}
    	&\np{\D v}{2}{B(0,\alpha R)}^2=\left|\int_{\alpha R}\bar{(v-{v}_{\partial B_{\alpha R}})}\partial_{\nu}v\,d\mathscr{H}^1\right|\leq \np{v-v_{\alpha R}}{1}{\partial B_{\alpha R}(0)}\np{\D v}{\infty}{\partial B_{\alpha R(0)}}\\
    	&\leq C_0\alpha R\hs{v}{\frac{1}{2}}{\partial B_{\alpha R}(0)}\times \frac{8}{\log(2)}\sqrt{\frac{3}{\pi}}\frac{1}{(1-\alpha)R}\left(1+\log^{\beta}\left(\frac{2}{\sqrt{3}}\frac{1}{1-\alpha}\right)\right)\znp{u}{2,\infty}{\log^{\beta}}{B(0,R)}\\
    	&\leq \frac{8 C_0}{\log(2)}\sqrt{\frac{3}{\pi}}\frac{\alpha}{(1-\alpha)}\left(1+\log^{\beta}\left(\frac{2}{\sqrt{3}}\frac{1}{1-\alpha}\right)\right)\np{\D v}{2}{B(0,\alpha R)}\znp{u}{2,\infty}{\log^{\beta}}{B(0,R)}.
    \end{align*}
    Therefore, we get
    \begin{align*}
    	\np{u}{2}{B(0,\alpha R)}\leq \frac{4 C_0}{\log(2)}\sqrt{\frac{3}{\pi}}\frac{\alpha}{(1-\alpha)}\left(1+\log^{\beta}\left(\frac{2}{\sqrt{3}}\frac{1}{1-\alpha}\right)\right)\znp{u}{2,\infty}{\log^{\beta}}{B(0,R)}
    \end{align*}
    which concludes the proof of the lemma.    
    \end{proof}
    We will also need the following lemma. 
    \begin{lemme}\label{lemma_holomorphe2}
    	Let $R>0$ and $u:B(0,R)\rightarrow \C$ by a holomorphic function such that $u\in \mathrm{L}^2(B(0,R))$. Then,  for all $0<\alpha<1$,  the following estimates hold:
    	\begin{align*}
    		\np{u}{2,1}{B(0,\alpha R)}&\leq \frac{4\alpha}{1-\alpha^2}\np{u}{2}{B(0,R)}\\
    		\np{\D u}{1}{B(0,\alpha R)}&\leq 4\sqrt{\pi}\frac{\alpha^2}{(1-\alpha^2)^{\frac{3}{2}}}\np{u}{2}{B(0,R)}. 
    	\end{align*}
    \end{lemme}
    \begin{proof}
    	Let $\ens{a_n}_{n\in\N}$ be such that 
    	\begin{align*}
    		u(z)=\sum_{n=0}^{\infty}a_nz^n. 
    	\end{align*}
        Taking complex coordinates, we deduce that 
        \begin{align*}
        	\int_{B(0,R)}|u(z)|^2|dz|^2=2\pi\sum_{n=0}^{\infty}\int_{0}^R|a_n|^2\rho^{2n+1}d\rho=\pi\sum_{n=0}^{\infty}\frac{|a_n|^2}{n+1}R^{2(n+1)}.
        \end{align*}
        A direct computation shows that for all $n\geq 0$, we have 
        \begin{align*}
        	\np{|z|^n}{2,1}{B(0,R)}=4\sqrt{\pi}R^{n+1}.
        \end{align*}
        Therefore, we deduce by the triangle inequality and the Cauchy-Schwarz inequality that 
        \begin{align*}
        	\np{u}{2,1}{B(0,\alpha R)}&\leq 4\sqrt{\pi}\sum_{n=0}^{\infty}|a_n|(\alpha R)^{n+1}\leq 4\sqrt{\pi}\left(\sum_{n=0}^{\infty}(n+1)\alpha^{2(n+1)}\right)^{\frac{1}{2}}\left(\sum_{n=0}^{\infty}\frac{|a_n|^2}{n+1}R^{2(n+1)}\right)^{\frac{1}{2}}\\
        	&=\frac{4\alpha}{1-\alpha^2}\np{u}{2}{B(0,R)}.
        \end{align*}
        Then, we compute 
        \begin{align*}
        	\p{z}u(z)=\sum_{n=1}^{\infty}na_nz^{n-1}=\sum_{n=0}^{\infty}(n+1)a_{n+1}z^n.
        \end{align*}
        Using that
        \begin{align}
        	\np{|z|^n}{1}{B(0,R)}=\frac{2\pi}{n+2}R^{n+2}, \quad \text{for all $n\geq 0$}\, ,
        \end{align}
       we deduce  that 
        \begin{align*}
        	\frac{1}{2}\np{\D u}{1}{B(0,\alpha R)}&=\np{\p{z}u}{1}{B(0,\alpha R)}\leq 2\pi \sum_{n=0}^{\infty}\frac{(n+1)}{(n+2)}|a_{n+1}|(\alpha R)^{n+2}\\
        	&\leq 2\pi \left(\sum_{n=0}^{\infty}\frac{(n+1)^2}{n+2}\alpha^{2(n+2)}\right)^{\frac{1}{2}}\left(\sum_{n=0}^{\infty}\frac{|a_{n+1}|^2}{n+2}R^{2(n+2)}\right)
        	\leq 2\sqrt{\pi}\frac{\alpha^2}{(1-\alpha^2)^\frac{3}{2}}\np{u}{2}{B(0,R)}
        \end{align*}
        where we used the following identities valid for $|a|<1$
        \begin{align*}
            \sum_{n=0}^{\infty}\frac{(n+1)^2}{n+2}a^n&=\sum_{n=0}^{\infty}(n+1)a^n-\sum_{n=0}^{\infty}\frac{n+1}{n+2}a^n=\sum_{n=0}^{\infty}(n+1)a^n-\sum_{n=0}^{\infty}a^n+\sum_{n=0}^{\infty}\frac{a^n}{n+2}\\
            &=\frac{1}{(1-a^2)}-\frac{1}{1-a}+\sum_{n=0}^{\infty}\frac{a^n}{n+2},
        \end{align*}
        and
        \begin{align*}
        	\sum_{n=0}^{\infty}\frac{a^n}{n+2}=\frac{1}{a^2}\sum_{n=0}^{\infty}\frac{a^{n+2}}{n+2}=\frac{1}{a^2}\left(\sum_{n=0}^{\infty}\frac{a^{n+1}}{n+1}-a\right)=\frac{1}{a^2}\left(-\log(1-a)-a\right)
        \end{align*}
        which shows that 
        \begin{align*}
        	\sum_{n=0}^{\infty}\frac{(n+1)^2}{n+2}a^{n+2}=\frac{a^3}{(1-a)^2}-\log(1-a)-a\leq \frac{a^3}{(1-a)^2}+\frac{a^2}{2(1-a)}=\frac{a^2+a^3}{2(1-a)^3}.
        \end{align*}
    \end{proof}

    \subsection*{Back to estimating $\psi_{k}$ defined in \eqref{eq:deffk}}

    Thanks to Lemmas \ref{lemme_holomorphe1} and \ref{lemma_holomorphe2}, recalling the estimates \eqref{est1}, \eqref{est2} and \eqref{est3}, we deduce that $\psi_{k}\in \mathrm{L}^{2,\infty}_{\log}(B(0,R_k))$. Since $\psi_{k}$ has a pole of order at most $1$ at zero, this implies that
    \begin{align*}
    	\psi_{k}(z)-\frac{a_{-1,k}}{z}\in \mathrm{W}^{1,1}\cap \mathrm{L}^{2,1}\left(B(0,{R_k}/{4})\right).
    \end{align*}
    Notice that $a_{-1,k}$ is bounded for $k\geq N$. Indeed, 
    $\Psi_{k}(z)=z\psi_{k}(z)$ is holomorphic, so the mean-value formula and \eqref{estimate_lebesgue} applied to $r=\frac{R_k}{2}$, yield that there exists $\rho\in[\frac{R}{2},R]$ such that 
    \begin{align*}
    	|a_{-1,k}|&=|u(0)|=\frac{1}{2\pi \rho}\left|\int_{\partial B_{\rho}(0)}\Psi_{k}(z)dz\right|\leq \frac{1}{2\pi}\int_{\partial B_{\rho}(0)}|\psi_{k}|d\mathscr{H}^1\\
	&\leq \frac{1}{2\pi }\frac{2\sqrt{3\pi}}{\log(2
    		)}\left(1+\log\left(\frac{4}{\sqrt{3}}\right)\right)\znp{\psi_{k}}{2,\infty}{\log}{B(0,R_k)}.
    \end{align*}

    \subsection*{A pointwise estimate on $\vec{V}_{k}$}

    Using  variants of Lemmas \ref{lemme_holomorphe1} and \ref{lemma_holomorphe2} for harmonic functions (see \cite{quantamoduli} and \cite{pointwise} for more details), we also deduce that $\vec{w}_k\in \mathrm{W}^{1,1}\cap \mathrm{L}^{2,1}(B(0,R_k/4))$. We obtain a decomposition on $B(0,R_k/4)$
    \begin{align}\label{previous_V}
         \D\Im\left(\vec{V}_k\right)=h_{1,k}(z)+h_{2,k}(z),
    \end{align}
    where 
    \begin{align*}
    	|h_{1,k}(z)|\leq C\frac{\log\left(\frac{R_k}{|z|}\right)}{|z|}
   \end{align*}
    and $h_{2,k}\in \mathrm{W}^{1,1}\cap \mathrm{L}^{2,1}(B(0,R_k/4))$. More precisely, there exists a constant $C_0>0$ independent of $k$ such that for all $k\geq N$, 
    \begin{align*}
    	\np{h_{2,k}}{2,1}{B(0,R_k/4)}+\np{\D h_{2,k}}{1}{B(0,R_k/4)}\leq C_0.
    \end{align*}
    Now, identifying $\C$ with $\R^2$, we can rewrite the equation for $\p{z}\vec{V}_k$ as 
    \begin{align}\label{new_system00}
    	\left\{\begin{alignedat}{1}
    		\p{x}\Re(\vec{V}_k)+\p{y}\Im(\vec{V}_k)&=2\,\Re\left(\vec{Y}_k\right)\\
    		-\p{y}\Re(\vec{V}_k)+\p{x}\Im(\vec{V}_k)&=2\,\Im\left(\vec{Y}_k\right)\,,
    	\end{alignedat}\right.
    \end{align}
    or, equivalently, 
    \begin{align*}
    	\D\Re(\vec{V}_k)-\D^{\perp}\Im(\vec{V_k})=2\,\bar{\vec{Y}_k}.
    \end{align*}    
    Therefore, we have 
    \begin{align}\label{est_V1}
    	\left\{\begin{alignedat}{1}
    	&\p{x}\Re\left(\vec{V}_k\right)=\p{y}\H_k-3\sum_{j=1}^{n-2}\s{\p{y}\H_k}{\n_{k,j}}\n_{k,j}+\star_h\left(\p{y}\n_k\wedge \H_k\right)-\p{y}\Im(\vec{V}_k)\\
    	&\p{y}\Re(\vec{L_k})=-\p{x}\H_k+3\sum_{j=1}^{n-2}\s{\p{x}\H_k}{\n_{k,j}}\n_{k,j}-\star_h\left(\p{x}\n_k\wedge \H_k\right)+\p{x}\Im(\vec{V}_k).
    	\end{alignedat}\right.
    \end{align}
    The identity \eqref{est_V1} implies that 
    \begin{align}\label{est_V2}
    	\p{r}\Re(\vec{V}_k)&=\cos(\theta)\p{x}\Re(\vec{V}_k)+\sin(\theta)\p{y}\Re(\vec{V}_k)=\left(\cos(\theta)\p{y}\H_k-\sin(\theta)\p{x}\H_k\right)\nonumber\\
    	&-3\sum_{j=1}^{n-2}\s{\cos(\theta)\p{y}\H_k-\sin(\theta)\p{x}\H_k}{\n_{k,j}}\n_{k,j}+\star_h\left((\cos(\theta)\p{y}\n_k-\sin(\theta)\p{x}\n_k)\wedge \H_k\right)\nonumber\\
    	&-\cos(\theta)\p{y}\Im(\vec{V}_k)+\sin(\theta)\Re(\vec{V}_k)\nonumber\\
    	&=\frac{1}{r}\p{\theta}\H_k-3\pi_{\n_k}\left(\frac{1}{r}\p{\theta}\H_k\right)+\star_{h}\left(\frac{1}{r}\p{\theta}\n_k\wedge \H_k\right)-\frac{1}{r}\p{\theta}\Im\left(\vec{V}_k\right),
    \end{align}
    since for any smooth function $u:B(0,1)\rightarrow \R^n$, and $a,b\in \C$ we have
    \begin{align*}
    	&\frac{1}{r}\p{\theta}u=-\sin(\theta)\p{x}u+\cos(\theta)\p{y}u\, .
    \end{align*}
    Now, recall the notation
    \begin{align*}
    	\Re(\vec{V}_k)_{\rho}=\dashint{\partial B(0,\rho)}\Re\left(\vec{V}_k\right)d\mathscr{H}^1=\frac{1}{2\pi\rho}\int_{\partial B(0,\rho)}\Re\left(\vec{V}_k\right)d\mathscr{H}^1\, , \quad \text{ for all $\alpha_0^{-1}r_k\leq \rho\leq \alpha_0 R_k$}.
    \end{align*}
    Since $\pi_{\n_k}(\H_k)=\H_k$, we have
    \begin{align}\label{est_V3}
    	\pi_{\n_k}\left(\p{\theta}\H_k\right)=\p{\theta}\left(\pi_{\n_k}(\H_k)\right)-\left(\p{\theta}\pi_{\n_k}\right)(\H_k)=\p{\theta}\H_k-\left(\p{\theta}(\pi_{\n_k})(\H_k)\right). 
    \end{align}
    Therefore, by \eqref{est_V1}, \eqref{est_V2} and \eqref{est_V3},  we deduce that
    \begin{align}\label{radial0}
    	&\frac{d}{d\rho}\Re(\vec{V}_k)_{\rho}=\frac{1}{2\pi}\int_{0}^{2\pi}\partial_{\rho}\vec{V}_k(\rho,\theta)d\theta=\frac{1}{2\pi}\int_{0}^{2\pi}\bigg(\frac{3}{\rho}(\partial_{\theta}\pi_{\n_k})\left(\H_k\right)+\star_h\left(\frac{1}{\rho}\partial_{\theta}\n_k\wedge \H_k\right)
    	\bigg)d\theta. 
    \end{align}
    Arguing as in the proof of \cite[Lemma VI.1]{quanta} using the  $\epsilon$-regularity Theorem \ref{eps_reg} (adapted from Rivière's original result \cite{riviere1}), we deduce that there exists $\epsilon_4(n),C_4(n)>0$ such that the following holds: if  
    \begin{align*}
    	\sup_{r_k<s<R_k}\int_{B_{2s}\setminus\bar{B}_s(0)}|\D\n_k|^2dx\leq \epsilon_4(n)
    \end{align*}
   then
    \begin{align}\label{pointwise1}
    	|\D\n_k(z)|^2\leq C_4^2(n)\frac{1}{|z|^2}\int_{B_{2|z|}\setminus\bar{B}_{\frac{|z|}{2}}}|\D\n_k|^2dx\leq \frac{C_4^2(n)\epsilon_3(n)}{|z|^2}, \quad \text{for all $2r_k\leq |z|\leq \frac{R_k}{2}$}\, .
    \end{align}
    This motivates us to introduce the function $\delta_k:(0,R_k)\rightarrow\R$ defined as
    \begin{align*}
    	\delta_k(\rho)=\left(\frac{1}{\rho^2}\int_{B_{2\rho}\setminus\bar{B}_{\frac{\rho}{2}}(0)}|\D\n_k|^2dx\right)^{\frac{1}{2}}=\frac{\sqrt{15\pi}}{2}\left(\dashint{B_{2\rho}\setminus\bar{B}_{\frac{\rho}{2}}(0)}|\D\n_k|^2dx\right)^{\frac{1}{2}}.
    \end{align*}
    We have in particular
    \begin{align}\label{delta_local}
    	\rho\,\delta_k(\rho)\leq \left(\int_{\Omega_k(1)}|\D\n_k|^2dx\right)^{\frac{1}{2}}\leq \sqrt{\Lambda(h)}. 
    \end{align}
    Using Fubini's theorem, we deduce that, for all $0<r<R<\frac{R_k}{2}$, it holds
    \begin{align}\label{delta}
    	\int_{r}^R\delta^2(\rho)\,\rho\, d\rho&=\int_{r}^{R}\left(\int_{B_{2R}\setminus\bar{B}_{\frac{r}{2}}(0)}|\D\n_k|^2\mathbf{1}_{\ens{\frac{\rho}{2}<|x|<2\rho}}dx\right)\frac{d\rho}{\rho}
    	=\int_{B_{2R}\setminus\bar{B}_{\frac{r}{2}}(0)}|\D\n_k|^2\left(\int_{\frac{|x|}{2}}^{2|x|}\frac{d\rho}{\rho}\right)dx\nonumber\\
    	&=\log(4)\int_{B_{2R}\setminus\bar{B}_{\frac{r}{2}}(0)}|\D\n_k|^2dx.
    \end{align}
    This implies by the Cauchy-Schwarz inequality that 
    \begin{align}\label{delta2}
    	\int_{r}^{R}\delta_k(\rho)\,\rho\,d\rho\leq \left(\int_{r}^{R}\rho\,d\rho\right)^{\frac{1}{2}}\left(\int_{r}^R\delta^2(\rho)\,\rho\,d\rho\right)^{\frac{1}{2}}\leq R_k\left(\int_{B_{R}\setminus \bar{B}_{\frac{r}{2}}(0)}|\D\n_k|^2dx\right)^{\frac{1}{2}}.
    \end{align}
    Furthermore, since $\H_k$ satisfies the Euler-Lagrange equation
    \begin{align*}
    	&\mathscr{L}_{g_k}\H_k=\frac{1}{2}\D_{g_k}^{\ast}\left(\D_{g_k}\H_k-3\,\D_{g_k}^{\perp}\H_k+\star_h\left(\left(\ast \D_{g_k}\n_k\right)\wedge \H_k\right)\right)\\
    	&=-\mathscr{R}_1^{\perp}(\H_k)+2\,\tilde{K}_h\,\H_k-2\,\mathscr{R}_2(d\phi_k)-(DR)(d\phi_k)-\mathscr{R}_3^{\perp}(\H_k),
    \end{align*}
    which is uniformly elliptic in all dyadic annuli in $\Omega_k(1)$ thanks to the hypothesis and the Harnack equation on the conformal parameter (Lemma \ref{neckfine}, \eqref{I3}), we deduce by standard elliptic regularity (\cite{gilbarg}, Theorem $3.9$) that there exists a constant $\Gamma_0(n)$ such that
    \begin{align}\label{cz1}
    	r\np{\D \H_k}{\infty}{\partial B(0,r)}\leq \Gamma_0(n)\left(\np{\H_k}{\infty}{B_{\frac{4}{3}r}\setminus\bar{B}_{\frac{3}{4}r}(0)}+r^2\np{\lg \H_k}{\infty}{B_{\frac{4}{3}r}\setminus\bar{B}_{\frac{3}{4}r}(0)}\right).
    \end{align}
    \begin{rem}
    	In fact, one rather applies elliptic regularity to the operator $\mathscr{L}_{\n}$ introduced below in \eqref{eqH2}, and to the equation \eqref{eqH}. One checks that using this equation yields the same estimate for $\D\H_k$. 
    \end{rem}
    With the $\epsilon$-regularity Theorem \ref{eps_reg} inspired from Rivière's one (\cite{riviere1}), we deduce that
    \begin{align}\label{cz2}
    	\np{e^{\lambda_k}\H_k}{\infty}{B_{\frac{4}{3}r}\setminus\bar{B}_{\frac{3}{4}r}(0)}\leq \Gamma_1(n)\np{e^{\lambda_k}\H_k}{2}{B_{2r}\setminus\bar{B}_{\frac{r}{2}}(0)}.
    \end{align}
    Now, since 
    \begin{align*}
    	e^{-A}|z|^{d_k}\leq e^{\lambda_k(z)}\leq e^{A}|z|^{d_k},
    \end{align*}
    we deduce that for all $0<\beta<1$ and $r_k\leq \beta r<\beta^{-1}r<R_k$, we have for all $\beta r\leq |z|\leq \beta^{-1}r$
    \begin{align*}
    	e^{-A}\alpha^{|d_k|}r^{d_k}\leq e^{\lambda_k(z)}\leq e^{A}\alpha^{-|d_k|}r^{d_k}. 
    \end{align*}
    And since $d_k\conv{k\rightarrow \infty}d$, we deduce that there exists $B\in \R$ independent of $k$ such that
    \begin{align}\label{cz3}
    	\sup_{z\in B_{2r}\setminus\bar{B}_{\frac{r}{2}}(0)}e^{\lambda_k(z)}\leq e^{B}\inf_{z\in B_{2r}\setminus\bar{B}_{\frac{r}{2}}(0)}e^{\lambda_k(z)}\,, \quad \text{for all $2r_k<r<\dfrac{R_k}{2}$}.
    \end{align}
   The combination of \eqref{cz2} and \eqref{cz3} gives
    \begin{align}\label{cz5}
    	\np{\H_k}{\infty}{B_{\frac{4}{3}r}\setminus\bar{B}_{\frac{3r}{4}}(0)}\leq e^{B}e^{-\lambda_k(r)}\Gamma_1(n)\np{e^{\lambda_k}\H_k}{2}{B_{2r}\setminus\bar{B}_{\frac{r}{2}}(0)}\leq e^{B}e^{-\lambda_k(r)}\Gamma_1(n)\delta(r). 
    \end{align}
    Likewise, we have
    \begin{align}\label{cz6}
    	\np{\lg \H_k}{\infty}{B_{\frac{4}{3}r}\setminus\bar{B}_{\frac{3}{4}r}(0)}&\leq C(h)\left(1+\np{\H_k}{\infty}{B_{\frac{4}{3}r}\setminus\bar{B}_{\frac{3r}{4}}(0)}\right)
    	\leq C(h)\left(1+e^{B}e^{-\lambda_k(r)}\Gamma_1(n)\delta(r)\right).
    \end{align}
    Since $r \delta(r)\leq \sqrt{\Lambda(h)}$, we finally deduce by \eqref{cz1}, \eqref{cz5} and \eqref{cz6} that
     \begin{align}\label{pointwise2}
    	\left\{\begin{alignedat}{2}
    		e^{\lambda_k(|z|)}|\H_k(z)|&\leq |\D\n_k(z)|\leq C_4(n)\delta(|z|) \qquad&&\text{for all}\;\,z\in \Omega_k\left(\frac{1}{2}\right)\\
    		e^{\lambda_k(|z|)}|\D\H_k(z)|&\leq C_5(n,h)\frac{\delta(|z|)}{|z|}\qquad&& \text{for all}\;\, z\in \Omega_k\left(\frac{1}{2}\right)\, .
    	\end{alignedat}\right.
    \end{align}  
    Therefore,
    \begin{align*}
    	\left|\frac{d}{d\rho}\Re(\vec{V}_k)_{\rho}\right|\leq Ce^{-\lambda_k(\rho)}\delta^2(\rho).
    \end{align*}
    This estimate implies that 
    \begin{align}\label{extra_step1}
    	\int_{2r_k}^{\frac{R_k}{2}}e^{\lambda_k(\rho)}\left|\frac{d}{d\rho}\Re(\vec{V}_k)_{\rho}\right|\,\rho\, d\rho\leq C\int_{\Omega_k(1)}|\D\n_k|^2dx. 
    \end{align}
    Define the function
    \begin{align*}
    		a:(0,R_k/2)&\rightarrow \R_+ \, ,\quad a(t)= \left|\Re\left(\vec{V}_k\right)_{t}\right|.
    \end{align*}
   As  $|a'(t)|=\left|\frac{d}{dt}\Re\left(\vec{V}_k\right)_t\right|,$ the combination of  \eqref{new_harnarck} and \eqref{extra_step1} gives that
    \begin{align*}
    	\int_{r}^{R}t^{1+d_k}|a'(t)|dt\leq C.
    \end{align*}
    Since $d_k\conv{k\rightarrow \infty}d>-1$,  from the identity 
    \begin{align}\label{estimate10}
    	\int_{r}^{R}t^{1+d_k}a'(t)dt=R^{1+d_k}a(R)-r^{1+d_k}a(r)-(1+d_k)\int_{r}^{R}t^{d_k}a(t)dt\, ,
    \end{align}
  we deduce   that 
    \begin{align*}
    	r^{1+d_k}a(r)\leq \left(\frac{R_k}{2}\right)^{1+d_k}a\left(\frac{R_k}{2}\right)+C.
    \end{align*}
    The estimate 
    \begin{align*}
    	|V_k(z)|\leq C_2\frac{\log\left(\frac{R_k}{|z|}\right)}{|z|}
    \end{align*}
    implies that
    \begin{align*}
    	\left(\frac{R_k}{2}\right)^{1+d_k}a\left(\frac{R_k}{2}\right)\leq C_2R_k^{d_k}\left|\log(R_k)\right|,
    \end{align*}
    which is bounded independently of $k\geq N$ since $d_k\conv{k\rightarrow \infty}d>-1$ and $R_k\conv{k\rightarrow \infty}R>0$. Finally, we deduce that for all $2r_k<|z|<\frac{R_k}{2}$, it holds
    \begin{align}\label{estimate_Vk1}
    	|z|^{d_k}\left|\Re(\vec{V}_k)_{|z|}\right|\leq \frac{C}{|z|}. 
    \end{align}
     \begin{lemme}\label{sobolev11}
    	For all $r>0$ and $u\in \mathrm{W}^{1,1}(\partial B(0,r),\R^n)$, it holds
    	\begin{align}\label{sobolev1}
    		\np{u-u_r}{\infty}{\partial B(0,r)}\leq n\int_{\partial B(0,r)}|\D u|d\mathscr{H}^1.
    	\end{align}
    \end{lemme}
    \begin{proof}
    	By scaling invariance, we can assume that $r=1$, which permits to see $u$ as a $2\pi$-periodic function $u=(u_1,\cdots,u_n):[0,2\pi]\rightarrow \R^n$. By the intermediate values theorem, for all $1\leq i\leq n$, there exists $a_i\in [0,2\pi]$ such that
    	\begin{align*}
    		u_i(a_i)=\dashint{0}^{2\pi}u_i\, d\leb^1=\frac{1}{2\pi}\int_{0}^{2\pi}u_i(\theta)d\theta.
    	\end{align*}
    	Therefore, we have for all $1\leq i\leq n$ and for all $\theta\in [0,2\pi]$
    	\begin{align*}
    		\left|u_i(\theta)-\dashint{0}^{2\pi}u_i\, d\leb^1\right|=|u_i(\theta)-u_i(a_i)|=\left|\int_{a_i}^{\theta}u_i'(t)dt\right|\leq \int_{0}^{2\pi}|u_i'(t)|dt,
    	\end{align*}
    	Therefore, by the triangle inequality, we have 
    	\begin{align*}
    		\left|u(\theta)-\dashint{0}^{2\pi}u\,d\leb^1\right|\leq \sum_{i=1}^{n}\left|u_i(\theta)-\dashint{0}^{2\pi}u_i' d\leb^1\right|\leq \sum_{i=1}^{n}\int_{0}^{2\pi}|u_i'(t)|dt\leq n\int_{0}^{2\pi}|u'(t)|dt, \quad \text{for all $\theta\in [0,2\pi]$}.
    	\end{align*}
    	Coming back to the initial inequality, we deduce that
    	\begin{align*}
    		\np{u-u_r}{\infty}{\partial B(0,r)}\leq n\int_{\partial B(0,r)}\left|\frac{1}{r}\partial_{\theta}u\right|d\mathscr{H}^1\leq n\int_{\partial B(0,r)}|\D u|d\mathscr{H}^1
    	\end{align*}
        which concludes the proof of the lemma. 
    \end{proof}
    Using the estimates $|\vec{Y}_k|\leq \dfrac{Ce^{-\lambda_k(|z|)}}{|z|^2}$, $|\D \Im(\vec{V}_k)|\leq \dfrac{C}{|z|}$ and \eqref{new_system00}, we deduce that for all $z\in \Omega_k(\frac{1}{4})$, we have
    \begin{align*}
    	\left|\Re\left(\vec{V}_k(z)\right)-\Re\left(\vec{V_k}\right)_{|z|}\right|&\leq n\int_{\partial B(0,|z|)}|\D\Re(\vec{V}_k)|d\mathscr{H}^1=n\int_{\partial B(0,|z|)}\left|2\bar{\vec{Y}_k}+\D^{\perp}\Im(\vec{V}_k)\right|d\mathscr{H}^1\\
    	&\leq n\int_{\partial B(0,|z|)}\left(2|\vec{Y}_k|+|\D \Im(\vec{V}_k)|\right)d\mathscr{H}^1\leq \frac{4\pi nC}{|z|^{1+d_k}}+2\pi nC_5\log\left(\frac{R_k}{|z|}\right)\\
    	&+n\int_{\partial B(0,|z|)}|h_2|d\mathscr{H}^1, 
    \end{align*}
    where $h_2$ is bounded in $\mathrm{W}^{1,1}\cap \mathrm{L}^{2,1}(B(0,R_k/4))$. In particular, by trace theory, we have $h_2\in \mathrm{L}^1(\partial B(0,|z|))$ and for some universal constant independent of $|z|$, we have by the Cauchy-Schwarz inequality
    \begin{align*}
    	\np{h_2}{1}{\partial B(0,|z|)}&\leq \Gamma_2\left(\frac{1}{R_k}\np{h_2}{1}{B(0,R_k/4)}+\np{\D h_2}{1}{B(0,R_k/4)}\right)\\
    	&\leq \Gamma_2 \left(\frac{\sqrt{\pi}}{4}\np{h_2}{2}{B(0,R_k/4)}+\np{\D h_2}{1}{B(0,R_k/4)}\right).
    \end{align*}
    Therefore, we obtain the estimate 
    \begin{align*}
    	\left|\Re(\vec{V}_k(z))-\Re\left(\vec{V}_k\right)_{|z|}\right|&\leq \frac{4\pi nC}{|z|^{1+d_k}}+2\pi nC_5\log\left(\frac{R_k}{|z|}\right)+C.
    \end{align*}
    Finally, 
    using that $d_k>-1+\epsilon$ and recalling the Harnack inequality \eqref{new_harnarck}, we deduce that
    \begin{align}\label{localV0}
    	e^{\lambda_k(z)}|\Re(\vec{V}_k)(z)|\leq \frac{C}{|z|}+2\pi nC_5|z|^{-1+\epsilon}\log\left(\frac{R_k}{|z|}\right)+C|z|^{-1+\epsilon}\leq \frac{C}{|z|}.
    \end{align}
    The previous estimate on $\Im(\vec{V}_k)$ (see \eqref{previous_V}) coupled with the same argument on averages implies that 
    \begin{align}\label{final_Va}
    	e^{\lambda_k(z)}|\vec{V}_k(z)|\leq \frac{C}{|z|}, \quad \text{for all $z\in \Omega_k(1/4)$,}
    \end{align}
    which, 
    recalling the Harnack inequality \eqref{new_harnarck} for the conformal parameters, in turn gives, 
    \begin{align}\label{final_V_2a}
    	|z|^{d_k}|\vec{V}_k(z)|\leq \frac{C}{|z|}\, ,  \quad \text{for all $z\in \Omega_k(1/4)$.}
    \end{align}

    \subsection*{Analysis of the other cases}
    \noindent
    Now, assume that $d_k>0$. We will  distinguish the case $d_k\conv{k\rightarrow\infty}0$ and $d_k\conv{k\rightarrow \infty}d>0$. 
    
    \textbf{Case 2:} $d_k\conv{k\rightarrow\infty}0$.  Define as above
    \begin{align*}
    	\vec{U}_k(z)=\frac{1}{\z^{2}}\left(-\frac{1}{\pi\bar{\zeta}}\star \left(\bar{\zeta}^{2}\vec{Y}_k(\zeta)\right)\right)=-\frac{1}{\pi \z^{2}}\int_{\C}\frac{\bar{\zeta}^{2}\vec{Y}_k(\zeta)}{\bar{z-\zeta}}|d\zeta|^2.
    \end{align*}
    Notice that $\p{z}\vec{U}_k=\vec{Y}_k$. Fix some $z\in B(0,R)$, where we recall that $r=2r_k$ and $R=\dfrac{R_k}{2}$. 
    \\ For $0<2|z|\leq r$, it holds
    \begin{align*}
    	\vec{U}_k(z)=-\frac{1}{\pi \z^{2}}\int_{B(0,R)\setminus\bar{B}(0,2|z|)}\frac{\bar{\zeta}^{2}\vec{Y}_k(\zeta)}{\bar{\zeta-z}}|d\zeta|^2=\sum_{l=0}^{\infty}\left(\frac{1}{\pi}\int_{B(0,R)\setminus\bar{B}(0,2|z|)}\frac{\bar{\zeta}^{2}\vec{Y}_k(\zeta)}{\bar{\zeta^{l+1}}}|d\zeta|^2\right)z^l=\sum_{l=0}^{\infty}a_{k,l}\z^l. 
    \end{align*} 
    Likewise, we have
    \begin{align*}
    	|a_{k,0}|\leq 2\pi C_0R\, .
    \end{align*}
    Then, we estimate for $l=1$
    \begin{align*}
    	\left|\int_{B(0,R)\setminus\bar{B}(0,2|z|)}\frac{\bar{\zeta}^2\vec{Y}_k(\zeta)}{\bar{\zeta}^{l+1}}|d\zeta|^2\right|\leq C_0\int_{B(0,R)\setminus\bar{B}(0,2|z|)}\frac{|d\zeta|^2}{|\zeta|^{2+d_k}}=\frac{2\pi C_0}{d_k}\left(\frac{1}{(2|z|)^{d_k}}-\frac{1}{R^{d_k}}\right).
    \end{align*}
    For $l\geq 2$, we have
    \begin{align*}
    	\left|\int_{B(0,R)\setminus\bar{B}(0,2|z|)}\frac{\bar{\zeta}^2\vec{Y}_k(\zeta)}{\bar{\zeta}^{l+1}}|d\zeta|^2\right|\leq C_0\int_{B(0,R)\setminus\bar{B(0,2|z|)}}\frac{|d\zeta|^2}{|\zeta|^{l+1+d_k}}=\frac{2\pi C_0}{l-1+d_k}\left(\frac{1}{(2|z|)^{l-1}}-\frac{1}{R^{l-1+d_k}}\right).
    \end{align*}
    Therefore, it holds
    \begin{align*}
    	\left|\sum_{l=2}^{\infty}a_{k,l}z^l\right|\leq 2C_0|z|.
    \end{align*}
    We deduce that 
    \begin{align*}
    	\left|\vec{U}_k(z)-\frac{a_{k,0}}{\z^2}\right|\leq \frac{2C_0}{|z|}\frac{1}{d_k}\left(\frac{1}{(2|z|)^{d_k}}-\frac{1}{R^{d_k}}\right)+\frac{2C_0}{|z|}.
	    \end{align*}
    Recalling the definition \eqref{def_Vk} of $\vec{V}_k(z)=\vec{U}_k(z)-\frac{a_{k,0}}{\z^2}$, we deduce that for all $\displaystyle|z|\leq \frac{r}{2}\leq \frac{R}{4}$, it holds
    \begin{align}\label{new_est_Vk1}
    	|\vec{V}_k(z)|\leq \frac{2C_0R^{-d_k}}{|z|}\frac{1}{d_k}\left(\left(\frac{R}{2|z|}\right)^{d_k}-1\right)+\frac{2C_0}{|z|}\leq \frac{4C_0}{|z|}\left(1+\frac{1}{d_k}\left(\left(\frac{R}{2|z|}\right)^{d_k}-1\right)\right)
    \end{align}
    for $k$ large enough. Notice that for all $z\neq 0$ we have
    \begin{align*}
    	\frac{1}{d_k}\left(\left(\frac{R}{2|z|}\right)^{d_k}-1\right)=\frac{1}{d_k}\left(e^{d_k\log\left(\frac{R}{2|z|}\right)}-1\right)=\log\left(\frac{R}{2|z|}\right)+O(d_k)\conv{k\rightarrow \infty}\log\left(\frac{R}{|z|}\right)\, ,
    \end{align*}
    which suggest at the light of the previous discussion that the function above in the right-hand side of \eqref{new_est_Vk1} belongs to a Lorentz space. Although one can effectively prove such an estimate by introducing a sequence of Lorentz spaces \enquote{converging} towards $\mathrm{L}^{2,\infty}_{\log}(B(0,R))$ and generalise Lemma \ref{lemme_holomorphe1} (see the Appendix \ref{app2}), we will not need of this fact. 
    
   Assume that $2|z|>r$. It holds
    \begin{align*}
    	\z^2\vec{U}_k(z)=\frac{1}{2\pi i}\int_{B(0,2|z|)\setminus \bar{B}(0,r)}\frac{\bar{\zeta}^2\vec{Y}_k(\zeta)}{\bar{\zeta-z}}d\bar{\zeta}\wedge d\zeta+\frac{1}{2\pi i}\int_{B(0,R)\setminus\bar{B}(0,2|z|)}\frac{\bar{\zeta}^2\vec{Y}_k(\zeta)}{\bar{\zeta-z}}d\bar{\zeta}\wedge d\zeta.
    \end{align*}
    We first estimate directly 
    \begin{align}\label{n_step1}
    	\left|\frac{1}{2\pi i}\int_{B(0,2|z|)\setminus\bar{B}(0,r)}\frac{\bar{\zeta}^2\vec{Y}_k(\zeta)}{\bar{\zeta-z}}d\bar{\zeta}\wedge d\zeta\right|&\leq \frac{C_0}{\pi}\int_{B(0,2|z|)}\frac{d\zeta|^2}{|\zeta|^{d_k}|z-\zeta|}\nonumber\\
    	&=\frac{C_0}{\pi}\int_{B(0,\frac{|z|}{2})}\frac{|d\zeta|^2}{|\zeta|^{d_k}|z-\zeta|}+\frac{C_0}{\pi}\int_{B(0,2|z|)\setminus\bar{B}(0,\frac{|z|}{2})}\frac{|d\zeta|^2}{|\zeta|^{d_k}|z-\zeta|}.
    \end{align}  
    For all $\zeta\in B(0,\frac{|z|}{2})$, we have by the triangle inequality
    \begin{align*}
    	|z-\zeta|\leq |z|-|\zeta|\geq \frac{|z|}{2},
    \end{align*}
    which implies that
    \begin{align}\label{n_step2}
    	\int_{B(0,\frac{2}{|z|})}\frac{|d\zeta|^2}{|\zeta|^{d_k}|z-\zeta|}\leq \frac{2}{|z|}\int_{B(0,\frac{|z|}{2})}\frac{|d\zeta|^2}{|z|^{d_k}}=\frac{4\pi}{|z|}\int_{0}^{\frac{|z|}{2}}r^{1-d_k}dr=\frac{4\pi}{2-d_k}|z|^{1-d_k},
    \end{align}
    while
    \begin{align}\label{n_step3}
    	\int_{B(0,2|z|)\setminus\bar{B}(0,\frac{|z|}{2})}\frac{|d\zeta|^2}{|\zeta|^{d_k}|z-\zeta|}&\leq \frac{2^{d_k}}{|z|^{d_k}}\int_{B(0,2|z|)\setminus\bar{B}(0,\frac{|z|}{2})}\frac{|d\zeta}{|z-\zeta|}\nonumber\\
    		&\leq \frac{2^{d_k}}{|z|^{d_k}}\int_{B(z,3|z|)}\frac{|d\zeta|^2}{|z-\zeta|}=2^{d_k+1}\cdot 3\pi |z|^{1-d_k}.
    \end{align}
    Since $d_k<1$, the combination of  \eqref{n_step1}, \eqref{n_step2}, \eqref{n_step3} and \eqref{n_step4} gives that 
    \begin{align}\label{n_step4}
    	\left|\frac{1}{2\pi i}\int_{B(0,2|z|)\setminus\bar{B}(0,r)}\frac{\bar{\zeta}^2\vec{Y}_k(\zeta)}{\bar{\zeta-z}}d\bar{\zeta}\wedge d\zeta\right|\leq {16 C_0} |z|^{1-d_k}.
    \end{align}
    The previous argument shows that (for the same constant $a_{k,0}$)
    \begin{align*}
    	\left|-\frac{1}{\pi \z^{2}}\int_{B(0,R)\setminus\bar{B}(0,2|z|)}\frac{\bar{\zeta}^{2}\vec{Y}_k(\zeta)}{\bar{\zeta-z}}|d\zeta|^2-\frac{a_{k,0}}{\z^2}\right|\leq \frac{4C_0}{|z|}\left(1+\frac{1}{d_k}\left(\left(\frac{R}{2|z|}\right)^{d_k}-1\right)\right),
    \end{align*}
    and
    \begin{align*}
    	\left|\vec{U}_k(z)-\frac{a_{k,0}}{\z^2}\right|\leq \frac{4C_0}{|z|}\left(1+\frac{1}{d_k}\left(\left(\frac{R}{2|z|}\right)^{d_k}-1\right)\right)+\frac{16C_0}{|z|^{1+d_k}}.
    \end{align*}
    Finally, we get that for all $z\in B(0,\frac{R}{2})$ the following estimate holds:
    \begin{align}\label{new_pointwise_Vk}
    	|\vec{V}_k(z)|\leq \frac{4C_0}{|z|}\left(1+\frac{1}{d_k}\left(\left(\frac{R}{2|z|}\right)^{d_k}-1\right)\right)+\frac{16C_0}{|z|^{1+d_k}}.
    \end{align}
    
    \textbf{Case 3:} $0<d_k\leq 1$ and $d_k\conv{k\rightarrow \infty}d\in (0,1)$. The estimate \eqref{new_pointwise_Vk} gives that
    \begin{align*}
    	|z|^{d_k}|\vec{V}_k(z)|\leq \frac{4C_0}{|z|}\left(|z|^{d_k}+\frac{1}{d_k}\left(\left(R^{d_k}-1\right)\right)\right)+\frac{16C_0}{|z|}\leq \frac{C}{|z|}
    \end{align*}
    which shows that our needed estimate holds uniformly in $k\in \N$. 
    
    \textbf{Case 4:} $d_k\conv{k\rightarrow \infty}d\geq 1$. By the previous cases, we can also assume that $d_k\geq 1$ for all $k\in \N$ too. Let $a\in \N$ such that $a\leq d_k<a+1$, and define 
    \begin{align*}
    	\vec{U}_k(z)=\frac{1}{\z^{a+2}}\left(-\frac{1}{\pi\bar{\zeta}}\ast \left(\bar{\zeta}^{a+2}\vec{Y}_k(\zeta)\right)\right)=-\frac{1}{\pi \z^{a+2}}\int_{\C}\frac{\bar{\zeta}^{a+2}\vec{Y}_k(\zeta)}{\bar{z-\zeta}}|d\zeta|^2.
    \end{align*}
    Observe that $\p{z}\vec{U}_k=\vec{Y}_k$. Fix some $z\in B(0,R)$, where we recall that $r=2r_k$ and $R=\dfrac{R_k}{2}$. 
    \\ For $0<2|z|\leq r$, it holds 
    \begin{align*}
    	\vec{U}_k(z)&=-\frac{1}{\pi \z^{a+2}}\int_{B(0,R)\setminus\bar{B}(0,2|z|)}\frac{\bar{\zeta}^{a+2}\vec{Y}_k(\zeta)}{\bar{\zeta-z}}|d\zeta|^2=\frac{1}{\z^{a+2}}\sum_{l=0}^{\infty}\left(\frac{1}{\pi}\int_{B(0,R)\setminus\bar{B}(0,2|z|)}\bar{\zeta}^{a-l+1}\vec{Y}_k(\zeta)|d\zeta|^2\right)\z^l\\
	&=\sum_{l=0}^{\infty}a_{k,l}\z^{l-a-2}. 
    \end{align*} 
    First, we have
    \begin{align*}
    	|a_{k,0}|\leq \frac{C_0}{\pi}\int_{B(0,R)\setminus\bar{B}(0,2|z|)}|\zeta|^{a-d_k-1}|d\zeta|^2=\frac{C_0}{a+1-d_k}\left(R^{a+1-d_k}-(2|z|)^{a+1-d_k}\right). 
    \end{align*}
    Then, for $l=1$:
    \begin{align*}
    	\left|\frac{1}{\pi}\int_{B(0,R)\setminus\bar{B}(0,2|z|)}\bar{\zeta}^{a-l+1}\vec{Y}_k(\zeta)|d\zeta|^2\right|&\leq \frac{C_0}{\pi}\int_{B(0,R)\setminus\bar{B}(0,2|z|)}\frac{|d\zeta|^2}{|\zeta|^{2+d_k-a}}\\
    	&=\left\{\begin{alignedat}{2}
    		&2C_0\log\left(\frac{R}{2|z|}\right)\qquad&& \text{if}\;\, d_k=a\\
    		&\frac{2C_0}{d_k-a}\left(\frac{1}{(2|z|)^{d_k-a}}-\frac{1}{R^{d_k-a}}\right)\qquad&&\text{if}\;\,d_k>a.
    	\end{alignedat}\right.
    \end{align*}
    For $l\geq 2$, we have:
    \begin{align*}
    	\left|\frac{1}{\pi}\int_{B(0,R)\setminus\bar{B}(0,2|z|)}\zeta^{a-l+1}\vec{Y}_k(\zeta)|d\zeta|^2\right|&\leq \frac{C_0}{\pi}\int_{B(0,R)\setminus\bar{B}(0,2|z|)}\frac{|d\zeta|^2}{|\zeta|^{d_k+l+1-a}}\\
    	& \hspace{-5em} =\frac{2C_0}{d_k+l-a-1}\left(\frac{1}{(2|z|)^{d_k+l-a-1}}-\frac{1}{R^{d_k+l-a-1}}\right)\leq \frac{2C_0}{l-1}\frac{1}{(2|z|)^{d_k-a-1}}\frac{1}{(2|z|)^{l}}\, .
    \end{align*}
    Therefore
    \begin{align*}
    	\left|\sum_{l=2}^{\infty}a_{k,l}z^l\right|\leq 2C_0(2|z|)^{a+1-d_k}\sum_{l=2}^{\infty}\frac{1}{l-1}\frac{1}{2^l}=C_0(2|z|)^{a+1-d_k}\sum_{l=0}^{\infty}\frac{1}{l+1}\frac{1}{2^{l+1}}=\log(2)C_0(2|z|)^{a+1-d_k}. 
    \end{align*}
    We deduce that for $d_k=a$, it holds
    \begin{align}\label{case1_Uk}
    	\left|\vec{U}_k(z)-\frac{a_{k,0}}{\z^{a+2}}\right|\leq \frac{4C_0}{|z|^{d_k+1}}+2C_0\frac{1}{|z|^{d_k+1}}\log\left(\frac{R}{2|z|}\right),
    \end{align}
    while for $d_k>a$, 
    \begin{align}\label{case2_Uk}
    	\left|\vec{U}_k(z)-\frac{a_{k,0}}{\z^{a+2}} \right|\leq \frac{4C_0}{|z|^{1+d_k}}+\frac{2C_0}{d_k-a}\frac{1}{|z|^{a+1}}\left(\frac{1}{|z|^{d_k-a}}-\left(\frac{2}{R}\right)^{d_k-a}\right).
    \end{align}
    Now, if $2|z|>r$, we can make the same decomposition 
    \begin{align*}
    	\z^{a+2}\vec{U}_k(z)=\frac{1}{2\pi i}\int_{B(0,2|z|)\setminus \bar{B}(0,r)}\frac{\bar{\zeta}^{a+2}\vec{Y}_k(\zeta)}{\bar{\zeta-z}}d\bar{\zeta}\wedge d\zeta+\frac{1}{2\pi i}\int_{B(0,R)\setminus\bar{B}(0,2|z|)}\frac{\bar{\zeta}^{a+2}\vec{Y}_k(\zeta)}{\bar{\zeta-z}}d\bar{\zeta}\wedge d\zeta.
    \end{align*}
    The second integral is estimated as above and we get 
    \begin{align}\label{gen_step1}
    	\left|\frac{1}{2\pi i}\int_{B(0,R)\setminus\bar{B}(0,2|z|)}\frac{\bar{\zeta}^2\vec{Y}_k(\zeta)}{\bar{\zeta-z}}d\bar{\zeta}\wedge d\zeta-\frac{a_{k,0}}{\z^{a+2}}\right|\leq \frac{4C_0}{|z|^{1+d_k}}+\frac{2C_0}{d_k-a}\frac{1}{|z|^{a+1}}\left(\frac{1}{|z|^{d_k-a}}-\left(\frac{2}{R}\right)^{d_k-a}\right).
    \end{align}
    Then, we have
    \begin{align*}
    	\left|\frac{1}{2\pi i}\int_{B(0,2|z|)\setminus\bar{B}(0,r)}\frac{\bar{\zeta}^{a+2}\vec{Y}_k(\zeta)}{\bar{\zeta-z}}|d\zeta|^2\right|\leq \frac{C_0}{\pi}\int_{B(0,2|z|)\setminus\bar{B(0,r)}}\frac{|d\zeta|^2}{|\zeta|^{d_k-a}|\zeta-z|}\leq 16C_0|z|^{a+1-d_k}
    \end{align*}
    using the same proof given the estimate \eqref{n_step4}. 
    Finally, we deduce that
    \begin{align}\label{final_estimate}
    	|\vec{V}_k(z)|\leq \frac{4C_0}{|z|}\left(1+\frac{1}{d_k}\left(\left(\frac{R}{2|z|}\right)^{d_k}-1\right)\right)+\frac{16C_0}{|z|^{1+d_k}}.
    \end{align}
    
    \textbf{Sub-case 1:} $d_k\conv{k\rightarrow \infty}d$ with $a<d\leq a+1$. Then the estimate \eqref{final_estimate} immediately implies the bound 
    \begin{align*}
    	|z|^{d_k}|\vec{V}_k(z)|\leq \frac{C}{|z|}, \quad \text{  for all $z\in \Omega_k\left(\frac{1}{2}\right)$.}
    \end{align*}

    \textbf{Sub-case 2:} $d_k=a$. Then we can apply the same argument on averages as in the case $a=-1$ and the rest of the proof is unchanged (notice that the proof of the estimate on $\Im(\vec{V}_k)$ is independent of the multiplicity $d_k$).
    
    \textbf{Sub-case 3:} $d_k\conv{k\rightarrow \infty}a$ and $d_k>a$ for all $k\in \N$ (notice that it covers all cases by the previous case). Then we apply the same averaging argument as before the statement of Lemma \ref{sobolev11} and find
    \begin{align*}
    	r^{1+d_k}a(r)\leq \left(\frac{R_k}{2}\right)^{1+d_k}a\left(\frac{R_k}{2}\right)+C,
    \end{align*}
    which shows by \eqref{case2_Uk} that 
    \begin{align*}
   	    r^{1+d_k}a(r)\leq 4C_0+C.
    \end{align*}
    Notice that if the estimate had been replaced by an estimate on $\Omega_k(\alpha_0)$ for all $0<\alpha_0<\dfrac{1}{2}$, we would have had instead
    \begin{align*}
    	r^{1+d_k}a(r)\leq 4C_0+\frac{2C_0}{d_k-a}\frac{1}{(\alpha_0R_k)^{a+1}}\left(\frac{1}{(\alpha_0 R_k)^{d_k-a}}-\left(\frac{2}{R_k}\right)^{d_k-a}\right).
    \end{align*}
    Since $d_k\conv{k\rightarrow \infty}a$ and $R_k\conv{k\rightarrow \infty}R\in (0,\infty)$, we deduce that 
    \begin{align*}
    	\frac{1}{(\alpha_0R_k)^{d_k-a}}=e^{-(d_k-a)\log(\alpha _0R_k)}=1-(d_k-a)\log (\alpha_0R_k)+O\left((d_k-a)^2\right)
    \end{align*}
    which gives that 
    \begin{align*}
    	\frac{1}{d_k-a}\left(\frac{1}{(\alpha_0 R_k)^{d_k-a}}-\left(\frac{2}{R_k}\right)^{d_k-a}\right)\conv{k\rightarrow \infty} \log\left(\frac{1}{2\alpha_0}\right),
    \end{align*}
    showing that 
    \begin{align*}
    	\limsup_{k\in \rightarrow\infty}\left(4C_0+\frac{2C_0}{d_k-a}\frac{1}{(\alpha_0R_k)^{a+1}}\left(\frac{1}{(\alpha_0 R_k)^{d_k-a}}-\left(\frac{2}{R_k}\right)^{d_k-a}\right)\right)=4C_0+\frac{2C_0}{(\alpha_0R)^{a+1}}\log\left(\frac{1}{2\alpha_0}\right)<\infty.
    \end{align*}
    
    In the next proposition, we summarise the point-wise estimates on $\vec{V}_k$ proved in this section.
    
    \begin{prop}\label{prop:summarySec}
    Under the hypothesis of Theorem \ref{L21_necks}, there exists $C=C(m,h,\Lambda)>0$ and $\alpha_0>0$ such that 
     $\vec{V}_k$ defined in \eqref{def_Vk}  satisfies the pointwise estimates
    \begin{align}\label{final_V}
        \left\{\begin{alignedat}{1}	e^{\lambda_k(z)}|\vec{V}_k(z)| &\leq \frac{C}{|z|}, \quad \text{for all $z\in \Omega_k(\alpha_0)$}\, , \\
        |z|^{d_k} \,  | \vec{V}_k(z)|& \leq \frac{C}{|z|}, \quad \text{for all $z\in \Omega_k(\alpha_0)$}\, .   
        \end{alignedat}\right.
    \end{align}   
   \end{prop}

 \subsection*{Estimates on $\vec{W}_{k}$ and  $\vec{L}_{k}$ }
 
 \noindent
    First,  will construct a function $\vec{U}_k$ on $B(0,\alpha_0R_k)$, for $\alpha_0>0$,  such that 
    \begin{equation}\label{eq:DzUk}
    	\D_z\vec{U}_k=\p{z}\vec{U}_k+\vec{f}_k(\vec{U}_k)=\p{z}\vec{V}_k\, , \quad \text{on $\Omega_k(\alpha_0)$}.
    \end{equation}
    Note this is equivalent to 
    \begin{align*}
    	\vec{U}_k=\vec{V}_k-\frac{1}{\pi\z}\ast f_k(\vec{U}_k)\, , \quad \text{on $\Omega_k(\alpha_0)$}.
    \end{align*}  
    This will be achieved by a fixed point argument similar to the one of  \cite[Lemma A.$1$]{mondinoriviere}. 
    
    First, extend by $0$ the restriction of $\vec{V}_k$ to $B(0,\alpha_0 R_k)$, and using a smooth non-negative cut-off $\eta$ such that $\eta=1$ on $B(0,\alpha_0R_k)$ and $\supp(\eta)\subset B(0,2\alpha_0R_k)$, we can assume that $\supp(\gamma_{j}^l)\subset B(0,2\alpha_0R_k)$ for all $1\leq j,l\leq m$. Then, define $\tilde{\lambda}_k:\C\rightarrow \R$ such that 
    \begin{align*}
    	\tilde{\lambda}_k(z)=\lambda_k(z)\eta(z)-(1-\eta(z))\log(1+|z|^2).
    \end{align*}
    Consider the normed space
    \begin{align*}
    	\mathrm{L}^{2,\infty}_{\tilde{\lambda}_k}(\C)=\mathrm{L}^{1}_{\mathrm{loc}}(\C)\cap\ens{u:e^{\tilde{\lambda}_k}u\in \mathrm{L}^{2,\infty}(\C)}\, , \quad \text{for all $k\in\N$},
    \end{align*}
    equipped with the norm 
    \begin{align*}
    	\znp{u}{2,\infty}{\tilde{\lambda}_k}{\C}=\np{e^{\tilde{\lambda}_k}u}{2,\infty}{\C}. 
    \end{align*}
    One checks immediately that $\mathrm{L}^{2,\infty}_{\tilde{\lambda}_k}(\C)$ is a Banach space (for example, by the series criterion using that $\mathrm{L}^{2,\infty}(\C)$ is a Banach space). Now, by the previous estimate \eqref{final_V}, we deduce that $\vec{V}_k\in \mathrm{L}^{(2,\infty)}_{\tilde{\lambda}_k}(\C)$. 
    \\ We introduce the operator $T:\mathrm{L}^{2,\infty}_{\tilde{\lambda}_k}(\C) \rightarrow  \mathrm{L}^{(2,\infty)}_{\tilde{\lambda}_k}(\C)$, defined by 
    \begin{equation*}
    	T(\vec{U})=\vec{V}_k-\frac{1}{\pi\z}\ast f_k(\vec{U}). 
    \end{equation*}
    Using the Young inequality for convolution, for all $2<p<\infty$, if $q<2$ is such that 
    \begin{align*}
    	\frac{1}{p}=\frac{1}{2}+\frac{1}{q}-1,
	    \end{align*}
    using the support assumption on $\gamma_j^l$ and that $|f_k(\vec{U}_k)|\leq C_0e^{\lambda_k}|\vec{U}_k|$, 
    we deduce that 
    \begin{align}\label{young}
    	&\np{\frac{1}{\pi\z}\ast f_k(\vec{U})}{p,\infty}{\C}\leq C(p)\np{\frac{1}{\pi\z}}{2,\infty}{\C}\np{f_k(\vec{U}_k)}{\frac{2p}{p+2},\infty}{\C}=2\sqrt{\pi}C(p)\np{f_k(\vec{U}_k)}{\frac{2p}{p+2}}{B(0,2\alpha_0R_k)}\nonumber\\
    	&\leq 2\sqrt{\pi}C(p)\np{1}{\frac{p+2}{p}}{B(0,2\alpha_0R_k)}\np{f_k(\vec{U}_k)}{2,\infty}{B(0,2\alpha_0R_k)}\nonumber\\
    	&\leq 2\sqrt{\pi}C_0C(p)\pi^{\frac{p}{p+2}}(2\alpha_0R_k)^{\frac{2p}{p+2}}\znp{\vec{U}}{2,\infty}{\tilde{\lambda}_k}{\C}. 
    \end{align}
    Since $e^{\lambda_k}\in \mathrm{L}^p(B(0,R_k))$ for all $p<\dfrac{2}{1-\epsilon}$, using Hölder's inequality we deduce that  for all $2<p<\dfrac{2}{1-\epsilon}$ it holds
    \begin{align*}
    	&\znp{T(\vec{U})-\vec{V}_k}{2,\infty}{\tilde{\lambda}_k}{\C}=\np{e^{\tilde{\lambda}_k}\left(T(\vec{U})-\vec{V}_k\right)}{2,\infty}{\C}\leq C(p)\np{e^{\tilde{\lambda}_k}}{p,\infty}{\C}\np{T(\vec{U})-\vec{V}_k}{\frac{p}{p-2},\infty}{\C}\\
    	&\leq C(p)\left(C_{\alpha_0}\np{r_k^{d_k}}{p}{B(0,\alpha_0^{-1}r_k)}+e^{A}\np{|z|^{d_k}}{p}{B(0,2\alpha_0R_k)}+\np{\frac{1}{1+|z|^2}}{2}{\C}\right)\np{\frac{1}{\pi\z}\ast f_k(\vec{U})}{\frac{p}{p-2},\infty}{\C}\\
    	&\leq C(p)\left(C_{\alpha_0}(2\pi)^{\frac{1}{p}}r_k^{\frac{2}{p}+d_k}+(2\pi)^{\frac{1}{p}}\frac{(2\alpha_0R_k)^{\frac{2}{p}+d_k}}{(2+pd_k)^{\frac{1}{p}}}+\sqrt{\pi}\right)\\
    	&\times 2\sqrt{\pi}C_0C\left(\frac{p}{p-2}\right)\pi^{\frac{p-2}{3p-4}}(2\alpha_0R_k)^{\frac{2p-4}{3p-4}}\znp{\vec{U}}{2,\infty}{\tilde{\lambda}_k}{\C}.
    \end{align*}
    In particular, choosing $p=2+\epsilon<\frac{2}{1-\epsilon}$, we obtain that there exists a constant $C(\epsilon)>0$ such that 
    \begin{align*}
    	\znp{T(\vec{U})-\vec{V}_k}{2,\infty}{\tilde{\lambda}_k}{\C}\leq C(\epsilon)\alpha_0^{\frac{\epsilon}{4}}\znp{\vec{U}}{2,\infty}{\tilde{\lambda}_k}{\C}.
    \end{align*}
    Therefore, taking $\alpha_0=\left(\frac{1}{2C(\epsilon)}\right)^{\frac{4}{\epsilon}}$, we get that 
    \begin{align*}
    	\znp{T(\vec{U})-\vec{V}_k}{2,\infty}{\tilde{\lambda}_k}{\C}\leq \frac{1}{2}\znp{\vec{U}}{2,\infty}{\tilde{\lambda}_k}{\C}.
    \end{align*}
    Since $\vec{V}_k\in \mathrm{L}^{2,\infty}_{\tilde{\lambda}_k}(\C)$, we conclude that $T:\mathrm{L}^{2,\infty}_{\tilde{\lambda}_k}(\C)\rightarrow \mathrm{L}^{2,\infty}_{\tilde{\lambda}_k}(\C)$ is a contraction which implies in particular that $T$ admits a unique fixed point $\vec{U}_k$.
    
    The function $\vec{U}_k\in \mathrm{L}^{2,\infty}_{\tilde{\lambda_k}}(\C)$ satisfies 
    \begin{align}\label{new_eq}
    	\vec{U}_k=\vec{V}_k-\frac{1}{\pi\z}\ast f_k(\vec{U}_k).
    \end{align}
    Since $f_k(\vec{U}_k)\in \mathrm{L}^{2,\infty}(B(0,\alpha_0R_k))$, the previous estimate \eqref{young} shows that 
    \begin{align*}
    	\frac{1}{\pi\z}\ast f_k(\vec{U}_k)\in \bigcap_{p<\infty}\mathrm{L}^p(\C).
    \end{align*}
    Since $e^{\lambda_k}\in \mathrm{L}^p(B(0,\alpha_0R_k))$ for all $p<\dfrac{2}{1-\epsilon}$, we deduce by Hölder's inequality that 
    \begin{align*}
    	e^{\lambda_k}\,\left(\frac{1}{\pi\z}\ast f_k(\vec{U}_k)\right)\in \bigcap_{p<\frac{2}{1-\epsilon}}\mathrm{L}^p(B(0,\alpha_0R_k)). 
    \end{align*}
    Now, recalling the pointwise estimate \eqref{final_V}, we deduce that there exists $f\in \mathrm{L}^p(B(0,\alpha_0R_k
    ))$ for all $p<\frac{2}{1-\epsilon}$ such that
    \begin{align*}
    	e^{\lambda_k}|\vec{U}_k|\leq \frac{C}{|z|}+f(z)\qquad\text{for all}\;\, z\in \Omega_k(\alpha_0). 
    \end{align*}
    Then, using an estimate presented in the proof of \cite[Lemma A.$2$]{mondinoriviere}, we deduce that 
    \begin{align*}
    	\np{\D\left(\frac{1}{\pi\z}\ast f_k(\vec{U}_k)\right)}{2,\infty}{\C}\leq C\np{f_k(\vec{U}_k)}{2,\infty}{\C}\leq C. 
    \end{align*}
    Since $\D\Im(\vec{V}_k)\in  \mathrm{L}^{2,\infty}(\C)$, by  \eqref{new_eq} we conclude that 
    \begin{align*}
    	\D\Im(\vec{U}_k)=\D\Im(\vec{V}_k)-\Im\left(\D\left(\frac{1}{\pi\z}\ast f_k(\vec{U}_k)\right)\right)\in \mathrm{L}^{2,\infty}(\C). 
    \end{align*}

    Now, we can apply the exact same proof as Lemma A.$1$ (see also Lemma \ref{A2} for more details) to obtain by a similar contraction argument as above the existence of $\vec{W}_k\in \mathrm{W}^{1,(2,\infty)}(\C)$ such that
    \begin{align*}
    	\left\{\begin{alignedat}{1}
    		\D_z\vec{W}_k&=0\qquad\text{in}\;\, B(0,\alpha_0R_k)\\
    		\Im(\vec{W}_k)&=\Im(\vec{U}_k)\qquad\text{on}\;\,\partial B(0,\alpha_0R_k). 
    	\end{alignedat}\right.
    \end{align*}
    Finally, defining $\vec{L}_k=\vec{U}_k-\vec{W}_k:B(0,\alpha_0R_k)\rightarrow \C^n$, we deduce that
    \begin{align*}
    	\left\{\begin{alignedat}{1}
    		&e^{\lambda_k}\vec{L}_k\in  \mathrm{L}^{2,\infty}(B(0,\alpha_0R_k))\\
    		&\Im(\vec{L}_k)\in \mathrm{W}^{1,(2,\infty)}(B(0,\alpha_0R_k))
    	\end{alignedat} \right.
    \end{align*}
    and
    \begin{align*}
    	\left\{\begin{alignedat}{2}
    		&\D_z\vec{L}_k=\vec{Y}_k\qquad&&\text{in}\;\, B(0,\alpha_0R_k)\\
    		&\Delta\Im\left(\vec{L}_k\right)=4\,\Im\left(\D_{\z}\vec{Y}_k\right)-4\,\Im\left(f_k(\vec{Y}_k)\right)-4\,\Im\left(\p{\z}\left(f_k(\vec{U}_k)\right)\right)\qquad&&\text{in}\;\,B(0,\alpha_0R_k)\\
    		&\Im(\vec{L}_k)=0\qquad&&\text{on}\;\,\partial B(0,\alpha_0R_k).
    	\end{alignedat}\right.
    \end{align*}
    
\noindent
    Furthermore, there exists 
    $l_k$     bounded in $\mathrm{L}^p(B(0,\alpha_0R_k))$ for all $p<\dfrac{2}{1-\epsilon}$ such that 
    \begin{align}\label{pointwise_Lk}
    	e^{\lambda_k(z)}|\vec{L}_k(z)|\leq \frac{C}{|z|} l_k(z), \quad \text{on }\Omega_k(\alpha_0)\, .
    \end{align}

    Now, we introduce a refinement of \cite[Lemma A.$2$]{mondinoriviere}.
    \begin{lemme}\label{A2}
    	Let $2<p<\infty$ and $1<q<\infty$ be fixed real numbers. There exists constants $\epsilon_3(n,p,q),C_3(n,p,q)>0$ with the following property. 
    	For all $j,k\in \ens{1,\cdots, n}$, let $\gamma_{j}^k\in \mathrm{L}^p(\C)\cap \mathrm{W}^{1,(2,\infty)}(\C)$ be such that $\supp(\gamma_{j}^k)\subset B(0,2)$ and $\np{\gamma_{j}^k}{2,1}{\C}\leq \epsilon_0$. For all $\vec{U}\in \mathrm{L}^{1}_{\mathrm{loc}}(\C)$, define
    	\begin{align*}
    		\left(\D_z \vec{U}\right)_j&=\p{z}\vec{U}_j+\sum_{k=1}^{n}\gamma_{j}^k\vec{U}_k\qquad\text{in}\;\,\mathscr{D}'(\C),\;\, \text{where}\;\, 1\leq j\leq n. 
    	\end{align*}
    	Then for all $\vec{Y}\in \left(\mathrm{L}^1\cap \mathrm{L}^{2,\infty}\right)(\C)$,
    	if $\np{\gamma_j^k}{p}{B(0,2)}\leq \epsilon_3$ there exists a unique $\vec{U}\in \mathrm{W}^{1,(2,\infty)}(B(0,1))$ such that $\Im(\vec{U})\in \mathrm{W}^{2,q}(B(0,1))$ satisfying
    	\begin{align*}
    		\left\{\begin{alignedat}{2}
    			\D_z\vec{U}&=\vec{Y}\qquad&&\text{in}\;\,\mathscr{D}'(B(0,1))\\
    			\Im(\vec{U})&=0\qquad&&\text{on}\;\,\partial B(0,1). 
    		\end{alignedat}
    		\right.
    	\end{align*}
    	Furthermore, we have the estimate
    	\begin{align*}
    		\wp{\vec{U}}{1,(2,\infty)}{B(0,1)}
    		\leq C_3
    		\left(
    		\np{\vec{Y}}{1}{\C}+\np{\vec{Y}}{2,\infty}{\C}\right)\,.
    	\end{align*}
    \end{lemme}
    \begin{proof}
    	As in the proof of Lemma \ref{A1}, the same argument shows that if $T$ is defined analogously, we have for all $\vec{U}\in \mathrm{L}^{2,\infty}(\C)$ the estimate 
    	\begin{align*}
    		\np{T(\vec{U})}{2,\infty}{\C}\leq C\np{\vec{Y}}{1}{\C}
    	\end{align*}
    	and that for all $\vec{U}_1,\vec{U}_2\in \mathrm{L}^{2,\infty}(\C)$ if $\epsilon_3\leq \delta_0$ is small enough, we have
    	\begin{align*}
    		\np{T(\vec{U}_1)-T(\vec{U}_2)}{2,\infty}{\C}\leq \frac{1}{4}\np{\vec{U}_1-\vec{U}_2}{2,\infty}{\C}.
    	\end{align*}
    	By the same argument as \cite{mondinoriviere}, we deduce that 
    	\begin{align*}
    		\np{\D T(\vec{U})_j}{2,\infty}{\C}&\leq \Gamma_2\np{\vec{Y}_j-\sum_{k=1}^{n}\gamma_k^k\vec{U}_k}{2,\infty}{\C}
    		\leq \Gamma_2\np{\vec{Y}_j}{2,\infty}{\C}+\sum_{j=1}^{n}\Gamma_2\np{\gamma_j^k\vec{U}_k}{2}{B(0,2)}\\
    		&\leq \Gamma_2\np{\vec{Y}}{2,\infty}{\C}+\sum_{j=1}^{n}\Gamma_2\np{\gamma_j^k}{p}{B(0,2)}\np{\vec{U}_k}{\frac{2p}{p-2}}{B(0,2)}.
    	\end{align*}
    	By the Sobolev embedding if $q<2$, we have $\mathrm{W}^{1,q}(B(0,2))\hooklongrightarrow \mathrm{L}^{q^{\ast}}(B(0,2))$, where
    	\begin{align*}
    		\frac{1}{q^{\ast}}=\frac{1}{q}-\frac{1}{2}=\frac{2-q}{2q}.
    	\end{align*}
    	Since we want to impose $q^{\ast}=\dfrac{2p}{p-2}$ where $p>2$, this implies that $q=p'=\frac{p}{p-1}$
    	Therefore, we deduce that there exits a universal constant $\Gamma_4$ such that
    	\begin{align*}
    		\np{\vec{U}_k}{\frac{2p}{p-2}}{B(0,2)}\leq \Gamma_4\left(\np{\vec{U}_k}{\frac{p}{p-1}}{B(0,2)}+\np{\D\vec{U}_k}{\frac{p}{p-1}}{B(0,2)}\right). 
    	\end{align*}
    	Now, let $r>1$ such that $rp'=2$, \emph{i.e.} $r=\dfrac{2}{p'}=\dfrac{2(p-1)}{p}$. By the $\mathrm{L}^{\frac{2(p-1)}{p-2},1}/\mathrm{L}^{\frac{2(p-1)}{p},\infty}$ duality, we deduce that 
    	\begin{align*}
    		\np{\vec{U}_k}{\frac{p}{p-1}}{B(0,2)}\leq \np{1}{\frac{2(p-1)}{p},1}{B(0,2)}\np{\vec{U}_k}{2,\infty}{B(0,2)}=2^{\frac{3p-4}{p-1}}\frac{(p-1)^2}{p(p-2)}\pi^{\frac{p-2}{2(p-1)}}\np{\vec{U}_k}{2,\infty}{B(0,2)},
    	\end{align*}
    	and likewise
    	\begin{align*}
    		\np{\D\vec{U}_k}{\frac{p}{p-1}}{B(0,2)}\leq 2^{\frac{3p-4}{p-1}}\frac{(p-1)^2}{p(p-2)}\pi^{\frac{p-2}{2(p-1)}}\pi^{\frac{p}{2(p-2)}}\np{\D\vec{U}_k}{2,\infty}{B(0,2)}.
    	\end{align*}
    	Therefore, we get
    	\begin{align*}
    		\np{\D T(\vec{U})}{2,\infty}{\C}&\leq \Gamma_2n\np{\vec{Y}}{2,\infty}{\C}\\
    		&+\Gamma_2n^2 2^{\frac{3p-4}{p-1}}\frac{(p-1)^2}{p(p-2)}\pi^{\frac{p-2}{2(p-1)}}\epsilon_4\left(\np{\vec{U}_k}{2,\infty}{B(0,2)}+\np{\D\vec{U}_k}{2,\infty}{B(0,2)}\right),
    	\end{align*}
    	and for all $\vec{U}_1,\vec{U}_2\in \mathrm{W}^{1,(2,\infty)}(\C)$, the estimate 
    	\begin{align*}
    		\np{\D T(\vec{U}_1)-\D T(\vec{U}_2)}{2,\infty}{\C}\leq \Gamma_2n^2 2^{\frac{3p-4}{p-1}}\frac{(p-1)^2}{p(p-2)}\pi^{\frac{p-2}{2(p-1)}}\pi^{\frac{p}{2(p-2)}}\epsilon_4\wp{\vec{U}_1-\vec{U}_2}{1,(2,\infty)}{\C}.
    	\end{align*}
    	Therefore, taking 
    	\begin{align*}
    		\epsilon_4=\min\ens{\delta_0, \frac{1}{4}\left(\Gamma_2n^22^{\frac{3p-4}{p-1}}\frac{(p-1)^2}{p(p-2)}\pi^{\frac{p-2}{2(p-1)}}\right)^{-1}},
    	\end{align*}
    	we deduce that 
    	\begin{align*}
    		\wp{T(\vec{U}_1)-T(\vec{U}_2)}{1,(2,\infty)}{\C}\leq \frac{1}{2}\wp{\vec{U}_1-\vec{U}_2}{1,(2,\infty)}{\C}.
    	\end{align*}
    	Therefore, $T$ has a unique fixed point that we denote $\vec{U}_0$. Define by $f:\C^n\rightarrow \C^n$ the linear map such that for all $\vec{X}\in C^{\infty}(\C,\C^n)$, we have
    	\begin{align*}
    		\D_z\vec{X}=\p{z}\vec{X}+f(\vec{X}). 
    	\end{align*}
    	Now, to get the boundary condition, define the operator $S:\mathrm{W}^{1,(2,\infty)}(B(0,1),\C^n)\rightarrow \mathrm{W}^{1,(2,\infty)}(B(0,1),\C^n)$ such that for all $\vec{U}\in \mathrm{W}^{1,(2,\infty)}(B(0,1),\C^n)$, $\vec{V}=S(\vec{U})$ is the unique solution of the equation 
    	\begin{align}\label{handnotes}
    		\left\{\begin{alignedat}{2}
    			\p{z}\vec{V}&=-f(\vec{U})\qquad&&\text{in}\;\, B(0,1)\\
    			\Im(\vec{V})&=\Im(\vec{U}_0)\qquad&&\text{on}\;\, \partial B(0,1). 
    		\end{alignedat}\right.
    	\end{align}
    	Since we only prescribe the imaginary part of $\vec{V}$, this system admits a unique solution. First, by the Cauchy formula, if 
    	\begin{align*}
    		\vec{V}_0(z)=\frac{1}{2\pi i}\int_{B(0,1)}\frac{f(\vec{U})}{\bar{\zeta-z}}d{\zeta}\wedge d\bar{\zeta},
    	\end{align*}
    	then we have $\p{z}\vec{V}_0=-f(\vec{U})$ and $\vec{V}_0=0$ on $\partial B(0,1)$. Furthermore, by the previous convolution estimates, we have 
    	\begin{align}\label{intermediate}
    		\wp{\vec{V}_0}{1,(2,\infty)}{B(0,1)}\leq \frac{1}{2}\wp{\vec{U}}{1,(2,\infty)}{B(0,1)}. 
    	\end{align}
    	Now, first solve in $\mathrm{W}^{1,(2,\infty)}(B(0,1),\R^n)$ the equation
    	\begin{align*}
    		\left\{\begin{alignedat}{2}
    			\Delta \vec{V}_1&=0\qquad&&\text{in}\;\, B(0,1)\\
    			\vec{V}_1&=\Im(\vec{U}_0)\qquad&&\text{on}\;\,\partial B(0,1).
    		\end{alignedat} \right.
    	\end{align*}
    	Then we have by Calder\'{o}n-Zygmund estimate $\vec{V}_1\in \mathrm{W}^{1,(2,\infty)}(B(0,1))$. If $\vec{V}_2$ is the harmonic conjugate of $\vec{V}_1$, then by interpolation theory, we also deduce that $\vec{V}_2\in \mathrm{W}^{1,(2,\infty)}(B(0,1))$ and that for some universal constant $C$, we have
    	\begin{align*}
    		\wp{\vec{V}_1}{1,(2,\infty)}{B(0,1)}+\wp{\vec{V}_2}{1,(2,\infty)}{B(0,1)}\leq C\wp{\Im(\vec{U}_0)}{1,(2,\infty)}{B(0,1)}. 
    	\end{align*}
    	By construction, the function $\vec{V}_3=\vec{V}_1+i\vec{V}_2=i(\vec{V}_2-i\vec{V}_1)$ is holomorphic, which implies that $\vec{V}_4=i\bar{\vec{V}_3}=\vec{V}_2+i\vec{V}_1$ is anti-holomorphic, \emph{i.e.} $\p{z}\vec{V}_4=0$. Furthermore, we have by construction $\Im(\vec{V}_4)=\vec{V}_1=\Im(\vec{U}_0)$ on $\partial B(0,1)$. Therefore, the function $\vec{V}_0+\vec{V}_4$ is the unique solution to the system \eqref{handnotes}. Furthermore, by the estimate \eqref{intermediate}, $S$ is a contraction, so we get a unique fixed point $\vec{U}_1$ of $S$, which satisfies $\D_z\vec{U}_1=0$ and $\Im(\vec{U}_1)=\Im\left(\vec{U}_0\right)$ on $\partial B(0,1)$. Therefore, the function $\vec{U}=\vec{U}_0-\vec{U}_1$ is the unique solution to the system of the theorem.
    \end{proof}

   Recall the expansion 
    \begin{align*}
    	\vec{L}_k=\vec{V}_k-\vec{W}_k-\frac{1}{\pi\z}\ast f_k(\vec{U}_k),
    \end{align*}
    where $f_k(\vec{U}_k)\in \mathrm{L}^{2,\infty}_{\tilde{\lambda}_k}(\C)$, $\vec{W}_k\in \mathrm{W}^{1,(2,\infty)}(B(0,\alpha_0R_k))$. Moreover,  
    \begin{align*}
    	e^{\lambda_k(z)}|\vec{V}_k(z)|\leq \frac{C}{|z|}\, , \quad \text{for all $z\in \Omega_k(\alpha_0)$}\, .
    \end{align*}
    From \cite[Lemma A.$1$]{mondinoriviere}, we deduce that
    \begin{align*}
    	\frac{1}{\pi\z}\ast f_k(\vec{U}_k)\in \mathrm{W}^{1,(2,\infty)}(\C),
    \end{align*}
    which finally implies that 
    \begin{align*}
    	\vec{L}_k=\vec{V}_k+\tilde{\vec{W}}_k,\quad \text{with $\tilde{\vec{W}}_k\in \mathrm{W}^{1,(2,\infty)}(B(0,\alpha_0R_k))$}\,.
    \end{align*}
    For simplicity, we shall rename $\tilde{\vec{W}}_k$ as $\vec{W}_k$. Summarising, we proved the following:
    
    \begin{theorem}\label{thm:L2infLW}
    	Under the hypothesis of Theorem \ref{L21_necks}, there exists $C_1(m,h,\Lambda)$, $\alpha_0>0$  and a measurable function $\vec{L}_k:B(0,\alpha_0R_k)\rightarrow \C^m$  satisfying the equation
    	\begin{align*}
    		\D_z\vec{L}_k=\vec{Y}_k\qquad\text{on}\;\,\Omega_k(\alpha_0).
    	\end{align*}
        Moreover, the following decomposition holds: $\vec{L}_k=\vec{V}_k+\vec{W}_k$, where 
        \begin{align*}
        	e^{\lambda_k(z)}|\vec{V}_k(z)|\leq \frac{C_1}{|z|}, \quad \text{for all $z\in \Omega_k(\alpha_0)$}, 
        \end{align*}
        and 
        \begin{align*}
        	\Im(\vec{L}_k)\;\,\text{and}\;\,\vec{W}_k\;\,\text{are bounded in $\mathrm{W}^{1,(2,\infty)}(B(0,\alpha_0R_k))$ uniformly in $k\in \N$.} 
        \end{align*}
    \end{theorem}

 \subsection{Conclusion of the proof of Theorem \ref{L21_necks} }\label{last_est}
    
    By Lemma \ref{A2}, we deduce that there exists $S_k\in \mathrm{W}^{1,(2,\infty)}(B(0,\alpha_0 R_k),\C)$ such that 
    \begin{align}\label{systemS}
    	\left\{\begin{alignedat}{2}
    		\p{z}S_k&=\bs{\p{z}\phi_k}{\bar{\vec{L}_k}}\qquad&&\text{in}\;\, B(0,\alpha_0 R_k)\\
    		\Im(S_k)&=0\qquad&&\text{on}\;\,\partial B(0,\alpha_0 R_k)
    	\end{alignedat}\right.
    \end{align}
    Furthermore,  by \cite[($6.7$)]{mondinoriviere} we have
    \begin{align*}
    	\Im\left(\p{\z}\left(\s{\p{z}\phi_k}{\bar{\vec{L}_k}}\right)\right)=-\frac{1}{2}e^{2\lambda_k}\bs{\H_k}{\Im\left(\vec{L}_k\right)}.
    \end{align*}
    Since $\Im(\vec{L}_k)\in \mathrm{L}^q(B(0,\alpha_0R_k))$ for all $q<\infty$, $e^{\lambda_k}\H_k\in \mathrm{L}^2(B(0,\alpha_0R_k))$ and $e^{\lambda_k}\in \mathrm{L}^p(\Omega_k(\alpha_0))$ for all $p<\frac{2}{1-\epsilon}$, we deduce that 
    \begin{align*}
    	e^{2\lambda_k}\bs{\H_k}{\Im\left(\vec{L}_k\right)}\qquad\text{is bounded in}\;\, \mathrm{L}^p(B(0,\alpha_0R_k))\;\, \text{for all}\;\, q<\frac{2}{2-\epsilon}. 
    \end{align*}
    Therefore, since 
    \begin{align*}
    	\left\{\begin{alignedat}{2}
    		\Delta \Im(S_k)&=4\,\Im\left(\p{\z}\left(\s{\p{z}\phi_k}{\bar{\vec{L}_k}}\right)\right) \qquad&&\text{on}\;\,B(0,\alpha_0R_k)\\
    		\Im(S_k)&=0\qquad&&\text{on}\;\,\partial B(0,\alpha_0R_k),
    	\end{alignedat} \right.
    \end{align*}
    the classical Calder\'{o}n-Zygmund estimates give that 
    \begin{align*}
    	\Im(S_k)\;\,\text{is bounded in}\;\, \mathrm{W}^{2,q}(B(0,\alpha_0R_k
    	))\;\,\text{for all}\;\,q<\frac{2}{2-\epsilon}.
    \end{align*}
    By Sobolev embedding, we deduce that 
    \begin{align*}
    	\D\Im(S_k)\;\,\text{is bounded in}\;\ \mathrm{L}^p(B(0,\alpha_0R_k))\qquad\text{for all}\;\, p<\frac{2}{1-\epsilon}. 
    \end{align*}
    Therefore,  for all $z\in \Omega_k(\alpha_0)$ it holds
    \begin{align}\label{local_ReS}
    	|\D\Re(S_k)|&=2|\p{z}\Re(S_k)|=2\left|-\p{z}\Im(S_k)+\s{\p{z}\phi_k}{\bar{\vec{L}_k}}\right|\nonumber\\
    	&\leq |\D\Im(S_k)|+2
    	e^{\lambda_k}|\vec{L}_k|\leq \frac{C}{|z|}+e^{\lambda_k}|g_k(z)|, 
    \end{align}
    where $g_k\in \mathrm{W}^{1,(2,\infty)}(B(0,\alpha_0R_k))$. 
    
    Now, using Lemma \ref{A2}, we deduce that there exists $\vec{R}_k\in \mathrm{W}^{1,(2,\infty)}(B(0,\alpha_0R_k), \Lambda^2\C^m)$ such that 
    \begin{align}\label{systemR}
    	\left\{
    	\begin{alignedat}{2}
    		\D_z\vec{R}_k&=\p{z}\phi_k\wedge \vec{L}_k-2i\,\p{z}\phi_k\wedge \H_k\qquad&&\text{in}\;\,B(0,\alpha R_k)\\
    		\Im(\vec{R}_k)&=0\qquad&&\text{on}\;\,\partial B(0,\alpha R_k).
    	\end{alignedat} \right.
    \end{align}
    Furthermore,  by \cite[($6.8$)]{mondinoriviere} we have
    \begin{align*}
    	\Im\left(\D_{\z}\left(\p{z}\phi_k\wedge \vec{L}_k-2i\p{z}\phi_k\wedge \H_k\right)\right)=-\frac{1}{2}e^{2\lambda_k}\H_k\wedge \Im(\vec{L}_k)\in \mathrm{L}^q(\Omega_k(\alpha_0))\qquad\text{for all}\;\,q<\frac{2}{2-\epsilon}.
    \end{align*}
    This give in turn
    \begin{align*}
    	\Delta\Im(\vec{R}_k)&=4\,\p{\z}\p{z}\vec{R}_k=4\,\Im\left(\p{\z}\left(\p{z}\phi_k\wedge \vec{L}_k-2i\,\p{z}\phi_k\wedge \H_k\right)\right)-4\,\Im\left(\p{\z}\left(F_k(\vec{R}_k)\right)\right)\\
    	&=4\,\Im\left(\D_{\z}\left(\p{z}\phi_k\wedge \vec{L}_k-2i\,\p{z}\phi_k\wedge \H_k\right)\right)-4\,\Im\left(\bar{F_k}\left(\p{z}\phi_k\wedge \Im(\vec{L}_k)-2i\,\p{z}\phi_k\wedge \H_k\right)\right)\\
    	&-4\,\Im\left(\p{\z}\left(F_k(\vec{R}_k)\right)\right). 
    \end{align*}
    Notice that 
    \begin{align*}
    	\left|\bar{F}_k\left(\p{z}\phi_k\wedge \Im(\vec{L}_k)-2i\,\p{z}\phi_k\wedge \H_k\right)\right|\leq Ce^{2\lambda_k}\left(|\vec{L}_k|+|\H_k|\right)\in \mathrm{L}^q(B(0,\alpha_0R_k))\quad \text{for all}\;\,q<\dfrac{2}{2-\epsilon}. 
    \end{align*}
    Recalling the Sobolev embedding
    \begin{align*}
    	\mathrm{W}^{1,(2,\infty)}(B(0,1))\hooklongrightarrow \bigcap_{q<\infty}\mathrm{L}^q(B(0,1)),
    \end{align*}
    we have $\vec{R}_k\in \mathrm{L}^{q}$ for all $q<\infty$. Since $e^{\lambda_k}\in \mathrm{L}^p(\Omega_k(\alpha_0))$ for all $p<\frac{2}{1-\epsilon}$, we deduce that \emph{a fortiori}  it holds
    \begin{align*}
    	|F_k(\vec{R}_k)|\leq Ce^{\lambda_k}|\vec{R}_k|\in \bigcap_{p<\frac{2}{1-\epsilon}}\mathrm{L}^{p}(\Omega_k(\alpha_0))\,.
    \end{align*}
    Thus $\Delta \Im(\vec{R}_k)\in \mathrm{W}^{-1,p}(B(0,\alpha_0R_k))$ for all $p<\dfrac{2}{1-\epsilon}$ and, by Calder\'{o}n-Zgymund estimates, we obtain that $\Im(\vec{R}_k)\in \mathrm{W}^{1,p}(B(0,\alpha_0R_k))$ for all $p<\dfrac{2}{1-\epsilon}$.

    Next, we sharpen the last estimate. To this aim, we first prove a pointwise bound for $\Re(\vec{R}_k)$.
    Writing
    \begin{align*}
    	\vec{R}_k=\vec{R}_k^1\wedge \vec{R}_k^2\, ,
    \end{align*}
    we deduce that
    \begin{align*}
    	\D_z\vec{R}_k&=\p{z}\vec{R}_k+f_k(\vec{R}_{k}^1)\wedge \vec{R}_{k}^2+\vec{R}_{k}^1\wedge f_k(\vec{R}_{k}^2)\\
    	&=\p{z}\vec{R}_k+\sum_{i,j,l=1}^{n}c_{i,j,l}\p{z}\phi_{k,l} \vec{R}_{k,i,j}\,e_i\wedge e_j\, ,
    \end{align*}
    where $(\e_1,\cdots,\e_m)$ is a basis of $\R^m$, and $c_{i,j,l}\in\R $ are finite linear combinations  of Christoffel symbols with integer weights. We compute:
    \begin{align*}
    	&\p{\z}F_k(\vec{R}_k)=\sum_{i,j,l=1}^{m}\left(\p{\z}c_{i,j,l}\p{z}\phi_{k,l}\,\vec{R}_{k,i,j}+\frac{1}{2}c_{i,j,l}e^{\lambda_k}\H_k\cdot e^{\lambda_k}\vec{R}_{k,i,j}+c_{i,j,l}\p{z}\phi_{k,l}\p{\z}\vec{R}_{k,i,j}\right)\,e_i\wedge e_j\\
    	&=\sum_{i,j,l=1}^n\bigg(\frac{1}{4}\left(\p{x}c_{i,j,l}+i\,\p{y}c_{i,j,l}\right)\left(\p{x}\phi_{k,l}-i\,\p{y}\phi_{k,l}\right)\left(\Re\left(\vec{R}_{k,i,j}\right)+i\,\Im\left(\vec{R}_{k,i,j}\right)\right)\\
    	&\quad +\frac{1}{2}c_{i,j,l}e^{\lambda_k}\H_k\cdot\left(\Re\left(\vec{R}_{k,i,j}\right)+i\,\Im\left(\vec{R}_{k,i,j}\right)\right)\\
    	&\quad +\frac{1}{4}c_{i,j,l}\left(\p{x}\phi_{k,l}-i\,\p{y}\phi_{k,l}\right)\left(\p{x}+i\,\p{y}\right)\left(\Re\left(\vec{R}_{k,i,j}\right)+i\,\Im\left(\vec{R}_{k,i,j}\right)\right)\bigg)\,e_i\wedge e_j\\
    	&=\sum_{i,j,l=1}^m\bigg(\frac{1}{4}\left(\s{\D c_{i,j,l}}{\D\phi_{k,l}}+i\,\s{\D c_{i,j,l}}{\D^{\perp}\phi_{k,l}}\right)\left(\Re\left(\vec{R}_{k,i,j}\right)+i\,\Im\left(\vec{R}_{k,i,j}\right)\right)\\
    	&\quad +\frac{1}{2}c_{i,j,l}e^{\lambda_k}\H_k\cdot\left(\Re\left(\vec{R}_{k,i,j}\right)+i\,\Im\left(\vec{R}_{k,i,j}\right)\right)\\
    	&\quad +\frac{1}{4}c_{i,j,l}\left(\p{x}\phi_{k,l}-i\,\p{y}\phi_{k,l}\right)\left(\p{x}\Re\left(\vec{R}_{k,i,j}\right)-\p{y}\Im\left(\vec{R}_{k,i,j}\right)\right. \\
    	&
    	\left.+i\left(\p{y}\Re\left(\vec{R}_{k,i,j}\right)+\p{x}\Im\left(\vec{R}_{k,i,j}\right)\right)\right)\bigg)\,e_i\wedge e_j\\
    	&=\sum_{i,j,l=1}^{m}\bigg(\frac{1}{4}\left(\s{\D c_{i,j,l}}{\D\phi_{k,l}}\Re\left(\vec{R}_{k,i,j}\right)-\s{\D c_{i,j,l}}{\D^{\perp}\phi_{k,l}}\Im\left(\vec{R}_{k,i,j}\right)\right)\\
    	&\quad +\frac{i}{4}\left(\s{\D c_{i,j,l}}{\D^{\perp}\phi_{k,l}}\Re\left(\vec{R}_{k,i,j}\right)+\s{\D c_{i,j,l}}{\D\phi_{k,l}}\Im\left(\vec{R}_{k,i,j}\right)\right)\\
    	&\quad + \frac{1}{2}c_{i,j,l}e^{\lambda_k}\H_k\cdot e^{\lambda_k}\left(\Re\left(\vec{R}_{k,i,j}\right)+i\,\Im\left(\vec{R}_{k,i,j}\right)\right)\\
    	&+\frac{1}{4}c_{i,j,l}\left(\s{\D\phi_k}{\D\Re\left(\vec{R}_{k,i,j}\right)}+\s{\D\phi_k}{\D^{\perp}\Im\left(\vec{R}_{k,i,j}\right)}\right)\\
    	&\quad +\frac{i}{4}c_{i,j,l}\left(-\s{\D\phi_{k,l}}{\D^{\perp}\Re\left(\vec{R}_{k,i,j}\right)}+\s{\D\phi_{k,l}}{\D\Im\left(\vec{R}_{k,i,j}\right)}\right)
    	\bigg)\,e_i\wedge e_j\, .
    \end{align*}
    Therefore, we have
    \begin{align*}
    	&\Im\left(\p{\z}F_k(\vec{R}_k)\right)=\frac{1}{4}\sum_{i,j,l=1}^{n}\bigg(\s{\D c_{i,j,l}}{\D^{\perp}\phi_{k,l}}\Re\left(\vec{R}_{k,i,j}\right)+\s{\D c_{i,j,l}}{\D\phi_{k,l}}\Im\left(\vec{R}_{k,i,j}\right)\\
    	&\qquad +2\,e^{\lambda_k}\H_k\cdot e^{\lambda_k}\Im\left(\vec{R}_{k,i,j}\right)+c_{i,j,l}\left(-\bs{\D\phi_{k,l}}{\D^{\perp}\Re\left(\vec{R}_{k,i,j}\right)}+\bs{\D\phi_{k,l}}{\D\Im\left(\vec{R}_{k,i,j}\right)}\right)
    	\bigg)\,e_i\wedge e_j. 
    \end{align*}
    Since $\vec{R}_{k}\in \mathrm{L}^p(B(0,\alpha_0R_k))$ for all $p<\infty$, we deduce that  for all $i,j,l\in\ens{1,\cdots,m}$
    \begin{align*}
    	\s{\D c_{i,j,l}}{\D^{\perp}\phi_{k,l}}\Re\left(\vec{R}_{k,i,j}\right)+\s{\D c_{i,j,l}}{\D\phi_{k,l}}\Im\left(\vec{R}_{k,i,j}\right)\in \bigcap_{p<\frac{1}{1-\epsilon}}\mathrm{L}^p(B(0,\alpha_0R_k
    	)), 
    \end{align*}
    Indeed, since for some $c_{i,j,l}^{\alpha,\beta,\gamma}\in \N$, we have
    \begin{align*}    
    	c_{i,j,l}(x)=\sum_{\alpha,\beta,\gamma}^{}c_{i,j,l}^{\alpha,\beta,\gamma}\Gamma_{\alpha,\beta}^{\gamma}(\phi_k(x)), \quad\text{}, 
    \end{align*}
    we deduce that
    \begin{align*}
    	|\D c_{i,j,l}|\leq C_0e^{\lambda_k},
    \end{align*}
    so that $\s{\D c_{i,j,l}}{\D^{\perp}\phi_{k,l}}\in \bigcap_{p<\frac{1}{1-\epsilon}}(B(0,\alpha_0R_k))\displaystyle$. 
    From $\Im(\vec{R}_k)\in C^0(B(0,\alpha_0R_k))$, we deduce that 
    \begin{align*}
    	\left|e^{\lambda_k}\H_k\cdot e^{\lambda_k}\Im\left(\vec{R}_{k,i,j}\right)\right|\leq C|z|^{d_k}\delta(|z|)\leq C|z|^{d_k-1}\,.
    \end{align*}
    Since $\D\lambda_k\in \mathrm{L}^{2,\infty}$ and $\D \vec{R}_k\in \mathrm{L}^{2,\infty}$, using the previous $\mathrm{L}^p$ bound on $e^{\lambda_k}$ (see \eqref{lambdak_p}), we get that 
    \begin{align*}
    	\p{\z}F_k(\vec{R}_k)\in \mathrm{L}^q(B(0,\alpha_0R_k
    	))\qquad\text{for all}\;\, q<\dfrac{2}{2-\epsilon}.
    \end{align*}
    Therefore, we have $\Im(\vec{R}_k)\in \mathrm{W}^{2,q}(B(0,\alpha_0R_k
    ))$ for all $q<\dfrac{2}{2-\epsilon}$, and by Sobolev embedding, we get $\D \Im(\vec{R}_k)\in \mathrm{L}^{p}(B(0,\alpha_0R_k))$ for all $p<\frac{2}{1-\epsilon}$, which shows in particular that $\D\Im(\vec{R}_k)\in \mathrm{L}^{2,1}(B(0,\alpha_0R_k))$ and $\Im(\vec{R}_k)\in C^0(B(0,\alpha_0R_k))$, which suffices to our purpose. Indeed, by the pointwise bounds \eqref{pointwise2} and \eqref{pointwise_Lk}, we deduce that
    \begin{align}\label{pointwise_Rk}
    	|\D \Re(\vec{R}_k)|&=2|\p{z}\Re(\vec{R}_k)|\leq e^{\lambda_k}|\vec{L}_k|+2e^{\lambda_k}|\H_k|+2|F_k(\vec{R}_k)|\nonumber\\
    	&\leq \frac{C}{|z|}+f_k(z)+C\delta(|z|)+Ce^{\lambda_k}|\vec{R}_k|
    	\leq \frac{C}{|z|}+h_k(z),
    \end{align}
    where $h_k$ is bounded in $\mathrm{L}^p(B(0,\alpha_0R_k))$ for all $p<\dfrac{2}{1-\epsilon}$.

    Furthermore, $\vec{R}_k$ and $S_k$ solve the system
    \begin{align}\label{systemRS1}
    	\left\{\begin{alignedat}{2}
    		\D_z\vec{R}_k&=i\left((-1)^{n+1}\star_h\left(\n_k\antires \D_z\vec{R}_k\right)+(\p{z}S_k)\star_h\n_k\right)\qquad&&\text{in}\;\,B(0,\alpha R_k)\\
    		\p{z}S_k&=-i\s{\D_z\vec{R}_k}{\star_h\n_k}\qquad&&\text{in}\;\,B(0,\alpha R_k).
    	\end{alignedat}\right.
    \end{align}
    
    Therefore, there exists a linear map $F_k:\Lambda^2\R^n\rightarrow \Lambda^2\R^n$ such that 
    \begin{align*}
    	\D_z\vec{R}_k=\p{z}\vec{R}_k+F_k(\vec{R}_k),
    \end{align*}
    where 
    \begin{align}\label{l2one}
    	|F_k(\vec{R}_k)|\leq Ce^{\lambda_k}|\vec{R}_k|.
    \end{align}
    Thus, we can rewrite the system \eqref{systemRS1} as
    \begin{align}\label{systemRS2}
    	\left\{\begin{alignedat}{1}
    		\p{z}\vec{R}_k&=i\left((-1)^{n+1}\star \left(\n_k\antires \p{z}\vec{R}_k\right)+(\p{z}S_k)\star_h\n_k+i\,F_k(\vec{R}_k)+(-1)^{n+1}\star_h\left(\n_k\antires F_k(\vec{R}_k)\right)\right)\\
		\p{z}S_k&=-i\s{\p{z}\vec{R}_k}{\star_h\n_k}-i\s{F_k(\vec{R}_k)}{\star_h\,\n_k} \, . 
    	\end{alignedat} \right.
    \end{align}
    Since $\vec{R}_k$ and $S_k$ are bounded in $\mathrm{W}^{1,(2,\infty)}(B(0,\alpha R_k))$, we deduce by the Sobolev embedding that for all $1<p<\infty$, there exists $C_p<\infty$ such that 
        \begin{align}\label{eq:LpBoundRkSk}
    	\np{\vec{R}_k}{p}{B(0,\alpha R_k)}+\np{S_k}{p}{B(0,\alpha R_k)}\leq C_p<\infty, \quad \text{for all $0<\alpha<\alpha_0$ and $k\geq N$.}
    \end{align}
      In order to make the notations easier to read, we prove the following lemma (see \cite{quanta}).  Denote:
       \begin{align*}
    	{u}_r=\dashint{\partial B(0,r)}u\,d\mathscr{H}^1=\frac{1}{2\pi r}\int_{\partial B(0,r)}u\,d\mathscr{H}^1. 
    \end{align*}
    \begin{lemme}\label{averaging_lemma}
    	Let $k,m\in \N$, $u\in \mathrm{W}^{1,1}(B(0,1),\C)$, $f\in \mathrm{L}^{2}(B(0,1),\C)$, $\vec{v}\in \mathrm{W}^{1,(2,\infty)}(B(0,1),\Lambda^k\C^m)$, $\vec{w}\in \mathrm{W}^{1,2}\cap \mathrm{L}^{\infty}(B(0,1),\Lambda^k\R^m)$  such that 
    	\begin{align}\label{eq:pzu}
    		\p{z}u=-i\left(\s{\p{z}\vec{v}}{\vec{w}}+f\right).
    	\end{align}
    	Let $0<r<R<\infty$ and set $\Omega=B_R\setminus\bar{B}_r(0)$. Assume that $\Im(\vec{v})\in \mathrm{W}^{1,2}(\Omega)$ and that 
    	\begin{align}\label{l2infty}
    		|\D \Re(\vec{v})(z)|\leq \frac{C_0}{|z|}\, , \quad \text{for all $r\leq |z|\leq R$}.
    	\end{align}
    	Then 
    	\begin{align}\label{averages_ineq}
    		\left(\int_{r}^R\left|\frac{d}{d\rho}\Re(u)_{\rho}\right|^2\rho\,d\rho\right)^{\frac{1}{2}}\leq \sqrt{2\pi}\binom{n}{k}C_0\np{\D\vec{w}}{2}{\Omega}+\frac{1}{\sqrt{2\pi}}\np{\vec{w}}{\infty}{\Omega}\np{\D\Im(\vec{v})}{2}{\Omega}+\frac{1}{\sqrt{2\pi}}\np{f}{2}{\Omega}.
    	\end{align}
    \end{lemme}
    \begin{proof}
    	Rewriting the equation \eqref{eq:pzu} as
    	\begin{align*}
    		\p{z}\,\Re(u)=-i\left(\s{\p{z}\vec{v}}{\vec{w}}+\p{z}\Im(u)+f\right),
    	\end{align*}
    	and recalling that $\p{z}=\frac{1}{2}\left(\p{x}-i\,\p{y}\right)$,    	we deduce that 
    	\begin{align*}
    		\left\{
    		\begin{alignedat}{3}
    			\p{x}\Re(u)&=&&2\,\Re\left(-i\left(\s{\p{z}\vec{v}}{\vec{w}}+\p{z}\Im(u)+f\right)\right)&&=2\,\Im\left(\s{\p{z}\vec{v}}{\vec{w}}+\p{z}\Im(u)+f\right)\\
    			\p{y}\Re(u)&=-&&2\,\Im\left(-i\left(\s{\p{z}\vec{v}}{\vec{w}}+\p{z}\Im(u)+f\right)\right)&&=2\,\Re\left(\s{\p{z}\vec{v}}{\vec{w}}+\p{z}\Im(u)+f\right). 
    		\end{alignedat}
    		\right.
    	\end{align*}
    	Recalling that
    	\begin{align*}
    		\begin{alignedat}{
    				4}
    			\p{r}\Re(u)&=&&\cos(\theta)\,\p{x}u\,&&+&&\sin(\theta)\,\p{y}u\\
    			\frac{1}{r}\p{\theta}\Re(u)&=-&&\sin(\theta)\,\p{x}u\,&&+&&\cos(\theta)\,\p{y}u,
    		\end{alignedat}
    	\end{align*}
    	we get
    	\begin{align}\label{no_mean1}
    		\p{r}\Re(u)&=2\,\cos(\theta)\Im\left(\s{\p{z}\vec{v}}{\vec{w}}+\p{z}\Im(u)+f\right)+2\,\sin(\theta)\Re\left(\s{\p{z}\vec{v}}{\vec{w}}+\p{z}\Im(u)+f\right)\nonumber\\
    		&=\frac{2}{|z|}\Im\left(\s{z\p{z}\vec{v}}{\vec{w}}+z\p{z}\Im(u)+zf(z)\right)\, .
    	\end{align}
    	Now, notice that
    	\begin{align*}
    		\frac{2}{|z|}\Im\left(\s{z\,\p{z}\vec{v}}{\vec{w}}\right)&=\Im\left(\s{\left(\cos(\theta)+i\sin(\theta)\right)\left(\p{x}-i\,\p{y}\right)\Re(\vec{v})}{\vec{w}}\right)\\
    		&+\Im\left(i\,\s{\left(\cos(\theta)+i\sin(\theta)\right)\left(\p{x}-i\,\p{y}\right)\Im(\vec{w})}{\w}\right)\\
    		&=\s{\sin(\theta)\p{x}\,\Re(\vec{v})-\cos(\theta)\p{y}\Re(\vec{v})}{\vec{w}}+\s{\cos(\theta)\p{x}\Im(\vec{v})+\sin(\theta)\p{y}\Im(\vec{v})}{\vec{w}}\\
    		&=-\frac{1}{r}\s{\p{\theta}\Re(\vec{v})}{\vec{w}}+\s{\p{r}\Im(\vec{v})}{\vec{w}}.
    	\end{align*}
    	Using \eqref{no_mean1}, the same computation for $z\p{z}\Im(u)$ yields 
    	\begin{align}\label{no_mean2}
    		\p{r}\,\Re(u)=-\frac{1}{r}\s{\p{\theta}\Re(\vec{v})}{\vec{w}}+\frac{1}{r}\p{\theta}\Im(u)+\s{\p{r}\Im(\vec{v})}{\vec{w}}+2\,\Im\left(\frac{z}{|z|}f(z)\right).
    	\end{align}
    	Therefore, we deduce that
    	\begin{align*}
    		\frac{d}{dr}{\Re(u)}_r&=\frac{1}{2\pi}\int_{0}^{2\pi}\partial_r \Re(u)(r,\theta)d\theta=-\frac{1}{2\pi 
    		}\int_{0}^{2\pi}\bs{\frac{1}{r}\partial_{\theta}\Re(\vec{v})}{\vec{w}}d\theta+\frac{1}{2\pi r}\int_{\partial B(0,t)}\partial_{\theta}\Im(u)d\theta\\
    		&+\frac{1}{2\pi}\int_{0}^{2\pi}\s{\p{r}\Im(\vec{v})}{\vec{w}}
    		+\frac{1}{2\pi}\int_{0}^{2\pi
    		}\Im\left(\frac{z}{|z|}f(z)\right)d\theta\\
    		&=-\frac{1}{2\pi }\int_{0}^{2\pi}\bs{\frac{1}{r}\partial_{\theta}\Re(\vec{v})}{\vec{w}-\vec{w}_t}d\theta+\frac{1}{2\pi}\int_{0}^{2\pi}\s{\p{r}\Im(\vec{v})}{\vec{w}}d\theta
    		+\frac{1}{2\pi}\int_{0}^{2\pi
    		}\Im\left(\frac{z}{|z|}f(z)\right)d\theta.
    	\end{align*}
    	Notice that by the Cauchy-Schwarz inequality, we have for all $\varphi\in \mathrm{L}^2(\Omega,\C)$
    	\begin{align*}
    		\int_{r}^R\left|\int_{0}^{2\pi}\varphi \frac{d\theta}{2\pi}\right|^2\rho\,d\rho\leq \frac{1}{2\pi}\int_{r}^R\int_{0}^{2\pi}|\varphi|^2\rho\,d\rho d\theta=\frac{1}{2\pi}\int_{\Omega}|\varphi|^2|dz|^2.
    	\end{align*}
    	Therefore,  by the Minkowski inequality and \eqref{l2infty} we have
    	\begin{align*}
    		&\left(\int_{r}^R\left|\frac{d}{d\rho}\Re(u)_{\rho}\right|^2\rho\,d\rho\right)^{\frac{1}{2}}\leq \frac{C_0}{\sqrt{2\pi}}\left(\int_{\Omega}\left|\vec{w}-\vec{w}_{|z|}\right|^2\frac{|dz|^2}{|z|^2}\right)^{\frac{1}{2}}+\frac{1}{\sqrt{2\pi}}\left(\int_{\Omega}|\s{\p{r}\Im(\vec{v})}{\vec{w}}|^2|dz|^2\right)^{\frac{1}{2}}\\
    		&+\frac{1}{\sqrt{2\pi}}\left(\int_{\Omega}\left|\Im\left(\frac{z}{|z|}f(z)\right)\right|^2|dz|^2\right)^{\frac{1}{2}}\\
    		&\leq \frac{C_0}{\sqrt{2\pi}}\left(\int_{\Omega}|\vec{w}-\vec{w}_{|z|}|^2\frac{|dz|^2}{|z|^2}\right)^{\frac{1}{2}}+\frac{1}{\sqrt{2\pi}}\np{\vec{w}}{\infty}{\Omega}\np{\D\Im(\vec{v})}{2}{\Omega}+\frac{1}{\sqrt{2\pi}}\np{f}{2}{\Omega}.
    	\end{align*}
    	Now, by \eqref{sobolev1}, we infer
    	\begin{align*}
    		\np{\w-\vec{w}_{|z|}}{\infty}{\partial B(0,|z|)}\leq \binom{n}{k}\int_{\partial B(0,|z|)}|\D\vec{w}|d\mathscr{H}^1, \quad \text{ for all $z\in B(0,1)$}.
    	\end{align*}
    	Using the Cauchy-Schwarz inequality and the co-area formula twice, we deduce that
    	\begin{align*}
    		&\int_{B_R\setminus\bar{B}_r(0)}|\vec{w}-\vec{w}_{|z|}|^2\frac{|dz|^2}{|z|^2}\leq \binom{n}{k}^2\int_{B_R\setminus\bar{B}_r(0)}\left(\int_{\partial B(0,|z|)}|\D\vec{w}|d\mathscr{H}^1\right)^2\frac{|dz|^2}{|z|^2}\\
    		&\leq 2\pi\,\binom{n}{k}^2\int_{B_R\setminus\bar{B}_r(0)}\int_{\partial B(0,|z|)}|\D\vec{w}|^2\frac{|dz|^2}{|z|}
    		=(2\pi)^2\binom{n}{k}^2\int_{r}^R\left(\int_{\partial B(0,t)}|\D\vec{w}|^2d\mathscr{H}^1\right)dt\\
    		&
    		=(2\pi)^2\binom{n}{k}^2\int_{B_R\setminus\bar{B}_r(0)}|\D\w|^2|dz|^2.
    	\end{align*}
    	We conclude that
    	\begin{align*}
    		\left(\int_{r}^R\left|\frac{d}{d\rho}\Re(u)_{\rho}\right|^2\rho\,d\rho\right)^{\frac{1}{2}}\leq \sqrt{2\pi}\binom{n}{k}C_0\np{\D\vec{w}}{2}{\Omega}+\frac{1}{\sqrt{2\pi}}\np{\vec{w}}{\infty}{\Omega}\np{\D\Im(\vec{v})}{2}{\Omega}+\frac{1}{\sqrt{2\pi}}\np{f}{2}{\Omega},
    	\end{align*}
    	which concludes the proof of the lemma. 
    \end{proof}
    Applying  Lemma \ref{averaging_lemma} to $S_k$ in the equation \eqref{systemRS2}, we deduce that 
    \begin{align*}
    	\left(\int_{\alpha_0^{-1}r_k}^{\alpha R_k}\left|\frac{d}{d\rho}\Re(S_{k,\rho})\right|^2\rho\,d\rho\right)^{\frac{1}{2}}&\leq n(n-1)\sqrt{\frac{\pi}{2}}C_3(n,\Lambda)\np{\D\n_k}{2}{\Omega_k(\alpha)}+\frac{1}{\sqrt{2}\pi}\np{\D\Im(\vec{R}_k)}{2}{\Omega_k(\alpha)}\\
    	&\qquad +\frac{1}{\sqrt{2\pi}}\np{\vec{R}_k}{2}{\Omega_k(\alpha)}. 
    \end{align*}
    Applying Lemma \ref{averaging_lemma} to each of the $\frac{m(m-1)}{2}$ components of $\vec{R}_k$, we get by the Minkowski inequality that 
    \begin{align*}
    	&\left(\int_{\alpha_0^{-1}r_k}^{\alpha_0R_k}\left|\frac{d}{d\rho}\Re(\vec{R}_{k,\rho})\right|^2\rho\,d\rho\right)^{\frac{1}{2}}\leq {n(n-1)}\Gamma_1(n)C_3(n\Lambda)\np{\D\n_k}{2}{\Omega_k(\alpha)}\\
    	&+\frac{n(n-1)}{2\sqrt{2\pi}}\left(\np{\D\Im(S_k)}{2}{\Omega_k(\alpha)}+\np{\D\Im(\vec{R}_k)}{2}{\Omega_k(\alpha)}\right)
    	+{n(n-1)}\sqrt{\frac{2}{\pi}}\np{\vec{R}_k}{2}{\Omega_k(\alpha)}. 
    \end{align*}
     Recall now the following generalization (see \cite{pointwise}) of \cite[Lemma VI.$2$]{quanta}  proved in \cite{angular} (see also \cite{quantamoduli}). 
    \begin{lemme}\label{newl2estimate}
    	There exists a universal constant $R_0>0$ with the following property.
    	Let $0<4r<R<R_0$, $\Omega=B(0,R)\setminus \bar{B}(0,r)\rightarrow \R$, $a,b:\Omega \rightarrow \R$ such that $\D a\in \mathrm{L}^{2,\infty}(\Omega)$ and $\D b\in \mathrm{L}^2(\Omega)$, and $\varphi:\Omega\rightarrow \R$ be a solution of 
    	\begin{align*}
    		\Delta \varphi=\D a\cdot \D^{\perp}b\qquad \text{in}\;\ \Omega.
    	\end{align*}
    	For $r\leq \rho\leq R$, define
    	\begin{align*}
    		{\varphi}_r=\dashint{\partial B_{\rho}(0)}\varphi\,d\mathscr{H}^1=\frac{1}{2\pi \rho}\int_{\partial B_{\rho}(0)}\varphi\,d\mathscr{H}^1.
    	\end{align*}
    	Then $\D \varphi\in \mathrm{L}^2(\Omega)$, and there exists a positive constant $C_0>0$ independent of $0<4r<R$ such that for all $\left(\dfrac{r}{R}\right)^{\frac{1}{2}}<\alpha<\dfrac{1}{2}$ it holds:
    	\begin{align*}
    		\np{\D \varphi}{2}{B_{\alpha R}\setminus \bar{B}_{\alpha^{-1}r}}\leq C_0\np{\D a}{2,\infty}{\Omega}\np{\D b}{2}{\Omega}+C_0\np{\D {\varphi}_{\rho}}{2}{\Omega}+C_0\np{\D \varphi }{2,\infty}{\Omega}.
    	\end{align*}
    \end{lemme}
  From \eqref{eq:LpBoundRkSk} we know that
    \begin{align*}
    	\np{S_k}{2}{\Omega_k(\alpha_0/2)}+\np{\vec{R}_k}{2}{\Omega_k(\alpha_0/2)}\leq C.
    \end{align*}
     Lemma \ref{newl2estimate}, the system \eqref{systemRS2} and the Cauchy-Schwarz inequality imply that
    \begin{align*}
    	&\int_{2\alpha_0^{-1}r_k}^{\frac{\alpha_0R_k}{2}}\left(\left|\frac{d}{d\rho}\Re(\vec{R}_{k,\rho})\right|+\left|\frac{d}{d\rho}\Re(S_{k,\rho})\right|\right)d\rho\leq C\int_{2\alpha_0^{-1}r_k}^{\frac{\alpha_0 R_k}{2}}\delta(\rho)\left(\int_{\partial B(0,\rho)}\left(|\D S_k|+|\D\vec{R}_k|\right)d\mathscr{H}^1\right)d\rho\\
    	&\leq C\left(\int_{2\alpha_0^{-1}r_k}^{\frac{\alpha_0R_k}{2}}\delta^2(\rho)\rho\,d\rho\right)^{\frac{1}{2}}\left(\np{S_k}{2}{\Omega_k(\alpha_0/2)}+\np{\vec{R}_k}{2}{\Omega_k(\alpha_0/2)}\right)\leq C.
    \end{align*}
    Furthermore, by the inequality \eqref{estimate_lebesgue}, there exists $r_0\in (\frac{\alpha_0R_k}{4},\frac{\alpha_0R_k}{2})$ such that if $r=\frac{\alpha_0R_k}{2}$, we have
    \begin{align}\label{estimate_lebesgue2}
    	\int_{\partial B_{r_0}(0)}\left(|S_k|+|\vec{R}_k|\right)d\mathscr{H}^1&\leq \frac{2\sqrt{3\pi}}{\log(2)}\left(\np{S_k}{2,\infty}{B_{2r}\setminus\bar{B}_r(0)}+\np{\vec{R}_k}{2,\infty}{B_{2r}\setminus\bar{B}_r(0)}\right) \nonumber \\
    	&\leq \frac{2\pi^{1-\frac{1}{2p}}(2r)^{1-\frac{1}{p}}}{\log(2)}\left(\np{S_k}{p}{B_{2r}\setminus\bar{B}_r(0)}+\np{\vec{R}_k}{p}{B_{2r}\setminus\bar{B}_r(0)}\right).
    \end{align}
    Therefore, it holds
    \begin{align*}
    	\left|\Re(\vec{R}_{k,r}-\Re(\vec{R}_{k,r_0})\right|+\left|\Re(S_{k,r}-\Re(S_{k,r_0})\right|\leq C_n\int_{2\alpha_0^{-1}r_k}^{r_0}\left(\left|\frac{d}{d\rho}\Re(\vec{R}_{k,\rho})\right|+\left|\frac{d}{d\rho}\Re(S_{k,\rho})\right|\right)d\rho\leq C
    \end{align*}
    and
    \begin{align*}
    	|\Re(\vec{R}_{k,r})|+|\Re(S_{k,r})|\leq C+\frac{C}{(\alpha_0R_k)^{\frac{1}{p}}}\left(\np{\vec{R}_k}{p}{\Omega_k(\alpha_0/2)}+\np{S_k}{p}{\Omega_k(\alpha_0/2)}\right).
    \end{align*}
    In particular, we get that $\Re(\vec{R}_k)_\rho,\Re(S_k)_\rho\in \mathrm{L}^{\infty}([4\alpha_0^{-1}R_k,\frac{\alpha_0R_k}{4}])$.
    We deduce that 
    \begin{align*}
    	&\left|\Re\left(\vec{R}_k\right)(z)-\Re\left(\vec{R}_k\right)_{|z|}\right|\leq n\int_{\partial B(0,|z|)}|\D\Re\left(\vec{R}_k\right)|d\mathscr{H}^1\leq n\int_{\partial B(0,|z|)}|\D\Im(\vec{R}_k)|d\mathscr{H}^1\\
    	&+2n\int_{\partial B(0,|z|)}e^{\lambda_k}\left(|\vec{L}_k|+|\H_k|
    	\right)d\mathscr{H}^1+C|z|^{d_k}\int_{\partial B(0,|z|)}\left|\Re(\vec{R}_k)\right|d\mathscr{H}^1+C|z|^{d_k}\int_{\partial B(0,|z|)}|\Im(\vec{R}_k)|d\mathscr{H}^1\\
    	&\leq C+|z|^{d_k}\int_{\partial B(0,|z|)}|\tilde{W}_k|d\mathscr{H}^1+C|z|^{d_k}\int_{\partial B(0,|z|)}|\Re(\vec{R}_k)|d\mathscr{H}^1\, .
    \end{align*}
    By the $\mathrm{L}^{2,1}/\mathrm{L}^{2,\infty}$ duality in dimension $1$ and  trace theory, since $\tilde{W}_k\in \mathrm{W}^{1,(2,\infty)}$, we deduce that 
    \begin{align*}
    	\int_{\partial B(0,|z|)}|\tilde{W}_k|d\mathscr{H}^1\leq \np{1}{2,1}{\partial B(0,|z|)}\np{\tilde{W}_k}{2,\infty}{\partial B(0,|z|)}\leq C|z|\wp{\tilde{W}_k}{1,(2,\infty)}{\Omega_k(\alpha_0)}\leq C|z|\, . 
    \end{align*}
    Since $d_k>-1+\epsilon$, we deduce by the $\mathrm{L}^{\infty}$ estimate for the means of $\Re(\vec{R}_k)$ that 
    \begin{align*}
    	|\Re(\vec{R}_k)(z)|\leq C+C|z|^{d_k}\int_{\partial B(0,|z|)}|\Re(\vec{R}_k)|d\mathscr{H}^1\, .
    \end{align*}
    Integrating this identity and using $d_k>-1+\epsilon$, we deduce that
    \begin{align*}
    	\int_{\partial B(0,|z|)}|\Re(\vec{R}_k)|d\mathscr{H}^1&\leq 2\pi C|z|+2\pi C|z|^{1+d_k}\int_{\partial B(0,|z|)}|\Re(\vec{R}_k)|d\mathscr{H}^1\\
	&\leq C|z|+C(\alpha_0R_k)^{\epsilon}\int_{\partial B(0,|z|)}|\Re(\vec{R}_k)|d\mathscr{H}^1,
    \end{align*}
    which shows that for $\alpha_0>0$ small enough, we have
    \begin{align*}
    	\int_{\partial B(0,|z|)}|\Re(\vec{R}_k)|d\mathscr{H}^1\leq C|z|,\, .
    \end{align*}
    We conclude that $\Re(\vec{R}_k)\in \mathrm{L}^{\infty}(B(0,\alpha_0R_k
    ))$. A similar argument (easier since we have a $\p{z}$ equation and not a $\D_z$ one) finally yields that
    \begin{align*}
    	\np{S_k}{\infty}{\Omega_k(\alpha_0/4)}+\np{\vec{R}_k}{\infty}{\Omega_k(\alpha_0/4)}\leq C. 
    \end{align*}    
    To complete the proof, recall from \cite{mondinoriviere} that in $B(0,\alpha_0R_k)$ it holds:
    \begin{align*}
    	\left\{
    	\begin{alignedat}{1}
    		\Delta\left(\Re(\vec{R}_k)\right)&=(-1)^n\star_h\left(\D\n_k\antires\D^{\perp}\Re(\vec{R}_k)\right)-\star_h\left(\D\n_k\antires\D^{\perp}(\Re(S_k))\right)+\vec{G}_{1,k}\\
    		\Delta\left(\Re(S_k)\right)&=\s{\D(\star_h\n_k)}{\D^{\perp}\Re(\vec{R}_k)}+G_{2,k}
    	\end{alignedat}\right.
    \end{align*}
    for some  $\vec{G}_{1,k}$ and $G_{2,k}$ which are bounded in $\mathrm{L}^p(B(0,\alpha_0 R_k))$, for all $1\leq p<2$.
    
   Recall also the following slight variant (see \cite{pointwise}) from a Lemma of \cite{angular}.
    \begin{lemme}\label{hardy}
    	Let $R_0>0$ be the constant of Lemma \ref{newl2estimate}.
    	Let $0<16r<R<R_0$, $\Omega=B(0,R)\setminus \bar{B}(0,r)\rightarrow \R$, $a,b:\Omega \rightarrow \R$ such that $\D a\in \mathrm{L}^{2}(\Omega)$ and $\D b\in \mathrm{L}^2(\Omega)$, and $\varphi:\Omega\rightarrow \R$ be a solution of 
    	\begin{align*}
    		\Delta \varphi=\D a\cdot \D^{\perp}b\quad \text{in}\;\ \Omega.
    	\end{align*}
    	Assume that $\np{\varphi}{\infty}{\partial\Omega}<\infty$. Then there exists a constant $C_1>0$ such that for all $\left(\dfrac{r}{R}\right)^{\frac{1}{2}}<\alpha<\dfrac{1}{4}$, 
    	\begin{align*}
    		\np{\varphi}{\infty}{\Omega}+\np{\D\varphi}{2,1}{B_{\alpha R}\setminus \bar{B}_{\alpha^{-1}r}(0)}+\np{\D^2\varphi}{1}{B_{\alpha R}\setminus\bar{B}_{\alpha^{-1}r}(0)}\leq C_1\left(\np{\D a}{2}{\Omega}\np{\D b}{2}{\Omega}+\np{\varphi}{\infty}{\partial\Omega}\right).
    	\end{align*}
    \end{lemme}
    From Lemma \ref{hardy}, we deduce that
    \begin{align*}
    	\np{\D S_k}{2,1}{\Omega_k(\alpha_0/2)}+\np{\D \vec{R}_k}{2,1}{\Omega_k(\alpha_0/2)}+\np{\D^2S_k}{1}{\Omega_k(\alpha_0/2)}+\np{\D^2\vec{R}_k}{1}{\Omega_k(\alpha_0/2)}\leq C. 
    \end{align*}
    Since $\Im(\vec{L}_k)\in \mathrm{W}^{1,(2,\infty)}(B(0,\alpha_0R_k))$ we deduce that $e^{\lambda_k}\Im(\vec{L}_k)\in \mathrm{L}^{2+\epsilon}(B(0,\alpha_0R_k))$.  Using the identity
    \begin{align*}
    	e^{\lambda_k}\H_k&=-\Im\left(\D_z\vec{R}_k\res e^{-\lambda_k}\p{\z}\phi_k\right)-\frac{1}{2}e^{\lambda_k}\Im(\vec{L}_k)\\
	&\quad -\Re\left(ie^{-\lambda_k}\p{\z}\phi_k\,\p{z}S_k\right)+\Re\left(\s{\p{z}\phi_k}{\Im(\vec{L}_k)}e^{-\lambda}\p{\z}\phi_k\right)\, ,
    \end{align*}
    we finally get 
    \begin{align*}
    	\np{e^{\lambda_k}\H_k}{2,1}{\Omega_k(\alpha_0/4)}\leq C
    \end{align*}    
    which concludes the proof of  Theorem \ref{L21_necks}. 
      \hfill$\Box$

     \section{Weak $\epsilon$-regularity for Willmore immersions with values into manifolds}
    By \cite[Lemma $3.2$ and Theorem $3.1$]{mondinoriviere}, the following identities are satisfied for any smooth immersion $\phi:\Sigma\rightarrow (M^m,h)$ 
     \begin{align*}
     	\vec{Y}=i\left(\D_z\H-3\,\D_z^{\perp}\H-i\star_h(\D_z\n\wedge \H)\right)&=-2i\left(\D_z^{\perp}\H+\s{\H}{\H_0}\p{\z}\phi\right)\\
     	4e^{-2\lambda}\Re\left(\D_{\z}\left(\D_z^{\perp}\H+\s{\H}{\H_0}\p{\z}\phi\right)\right)&=\Delta_g^{\perp}\H-2|\H|^2\H+\mathscr{A}(\H)+8\,\Re\left(\s{R(\e_{\z},\e_z)\e_z}{\H}\e_{\z}\right). 
     \end{align*}
     We deduce that
     \begin{align*}
     	\Im(\D_{\z}\vec{Y})&=\Im\left(-2i\D_{\z}\left(\D_z^{\perp}\H+\s{\H}{\H_0}\p{\z}\phi\right)\right)=-2\,\Re\left(\D_{\z}\left(\D_z^{\perp}\H+\s{\H}{\H_0}\p{\z}\phi\right)\right)\\
     	&=-\frac{1}{2}e^{2\lambda}\left(\Delta_g^{\perp}\H-2|\H|^2\H+\mathscr{A}(\H)\right)-4e^{2\lambda}\,\Re\left(\s{R(\e_{\z},\e_z)\e_z}{\H}\e_{\z}\right).
     \end{align*}
     Assuming that $\phi$ is a Willmore immersion, from \eqref{el1} we deduce that
     \begin{align}\label{identity_y2}
     	\Im(\D_{\z}\vec{Y})=\frac{1}{2}e^{2\lambda}\left(\mathscr{R}_1^{\perp}(\H)-2\,\tilde{K}_h\,\H+2\,\mathscr{R}_2(d\phi)+(DR)(d\phi)-8\,\Re\left(\s{R(\e_{\z},\e_z)\e_z}{\H}\e_{\z}\right)\right).
     \end{align}
     As before, let $f:\C^m\rightarrow \C^m$ be the linear map such that for all $X\in C^{\infty}(B(0,1),\C^m)$, it holds
     \begin{align*}
     	\D_z\vec{X}=\p{z}\vec{X}+f(\vec{X})=\p{z}\vec{X}+\left(\sum_{l=1}^m\gamma^j_l\vec{X}^l\right)_{1\leq j\leq m}\, ,
     \end{align*}
     where 
    $ 	\gamma^j_l=\sum_{q=1}^{m}\Gamma_{l,q}^j\p{z}\phi^q $
     and $\Gamma_{l,q}^j$ are the Christoffel symbols of the ambient space $(M^m,h)$.
     \\ Likewise, there exists linear maps $F:\Lambda^2\C^m\rightarrow \Lambda^2\C^m$ and $G:\Lambda^{m-2}\C^m\rightarrow \Lambda^{m-2}\C^m$ (notice that $G=f$ is $m=3$) such that for all $\vec{Y}\in C^{\infty}(B(0,1),\Lambda^2\C^m)$ and $\vec{Z}\in C^{\infty}(B(0,1),\Lambda^{m-2}\C^m)$ it holds:
     \begin{align*}
     	\D_z\vec{Y}=\p{z}\vec{Y}+F(\vec{Y})\quad \text{and} \quad \D_z\vec{Z}=\p{z}\vec{Z}+G(\vec{Z})\,.
     \end{align*}
    Notice that 
     \begin{align*}
     	\Im(\D_{\z}\vec{Y})=\Im(\p{\z}\vec{Y}+\bar{f_k}(\vec{Y})),
     \end{align*}
     that 
     \begin{align*}
     	\vec{Y}=i\left(\p{z}\vec{H}-3\pi_{\n}\left(\p{z}\vec{H}\right)-i\star_h\left(\p{z}\n\wedge\H\right)-2f_k(\H)-i\star_h\left(G(\n)\wedge \H\right)\right). 
     \end{align*}
    and that
     \begin{align*}
     	\Im\left(\p{\z}\left(\vec{Y}\right)\right)=-\Re\left(\p{\z}\left(\p{z}\H-3\,\pi_{\n}(\p{z}\H)-i\,\star_h(\p{z}\n\wedge \H)\right)\right)+\Re\left(\p{\z}\left(f(\H)-i\star_h\left(G(\n)\wedge\H\right)\right)\right). 
     \end{align*}
     We immediately get
     \begin{align*}
     	\Re\left(\p{\z}\left(\p{z}\H-3\,\pi_{\n}\left(\p{z}\H\right)\right)\right)=\frac{1}{4}\dive\left(\D\H-3\,\pi_{\n}(\D\H)\right)
     \end{align*}
     and
     \begin{align*}
     	\Re\left(\p{\z}\left(-i\star_h\left(\p{z}\n\wedge \H\right)\right)\right)&=\Re\left(-\frac{i}{4}\star_h\left(\Delta\n\wedge \H\right)+\Re\left(-i\,\star_h\left(\p{z}\n\wedge \p{\z}\H\right)\right)\right)\\
     	&=\frac{1}{4}\star_h\left(\D^{\perp}\n\wedge \D\H\right)=\frac{1}{4}\dive\left(\star_h\left(\D^{\perp}\n\wedge \H\right)\right). 
     \end{align*}
     Finally, we get
     \begin{align*}
     	\Im\left(\p{\z}\left(\vec{Y}\right)\right)=-\frac{1}{4}\dive\left(\D\H-3\,\pi_{\n}(\D\H)+\star_h\left(\D^{\perp}\n\wedge \H\right)\right)+\Re\left(\p{\z}\left(f(\H)-i\ast_h\left(G(\n)\wedge\H\right)\right)\right).
     \end{align*}
     Now, we have
     \begin{align*}
     	\vec{Y}&=i\left(\p{z}\H-3\pi_{\n}(\p{z}\H)-2f_k(\H)-i\star (\p{z}\n\wedge \H)-i\star_h\left(G(\n)\wedge\H\right)\right)\\
     	&=i\left(-2\,\p{z}\H+3\left(\p{z}\pi_{\n}\right)\H-2f(\H)-i\star (\p{z}\n\wedge \H)-i\star_h\left(G(\n)\wedge\H\right)\right).
     \end{align*}
     We compute 
      \begin{align*}
          	&\Im\left(\bar{f}\left((\p{z}\pi_{\n})\H\right)\right)_j=\frac{1}{4}\sum_{l=1}^{n}\sum_{q=1}^n\Gamma_{j,q}^l\s{\D\phi_q}{\D\pi_{\n}}\H_j=\frac{1}{4}\left(\vec{A}_2(\n)\cdot\H\right)_j\\
          	&\Im\left(-i\,\bar{f}\left(f(\H)\right)\right)_j=\Re\left(\sum_{l=1}^{n}\sum_{q=1}^n\Gamma_{j,q}^l\p{\z}\phi_qf(\H)_l\right)=\Re\left(\sum_{l=1}^{n}\sum_{q=1}^n\sum_{l'=1}^n\sum_{q'=1}^n\Gamma_{j,q}^l\Gamma_{l,q'}^{l'}\p{\z}\phi_q\p{z}\phi_{q'}\H_{l'}\right)\\
          	&=\frac{1}{4}\sum_{l=1}^{n}\sum_{q=1}^n\sum_{l'=1}^n\sum_{q'=1}^n\Gamma_{j,q}^l\Gamma_{l,q'}^{l'}\s{\D\phi_{q}}{\D\phi_{q'}}\H_{l'}=\frac{1}{4}\left(\vec{A}_3\cdot\H\right)_j\\
          	&\Im\left(\bar{f}\left(\star_h\left(\p{z}\n\wedge\H\right)\right)\right)_j=\frac{1}{4}\sum_{l=1}^n\sum_{q=1}^{n}\Gamma_{j,q}^l\bs{\D\phi_q}{\star_h\left(\D^{\perp}\n\wedge\H\right)_l}=\frac{1}{4}\left(\vec{A}_4\cdot \star_h\left(\D^{\perp}\n\wedge\H\right)\right)_j\\
          	&\Im\left(\bar{f}\left(\star_h\left(G(\n)\wedge\H\right)\right)\right)=\frac{1}{4}\vec{A}_4\cdot (\vec{W}\wedge\H).
          \end{align*}
         Likewise, we have for all $1\leq j\leq n$
          \begin{align*}
          	&\Im\left(\bar{f}(-2i\,\p{z}\H)\right)_j=\Im\left(-2i\sum_{q=1}^{n}\Gamma_{j,q}^l\p{\z}\phi_{q}\p{z}\H_{l}\right)=-2\,\sum_{l=1}^{n}\sum_{q=1}^n\Gamma_{j,q}^l\Re\left(\p{\z}\phi_{q}\p{z}\H_{l}\right)\\
          	&=-\sum_{l=1}^{n}\frac{1}{2}\sum_{q=1}^{n}\Gamma_{j,q}^l\s{\D\phi_{q}}{\D\H_{l}}\\
          	&=-\frac{1}{2}\dive\left(\sum_{l=1}^{n}\sum_{q=1}^{n}\Gamma_{j,q}^l\D\phi_{q}\,\H_{l}\right)+\frac{1}{2}\sum_{l=1}^{n}\sum_{q=1}^{n} \s{\D\Gamma_{j,q}^l}{\D\phi_{q}}\H_{l}+\frac{1}{2}\sum_{l=1}^{n}\sum_{q=1}^n\Gamma_{j,q}^l\Delta\phi_{q}\,\H_{l}\\
          	&=-\frac{1}{2}\dive\left(\sum_{l=1}^{n}\sum_{q=1}^{n}\Gamma_{j,q}^l\D\phi_{q}\,\H_{l}\right)+\frac{1}{2}\sum_{l=1}^{n}\sum_{q=1}^{n} \s{\D\Gamma_{j,q}^l}{\D\phi_{q}}\H_{l}+\sum_{l=1}^{n}\sum_{q=1}^{n}\Gamma_{j,q}^le^{2\lambda}\H_{q}\H_{l}\\
          	&=-\frac{1}{2}\dive\left(\vec{A}_0\cdot \H\right)+\frac{1}{4}\vec{A}_1\cdot \H+\frac{1}{4}B(\H,\H),
          \end{align*}
          where $\vec{A}_0,\vec{A}_1\in M_n(\R)\otimes\R^2$ and $B$ is a bilinear map, such that for some universal constant $C_2=C_2(n,h)$,  have
          \begin{align*}
          	|\vec{A}_0|+|\vec{A}_1|+\sqrt{\Vert B\Vert}\leq C_2e^{\lambda}. 
          \end{align*} 
          Then, we have 
          \begin{align*}
          	\Re\left(\p{\z}f(\H)\right)=\frac{1}{4}\dive\left(\sum_{l=1}^{n}\sum_{q=1}^{n}\Gamma_{j,q}^l\D\phi_q\H_l\right)_{1\leq j\leq n}=\frac{1}{4}\dive\left(\vec{A}_0\cdot\H\right),
          \end{align*}
          and
          \begin{align*}
          	&\Re\left(\p{\z}\left(-i\star_h\left(G(\n)\wedge \H\right)\right)\right)\\
          	&=\frac{1}{4}\dive\left(\star_h\left(\sum_{k=1}^{n-2}\sum_{l=1}^{n}\sum_{q=1}^{n}\Gamma_{j,q}\left(\n_1\wedge\cdots\wedge \n_{k-1}\wedge \D^{\perp}\phi_q\n_{k,l}\right)\wedge \n_{k+1}\wedge\cdots\wedge \n_{n-2}\right)\wedge \H\right)_{1\leq j\leq n}\\
          	&=-\frac{1}{4}\dive\left(\star_h\left(\vec{V}\wedge \H\right)\right)
          \end{align*}
          where $\vec{V}\in \Lambda^{n-2}\R^n$. 
  Finally, we deduce that 
     \begin{align}\label{operator_bounds}
     	\left\{\begin{alignedat}{1}
	&|\vec{A}_0|+|\vec{A}_1|+|\vec{A}_4|+\sqrt{\Vert B\Vert}+|\vec{V}|+|\vec{W}|\leq C e^{\lambda}\\
     	&|\vec{A}_2(\n)|\leq C e^{\lambda}|\D\n|\\
     	&|\vec{A}_3|\leq C e^{2\lambda}\, ,
     	     \end{alignedat}\right.
     \end{align}
     for some $C=C(m,h)>0$.
 We deduce that 
     \begin{align}\label{eqH}
     	&\dive\left(\D\H-3\,\pi_{\n}(\D\H)+\star_h\left(\D^{\perp}\n\wedge\H\right)+\vec{A}_0\cdot\H+\star_h(\vec{V}\wedge\H)\right)\nonumber\\
	&= -4 \Im(\D_{\z}\vec{Y})\nonumber\\
     	&=-2e^{2\lambda}\left(\mathscr{R}_1^{\perp}(\H)-2\,\tilde{K}_h\,\H+2\,\mathscr{R}_2(d\phi)+(DR)(d\phi)-8\,\Re\left(\s{R(\e_{\z},\e_z)\e_z}{\H}\e_{\z}\right)\right)\nonumber\\
     	&\quad -\vec{A}_1\cdot \H-{3}\,\vec{A}_2(\n)\cdot\H-2\,\vec{A}_3\cdot\H-\vec{A}_4\cdot\star_h\left(\D^{\perp}\n\wedge\H\right)-\vec{A}_5\cdot \star_h\left(\vec{W}\wedge\H\right)-B(\H,\H).
     \end{align}
     The bounds \eqref{operator_bounds} give
     \begin{align*}
     	&\bigg|-2e^{2\lambda}\left(\mathscr{R}_1^{\perp}(\H)-2\,\tilde{K}_h\,\H+2\,\mathscr{R}_2(d\phi)+(DR)(d\phi)-8\,\Re\left(\s{R(\e_{\z},\e_z)\e_z}{\H}\e_{\z}\right)\right)\\
     	&-\vec{A}_1\cdot \H-{3}\,\vec{A}_2(\n)\cdot\H-2\,\vec{A}_3\cdot\H-\vec{A}_4\cdot\star_h\left(\D^{\perp}\n\wedge\H\right)-\vec{A}_5\cdot \star_h\left(\vec{W}\wedge\H\right)-B(\H,\H)\bigg|\\
     	&\leq C_3\left(e^{\lambda}+e^{2\lambda}+e^{\lambda}|\H|+e^{2\lambda}|\H|+|\D\n|e^{\lambda}|\H|+e^{2\lambda}|\H|^2\right)\in \mathrm{L}^1(B(0,1)).
     \end{align*}
    Following the strategy of \cite{riviere1}, we next study the linear operator $\mathscr{L}_{\n}$ defined for all vector-fields $\vec{u}:B(0,1)\rightarrow TM^m$ by 
     \begin{align}\label{eqH2}
     	\leb_{\n}\vec{u}=\Delta \vec{u}+\dive\left(-3\,\pi_{\n}(\D\vec{u})+\star_h\left(\D^{\perp}\n\wedge\vec{u}\right)+\vec{A}_0\cdot\vec{u}+\star_h(\vec{V}\wedge\vec{u})\right).
     \end{align}
     Compared to \cite{riviere1}, $\leb_{\n}$ only differs by lower-order terms that are in $\mathrm{L}^2$ (if $\lambda\in \mathrm{L}^{\infty}(B(0,1))$ and $\vec{u}\in \mathrm{L}^2(B(0,1))$), and one checks that it is an elliptic operator for vector-fields with values into $M^m$.
     It will be convenient to use Nash embedding to isometrically embed all the objects in $\R^{n}$. By \eqref{second_form1}, calling $\tilde{\n}:\Sigma\rightarrow \Lambda^{n-2}\R^n$ the induced unit normal of the immersion $\iota\circ\phi:B(0,1)\rightarrow \R^n$, we have 
     \begin{align*}
     	\pi_{\tilde{n}}(\D\H)=\pi_{\n}(\D\H).
     \end{align*}
     Furthermore, using the function $\vec{V}_0:B(0,1)\rightarrow \Lambda^{n-m}\R^n$ defined by 
     \begin{align*}
     	\vec{V}_0=\star_{\R^n}\left(\e_1\wedge\e_2\wedge\n\right)=\n_{\iota}\circ \phi,
     \end{align*}
     where $\n_{\iota}:M^m\rightarrow \Lambda^{n-m}\R^n$ is the unit normal of the inclusion $\iota:M^m\rightarrow \R^n$,
     we deduce that there exist $\vec{A}_1\in \mathrm{L}^{\infty}(B(0,1),M_n(\R)\otimes\R^2)$ and $\vec{V}_1\in \mathrm{L}^{\infty}(B(0,1),\Lambda^{n-2}\R^n)$ such that the operator $\mathscr{L}$ defined for all $\vec{u}\in C^{\infty}(B(0,1),\R^n)$  by 
     \begin{align*}
     	\mathscr{L}_{1}\vec{u}=\Delta\vec{u}+\dive\left(-3\,\pi_{\tilde{\n}}(\D\vec{u})+\star_{\R^n}\left(\D^{\perp}\n\wedge\vec{V}_0\wedge\vec{u}\right)+\vec{A}_1\cdot\vec{u}+\star_{\R^n}\left(\vec{V}_1\wedge \vec{u}\right) \right)\, ,
     \end{align*}
     satisfies  for all $\vec{u}\in \Gamma(TM^m)$ the identity
     \begin{align*}
     	\mathscr{L}_1(\vec{u})=\mathscr{L}_{\n}(\vec{u}). 
     \end{align*}
     However, we need to have the critical part of $\mathscr{L}_1$ self-adjoint, so we replace it by 
     \begin{align*}
     	\mathscr{L}\vec{u}=\Delta\vec{u}+\dive\left(-3\,\pi_{\tilde{\n}}(\D\vec{u})+\star_{\R^n}\left(\D^{\perp}\n\wedge\vec{V}_0\wedge\vec{u}\right)+\vec{A}_1\cdot\vec{u}+\star_{\R^n}\left(\vec{V}_1\wedge \vec{u}\right) \right)+\star_{\R^n}\left(\D^{\perp}\n\wedge \vec{V}_2\wedge \vec{u}\right)
     \end{align*}
     for some $\vec{V}_2$ to be determined later. The \emph{critical part} of $\mathscr{L}$ is defined by
     \begin{align*}
     	\mathscr{L}_0\vec{u}=\Delta\vec{u}-3\,\dive\left(\pi_{\tilde{n}}(\D\vec{u})\right)+\dive\left(\star_{\R^n}\left(\D^{\perp}\n\wedge \vec{V}_0\wedge\vec{u}\right)\right)+\star_{\R^n}\left(\D^{\perp}\n\wedge\vec{V}_2\wedge\vec{u}\right).
     \end{align*}
     Making the same computation as in the proof of  of \cite[Lemma A.$1$]{riviere1} (see (A.$28$)), we deduce that for all $\vec{v}\in C^{\infty}_c(B(0,1),\R^n)$, it holds
     \begin{align}\label{selfadjoint}
     	&\int_{B(0,1)}\s{\mathscr{L}_0\vec{u}}{\vec{v}}dx=\int_{B(0,1)}\bigg(\bs{\vec{u}}{\Delta\vec{v}}-3\,\bs{\vec{u}}{\dive\left(\pi_{\tilde{\n}}(\D\vec{u})\right)}-\bs{\star_{\R^n}\left(\D^{\perp}\n\wedge\vec{V}_0\wedge \vec{u}\right)}{\D\vec{v}}\nonumber\\
     	&+\bs{\star_{\R^n}\left(\D^{\perp}\n\wedge \vec{V}_2\wedge \vec{u}\right)}{\vec{v}}\bigg)dx\nonumber\\
     	&=\int_{B(0,1)}\bigg(\bs{\vec{u}}{\Delta\vec{v}}-3\,\bs{\vec{u}}{\dive\left(\pi_{\tilde{\n}}(\D\vec{u})\right)}+\bs{\vec{u}}{\star_{\R^n}\left(\D^{\perp}\n\wedge\vec{V}_0\wedge \vec{v}\right)}
     	-\bs{\vec{u}}{\star_{\R^n}\left(\D^{\perp}\n\wedge \vec{V}_2\wedge \vec{v}\right)}\bigg)dx\nonumber\\
     	&=\int_{B(0,1)}\left(\bs{\vec{u}}{\Delta\vec{v}-3\,\dive\left(\pi_{\tilde{\n}}(\D\vec{u})\right)+\dive\left(\star_{\R^n}\left(\D^{\perp}\n\wedge\vec{V}_0\wedge\D\vec{v}\right)-\star_{\R^n}\left(\D^{\perp}\n\wedge(\D\vec{V}_0+\vec{V}_2)\wedge \vec{v}\right)\right)}\right)dx\nonumber\\
     	&=\int_{B(0,1)}\s{\vec{u}}{\mathscr{L}_0\vec{v}}dx
     \end{align}
     if and only if 
     \begin{align*}
     	\vec{V}_2=-\frac{1}{2}\D\vec{V}_0\, .
     \end{align*}
     Notice that  $\n_{\iota}\in C^{\infty}(B(0,1),\Lambda^{n-m}\R^n)$, since $(M^m,h)$ is smooth.  Assuming without loss of generality that $\phi\in \mathrm{W}^{1,\infty}(B(0,1),M^m)$, this implies that $\D\vec{V}_0\in \mathrm{L}^{\infty}(B(0,1))$ and that $\mathscr{L}$ and $\mathscr{L}_0$ are well-defined in the distributional sense for all $\vec{u}\in \mathrm{L}^2(B(0,1),\R^n)$ since 
     \begin{align*}
     	\D^{\perp}\n\wedge\D\vec{V}_0\wedge \vec{u}\in \mathrm{L}^1(B(0,1)). 
     \end{align*}
 Next, we will study the properties of the operator $\mathscr{L}$.  
 Applying \cite[Lemma A.$2$ and Lemma A.4]{riviere1} directly  to the self-adjoint operator $\mathscr{L}_0$, one easily checks that  the arguments \cite[Lemma A.$1$ and Lemma A.3]{riviere1}  can be adapted to $\mathscr{L}$. The point is that $\mathscr{L}$  is \emph{not} self-adjoint, however $\mathscr{L}^{-1}:\mathrm{H}^{-1}(B(0,1),\R^n)\rightarrow \mathrm{W}^{1,2}_0(B(0,1),\R^n)$ is still a compact operator on $\mathrm{L}^{2}$ and this is the key to make the strategy work. Moreover,  using that $\mathscr{L}_0$ is indeed self-adjoint, one can adapt \cite[Lemma A.$8$]{riviere1}.  
     
     \begin{lemme}\label{new_A2}
     	For all $0<C_1<\infty$, there exists $\epsilon_0=\epsilon_0(n)>0$ and $C_0=C_0(n)$ with the following property. Let $2\leq m\leq n$. For all $\n_1\in \mathrm{W}^{1,2}(B(0,1),\Lambda^{m-2}\R^n)$ and $\n_2\in \mathrm{W}^{1,\infty}(B(0,1),\Lambda^{n-m}\R^n)$ letting $\n=\n_1\wedge\n_2\in \mathrm{W}^{1,2}(B(0,1),\Lambda^{n-2}\R^n)$, assume that
     	\begin{align}\label{eps_reg1}
     		&\int_{B(0,1)}|\D\n_1|^2dx+\int_{B(0,1)}|\D\n_2|^2dx\leq \epsilon_0,\nonumber\\
     		&\int_{B(0,1)}|\D\n_2|^4dx\leq C_1\,.
     	\end{align}
        Then, for all $\vec{f}\in \mathrm{H}^{-1}(B(0,1),\R^n)$, there exists a unique map $\vec{u}\in \mathrm{W}^{1,(2,\infty)}_0(B(0,1),\R^n)$ such that
        \begin{align}\label{solve}
        	\left\{\begin{alignedat}{2}
        		\Delta \vec{u}-3\,\dive\left(\pi_{\n}(\D\vec{u})\right)+\dive\left(\star(\D^{\perp}\n_1\wedge\n_2\wedge \vec{u})\right)-\frac{1}{2}\star\left(\D^{\perp}\n_1\wedge\D\n_2\wedge \vec{u}\right)&=\vec{f}\qquad&&\text{in}\;\, B(0,1)\\
        		\vec{u}&=0\qquad&&\text{on}\;\,\partial B(0,1)
        	\end{alignedat} \right.
        \end{align}
        and 
        \begin{align}\label{lemma_ineq1}
        	\np{\D \vec{u}}{2,\infty}{B(0,1)}\leq C_0\hs{\vec{f}\,}{-1}{B(0,1)}.
        \end{align}
        Furthermore, the operator $\mathscr{L}_0^{-1}:\vec{f}\rightarrow \vec{u}$ is a self-adjoint and compact operator from $\mathrm{L}^{2}(B(0,1),\R^n)$ into itself. 
     \end{lemme}
     \begin{proof}
     	First define 
     	\begin{align*}
     		\tilde{\mathscr{L}}_0\vec{u}=\Delta \vec{u}-3\,\dive\left(\pi_{\n}(\D\vec{u})\right)+\dive\left(\star(\D^{\perp}\n_1\wedge\n_2\wedge \vec{u})\right).
     	\end{align*}
     	Under the assumption \eqref{eps_reg1}, by the proofs of   \cite[Lemma A.$1$ and  Lemma A.$2$]{riviere1}, we deduce that there exists a unique $\vec{u}_0\in \mathrm{W}^{1,2}_0(B(0,1),\R^n)$ such that $\tilde{\mathscr{L}}_0\vec{u}_0=\vec{f}$ and
     	\begin{align*}
     		\np{\D\vec{u}_0}{2}{B(0,1)}\leq C_0\hs{\vec{f}}{-1}{B(0,1)}. 
     	\end{align*}
        Furthermore, by the proofs of  \cite[Lemma A.$3$ and  Lemma A.$4$]{riviere1},  there exists a unique $\vec{v}\in \mathrm{W}^{1,(2,\infty)}_0(B(0,1),\R^n)$ such that $\tilde{\mathscr{L}}_0\vec{v}=\vec{g}$ and 
        \begin{align*}
        	\np{\D\vec{v}}{2,\infty}{B(0,1)}\leq C_0\np{\vec{g}}{1}{B(0,1)}.
        \end{align*}
        Now, we will use a fixed point argument to conclude. We want to construct a function $\vec{u}$ such that 
        \begin{align*}
        	\mathscr{L}_0\vec{u}=\tilde{\mathscr{L}}_0\vec{u}_0
        \end{align*}
        or, equivalently, such that
        \begin{align*}
        	\tilde{\mathscr{L}}_0(\vec{u}-\vec{u}_0)=\frac{1}{2}\star\left(\D^{\perp}\n_1\wedge\D\n_2\wedge\vec{u}\right).
        \end{align*}
        By induction on $k\geq 1$,  we define  $\vec{u}_k$ to be the unique $\mathrm{W}^{1,(2,\infty)}_0(B(0,1),\R^n)$ function such that
        \begin{align*}
        	\mathscr{\mathscr{L}}_0\vec{u}_k=\frac{1}{2}\star\left(\D^{\perp}\n_1\wedge\D\n_2\wedge\vec{u}_{k-1}\right).
        \end{align*}
        Using that $\displaystyle \mathrm{W}^{1,(2,\infty)}(B(0,1))\hookrightarrow\bigcap_{p<\infty}\mathrm{L}^p(B(0,1))$,  by the Sobolev inequality and interpolation, we have
        \begin{align*}
        	&\np{\D\vec{u}_k}{2,\infty}{B(0,1)}\leq C_0\np{\D^{\perp}\n_1\wedge\D\n_2\wedge\vec{u}_{k-1}}{1}{B(0,1)}\\
        	&\leq C_0\np{\D\n_2}{4}{B(0,1)}\np{\D\n_1}{2}{B(0,1)}\np{\vec{u}_{k-1}}{4}{B(0,1)}
        	\leq C_0C_1C_{S}\epsilon_0\np{\D\vec{u}_{k-1}}{2,\infty}{B(0,1)},
        \end{align*}
        which implies that choosing $\epsilon_0=\min\ens{\epsilon_1,\frac{1}{2C_0C_1C_S}}$, the series   $\sum_{k=0}^{n}\vec{u}_k$ converges to some limit
        \begin{align*}
        	\vec{u}=\sum_{k=0}^{\infty}\vec{u}_k
        \end{align*}
        which satisfies
        \begin{align*}
        	\tilde{\mathscr{L}}_0\vec{u}=\tilde{\mathscr{L}}_0\vec{u}_0+\sum_{k=1}^{\infty}\frac{1}{2}\star\left(\D^{\perp}\n_1\wedge\D\n_2\wedge\vec{u}_{k-1}\right)=\vec{f}+\frac{1}{2}\star\left(\D^{\perp}\n_1\wedge\D\n_2\wedge\vec{u}\right),
        \end{align*}
        so that
        \begin{align*}
        	\mathscr{L}_0\vec{u}=\vec{f}. 
        \end{align*}
        Furthermore, we have
        \begin{align*}
        	\np{\D\vec{u}}{2,\infty}{B(0,1)}\leq 2\np{\D\vec{u}_0}{2,\infty}{B(0,1)}\leq 2C_0\hs{\vec{f}}{-1}{B(0,1)}
        \end{align*}
        which implies the needed estimate up to replacing $C_0$ by $2C_0$ in the statement of the theorem. Finally, the self-adjointness follows from the above computations in \eqref{selfadjoint}.      
     \end{proof}

     \begin{rem}
     	Notice that in fact 
     	\begin{align*}
     		\frac{1}{2}\star\left(\D^{\perp}\n_1\wedge\D\n_2\wedge\vec{u}\right)\in\bigcap_{p<2}\mathrm{L}^p(B(0,1))
     	\end{align*}
     	so we could also do the argument using Calder\'{o}n-Zgymund estimates and get a $\mathrm{W}^{1,2}_0$ solution to the original equation, but a $\mathrm{W}^{1,(2,\infty)}$ regularity suffices to our purposes. 
     \end{rem}
 
     Now, we easily obtain the following extension of Lemma \ref{new_A2}.
     
     \begin{lemme}\label{new_A1}
     	For all $0<C_1<\infty$ and $2<p<\infty$, there exists $\epsilon_1=\epsilon_1(p,n)>0$ and $C_1=C_1(n)$ with the following property. Let $2\leq m\leq n$. For all $\n_1\in \mathrm{W}^{1,2}(B(0,1),\Lambda^{m-2}\R^n)$ $\n_2\in \mathrm{W}^{1,\infty}(B(0,1),\Lambda^{n-m}\R^n)$, $\vec{A}_1\in \mathrm{L}^{\infty}(B(0,1),M_n(\R)\otimes\R^2)$ and $\vec{V}_1\in \mathrm{L}^{\infty}(B(0,1),\Lambda^{n-2}\R^n\otimes\R^2)$ letting $\n=\n_1\wedge\n_2\in \mathrm{W}^{1,2}(B(0,1),\Lambda^{n-2}\R^n)$, assume that
     	\begin{align}\label{eps_reg2}
     		\left\{\begin{alignedat}{1}
     		&\int_{B(0,1)}|\D\n_1|^2dx+\int_{B(0,1)}|\D\n_2|^2dx+\int_{B(0,1)}\left(|\vec{A}_1|^p+|\vec{V}_1|^p\right)dx\leq \epsilon_1,\\
     		&\int_{B(0,1)}|\D\n_2|^4dx\leq C_1\,.
     	\end{alignedat}\right.
     	\end{align}
     	Then, for all $\vec{f}\in \mathrm{H}^{-1}(B(0,1),\R^n)$, there exists a unique map $\vec{u}\in \mathrm{W}^{1,(2,\infty)}_0(B(0,1),\R^n)$ such that
     	\begin{align}\label{solve2}
     		\left\{\begin{alignedat}{2}
     			\Delta \vec{u}-3\,\dive\left(\pi_{\n}(\D\vec{u})\right)+\dive\left(\star(\D^{\perp}\n_1\wedge\n_2\wedge \vec{u})\right) \qquad \qquad &\\
     			-\frac{1}{2}\star\left(\D^{\perp}\n_1\wedge\D\n_2\wedge \vec{u}\right)
     			+\dive\left(\vec{A}_1\cdot\vec{u}+\star\left(\vec{V}_1\wedge\vec{u}\right)\right)&=\vec{f}\qquad&&\text{in}\;\, B(0,1)\\
     			\vec{u}&=0\qquad&&\text{on}\;\,\partial B(0,1)
     		\end{alignedat} \right.
     	\end{align}
     	and 
     	\begin{align}\label{lemma_ineq2}
     		\np{\D \vec{u}}{2,\infty}{B(0,1)}\leq C_1\hs{\vec{f}\,}{-1}{B(0,1)}.
     	\end{align}
     \end{lemme}
     \begin{proof}
     	By Lemma \ref{new_A2}, we deduce that there exists $\vec{u}_0\in \mathrm{W}^{1,(2,\infty)}_0(B(0,1))$ such that $\mathscr{L}_0\vec{u}_0=\vec{f}$. Then notice that by Hölder's inequality and the Sobolev inequality, we have for all $\vec{u}\in \mathrm{W}^{1,(2,\infty)}_0(B(0,1),\R^n)$
     	\begin{align*}
     		&\hs{\dive\left(\vec{A}_1\cdot\vec{u}\right)}{-1}{B(0,1)}\leq C_1\np{\vec{A}_1\cdot\vec{u}}{2}{B(0,1)}\leq C_1\np{\vec{A}_1}{p}{B(0,1)}\np{\vec{u}}{\frac{2p}{p-2}}{B(0,1)}\\
     		&\leq CC_p\np{\vec{A}_1}{p}{B(0,1)}\np{\D\vec{u}}{2,\infty}{B(0,1)}\leq C_1C_p\epsilon_1^{\frac{1}{p}}\np{\D\vec{u}}{2,\infty}{B(0,1)}.
     	\end{align*}
        Likewise, we have 
        \begin{align*}
        	\hs{\dive\left(\star\left(\vec{V}_1\wedge \vec{u}\right)\right)}{-1}{B(0,1)}\leq C_1C_p\epsilon_1^{\frac{1}{p}}\np{\D\vec{u}}{2,\infty}{B(0,1)}.
        \end{align*}
        Therefore, we deduce that for $\epsilon_1=\min\ens{\epsilon_0,\frac{1}{(4C_1C_p)^p}}$, we can use the same argument as in the proof of Lemma \ref{new_A2} to deduce that there exists $\vec{u}$ such that 
        \begin{align*}
        	\mathscr{L}\vec{u}=\mathscr{L}_0\vec{u}_0=\vec{f}
        \end{align*}
        and the estimate follows as above. 
     \end{proof}
 
     Adapting \emph{mutadis mutandis} the above results and the proof of  \cite[Lemma A.$3$]{riviere1}, we deduce the following two lemmas. 
 
      \begin{lemme}\label{new_A4}
     	For all $0<C_1<\infty$, there exists $\epsilon_0=\epsilon_0(n)>0$ and $C_0=C_0(n)$ with the following property. Let $2\leq m\leq n$. For all $\n_1\in \mathrm{W}^{1,2}(B(0,1),\Lambda^{m-2}\R^n)$ and $\n_2\in \mathrm{W}^{1,\infty}(B(0,1),\Lambda^{n-m}\R^n)$ letting $\n=\n_1\wedge\n_2\in \mathrm{W}^{1,2}(B(0,1),\Lambda^{n-2}\R^n)$, assume that
     	\begin{align}\label{eps_reg3}
     		&\int_{B(0,1)}|\D\n_1|^2dx+\int_{B(0,1)}|\D\n_2|^2dx\leq \epsilon_0,\nonumber\\
     		&\int_{B(0,1)}|\D\n_2|^4dx\leq C_1\, .
     	\end{align}
     	Then for all $\vec{f}\in \mathrm{L}^1(B(0,1),\R^n)$, there exists a unique map $\vec{u}\in \mathrm{W}^{1,(2,\infty)}_0(B(0,1),\R^n)$ such that
     	\begin{align}\label{solve3}
     		\left\{\begin{alignedat}{2}
     			\Delta \vec{u}-3\,\dive\left(\pi_{\n}(\D\vec{u})\right)+\dive\left(\star(\D^{\perp}\n_1\wedge\n_2\wedge \vec{u})\right)-\frac{1}{2}\star\left(\D^{\perp}\n_1\wedge\D\n_2\wedge \vec{u}\right)&=\vec{f}\qquad&&\text{in}\;\, B(0,1)\\
     			\vec{u}&=0\qquad&&\text{on}\;\,\partial B(0,1)
     		\end{alignedat} \right.
     	\end{align}
     	and 
     	\begin{align}\label{lemma_ineq3}
     		\np{\D \vec{u}}{2,\infty}{B(0,1)}\leq C_0\np{\vec{f}\,}{1}{B(0,1)}.
     	\end{align}
    \end{lemme}     
     
     \begin{lemme}\label{new_A3}
     	For all $0<C_1<\infty$ and $2<p<\infty$, there exists $\epsilon_1=\epsilon_1(p,n)>0$ and $C_1=C_1(n)$ with the following property. Let $2\leq m\leq n$. For all $\n_1\in \mathrm{W}^{1,2}(B(0,1),\Lambda^{m-2}\R^n)$ $\n_2\in \mathrm{W}^{1,\infty}(B(0,1),\Lambda^{n-m}\R^n)$, $\vec{A}_1\in \mathrm{L}^{\infty}(B(0,1),M_n(\R)\otimes\R^2)$ and $\vec{V}_1\in \mathrm{L}^{\infty}(B(0,1),\Lambda^{n-2}\R^n\otimes\R^2)$ letting $\n=\n_1\wedge\n_2\in \mathrm{W}^{1,2}(B(0,1),\Lambda^{n-2}\R^n)$, assume that
     	\begin{align}\label{eps_reg4}
     		\left\{\begin{alignedat}{1}
     			&\int_{B(0,1)}|\D\n_1|^2dx+\int_{B(0,1)}|\D\n_2|^2dx+\int_{B(0,1)}\left(|\vec{A}_1|^p+|\vec{V}_1|^p\right)dx\leq \epsilon_1,\\
     			&\int_{B(0,1)}|\D\n_2|^4dx\leq C_1\,.
     		\end{alignedat}\right.
     	\end{align}
     	Then, for all $\vec{f}\in \mathrm{L}^1(B(0,1),\R^n)$, there exists a unique map $\vec{u}\in \mathrm{W}^{1,(2,\infty)}_0(B(0,1),\R^n)$ such that
     	\begin{align}\label{solve4}
     		\left\{\begin{alignedat}{2}
     			\Delta \vec{u}-3\,\dive\left(\pi_{\n}(\D\vec{u})\right)+\dive\left(\star(\D^{\perp}\n_1\wedge\n_2\wedge \vec{u})\right) \qquad \qquad&\\
     			-\frac{1}{2}\star\left(\D^{\perp}\n_1\wedge\D\n_2\wedge \vec{u}\right)
     			+\dive\left(\vec{A}_1\cdot\vec{u}+\star\left(\vec{V}_1\wedge\vec{u}\right)\right)&=\vec{f}\qquad&&\text{in}\;\, B(0,1)\\
     			\vec{u}&=0\qquad&&\text{on}\;\,\partial B(0,1)
     		\end{alignedat} \right.
     	\end{align}
     	and 
     	\begin{align}\label{lemma_ineq4}
     		\np{\D \vec{u}}{2,\infty}{B(0,1)}\leq C_1\np{\vec{f}\,}{1}{B(0,1)}.
     	\end{align}
     \end{lemme}
     
     Now, we can show in a similar way the improved regularity for eigenvectors of $\mathscr{L}_0$.
     
     \begin{lemme}\label{pre_unique}
     	For all $0<C_1<\infty$, there exists $\epsilon_0=\epsilon_0(n)>0$ and $C_0=C_0(n)$ with the following property. Let $2\leq m\leq n$. For all $\n_1\in \mathrm{W}^{1,(2,1)}(B(0,1),\Lambda^{m-2}\R^n)$ and $\n_2\in \mathrm{W}^{1,\infty}(B(0,1),\Lambda^{n-m}\R^n)$ letting $\n=\n_1\wedge\n_2\in \mathrm{W}^{1,2}(B(0,1),\Lambda^{n-2}\R^n)$, assume that 
     	\begin{align}\label{eps_reg5}
     		&\int_{B(0,1)}|\D\n_1|^2dx+\int_{B(0,1)}|\D\n_2|^2dx\leq \epsilon_0,\nonumber\\
     		&\np{\D\n_2}{\infty}{B(0,1)}\leq C_1.
     	\end{align}
     	Assume that $\vec{\varphi}\in \mathrm{W}^{1,(2,\infty)}_0(B(0,1),\R^n)$ is an eigen-vector of $\mathscr{L}_0$ \emph{i.e.} that there exists $\lambda\in  \R$ such that
     	\begin{align}\label{solve5}
     		\left\{\begin{alignedat}{2}
     			\Delta \vec{\varphi}-3\,\dive\left(\pi_{\n}(\D\vec{\varphi})\right)+\dive\left(\star(\D^{\perp}\n_1\wedge\n_2\wedge \vec{\varphi})\right)-\frac{1}{2}\star\left(\D^{\perp}\n_1\wedge\D\n_2\wedge \vec{\varphi}\right)&=\lambda\,\vec{\varphi}\qquad&&\text{in}\;\, B(0,1)\\
     			\vec{\varphi}&=0\qquad&&\text{on}\;\,\partial B(0,1) \,.
     		\end{alignedat} \right.
     	\end{align}
     	Then
     	\begin{align}\label{eigen3}
     		\np{\D^2\vec{\varphi}\,}{2,1}{B(0,1)}\leq C_0\np{\D\vec{\varphi}\,}{2,\infty}{B(0,1)}.
     	\end{align}
     \end{lemme}
     \begin{proof}
     	The first part of the proof goes as in \cite{riviere1} and we deduce by the Sobolev embedding
     	\begin{align*}
     		\mathrm{W}^{1,2}_0(B(0,1))\hookrightarrow\bigcap_{p<\infty}\mathrm{L}^p(B(0,1))
     	\end{align*}
        that 
        \begin{align*}
        	\vec{\varphi}\in\bigcap_{p<\infty}\mathrm{W}^{1,p}_0(B(0,1)). 
        \end{align*}
        Now, as in \cite{riviere1}, let $x_0\in B(0,1)$. Since $\D\n\in \mathrm{L}^{2,1}(B(0,1))$ and $\D\n_2\in \mathrm{L}^4(B(0,1))$, for all $\epsilon>0$, there exists $\rho>0$, there exists $\rho>0$ such that
        \begin{align*}
        	\np{\D\n_1}{2,1}{B(x_0,\rho)}+\np{\D\n_2}{2,1}{B(x_0,\rho)}\leq \epsilon.
        \end{align*}
        We now define $\vec{u}:B(0,1)\rightarrow \R^n$ by
        \begin{align*}
        	\vec{u}(x)=\vec{\varphi}(x_0+\rho\,x)\eta(x)=\vec{\psi}(x)\,\eta(x),
        \end{align*}
        where $\eta$ is a smooth cut-off function such that $\eta=1$ on $B(0,\frac{1}{2})$. Then we compute (without changing notations for the rescaled versions of $\n$)
        \begin{align*}
        	\mathscr{L}_0\vec{u}&=\lambda \vec{u}+2\D\vec{\psi}\cdot\D\eta+\vec{\psi} \Delta\eta-3\dive\left(\pi_{\n}(\vec{\psi})\D\eta+\eta\pi_{\n}(\D\vec{\psi})\right)+3\,\eta\dive\left(\pi_{\n}(\D\vec{\psi})\right)\\
        	&\qquad+\star\left(\D^{\perp}\n_1\wedge\D\n_2\wedge\vec{\psi}\right)\cdot\D\eta
        	 +\star\left(\D^{\perp}\n_1\wedge\n_2\wedge\vec{\psi}\right)\Delta\eta\\
        	&=\lambda\,\vec{u}+2\D\vec{\psi}\cdot\D\eta+\vec{\psi}\Delta\eta-3\,\D\pi_{\n}(\vec{\psi})\cdot\D\eta-6\,\dive(\pi_{\n}(\D\vec{\psi}))\cdot\D\eta-3\,\pi_{\n}(\vec{\psi})\Delta\eta\\
        	&\qquad +\star\left(\D^{\perp}\n_1\wedge\D\n_2\wedge\vec{\psi}\right)\cdot\D\eta+\star\left(\D^{\perp}\n_1\wedge\n_2\wedge\vec{\psi}\right)\Delta\eta\\
        	&=\vec{f}.
        \end{align*} 
        Since $\D\n_1\in \mathrm{L}^{2,1}(B(0,1))$, $\D\n_2\in \mathrm{L}^{\infty}(B(0,1))$ while $\vec{\psi}\in \mathrm{W}^{1,p}_0(B(0,1))$ for all $p<\infty$ (which implies in particular that $\vec{\psi}\in C^0(B(0,1))$), we deduce that $\vec{f}\in \mathrm{L}^{2,1}(B(0,1))$. Therefore, $\vec{u}\in \mathrm{W}^{1,p}_0(B(0,1),\R^n)$ (for all $p<\infty$) solves the system
        \begin{align}
        	\left\{\begin{alignedat}{2}
        		\mathscr{L}_0\vec{u}&=\vec{f}\qquad&&\text{in}\;\,B(0,1)\\
        		\vec{u}&=0\qquad&&\text{on}\;\, \partial B(0,1).
        \end{alignedat}\right.
        \end{align}
        Following \cite{riviere1}, introduce the Hodge decomposition
        \begin{align*}
        	\D\vec{u}-3\pi_{\n}(\D\vec{u})=\D\vec{A}+\D^{\perp}\vec{B},
        \end{align*}
        where $\vec{A}=0$ and $\partial_{\nu}\vec{B}=0$ on $\partial B(0,1)$. Then we directly have
        \begin{align*}
        	\Delta\vec{A}=\Delta\vec{u}-3\,\dive\left(\pi_{\n}(\D\vec{u})\right)=-\dive\left(\star\left(\D^{\perp}\n_1\wedge\n_2\wedge\vec{u}\right)\right)+\frac{1}{2}\star\left(\D^{\perp}\n_1\wedge\D\n_2\wedge\vec{u}\right)+\vec{f}
        \end{align*}
        so that
        \begin{align*}
        	\left\{\begin{alignedat}{2}
        		\Delta\vec{A}&=-\dive\left(\star\left(\D^{\perp}\n_1\wedge\n_2\wedge\vec{u}\right)\right)+\frac{1}{2}\star\left(\D^{\perp}\n_1\wedge\D\n_2\wedge\vec{u}\right)+\vec{f}\qquad&&\text{in}\;\,B(0,1)\\
        		\vec{A}&=0\qquad&&\text{on}\;\,\partial B(0,1). 
        	\end{alignedat}\right.
        \end{align*}
        Likewise, we directly obtain
        \begin{align*}
        	\left\{\begin{alignedat}{2}
        		\Delta\vec{B}&=3\,\dive\left(\pi_{\n}(\D^{\perp}\vec{u})\right)\qquad&&\text{in}\;\,B(0,1)\\
        		\vec{B}&=0\qquad&&\text{on}\;\,\partial B(0,1).
        	\end{alignedat}\right.
        \end{align*}
        By Calder\'{o}n-Zygmund estimates, interpolation, and the Sobolev embedding $\mathrm{W}^{1,(2,1)}_0(B(0,1))\hookrightarrow C^0(B(0,1))$, we deduce that 
        \begin{align*}
        	&\np{\D\vec{A}\,}{\infty}{B(0,1)}+\np{\D^2\vec{A}\,}{2,1}{B(0,1)}\leq  C\np{\Delta\vec{A}\,}{2,1}{B(0,1)}\\
        	&\leq C\left(\np{\D\n_1}{2,1}{B(0,1)}\left(\np{\D\n_2}{\infty}{B(0,1)}\np{\vec{u}}{\infty}{B(0,1)}+\np{\D\vec{u}}{\infty}{B(0,1)}\right)
        	+\np{\vec{f}\,}{2,1}{B(0,1)}\right)\\
        	&\np{\D\vec{B}\,}{\infty}{B(0,1)}+\np{\D^2\vec{B}\,}{2,1}{B(0,1)}\leq C\np{\Delta\vec{B}\,}{2,1}{B(0,1)}\leq C\np{\D\n}{2,1}{B(0,1)}\np{\D\vec{u}}{\infty}{B(0,1)}
        \end{align*}
       From this point we can follow verbatim the proof of \cite[Lemma A.$7$]{riviere1}  to conclude.  
     \end{proof} 
 
     Finally, we deduce the counterpart of \cite[Lemma A.$8$]{riviere1} for the operator $\mathscr{L}_0$. 
     
     \begin{lemme}\label{A81}
     	For all $0<C_1<\infty$, there exists $\epsilon_0=\epsilon_0(n)>0$ and $C_0=C_0(n)$ with the following property. Let $2\leq m\leq n$. For all $\n_1\in \mathrm{W}^{1,(2,1)}(B(0,1),\Lambda^{m-2}\R^n)$ and $\n_2\in \mathrm{W}^{1,\infty}(B(0,1),\Lambda^{n-m}\R^n)$ letting $\n=\n_1\wedge\n_2\in \mathrm{W}^{1,2}(B(0,1),\Lambda^{n-2}\R^n)$, assume that
     	\begin{align}\label{eps_reg6}
     		&\int_{B(0,1)}|\D\n_1|^2dx+\int_{B(0,1)}|\D\n_2|^2dx\leq \epsilon_0,\nonumber\\
     		&\np{\D\n_2}{\infty}{B(0,1)}\leq C_1.
     	\end{align}
     	Assume that $\vec{u}\in \mathrm{L}^2(B(0,1),\R^n)$ and that $\D\vec{u}$ is the sum of a 
     	compactly supported distribution in $B(0,1)$ and of a function in $\mathrm{L}^{2,\infty}(B(0,1))$. Suppose that $\vec{u}$ satisfies the following equation in the distributional sense
     	\begin{align}\label{solve6}
     		\left\{\begin{alignedat}{2}
     			\Delta \vec{u}-3\,\dive\left(\pi_{\n}(\D\vec{u})\right)+\dive\left(\star(\D^{\perp}\n_1\wedge\n_2\wedge \vec{u})\right)-\frac{1}{2}\star\left(\D^{\perp}\n_1\wedge\D\n_2\wedge \vec{u}\right)&=0\qquad&&\text{in}\;\, B(0,1)\\
     			\vec{u}&=0\qquad&&\text{on}\;\,\partial B(0,1).
     		\end{alignedat} \right.
     	\end{align}
        Then $\vec{u}=0$ on $B(0,1)$. 
     \end{lemme} 
     \begin{proof}
          We can follow verbatim the proof of  \cite[Lemma A.$8$]{riviere1} since only the self-adjointness of the operator and the regularity result of Lemma \ref{pre_unique} are used.
     \end{proof}
     
     Finally, we deduce that Lemma \ref{A81} holds with $\mathscr{L}_0$ replaced by $\mathscr{L}$. 
     
     \begin{lemme}\label{A82}
     	For all $0<C_1<\infty$, there exists $\epsilon_0=\epsilon_0(n)>0$ and $C_0=C_0(n)$ with the following property. Let $2\leq m\leq n$. For all $\n_1\in \mathrm{W}^{1,(2,1)}(B(0,1),\Lambda^{m-2}\R^n)$ and $\n_2\in \mathrm{W}^{1,\infty}(B(0,1),\Lambda^{n-m}\R^n)$ letting $\n=\n_1\wedge\n_2\in \mathrm{W}^{1,2}(B(0,1),\Lambda^{n-2}\R^n)$, assume that 
     	\begin{align}\label{eps_reg7}
     		&\int_{B(0,1)}|\D\n_1|^2dx+\int_{B(0,1)}|\D\n_2|^2dx\leq \epsilon_0,\nonumber\\
     		&\np{\D\n_2}{\infty}{B(0,1)}\leq C_1.
     	\end{align}
     	Assume that $\vec{u}\in \mathrm{L}^2(B(0,1),\R^n)$ and that $\D\vec{u}$ is the sum of a 
     	compactly supported distribution in $B(0,1)$ and of a function in $\mathrm{L}^{2,\infty}(B(0,1))$. Suppose that $\vec{u}$ satisfies the following equation in the distributional sense
     	\begin{align}\label{solve7}
     		\left\{\begin{alignedat}{2}
     			\Delta \vec{u}-3\,\dive\left(\pi_{\n}(\D\vec{u})\right)+\dive\left(\star(\D^{\perp}\n_1\wedge\n_2\wedge \vec{u})\right) \qquad \qquad&\\
     			-\frac{1}{2}\star\left(\D^{\perp}\n_1\wedge\D\n_2\wedge \vec{u}\right)
     			+\dive\left(\vec{A}_1\cdot\vec{u}+\star\left(\vec{V}_1\wedge\vec{u}\right)\right)&=0\qquad&&\text{in}\;\, B(0,1)\\
     			\vec{u}&=0\qquad&&\text{on}\;\,\partial B(0,1)
     		\end{alignedat} \right.
     	\end{align}
        Then $\vec{u}=0$ on $B(0,1)$.
     \end{lemme}
     \begin{proof}
        Since $\vec{u}\in \mathrm{L}^2(B(0,1))$ and $\vec{A}_1,\vec{V}_1\in \mathrm{L}^{\infty}(B(0,1))$, we deduce that 
        \begin{align*}
        	\vec{f}=\dive\left(\vec{A}_1\cdot\vec{u}+\star\left(\vec{V}_1\wedge\vec{u}\right)\right)\in \mathrm{H}^{-1}(B(0,1)). 
        \end{align*}
        Using Lemma \ref{new_A2}, we deduce that there exists a unique $\vec{v}\in \mathrm{W}^{1,2}_0(B(0,1))$ such that 
        \begin{align*}
        	\mathscr{L}_0\vec{v}=\vec{f}.
        \end{align*}
        Therefore, we get that $\vec{w}=\vec{u}-\vec{v}$ satisfies 
        \begin{align*}
        	\left\{\begin{alignedat}{2}
        		\Delta \vec{w}-3\,\dive\left(\pi_{\n}(\D\vec{w})\right)+\dive\left(\star(\D^{\perp}\n_1\wedge\n_2\wedge \vec{w})\right)-\frac{1}{2}\star\left(\D^{\perp}\n_1\wedge\D\n_2\wedge \vec{w}\right)&=0\qquad&&\text{in}\;\, B(0,1)\\
        		\vec{w}&=0\qquad&&\text{on}\;\,\partial B(0,1).
        	\end{alignedat} \right.
        \end{align*}
        Since $\vec{v}\in \mathrm{L}^{2,\infty}(B(0,1))$, we deduce that $\D\vec{w}$ is also the sum of a compactly supported distribution and of a $\mathrm{L}^{2,\infty}(B(0,1))$ function. Therefore, by Lemma \ref{A81}, we obtain that $\vec{w}=0$, \emph{i.e.}  $\vec{u}=\vec{v}\in \mathrm{W}^{1,(2,\infty)}_0(B(0,1))$. 
        Using the uniqueness of the weak solution in $\mathrm{W}^{1,(2,\infty)}_0(B(0,1))$ of the Dirichlet problem associated to $\mathscr{L}$ (see Lemma \ref{new_A1} and Lemma \ref{new_A3}), we conclude that $\vec{u}=0$. 
     \end{proof}

   Finally, we can prove the $\epsilon$-regularity result (see \cite[Theorem I.5]{riviere1} for Willmore immersions in the Euclidean space).
    
     \begin{theorem}\label{eps_reg}
     	Let $(M^m,h)$ be a smooth closed Riemannian manifold that we assume isometrically embedded in $\R^n$.. There exists constants $\epsilon_0>0$ and $\ens{C_k}_{k\in \N}\subset (0,\infty)$ with the following property. Let $\phi:B(0,1)\rightarrow (M^m,h)$ be a conformal, Lipschitz ; denote with  $\n:B(0,1)\rightarrow \Lambda^{m-2}TM^m$  the Gauss map associated to $\phi$ and assume that 
     	\begin{align}\label{eq:AssumEpsReg}
     		\int_{B(0,1)}|\D\n|^2dx+\mathrm{Area}(\phi(B(0,1)))\leq \epsilon_0\, .
     	\end{align}
        Then, for all $k\in\N$, it holds
        \begin{align*}
        	\np{\D^k\n}{\infty}{B(0,\frac{1}{2})}^2\leq C_k\int_{B(0,1)}|\D\n|^2dx.
        \end{align*}
     \end{theorem}
     \begin{proof}  
     \textbf{Step 1}. Proof that $\phi\in \mathrm{W}^{2,(2,1)}_{\mathrm{loc}}(B(0,1))$.
\\       Using \cite[Lemma $6.1$ and $6.2$]{mondinoriviere}, we deduce that there exists $\vec{L}\in \mathrm{L}^{2,\infty}(B(0,1),\C^n)$ such that $\D\Im(\vec{L})\in \mathrm{W}^{1,(2,\infty)}(B(0,1))$ and 
        \begin{align*}
        	\left\{\begin{alignedat}{2}
        		\D_z\vec{L}&=\vec{Y}=i\left(\D_z\H-3\,\D_z^{\perp}\H-i\star_h(\D_z\n\wedge \H)\right)\qquad&&\text{in}\;\, B(0,1)\\
        		\Im(\vec{L})&=0\qquad&&\text{on}\;\,\partial B(0,1). 
        	\end{alignedat}\right.
        \end{align*}
        Likewise, there exists $S\in \mathrm{W}^{1,(2,\infty)}(B(0,1),\C)$ such that $\Im(S)\in \mathrm{W}^{2,q}(B(0,1),\R)$ for all $q<2$ such that 
        \begin{align*}
        	\left \{\begin{alignedat}{2}
        		\p{z}S&=\s{\p{z}\phi}{\bar{\vec{L}}}\qquad&&\text{in}\;\, B(0,1)\\
        		\Im(S)&=0\qquad&&\text{on}\;\,\partial B(0,1)
        \end{alignedat}\right.
        \end{align*}
        and $\vec{R}\in \mathrm{W}^{1,(2,\infty)}(B(0,1),\Lambda^2\C^n)$ such that $\Im(\vec{R})\in \mathrm{W}^{2,q}(B(0,1),\Lambda^2\R^n)$ for all $q<2$ and 
        \begin{align*}
        	\left\{\begin{alignedat}{2}
             	\D_z\vec{R}&=\p{z}\phi\wedge\bar{\vec{L}}-2i\,\p{z}\phi\wedge\H\qquad&&\text{in}\;\,B(0,1)\\
            	\vec{R}&=0\qquad&&\text{on}\;\,\partial B(0,1). 
            \end{alignedat}\right.
        \end{align*}
        Furthermore, using  \cite[Proposition $6.1$]{mondinoriviere}, we deduce that 
        \begin{align*}
        	\left\{\begin{alignedat}{2}
        	\Delta\Re(S)&=\s{\D(\star_h\n)}{\D^{\perp}\vec{R}}+F\qquad\text{in}\;\, B(0,1)\\
        	\Delta\Re(\vec{R})&=(-1)^m\star_h\left(\D\n\res \D^{\perp}\Re(\vec{R})\right)-\star_h\s{\D\n}{\D^{\perp}\Re(S)}+\vec{G}\qquad\text{in}\;\, B(0,1),
        	\end{alignedat}\right. 
        \end{align*}
        where $F\in \mathrm{L}^p(B(0,1))$ for all $p<2$ and $\vec{G}\in \mathrm{L}^p(B(0,1),\Lambda^2\R^n)$ for all $p<2$. 
        Therefore, we can perform a decomposition $\Re(S)=u+v+w$, where 
        \begin{align*}
        	\left\{\begin{alignedat}{2}
        		\Delta u&=\s{\D\left(\star_ h\n\right)}{\D^{\perp}\vec{R}}\qquad&&\text{in}\;\, B(0,1)\\
        		u&=0\qquad&&\text{on}\;\,\partial B(0,1)\, 
        \end{alignedat}\right. \qquad \text{and} \qquad 
        	\left\{\begin{alignedat}{2}
        		\Delta v&=F\qquad&&\text{in}\;\, B(0,1)\\
        		v&=0\qquad&&\text{on}\;\,\partial B(0,1).
        	\end{alignedat}\right.
        \end{align*}
        Since $\vec{S},\vec{R}\in \mathrm{W}^{1,(2,\infty)}(B(0,1))$ and $\n\in \mathrm{W}^{1,2}(B(0,1))$, by the $\mathrm{L}^2$-$\mathrm{L}^{2,\infty}$ Wente estimate (\cite[Théorème $3.4.5$]{helein}), we deduce that $\D u\in \mathrm{L}^2(B(0,1))$ and that there exists a universal constant $C_1>0$ such that
        \begin{align*}
        	\np{\D u}{2}{B(0,1)}\leq C_1\np{\D\vec{R}}{2,\infty}{B(0,1)}\np{\D\n}{2}{B(0,1)}.
        \end{align*}
        By Calder\'{o}n-Zygmund estimates, we get that $v\in \mathrm{W}^{2,q}(B(0,1))$ for all $q<2$. Also, since $w$ is harmonic, we obtain (see \cite{quanta} and \cite{quantamoduli}, or  \cite[Lemma $2.2$ and $2.3$] {pointwise}) that  for all $r<1$ there exists $C_{r}<\infty$ such that
        \begin{align*}
        	\np{\D w}{2,1}{B(0,r)}+\np{\D^2w}{1}{B(0,r)}\leq C_{r}\np{\D w}{2,\infty}{B(0,1)}.
        \end{align*}
 We deduce that $\Re(S)\in \mathrm{W}^{1,2}(B(0,r))$ for all $r<1$, and a similar decomposition also yields that $\Re(\vec{R})\in \mathrm{W}^{1,2}(B(0,r))$ for all $r<1$. Using the Coifman-Lions-Meyer-Semmes compensation compactness result \cite{meyer}, we deduce that $\vec{R},\vec{S}\in \mathrm{W}^{2,1}(B(0,r))$ for all $r<1$ which implies in particular that $\vec{R},\vec{S}\in \mathrm{W}^{1,(2,1)}\cap C^0(B(0,r))$ for all $r<1$.
 
 By \cite[Identity (6.52)]{mondinoriviere}, we have
        \begin{align}\label{localHk}
        	e^{2\lambda}\H=-\Im\left(\D_z\vec{R}\res \p{\z}\phi\right)-\frac{1}{2}e^{2\lambda}\Im(\vec{L})-\Re\left(i\p{\z}\phi\,\p{z}S\right)+\Re\left(\s{\p{z}\phi}{\Im(\vec{L})}\p{\z}\phi\right)\,.
        \end{align}
        Using that $\Im(\vec{L})\in \mathrm{L}^p(B(0,1))$ for all $p<\infty$,  we deduce that 
        \begin{align*}
        	e^{\lambda}\H\in \mathrm{L}^{2,1}(B(0,r))\qquad\text{for all}\;\,r<1. 
        \end{align*}
        Recalling that $\Delta\phi=e^{2\lambda}\H$ and using that $\lambda\in \mathrm{L}^{\infty}(B(0,1))$, by Calder\'{o}n-Zygmund estimates  we deduce that $\phi\in \mathrm{W}^{2,(2,1)}(B(0,r))$ for all $r<1$, which implies in particular that $\D\n\in \mathrm{L}^{2,1}(B(0,1))$.

       \textbf{Step 2}.  Proof that  $\H\in \mathrm{W}^{1,(2,\infty)}(B(0,\frac{1}{2}))$.
       \\        To simplify notations, write $\n=\n_1$ and $\n=\n_1\wedge\n_2$, where $\n_2=\n_{\iota}\circ\phi$. Since $M^m$ is a smooth compact submanifold of $\R^n$, we have that $\n_{\iota}\in C^{\infty}(M^m,T\R^n)$ which implies that $\n_2\in \mathrm{W}^{1,\infty}(B(0,1),\Lambda^{n-m}\R^n)$, since $\phi$ is Lipschitz. 
        Next we will apply the lemmas above, exploiting the assumption \eqref{eq:AssumEpsReg}.
\\   By \eqref{operator_bounds}, we easily check that the assumptions on $\vec{A}_1$ and $\vec{V}_1$ are satisfided in Lemmas \ref{new_A1}, \ref{new_A3} and \ref{A82} for some $p>2$ small enough. Indeed, we have $|\vec{A}_1|+|\vec{V}_1|\leq C_2e^{\lambda}$, which implies that 
        \begin{align*}
        	\int_{B(0,1)}\left(|\vec{A}_1|^p+|\vec{V}_1|^p\right)dx&\leq 2^pC_2^pe^{(p-2)\np{\lambda}{\infty}{B(0,1)}}\int_{B(0,1)}e^{2\lambda}dx=C_2^pe^{(p-2)\np{\lambda}{\infty}{B(0,1)}}\mathrm{Area}(\phi(B(0,1)))\\
        	&\leq 4C_2^2(1+C_2)\epsilon_0
        \end{align*}
        for $p>2$ small enough. However, that would make 
        $\epsilon_0$ and $C_{k}$ depend also on $\np{\lambda}{\infty}{B(0,1)}$.  To circumvent that, in the argument below we will instead use the Lemmas \ref{new_A2}, \ref{new_A4} and \ref{A81}.
        By \eqref{eqH}, we have
        \begin{align*}
        	&\dive\left(\D\H-3\,\pi_{\n}(\D\H)+\star_h\left(\D^{\perp}\n\wedge\H\right)+\vec{A}_0\cdot\H+\star_h(\vec{V}\wedge\H)\right)\nonumber\\
        	&=-2e^{2\lambda}\left(\mathscr{R}_1^{\perp}(\H)-2\,\tilde{K}_h\,\H+2\,\mathscr{R}_2(d\phi)+(DR)(d\phi)-8\,\Re\left(\s{R(\e_{\z},\e_z)\e_z}{\H}\e_{\z}\right)\right)\nonumber\\
        	&\quad -\vec{A}_1\cdot \H-{3}\,\vec{A}_2(\n)\cdot\H-2\,\vec{A}_3\cdot\H-\vec{A}_4\cdot\star_h\left(\D^{\perp}\n\wedge\H\right)-\vec{A}_5\cdot \star_h\left(\vec{W}\wedge\H\right)-B(\H,\H),
        \end{align*}
        which can be rewritten as 
        \begin{align*}
        	&\mathscr{L}(\H)=\Delta\vec{H}+\dive\left(-3\,\pi_{\n}(\D\vec{H})+\star \left(\D^{\perp}\n_1\wedge\n_2\wedge\vec{H}\right)+\vec{A}_1\cdot\vec{H}+\star \left(\vec{V}_1\wedge \vec{H}\right) \right)\\
        	&-\frac{1}{2}\star\left(\D^{\perp}\n_1\wedge \D\n_2\wedge \vec{H}\right)\\
        	&=-2e^{2\lambda}\left(\mathscr{R}_1^{\perp}(\H)-2\,\tilde{K}_h\,\H+2\,\mathscr{R}_2(d\phi)+(DR)(d\phi)-8\,\Re\left(\s{R(\e_{\z},\e_z)\e_z}{\H}\e_{\z}\right)\right)\nonumber\\
        	&\quad -\vec{A}_1\cdot \H-{3}\,\vec{A}_2(\n_1)\cdot\H-2\,\vec{A}_3\cdot\H-\vec{A}_4\cdot\star\left(\D^{\perp}\n_1\wedge\n_2\wedge\H\right)-\vec{A}_5\cdot \star\left(\tilde{\vec{W}}\wedge\H\right)-B(\H,\H)\\
        	&\quad -\frac{1}{2}\left(\D^{\perp}\n_1\wedge \D\n_2\wedge \vec{H}\right)\\
	&=\vec{f}.
        \end{align*}
        Notice that by hypothesis, we have $\vec{f}\in \mathrm{L}^1(B(0,1))$. 
        
        Since $\mathscr{L}(\H)=\vec{f}$, if $\eta$ is a smooth cut-off function 
        supported on $B(0,1)$ such that $\eta=1$ on $B(0,\frac{1}{2})$, we deduce that the following identity holds in the distributional sense
        \begin{align*}
        	\dive\left(\pi_{\n}\left(\D\left(\eta\H\right)\right)\right)&=\dive\left(\eta\pi_{\n}(\D\H)+\D\eta\H\right)=\eta\,\dive\left(\pi_{\n}(\D\H)\right)+\D\eta\cdot\pi_{\n}(\D\H)+\dive\left(\D\eta\H\right)\\
        	&=\eta\dive\left(\pi_{\n}(\D\H)\right)+2\,\dive\left(\D\eta\H\right)-\Delta\,\eta\H-\D\eta\cdot\D\pi_{\n}(\H),
        \end{align*}
        thus
        \begin{align*}
        	\Delta(\eta\H)=\eta\Delta\H+2\s{\D\eta}{\D\H}+\Delta\eta\,\H=\eta\Delta\H+2\,\dive\left(\D\eta\H\right)-\Delta\eta\H\, .
        \end{align*}
        Therefore, we have
        \begin{align*}
        	\mathscr{L}(\eta\H)&=\eta\vec{f}_1+2\,\dive\left(\D\eta\H\right)-\Delta\eta\,\H-3\left(2\,\dive\left(\D\eta\H\right)-\Delta\eta\H-\D\eta\,\cdot\D\pi_{\n}(\H)\right)\\
        	&\qquad +\s{\D\eta}{\star(\D^{\perp}\n_1\wedge\n_2\wedge\H)}+\s{\D\eta}{\vec{A}_1\cdot\H}+\s{\D\eta}{\star(\vec{V}_1\wedge\H)}\\
        	&=\eta\vec{f}+2\,\Delta\eta\,\H+3\D\eta\cdot\D\pi_{\n}(\H)+\s{\D\eta}{\star(\D^{\perp}\n_1\wedge\n_2\wedge\H)}+\s{\D\eta}{\vec{A}_1\cdot\H}+\s{\D\eta}{\star(\vec{V}_1\wedge\H)}\\
        	&\qquad -4\,\dive\left(\D\eta\H\right).
        \end{align*}
        Finally, we have
        \begin{align*}
        	&\mathscr{L}_0(\eta\H)=\mathscr{L}(\H)-\dive\left(\eta\left(\vec{A}_1\cdot\H+\star(\vec{V}_1\wedge\H)\right) \right)\\
        	&=\eta\vec{f}+2\,\Delta\eta\,\H+3\D\eta\cdot\D\pi_{\n}(\H)+\s{\D\eta}{\star(\D^{\perp}\n_1\wedge\n_2\wedge\H)}+\s{\D\eta}{\vec{A}_1\cdot\H}+\s{\D\eta}{\star(\vec{V}_1\wedge\H)}\\
        	&\quad -4\,\dive\left(\D\eta\H\right)-\dive\left(\eta\left(\vec{A}_1\cdot\H+\star(\vec{V}_1\wedge\H)\right) \right)\,.
        \end{align*}
        Note that
        \begin{align*}
        	\vec{f}_1=\eta\vec{f}+2\,\Delta\eta\,\H+3\D\eta\cdot\D\pi_{\n}(\H)+\s{\D\eta}{\star(\D^{\perp}\n_1\wedge\n_2\wedge\H)}+\s{\D\eta}{\vec{A}_1\cdot\H}+\s{\D\eta}{\star(\vec{V}_1\wedge\H)}
	\end{align*}
        satisfies $\vec{f}_1\in \mathrm{L}^1(B(0,1))$ with the estimate
        \begin{align}\label{est_f1}
        	\np{\vec{f}_1}{1}{B(0,1)}\leq C_3^2\int_{B(0,1)}|\D\n||\H|dx\,.
        \end{align} 
        Analogously, 
        \begin{align*}
        	\vec{f}_2=-4\,\dive\left(\D\eta\H\right)-\dive\left(\eta\left(\vec{A}_1\cdot\H+\star(\vec{V}_1\wedge\H)\right) \right)\in \mathrm{H}^{-1}(B(0,1))
        \end{align*}
         can be estimated by
        \begin{align}\label{est_f2}
        	\hs{\vec{f}_2}{-1}{B(0,1)}\leq C_3\left(\int_{B(0,1)}|\H|^2\right)^{\frac{1}{2}}.
        \end{align}
        Applying Lemma \ref{new_A2} and \ref{new_A4} (respectively with $\vec{f}=\vec{f}_1$ and $\vec{f}=\vec{f}_2$), we obtain that there exist $\vec{u}_1,\vec{u}_2\in \mathrm{W}^{1,(2,\infty
        	)}_0(B(0,1))$ such that $\mathscr{L}\vec{u}_1=\vec{f}_1$ and $\mathscr{L}\vec{u}_2=\vec{f}_2$ on $B(0,1)$ and satisfying the estimates  (thanks to \eqref{est_f1} and \eqref{est_f2})
        \begin{align*}
        	\np{\D\vec{u}_1}{2,\infty}{B(0,1)}+\np{\D\vec{u}_2}{2,\infty}{B(0,1)}&\leq C_0C_3\left(\int_{B(0,1)}|\D\n||\H|dx+\left(\int_{B(0,1)}|\H|^2\right)^{\frac{1}{2}}\right)\\
        	&\leq C_0C_3(1+\epsilon_0)\left(\int_{B(0,1)}|\H|^2\right)^{\frac{1}{2}}.
        \end{align*}
        Invoking the uniqueness result of Lemma \ref{A81}, we deduce that $\eta\H=\vec{u}_1+\vec{u}_2\in \mathrm{W}^{1,(2,\infty)}_0(B(0,1))$
        and that
        \begin{align*}
        	\np{\D(\eta\H)}{2,\infty}{B(0,1)}\leq C\left(\int_{B(0,1)}|\H|^2dx\right)^{\frac{1}{2}}.
        \end{align*} 
        We conclude that $\H\in \mathrm{W}^{1,(2,\infty)}(B(0,\frac{1}{2}))$.  Now, one can complete the proof by following verbatim the arguments in the Euclidean $\varepsilon$-regularity theorem \cite{riviere1}.
     \end{proof}

     \section{  $\mathrm{L}^{2,\infty}$ quantization of  energy} 

  The goal of this section is to prove the following result, establishing the $\mathrm{L}^{2,\infty}$ quantization of energy. 
    
          \begin{theorem}\label{L2infty_necks}
        Let $(M^m,h)$ be a closed $m$-dimensional Riemannian manifold. 
     	There exists $\epsilon_0>0$ with the following property. Let $\ens{r_k}_{k\in\N},\ens{R_k}_{k\in\N}\subset (0,\infty)$ be such that $\limsup_{k\rightarrow \infty}R_k\in (0,\infty)$, $r_k\conv{k\rightarrow \infty}0$ and set $\Omega_k(\alpha)=B_{\alpha R_k}\setminus\bar{B}_{\alpha^{-1}r_k}(0)$ for all $0<\alpha\leq 1$.  Let $\ens{\phi_k}_{k\in\N}:B(0,R_k)\rightarrow (M^m,h)$ be a sequence of Willmore immersions such that
     	\begin{align*}
     		\Lambda=\sup_{k\in\N}\left(\np{\D\lambda_k}{2,\infty}{B(0,R_k)}+\int_{B(0,R_k)}|\D\n_k|^2dx+\mathrm{Area}(\phi_k(B(0,R_k)))\right)<\infty\, ,
     	\end{align*}
        and 
        \begin{align*}
        	\sup_{r_k<s<\frac{R_k}{2}}\int_{B_{2s}\setminus\bar{B}_s(0)}|\D\n_k|^2dx\leq \epsilon_0.
        \end{align*}
        Then 
        \begin{align*}
        	\limsup_{\alpha\rightarrow 0}\limsup_{k\rightarrow \infty}\np{|x|\D\n_k(x)}{\infty}{\Omega_k(\alpha)}=0.
        \end{align*}
        In particular, it holds
        \begin{align*}
        	\limsup_{\alpha\rightarrow 0}\limsup_{k\rightarrow \infty}\np{\D\n_k}{2,\infty}{\Omega_k(\alpha)}=0. 
        \end{align*}
     \end{theorem}
     \begin{proof}
     	By the $\epsilon$-regularity Theorem \ref{eps_reg}, we deduce that 
     	\begin{align*}
     		|\D\n_k(z)|^2\leq \frac{C_1}{|z|^2}\int_{B_{2|z|}\setminus\bar{B}_{\frac{|z|}{2}}}|\D\n_k|^2dx\leq \frac{2C_1\epsilon_0}{|z|^2}.
     	\end{align*}
        In particular, we have 
        \begin{align*}
        	\np{\D\n_k}{2,\infty}{\Omega_k(\frac{1}{2})}\leq 4\sqrt{\pi C_1 \epsilon_0}.
        \end{align*}
     Now, assume by contradiction that there exists $\epsilon_1>0$, a sequence $\{\phi_k\}_{k\in\N}\subset C^{\infty}(B(0,R_k),M^m)$ of smooth Willmore immersions and a sequence $\ens{z_k}_{k\in\N}\in \Omega_k(\frac{1}{2})$ such that
        \begin{align*}
        	\log\left|\frac{|z_k|}{r_k}\right|\conv{k\rightarrow \infty}\infty\qquad\text{and}\qquad\log\left|\frac{R_k}{|z_k|}\right|\conv{k\rightarrow \infty}\infty\, , 
        \end{align*}
        and
        \begin{align}\label{contradiction_final}
        	|z_k||\D\n_k(z_k)|\geq \epsilon_1>0\, . 
        \end{align}
        Notice that,  in particular,  $z_k\conv{k\rightarrow \infty}0\in \C$.
        Therefore, 
        applying again the $\epsilon$-regularity Theorem \ref{eps_reg}, we deduce 
        \begin{align}\label{contradiction_corollary}
        	\int_{B_{2|z_k|}\setminus\bar{B}_{\frac{|z_k|}{2}}}|\D\n_k|^2dx\geq C_0^{-1}|z_k|^2|\D\n_k(z_k)|^2\geq C_0^{-1}\epsilon_1^2\, .
        \end{align}
        Using the previous result (Theorem \ref{thm:L2infLW} and the end of Section \ref{thm:L2infLW})
         and that $\displaystyle\limsup_{k\rightarrow \infty}R_k<\infty$, we get that there exist $\alpha_0>0$,  $\vec{L}_k\in \mathrm{L}^{2,\infty}_{\lambda_k}(B(0,\alpha_0R_k),\C^n)$, $S_k\in \mathrm{W}^{1,(2,\infty)}(B(0,\alpha_0R_k),\C)$ and $\vec{R}_k\in \mathrm{W}^{1,(2,\infty)}(B(0,\alpha_0R_k),\Lambda^2\C^n)$ such that
        \begin{align}\label{first_order}
        	\left\{\begin{alignedat}{2}
        		\D_z\vec{L}_k&=i\left(\D_z\H_k-3\,\pi_{\n_k}(\D_z\vec{H}_k)-i\star_h\left(\D_z\n_k\wedge\H_k\right)\right)\qquad&&\text{in}\;\, \Omega_k(\alpha_0)\\
        		\p{z}S_k&=\s{\p{z}\phi_k}{\bar{\vec{L}_k}}\qquad&&\text{in}\;\,\Omega_k(\alpha_0)\\
        		\D_z\vec{R}_k&=\p{z}\phi_k\wedge\bar{\vec{L}_k}-2i\,\p{z}\phi_k\wedge\H_k\qquad&&\text{in}\;\,\Omega_k(\alpha_0)
        	\end{alignedat} \right.
        \end{align}
        and satisfying the bounds:
        \begin{align*}
     &   	\np{e^{\lambda_k}\vec{L}_k}{2,\infty}{\Omega_k(\alpha_0)}+\np{\D S_k}{2,1}{\Omega_k(\alpha_0)}+\np{\D\vec{R}_k}{2,1}{\Omega_k(\alpha_0)}\leq C \\
     &   	\wp{\Im(\vec{L}_k)}{1,(2,\infty)}{\Omega_k(\alpha_0)}+\wp{\Im(S_k)}{2,q}{B(0,\alpha_0R_k)}+\wp{\Im(\vec{R}_k)}{2,q}{B(0,\alpha_0R_k)}\leq C_q, \; \text{ for all $q<\frac{2}{2-\epsilon}$.}
        \end{align*}
        Furthermore, $\Im(S_k)$ and $\Im(\vec{R}_k)$ solve the equations 
        \begin{align}\label{imaginary_system}
        	\left\{\begin{alignedat}{2}
        		\Delta \Im(S_k)&=-2e^{2\lambda_k}\s{\H_k}{\Im(\vec{L}_k)}\qquad&&\text{in}\;\, \Omega_k(\alpha_0)\\
        		\Delta\Im(\vec{R}_k)&=4\,\Im\left(\D_{\z}\left(\p{z}\phi_k\wedge \vec{L}_k-2i\,\p{z}\phi_k\wedge \H_k\right)\right) -4\,\Im\left(\p{\z}\left(F_k(\vec{R}_k)\right)\right)\\
        		&\qquad -4\,\Im\left(\bar{F_k}\left(\p{z}\phi_k\wedge \Im(\vec{L}_k)-2i\,\p{z}\phi_k\wedge \H_k\right)\right)
        		\qquad&&\text{in}\;\,\Omega_k(\alpha_0). 
        	\end{alignedat} \right.
        \end{align}
        Define the function $\vec{\Psi}_k:\Omega_k(\alpha_0|z_k|^{-1})\rightarrow \R^n$ by 
        \begin{align*}
        	\vec{\Psi}_k(w)=e^{-\lambda_k(z_k)-\log|z_k|}\left(\phi_k(|z_k|w)-\phi_k(z_k)\right). 
        \end{align*}
        A direct computation shows that
        \begin{align}\label{der1}
        	\p{w}\vec{\Psi}_k(w)=|z_k|e^{-\lambda_k(z_k)-\log|z_k|}\p{z}\phi_k(|z_k|w)=e^{-\lambda_k(z_k)}\p{z}\phi_k(|z_k|w).
        \end{align}
        Therefore, the conformal parameter $\mu_k$ of $\vec{\Psi}_k$ satisfies
        \begin{align*}
        	\mu_k(w)=\lambda_k(|z_k|w)-\lambda_k(z_k)\,.
        \end{align*}
     By the uniform Harnack inequality on the conformal parameters (see Theorem \ref{neckfine}),   there exists $C>0$ independent of $k\geq N$ such that
        \begin{align*}
        	\np{\D\left(\lambda_k-d_k\log|z|\right)}{2,1}{\Omega_k(\alpha_0)}+\np{\lambda_k-d_k\log|z|}{\infty}{\Omega_k(\alpha_0)}\leq C,
        \end{align*}
        where   $    	d_k\to d\in (-1,\infty)$, as $k\to \infty$. 
        In particular, it holds
        \begin{align*}
        	|\mu_k(w)|\leq \left|\lambda_k(|z_k|w)-d_k\log|z_kw|\right|+\left|d_k\log|z_k|-\lambda_k(z_k)\right|+|d_k\log|w||\leq (|d|+1)|\log|w||+2C,
        \end{align*}
        which is uniformly bounded on any compact 
        subset $K\subset\C\setminus\ens{0}$. Now, by \eqref{der1}, we deduce that
        \begin{align}\label{blow_up3}
        	\tilde{\n}_k(w)=\n_{\vec{\Psi}_k}(w)=\n_k(|z_k|w).
        \end{align}
        Then, we compute
        \begin{align*}
        	\p{w\w}^2\vec{\Psi}(w)=|z_k|e^{-\lambda_k(z_k)}\p{z\z}^2\phi_k(|z_k|w)=|z_k|e^{-\lambda_k(z_k)}\times 2e^{2\lambda_k(|z_k|w)}\H_k(|z_k|w)=2|z_k|e^{2\lambda_k(|z_k|w)-\lambda_k(z_k)}\H_k(|z_k|w),
        \end{align*}
        which implies that 
        \begin{align}\label{blow_up4}
        	\tilde{\H}_k(w)=\frac{1}{2}e^{-2\mu_k(w)}\p{w\w}^2\vec{\Psi}_k(w)=|z_k|e^{\lambda_k(z_k)}\H_k(|z_k|w). 
        \end{align} 
        We deduce that
        \begin{align*}
        	e^{\mu_k}\tilde{\H}_k(w)=|z_k|e^{\lambda_k(z_k)}\H_k(|z_k|w)\,, 
        \end{align*}  
        which shows that, after the linear change of variable $z=|z_k|w$, it holds
        \begin{align*}
        	\int_{\Omega_k(\alpha_0|z_k|^{-1})}e^{2\mu_k(w)}|\tilde{\H}_k(w)|^2|dw|^2&=\int_{\Omega_k(\alpha_0|z_k|^{-1})}|z_k|^2e^{2\lambda_k(|z_k|w)}|\H_k(|z_k|w)|^2|dw|^2\\
        	&=\int_{\Omega_k(\alpha_0)}e^{2\lambda_k(z)}|\H_k(z)|^2|dz|^2=\int_{\Omega_k(\alpha_0)}|\H_k|^2d\vg. 
        \end{align*}
        We deduce that
        \begin{align}\label{blow_up1}
        	\p{w}\tilde{\H}_k(w)=|z_k|^2e^{\lambda_k(z_k)}\p{z}\H_k(|z_k|w)
        \end{align}
        and that
        \begin{align}\label{blow_up2}
        	|z_k|^2e^{\lambda_k(z_k)}\Gamma_{l,q}^j(\phi_k(|z_k|w))\p{z}\phi_{k}(|z_k|w)^q\H_{k}(|z_k|w)^l=\Gamma_{l,q}^j(\phi_k(|z_k|w))|z_k|\p{z}\phi_{k}(|z_k|w)^q\tilde{\H}_k(w)^l.
        \end{align}
        Since $\ens{\mu_k}_{k\in\N}$ is uniformly bounded in $\mathrm{L}^{\infty}_{\mathrm{loc}}(\C\setminus\ens{0})$, we deduce by the $\epsilon$-regularity that $\ens{\tilde{\H}_k}_{k\in\N}$ is bounded in $\mathrm{L}^{\infty}_{\mathrm{loc}}(\C\setminus\ens{0})$. Furthermore, by the Harnack inequality \eqref{new_harnarck}, we deduce that
        \begin{align*}
        	e^{-A}|z_k|^{d_k+1}|w|^{d_k}\leq |z_k||\p{z}\phi_{k,q}(|z_k|w)|\leq e^{A}|z_k|^{d_k+1}|w|^{d_k}\,.
        \end{align*}
        Since $d_k\conv{k\rightarrow \infty}d>-1$, and $z_k\conv{k\rightarrow\infty}0$, while $\Gamma_{l,q}^j$ are bounded, we conclude that for all compact subset $K\subset\C\setminus\ens{0}$,  it holds
        \begin{align}\label{blow_up5}
        	\lim\limits_{k\rightarrow \infty}\np{|z_k|^2e^{\lambda_k(z_k)}\Gamma_{l,q}^j(\phi_k(|z_k|w))\p{z}\phi_{k}(|z_k|w)^q\H_{k}(|z_k|w)^l}{\infty}{K}=0.
        \end{align}
        Now, the previous scaling considerations prompt us to introduce $\tilde{\vec{L}}_k:\Omega_k(\alpha_0|z_k|^{-1})\rightarrow \C^n$ defined by 
        \begin{align*}
        	\tilde{\vec{L}}_k(w)=|z_k|e^{\lambda_k(z_k)}\vec{L}_k(|z_k|w).
        \end{align*}
        It is immediate to check that
        \begin{align*}
        	e^{\mu_k}\tilde{\vec{L}}_k(w)=|z_k|e^{\lambda_k(|z_k|w)}\vec{L}_k(|z_k|w)=|z_k|e^{\lambda_k(|z_k|w)}\vec{V}_k(|z_k|w)+|z_k|e^{\lambda_k(|z_k|w)}\vec{W}_k(|z_k|w),
        \end{align*}
        where $\vec{W}_k\in \mathrm{W}^{1,(2,\infty
        	)}(B(0,\alpha_0R_k
        ))$ and $\vec{V}_k$ satisfy for all  $z\in \Omega_k(\alpha_0)$ the estimate
        \begin{align*}
        	e^{\lambda_k(z)}|\vec{V}_k(z)|\leq \frac{C}{|z|}
	        \end{align*}
       for some constant $C>0$ independent of $k$. We deduce that
        \begin{align*}
        	|z_k|e^{\lambda_k(|z_k|w)}\left|\vec{V}_k(|z_k|w)\right|\leq |z_k|\times \frac{C}{||z_k|w|}= \frac{C}{|w|}\,.
        \end{align*}
        Also, by defining $\tilde{\vec{W}}_k:B(0,\alpha_0|z_k|^{-1}R_k)\rightarrow \C^n$ as
        \begin{align*}
        	\tilde{\vec{W}}_k(w)=\vec{W}_k(|z_k|w)\, ,
        \end{align*}
        we have
        \begin{align*}
        	\np{\D\tilde{\vec{W}}_k}{2,\infty}{B(0,\alpha_0|z_k|^{-1}R_k)}=\np{\D\vec{W}_k}{2,\infty}{B(0,\alpha_0R_k)}\leq C\, ,
        \end{align*}
        giving that $\vec{W}_k$ is bounded in $\mathrm{W}^{1,(2,\infty)}(B(0,\alpha_0|z_k|^{-1}R_k))$. Furthermore, by the Harnack inequality \eqref{new_harnarck}, we deduce that 
        \begin{align*}
        	|z_k|e^{\lambda_k(|z_k|w)}\leq e^{A}|z_k|^{d_k+1}|w|^{d_k}\conv{k\rightarrow \infty}0\qquad\text{in}\;\, \mathrm{L}^{\infty}_{\mathrm{loc}}(\C\setminus\ens{0}). 
        \end{align*}
        Therefore, we get that 
        \begin{align*}
        	&\np{|z_k|e^{\lambda_k(|z_k|\,\cdot\,)}\tilde{W}_k}{2,\infty}{\Omega_k(\alpha_0|z_k|^{-1})}\conv{k\rightarrow \infty}0\\
        	&|z_k|e^{\lambda_k(|z_k|\,\cdot\,)}\tilde{W}_k\conv{k\rightarrow\infty}0\qquad\text{in}\;\, \mathrm{L}^{p}_{\mathrm{loc}}(\C\setminus\ens{0})\;\,\text{for all}\;\,p<\infty. 
        \end{align*}
        Then, we have 
        \begin{align*}
        	\D\Im(\tilde{\vec{L}}_k)(w)=|z_k|^2e^{\lambda_k(z_k)}\D\vec{L}_k(|z_k|w)
        \end{align*}
        which implies that 
        \begin{align*}
        	\np{\D\Im(\tilde{\vec{L}}_k)}{2,\infty}{\Omega_k(\alpha_0)}&=|z_k|e^{\lambda_k(z_k)}\np{\D\Im(\vec{L}_k)}{2,\infty}{\Omega_k(\alpha_0)}\\
        	&\leq e^{A}|z_k|^{d_k+1}\np{\D\Im(\vec{L}_k)}{2,\infty}{\Omega_k(\alpha_0)}\conv{k\rightarrow\infty}0,
        \end{align*}
        where we used $|z_k|\conv{k\rightarrow \infty}0$ and $d_k+1\conv{k\rightarrow \infty}d+1>0$. Finally, we can manipulate the equation \eqref{first_order}  by using \eqref{blow_up1} and \eqref{blow_up2} to obtain:
        \begin{align}\label{blow_up6}
        	\p{w}\tilde{\vec{L}}_k-i\left(\p{w}\tilde{\H}_k-3\,\pi_{\tilde{\n}_k}(\p{w}\H_k)-i\star_h\left(\p{z}\tilde{\n}_k\wedge\tilde{\H}_k\right)\right)=\vec{Z}_k
        \end{align}
        where 
        \begin{align*}
        	\left(\vec{Z}_k\right)_j&=\sum_{l,q=1}^{n}|z_k|^2e^{\lambda_k(z_k)}\Gamma_{j,q}^l(\phi_k(|z_k|w))\p{z}\phi_{k,q}(|z_k|w)\bigg(\vec{L}_{k,l}(|z_k|w)\\
        	&-i\,\H_{k,l}(|z_k|w)+3i\,\pi_{\n_k(|z_k|w)}(\H_{k}(|z_k|w))_l-\star_h\left(\n_{k}(|z_k|w)\wedge\H_k(|z_k|w)\right)_l\bigg)\\
        	&=\sum_{l,q=1}^{n}\Gamma_{j,q}^l(\phi_k(|z_k|w))|z_k|\p{z}\phi_{k,q}(|z_k|w)\left(\tilde{\L}_{k,l}(w)-i\,\tilde{H}_{k,l}(w)+3i\,\pi_{\tilde{n}_k(w)}(\tilde{\H}_k(w))_l\right.\\
        	&\left.-\star\left(\tilde{\n_k}(w)\wedge\tilde{\H}_k(w)\right)_l\right). 
        \end{align*}
        By the $\epsilon$-regularity  Theorem \ref{eps_reg}, as in \eqref{blow_up5}, we get that for any compact subset $K\subset \C\setminus\ens{0}$, it holds
        \begin{align*}
        	\np{\Gamma_{l,q}^j(\phi_k(|z_k|w))|z_k|\p{z}\phi_{k}(|z_k|w)^q\bigg(-i\,\tilde{H}_{k}(w)^l+3i\,\pi_{\tilde{n}_k(w)}(\tilde{\H}_k(w))^l-\star\left(\tilde{\n_k}(w)\wedge\tilde{\H}_k(w)\right)^l\bigg)}{\infty}{K}=0.
        \end{align*}
        Furthermore, 
        \begin{align*}
        	|\vec{L}_{k}(w)|\leq \frac{C}{|w|^{d_k+1}}+f_k(w)\, ,
        \end{align*}
        where $f_k$ is bounded in $\mathrm{W}^{1,(2,\infty)}$ and in particular in $\mathrm{L}^{p}$ for all $p<\infty$. Therefore, we have
        \begin{align*}
        	|z_k|\, |\p{z}\phi_{k}(|z_k|w)|\, |\vec{L}_{k}(w)|\leq \frac{|z_k|^{d_k+1}}{|w|}+|z_k|^{d_k+1}|w|^{d_k}f_k(w)
        \end{align*}
        which implies that for any compact subset $K\subset \C\setminus\ens{0}$ and any $p<\infty$, it holds 
        \begin{align*}
        	\np{\Gamma_{l,q}^j(\phi_k(|z_k|w))|z_k|\p{z}\phi_{k}(|z_k|w)^q\tilde{\L}_{k}(w)^l}{p}{K}\conv{k\rightarrow\infty}0.
        \end{align*}
        Now define $\tilde{S}_k:\Omega_k(\alpha_0|z_k|^{-1})\rightarrow \C$ and $\tilde{\vec{R}}_k:\Omega_k(\alpha_0|z_k|^{-1})\rightarrow\Lambda^2\C^n$ by 
        \begin{align*}
        	\tilde{S}_k(w)=S_k(|z_k|w),\qquad\tilde{\vec{R}}_k(w)=\vec{R}_k(|z_k|w).
        \end{align*} 
        By scaling invariance, we have
        \begin{align*}
        	&\int_{\Omega_k(\alpha_0|z_k|^{-1})}|\D\n_{\vec{\Psi}_k}|^2dx=\int_{\Omega_k(\alpha_0)}|\D\n_k|^2dx\leq C\\
        	&\np{\D\tilde{S}_k}{2,1}{\Omega_k(\alpha_0|z_k|^{-1})}=\np{\D S_k}{2,1}{\Omega_k(\alpha_0)}\leq C\\
        	&\np{\D\tilde{\vec{R}}_k}{2,1}{\Omega_k(\alpha_0|z_k|^{-1})}=\np{\D \vec{R}_k}{2,1}{\Omega_k(\alpha_0)}\leq C.
        \end{align*}
        Therefore, by the $\epsilon$-regularity of Theorem \ref{eps_reg} to deduce that for any compact subset $K\subset \C\setminus\ens{0}$, and for any $l\in \N$, there exists $C_l(K)<\infty$ such that
        \begin{align*}
        	\np{\D^l\n_{\vec{\Psi}_k}}{\infty}{K}\leq C_l(K).
        \end{align*}
        Therefore, up to a subsequence, $\vec{\Psi}_k\conv{k\rightarrow \infty}\vec{\Psi}_{\infty}$ in $C^l_{\mathrm{loc}}(\C\setminus\ens{0})$ and by lower semi-continuity and conformal invariance of the Dirichlet energy (or $\mathrm{L}^{2,1}$ norm of the gradient), we deduce that 
        \begin{align*}
        	\np{e^{\mu_{\infty}}\tilde{\vec{L}}_{\infty}}{2,\infty}{\C}+\np{\D \tilde{S}_{\infty}}{2,1}{\C}+\np{\D\tilde{\vec{R}}_{\infty}}{2,1}{\C}\leq C.
        \end{align*}
        Furthermore, recalling that $\Im(S_k),\Im(\vec{R}_k)\in \mathrm{W}^{2,q}(B(0,\alpha_0R_k))$ for all $q<\frac{2}{2-\epsilon}$, and using that $\Im(\tilde{S}_k)(w)=\Im(\tilde{S}_k)(|z|_kw)$, we deduce that for all $q<\frac{2}{1-\epsilon}$ 
        \begin{align*}
        	&\int_{B(0,\alpha_0|z_k|^{-1}R_k)}\left(|\D^2\Im(\tilde{S}_k)|^{q}+|\D^2\Im(\tilde{\vec{R}}_k)|^q\right)dx\\
        	&=|z_k|^{2(q-1)}\int_{B(0,\alpha_0R_k)}\left(|\D^2\Im({S}_k)|^{q}+|\D^2\Im({R}_k)|^q\right)dx
        	\leq 2(C_q)^q|z_k|^{2(q-1)}\conv{k\rightarrow \infty}0, 
        \end{align*}
       It follows that $\D^2\Im(\tilde{S}_{\infty})=0$ and $\D^2\Im(\tilde{\vec{R}}_{\infty})=0$ and, since $\D\Im(\tilde{S}_{\infty})\in \mathrm{L}^{2,1}(\C)$ and $\D\Im(\tilde{\vec{R}}_{\infty})\in \mathrm{L}^{2,1}(\C)$, this implies that $\Im(\tilde{S}_{\infty})$ and $\Im(\tilde{\vec{R}}_{\infty})$ are constant. 
        Recalling that \eqref{first_order} can be rewritten (see \cite[Lemma $6.2$]{mondinoriviere}) as
        \begin{align*}
        	\left\{\begin{alignedat}{1}
        	\D_z\vec{R}_k&=(-1)^{m+1}\star_h\left(\n_k\antires i\,\D_z\vec{R}\right)+i\,\p{z}S_k\,\star_h\n_k\\
        	\p{z}\vec{S}_k&=-i\s{\D_z\vec{R}_k}{\star_h\n_k},
        	\end{alignedat}\right.
        \end{align*}
        an expansion similar to the one made in \eqref{blow_up6} shows that the system passes to the limit and yields 
        \begin{align*}
        	\left\{\begin{alignedat}{1}
        		\p{z}\tilde{\vec{R}}_{\infty}&=(-1)^{m+1}\star_h\left(\tilde{\n}_{\infty}\antires i\,\p{z}\tilde{\vec{R}}_{\infty}\right)+i\,\p{z}\tilde{S}_{\infty}\,\star_h\n_k\\
        		\p{z}\tilde{S}_{\infty}&=-i\s{\p{z}\tilde{\vec{R}}_{\infty}}{\star_h\tilde{\n}_{\infty}}\,.
        	\end{alignedat}\right.
        \end{align*}
        Since both $\tilde{\vec{R}}_{\infty}$ and $\tilde{S}_{\infty}$ are real, this system can be rewritten as 
        \begin{align*}
        	\left\{\begin{alignedat}{1} 
            	\D\tilde{\vec{R}}_{\infty}&=(-1)^m\star_h\left(\tilde{\n}_{\infty}\antires \D^{\perp}\tilde{\vec{R}}_{\infty}\right)-\D^{\perp}\tilde{S}_{\infty}\,\star_h\tilde{\n}_{\infty}\\
            	\D\tilde{S}_{\infty}&=\s{\D^{\perp}\tilde{\vec{R}}_{\infty}}{\star_h\tilde{\n}_{\infty}}. 
        	\end{alignedat}\right.
        \end{align*}
        We deduce that the following Jacobian system holds:
        \begin{align*}
        	\left\{\begin{alignedat}{1}
        		\Delta\tilde{\vec{R}}_{\infty}&=(-1)^m\star_h\left(\D\tilde{\n}_{\infty}\antires\D^{\perp}\tilde{\vec{R}}_{\infty}\right)-\D^{\perp}\tilde{S}_{\infty}\cdot \D \left(\star \tilde{\n}_{\infty}\right)\\
        		\Delta\tilde{S}_{\infty}&=\s{\D^{\perp}\tilde{\vec{R}}_{\infty}}{\D\left(\star_h\tilde{\n}_{\infty}\right)}.
        	\end{alignedat} \right.
        \end{align*}
        Using an improved Wente estimate as in \cite{quanta}, we deduce that 
        \begin{align*}
        	\np{\D\tilde{\vec{R}}_{\infty}}{2,1}{\C}+\np{\D\tilde{S}_{\infty}}{2,1}{\C}&\leq C\left(\np{\D\tilde{\vec{R}}_{\infty}}{2,1}{\C}+\np{\D\tilde{S}_{\infty}}{2,1}{\C}\right)\np{\D\tilde{n}_{\infty}}{2,\infty}{\C}\\
        	&\leq C\epsilon_0 \left(\np{\D\tilde{\vec{R}}_{\infty}}{2,1}{\C}+\np{\D\tilde{S}_{\infty}}{2,1}{\C}\right)\,.
        \end{align*}
        Therefore, taking $\epsilon_0=\frac{1}{2C}$, we get  that $\tilde{\vec{R}}_{\infty}$ and $\tilde{S}_{\infty}$ are constant. 
        Since $\Im(\vec{L}_{\infty})=0$, we have the identity 
        \begin{align*}
        	e^{2\mu_{\infty}}\tilde{\H}_{\infty}&=-\Im\left(\p{z}\tilde{\vec{R}}_{\infty}\res \p{\z}\vec{\Psi}_{\infty}\right)-\Re\left(\p{z}\tilde{S}_{\infty}(i\,\p{\z}\vec{\Psi}_{\infty})\right)
        	+\Re\left(\s{\p{z}\vec{\Psi}_{\infty}}{\Im(\tilde{L}_{\infty})}\p{\z}\vec{\Psi}_{\infty}\right)=0.
        \end{align*}
        We deduce that $\tilde{\H}_{\infty}=0$ which, by the exact same proof as in \cite{quanta},  yields that $\D\tilde{\n}_{\infty}=0$, contradicting the estimate 
        \begin{align*}
        	\int_{B_{2}\setminus\bar{B}_1(0)}|\D\tilde{n}_{\infty}|^2dx\geq C_0^{-1}\epsilon_1^2
        \end{align*}
        that passed to the limit thanks to the strong convergence on $\C\setminus\ens{0}$. 
        \end{proof}

       \section{Proof of the main Theorem \ref{thm:MainThm}}

       We are finally in position to prove the main result of the paper (Theorem \ref{thm:MainThm}), namely the quantization of the Willmore energy for Willmore spheres in Riemannian manifolds. The proof will combine all the main technical results proved in the paper: the $\mathrm{L}^{2,1}$ uniform bounds on the Willmore integrand in neck regions (Theorem \ref{L21_necks}), the $\epsilon$-regularity Theorem  \ref{eps_reg}, and the $\mathrm{L}^{2,\infty}$ quantization of energy (Theorem \ref{L2infty_necks}).       
        \begin{proof}[Proof of Theorem \ref{thm:MainThm}]
   Let us first consider the case when $\displaystyle\inf_{k\in\N}\textrm{Area}(\phi_{k}(S^{2}))>0$.

Thanks to the pre-compactness Theorem \ref{thm:compS2}, we know that there exist a sequence of Lipschitz diffeomorphisms $\{f_{k}\}$ of $S^{2}$ and a weak immersion $\vec{\xi}_{\infty}$ of $(S^{2} \setminus \{ a_1,\cdots,a_N\})$, possibly branched at the finitely many points $a_1,\cdots,a_N$,  into $(M^{m},h)$ such that 
            \begin{equation}\label{eq:weakconvergenceXik}
            \vec{\xi}_{k}= \phi_{k}\circ f_{k}  \rightharpoonup  \vec{\xi}_{\infty}\quad  \text{weakly in $\mathrm{W}^{2,2}_{\mathrm{loc}} (S^{2} \setminus \{ a_1,\cdots,a_N\})$ ,}
            \end{equation}
           Thanks to  the $\epsilon$-regularity Theorem \ref{eps_reg}, we  can improve the weak $\mathrm{W}^{2,2}_{\mathrm{loc}}$ convergence  to local smooth convergence:
            \begin{align}\label{eq:phiktophiinftysmooth}
             \vec{\xi}_{k}\conv{k\rightarrow \infty} \vec{\xi}_{\infty}\qquad\text{in}\;\, C^{l}_{\mathrm{loc}}(S^2\setminus\ens{a_1,\cdots,a_N})\,\;\text{for all}\;\,l\in \N.
            \end{align}
            Following verbatim the arguments at \cite[pp.129--130]{quanta}, one can extend the map $ \vec{\xi}_{\infty}$ to the whole $S^{2}$, so that the extension $\ \vec{\xi}_{\infty}: S^{2}\rightarrow (M^{m}, h)$ realises  a Willmore immersion of $S^{2}$ into  $(M^{m}, h)$, possibly branched at the points $a_1,\cdots,a_N$. 
        \\      Now, in a neck-region conformally equivalent to $\Omega_k(\alpha)=B_{\alpha R_k}\setminus\bar{B}_{\alpha^{-1}r_k}(0)$, we can apply Theorem \ref{L21_necks} and Theorem \ref{L2infty_necks} to deduce that there exists $\alpha_0>0$ and $C>0$ independent of $k$ such that 
            \begin{align*}
            	\np{e^{\lambda_k}\H_k}{2,1}{\Omega_k(\alpha_0)}\leq C\quad\text{and}\quad \lim\limits_{\alpha\rightarrow 0}\limsup_{k\rightarrow \infty}\np{\D\n_k}{2,\infty}{\Omega_k(\alpha)}=0\, .
            \end{align*}
           By the $\mathrm{L}^{2,1}/\mathrm{L}^{2,\infty}$ duality, we deduce that for all $0<\alpha<\alpha_0$
            \begin{align*}
            	\int_{\Omega_k(\alpha)}e^{2\lambda_k}|\H_k|^2dx\leq \np{e^{\lambda_k}\H_k}{2,1}{\Omega_k(\alpha)}\np{e^{\lambda_k}\H_k}{2,\infty}{\Omega_k(\alpha)}\leq C\np{\D\n_k}{2,\infty}{\Omega_k(\alpha)},
            \end{align*}
            which implies that 
            \begin{align}\label{final_quanta1}
            	\lim\limits_{\alpha\rightarrow 0}\limsup_{k\rightarrow \infty}\int_{\Omega_k(\alpha)}|\H_k|^2d\mathrm{vol}_{g_k}=0\, .
            \end{align}
            By  \cite[Lemma V.$1$]{quanta} (which does not use the Willmore equation and is valid for any weak immersion), we also have
            \begin{align}\label{final_quanta2}
            	\lim\limits_{\alpha\rightarrow 0}\limsup_{k\rightarrow \infty}\left|\int_{\Omega_k(\alpha)}K_{g_k}d\mathrm{vol}_{g_k}\right|=0. 
            \end{align}
            Moreover, by the proof of Theorem \ref{L21_necks} 
            (see \eqref{eq:NoNeckArea}), it holds
            \begin{align}\label{final_quanta3}
            	\lim\limits_{\alpha\rightarrow 0}\limsup_{k\rightarrow \infty}\mathrm{Area}(\vec{\xi}_k(\Omega_k(\alpha)))=0\,.
            \end{align}
            Using the point-wise identity
            \begin{align*}
            	|\D\n_k|^2=4|\H_k|^2-2K_{g_k}+2K(\vec{\xi}_{k,\ast}(TS^{2}))\, ,
            \end{align*}
           together with \eqref{final_quanta1}, \eqref{final_quanta2} and  \eqref{final_quanta3}, we deduce that 
            \begin{align}\label{eq:no-neck-energy}
            	\lim\limits_{\alpha\rightarrow 0}\limsup_{k\rightarrow \infty}\int_{\Omega_k(\alpha)}|\D\n_k|^2dx=0\, . 
            \end{align} 
            This is the \emph{no-neck energy} which will give below the desired quantization result. 

             Using that $(M^m,h)$ is isometrically embedded into $\R^n$ and the conformal invariance of the Willmore energy to obtain a suitable convergence result for the energy of bubbles. Indeed, seeing $\phi_k:\Sigma\rightarrow (M^m,h)\subset \R^n$ as an immersion of $\R^n$, since by assumption $h=\iota^{\ast}g_{\R^n}$, where $\iota:M^m\hooklongrightarrow \R^n$ is the inclusion map, we deduce that for any open subset $\Omega\subset \Sigma$ it holds
            \begin{align*}
            	\int_{\Omega}\left(|\H_k|^2+K(\phi_{k,\ast}(T\Sigma))\right)d\mathrm{vol}_{g_k}=W_{(M^m,h)}(\phi_k|\Omega)
            	=W_{\R^n}(\iota \circ \phi_k|\Omega)=\int_{\Omega}|\H_{\iota\circ \phi_k}|^2d\mathrm{vol}_{g_k},
            \end{align*}
           Let 
           \begin{equation*}
    	B(i,j,\alpha,k)=B_{\alpha^{-1}\rho_k^{i,j}}(x_k^{i,j})\setminus\bigcup_{j'\in I^{i,j}}B_{\alpha \rho_k^{i,j}}(x_k^{i,j'})
    \end{equation*}
    be a bubble domain (for more details, see \cite{quanta} or the discussion after \eqref{eq:defBijak}). Recall that $\rho^{i,j}_{k}\to 0$ as $k\to \infty$ and that the indices $i,j$ vary within a finite set. The Harnack inequality \eqref{eq:HarnackBubble} (proved in  \cite[Display (VIII.$10$)]{quanta}) gives that  for all $0<\alpha<1$, there exists $A_{\alpha}=A(i,j,\alpha)\geq 0$ such that
            \begin{align}\label{harnack_final}
            	\sup_{z\in B(i,j,\alpha,k)}e^{\lambda_k(z)}\leq e^{A_{\alpha}}\inf_{z\in B(i,j,\alpha,k)}e^{\lambda_k(z)}.
            \end{align}
            Choose an arbitrary point $z_{k}^{i,j}\in B(i,j,\alpha,k)$ and set $\lambda(i,j,\alpha,k)= \lambda_k(z_{k}^{i,j})$.
          The uniform area bound implies that    
            \begin{equation}\label{eq:lambdarho<infty}
        \limsup_{k\to \infty}   e^{2\lambda(i,j,\alpha,k)} (\rho_k^{i,j})^{2} <\infty \, ,  \quad \text{for all $\alpha\in (0,1)$}.
           \end{equation}
           Thus we have two cases.
           
            \textbf{Case 1}.    $ \limsup_{k\to \infty}   e^{2\lambda(i,j,\alpha,k)} (\rho_k^{i,j})^{2} >0 \, ,$  for some $\alpha\in (0,1)$.     
            \\This case corresponds to a macroscopic bubbles forming in the region  $B(i,j,\alpha,k)$. By performing a good gauge extraction procedure along the lines of \cite[Lemma 4.1]{mondinoriviere}, we can find positive Möbius transformations $\tilde{f}_{k}$ of $S^{2}$ such that the reparametrised immersions (up to a subsequence)
            \begin{equation*}
            \tilde{\vec{\xi}}_{k}= \xi_{k} \circ \tilde{f}_{k}: S^{2}\rightarrow (M^{m}, h)
            \end{equation*}
            converge weakly in $\mathrm{W}^{2,2}$ (and then smoothly, by the $\varepsilon$-regularity Theorem) outside finitely many points $\{a^{i,j}_{1}, \ldots, a^{i,j}_{N_{i,j}}\}$ to a Willmore immersion $\vec{\Psi}_{i,j}: S^{2}\setminus \big\{a_1^{i,j},\cdots,a_{N_{i,j}}^{i,j} \big\}\rightarrow (M^{m}, h)$:
             \begin{align}\label{eq:phiktophiinftysmooth}
              \tilde{\vec{\xi}}_{k}\conv{k\rightarrow \infty} \vec{\Psi}_{i,j}\qquad\text{in}\;\, C^{l}_{\mathrm{loc}}(S^{2}\setminus\big\{a^{i,j}_1,\cdots,a^{i,j}_{N_{i,j}}\big\})\,\;\text{for all}\;\,l\in \N\,.
            \end{align}
            Following verbatim the arguments at \cite[pp. 129--130]{quanta}, one can extend the map to the whole $S^{2}$, so that the extension $\vec{\Psi}_{i,j}: S^{2}\rightarrow (M^{m}, h)$ realises  a Willmore immersion of $S^{2}$ into  $(M^{m}, h)$, possibly branched at the finitely many points $a^{i,j}_1,\cdots,a^{i,j}_{N_{i,j}}$.
            Moreover, the no neck energy identity \eqref{eq:no-neck-energy} ensures that
             \begin{align}\label{eq:WBcase1a}
            \lim_{\alpha\rightarrow 0}\lim\limits_{k\rightarrow \infty}  W_{(M,h)}\left(\vec{\xi}_{k}|B(i,j,k,\alpha)\right)=  W_{(M,h)}(\vec{\Psi}_{i,j}). 
            \end{align} 
            Such a branched Willmore immersion corresponds to a Riemannian bubble $\vec{\Psi}_{j}$, $j=1,\ldots, u\in \N$ in the statement of Theorem \ref{thm:MainThm}.

           \textbf{Case 2}.     $ \lim_{k\to \infty}   e^{2\lambda(i,j,\alpha,k)} (\rho_k^{i,j})^{2} =0 \, ,$  for all $\alpha\in (0, 1)$.
           \\  In this case, there exists a point $\bar{x}_{i,j}\in M$ such that (again, up to a subsequence in $k$)
           \begin{equation*}
           \vec{\xi}_{k}(B(i,j,\alpha,k)) \to \bar{x}_{i,j}\, \quad \text{in Hausdorff distance sense, as $k\to \infty$, for all $\alpha\in (0,1)$.}  
           \end{equation*}
   Let ${\rm Exp}_{\bar{x}_{i,j}}:B_{\epsilon}^{\R^{m}}(0)\to M$ denote the exponential map of $(M,h)$ based at the point $\bar{x}_{i,j}$.
           Consider the rescaled immersions (with values in $T_{p}M \simeq \R^{m}$)
            \begin{align*}
            	\vec{\xi}_{k}^{i,j}(w)=e^{-\lambda(i,j,\alpha,k)-\log\,\rho_k^{i,j}} \, {\rm Exp}_{\bar{x}_{i,j}}^{-1} \left(\vec{\xi}_k(\rho_k^{i,j}w+z_{k}^{i,j})-\vec{\xi}_k(z_k^{i,j})\right), \; \forall w\in (\rho_k^{i,j})^{-1} \big(B(i,j,\alpha,k)-z^{k}_{i,j} \big)\subset \C\,.
            \end{align*}
            It is easily seen that  \eqref{harnack_final}  implies
            \begin{align*}
            	e^{-A_{\alpha}-1}\leq |\p{w}\vec{\xi}_{k}^{i,j}|\leq e^{A_{\alpha}+1}\, , \quad \text{for $k\in \N$ sufficiently large},
            \end{align*}
            and that the assumption of case 2 yields $e^{-\lambda(i,j,\alpha,k)-\log\,\rho_k^{i,j}}\to + \infty$.
\\            Notice that $\vec{\xi}_{k}^{i,j}$ are Willmore immersions in $(\R^{m}, g(k,i,j))$, where the Riemannian metrics  $g(k,i,j)$ converge to the Euclidean metric as $k\to \infty$, in $C^{l}_{\mathrm{loc}}(\R^{m})$ topology, for every $l\in \N$. Moreover, the scaling invariance of the Willmore functional implies that
            \begin{equation}\label{eq:WMhWRm}
            W_{(M,h)}\left(\vec{\xi}_{k}|B(i,j,k,\alpha)\right)= W_{(\R^{m},g(k,i,j))} \left(\vec{\xi}_{k}^{i,j} | (\rho_k^{i,j})^{-1} \big(B(i,j,\alpha,k)-z^{k}_{i,j} \big)\right). 
            \end{equation}

            Using the aforementioned $C^{l}_{\mathrm{loc}}(\R^{m})$ convergence of the ambient Riemannian metrics, one can immediately adapt the proof of the $\epsilon$-regularity Theorem \ref{eps_reg} to deduce that there exists a finite set of points $\big\{a_1^{i,j},\cdots,a_{N_{i,j}}^{i,j}\big\}\subset \C$ with $N_{i,j}\in \N$ such that
            \begin{align*}
            	\vec{\xi}_{k}^{i,j}\conv{k\rightarrow \infty}\vec{\xi}_{\infty}^{i,j}\quad\text{in}\;\, C^{l}_{\mathrm{loc}}\left(\C\setminus\big\{a_1^{i,j},\cdots,a_{N_{i,j}}^{i,j} \big\} \right),\; \text{for all}\;\, l\in \N,
            \end{align*}
            where $\vec{\xi}_{\infty}^{i,j}$ is a smooth Willmore immersion of $\C\setminus \big\{a_1^{i,j},\cdots,a_{N_{i,j}}^{i,j} \big\}$ in the Euclidean space $\R^{m}$.
            \\We can now follow verbatim the arguments in \cite[pp. 130--132]{quanta} and deduce that:
             \begin{itemize}
            \item[(1)] In case $\int_{\C} |\nabla \vec{\xi}_{\infty}^{i,j}|^{2} dx<\infty$, using the stereographic projection of $S^{2}$ to $\C$, the limit map $\vec{\xi}_{\infty}^{i,j}$ extends to a smooth Willmore immersion of $S^{2}$ into $\R^m$,  possibly branched at $\big\{a_1^{i,j},\cdots,a_{N_{i,j}}^{i,j} \big\}$.  Moreover (see also \eqref{eq:WMhWRm}): 
             \begin{align}\label{eq:WBcase1}
            \lim_{\alpha\rightarrow 0}\lim\limits_{k\rightarrow \infty}  W_{(M,h)}\left(\vec{\xi}_{k}|B(i,j,k,\alpha)\right)&=	\lim_{\alpha\rightarrow 0}\lim\limits_{k\rightarrow \infty}\int_{(\rho_k^{i,j})^{-1} \big(B(i,j,\alpha,k)-z^{k}_{i,j} \big)} |\H_{\vec{\xi}_{k}^{i,j}}|^2e^{2\lambda_{\vec{\xi}_{k}^{i,j}}}dx \nonumber \\
            &= W_{\R^m}(\vec{\xi}_{\infty}^{i,j}). 
            \end{align}
            Such a branched Willmore immersion corresponds to a Euclidean bubble $\vec{\eta}_{s}$, $s=1,\ldots, p\in \N$ in the statement of Theorem \ref{thm:MainThm}.
            \item[(2)] In case $\int_{\C} |\nabla \vec{\xi}_{\infty}^{i,j}|^{2} dx=\infty$, one finds suitable inversions ${\mathcal I}_{k,i,j}, {\mathcal I}_{\infty,i,j} $ in $\R^{m}$ such that
             \begin{equation*}
             {\mathcal I}_{k,i,j}\circ \vec{\xi}_{k}^{i,j} \to   {\mathcal I}_{\infty,i,j}\circ \vec{\xi}_{\infty}^{i,j}, \qquad\text{in}\;\, C^{l}_{\mathrm{loc}}\left(\C\setminus\big\{a_1^{i,j},\cdots,a_{N_{i,j}}^{i,j} \big\} \right),\;\,\text{for all}\;\, l\in \N.
             \end{equation*}
            Furthermore, one obtains that (pre-composing with the stereographic projection)  $ {\mathcal I}_{\infty,i,j}\circ \vec{\xi}_{\infty}^{i,j}$ extends to a smooth Willmore immersion of $S^{2}$ into $\R^{m}$, possibly branched at $\big\{a_1^{i,j},\cdots,a_{N_{i,j}}^{i,j} \big\}$. Moreover, 
            \begin{align}\label{eq:WBcase2}
            \lim_{\alpha\rightarrow 0}\lim\limits_{k\rightarrow \infty}  W_{(M,h)}\left(\vec{\xi}_{k}|B(i,j,k,\alpha)\right)&=	\lim_{\alpha\rightarrow 0}\lim\limits_{k\rightarrow \infty}\int_{(\rho_k^{i,j})^{-1} \big(B(i,j,\alpha,k)-z^{k}_{i,j} \big)}|\H_{\vec{\xi}_{k}^{i,j}}|^2e^{2\lambda_{\vec{\xi}_{k}^{i,j}}}dx \nonumber \\
            &= W_{\R^m}({\mathcal I}_{\infty,i,j}\circ \vec{\xi}_{\infty}^{i,j})- 4\pi \theta_{i,j}\,,  
            \end{align}
            where $\theta_{i,j}$ is the integer density of ${\mathcal I}_{\infty,i,j}\circ \vec{\xi}_{\infty}^{i,j}$ at the image point $0\in \R^{m}$.
            
            Such an inverted (compact)  branched Willmore immersion of $S^{2}$ corresponds to a  Euclidean bubble $\vec{\zeta}_{t}$, $t=1,\ldots, q\in \N$ in the statement of Theorem \ref{thm:MainThm}.
             \end{itemize}
            The combination of \eqref{eq:phiktophiinftysmooth}, \eqref{eq:WBcase1a}, \eqref{eq:WBcase1} and \eqref{eq:WBcase2} gives the desired energy identity (up to a subsequence in $k$):
        \begin{align*}
        	\lim\limits_{k\rightarrow \infty}W_{(M^{m},h)}(\phi_k)&=W_{(M^{m},h)}({\vec{\xi}_{\infty}})+ \sum_{j=1}^{u} W_{(M^{m},h)}(\vec{\Psi}_{j}) +\sum_{s=1}^{p}W_{\R^{m}}(\vec{\eta}_s)+\sum_{t=1}^{q}\left(W_{\R^{m}}(\vec{\zeta}_t)-4\pi\theta_{0,t}\right)\, .
        \end{align*}  
        \end{proof}

        Arguing along the lines of the proof of Theorem \ref{thm:MainThm}, one can prove the following quantization result for surfaces of arbitrary genus, under the assumption of weak convergence to a limit surface and a bound on the conformal structures.

    \begin{theorem}\label{thm:MainThm2}
    	Let $\Sigma$ be a closed Riemann surface, $(M^m,h)$ be a smooth compact Riemannian manifold of dimension $m\geq 3$, and let $\{\phi_k\}_{k\in\N}\subset \mathrm{Imm}(\Sigma,M^m)$ be a sequence of conformally parametrised Willmore immersions. 
    	Assume that the conformal classes of $\ens{\phi_k^{\ast}h}_{k\in\N}$ remain within a compact region of the moduli space of $\Sigma$ and that there exists a weak, possibly branched, immersion $\phi_{\infty}: \Sigma \to (M^{m},h)$ such that
	\begin{equation}
	 \phi_k\conv{k\rightarrow \infty} \phi_{\infty}\quad \text{weakly in } \mathrm{W}^{2,2}_{\mathrm{loc}}(\Sigma\setminus\ens{a_1,\cdots,a_N}) \text { and weakly$^{*}$ in } \mathrm{W}^{1,\infty}_{\mathrm{loc}}(\Sigma\setminus\ens{a_1,\cdots,a_N})
	\end{equation}
	where $\ens{a_1,\cdots,a_N}\subset \Sigma$ is a finite set.
    	Then the following identity holds:
        \begin{align*}
        	\lim\limits_{k\rightarrow \infty}W_{(M^{m},h)}(\phi_k)&=W_{(M^{m},h)}(\phi_{\infty})+ \sum_{j=1}^{u} W_{(M^{m},h)}(\vec{\Psi}_{j}) +\sum_{s=1}^{p}W_{\R^{m}}(\vec{\eta}_s)+\sum_{t=1}^{q}\left(W_{\R^{m}}(\vec{\zeta}_t)-4\pi\theta_{0,t}\right)\, ,
        \end{align*}
        where: 
      \begin{itemize}
      \item[\rm{(1)}]  The map  $\phi_{\infty}$ is a smooth, possibly branched, Willmore immersion of $\Sigma$ into $(M^{m},h)$ and   
        \begin{equation*}
        \phi_k\conv{k\rightarrow \infty} \phi_{\infty}\quad \text{ in }C^l_{\mathrm{loc}}(\Sigma\setminus\ens{a_1,\cdots,a_N}), \quad \forall  l\in \N.
        \end{equation*}
        Furthermore, it holds
         \begin{equation*}
       \lim_{k\to \infty}W_{(M^{m},h)}(\phi_{k})= W_{(M^{m},h)}(\phi_{\infty})\;  \Longleftrightarrow \; \phi_k\conv{k\rightarrow \infty} \phi_{\infty} \text{ in }C^l(\Sigma),  \quad \forall l\in \N.
        \end{equation*}
      \item[\rm{(2)}] The maps $\vec{\Psi}_{j}:  S^2\rightarrow (M^{m},h)$ are  smooth, possibly branched, Willmore immersions.
      \item[\rm{(3)}]      The maps $\vec{\eta}_{s}: S^2\rightarrow \R^m$ and $\vec{\zeta}_t: S^2\rightarrow \R^m$ are smooth, possibly branched, Willmore immersions and $\theta_{0,t}=\theta_0(\vec{\zeta}_t,x_t)\in \N$ is the multiplicity of $\vec{\zeta}_t$ at some some point $x_t\in \R^m$.
      \item[\rm{(4)}]  The Riemannian Willmore bubbles $\vec{\Psi}_j:S^{2}\to M^{m}$ are obtained as follows: there exist a sequence of unit area and constant curvature metrics $h_{k}$ on $\Sigma$ conformally equivalent to $\vec{\xi}_{k}^{*}h$ and strongly converging in $C^{l}(\Sigma)$ such that for any $j\in \{1,\ldots,u\}$, there exists a sequence of points $x^{u}_{k}\in \Sigma$ converging to one of $\ens{a_1,\cdots,a_N}$, a sequence of radii $\rho^{j}_{k}$  converging to zero such that (in converging $h_{k}$ conformal coordinates $\varphi_{k}$ around the given point in $\ens{a_1,\cdots,a_N}$):
     \begin{equation*}
       \vec{\xi}_k \circ \varphi_{k} (\rho^{j}_{k} y+ \varphi_{k}^{-1}(x^{j}_{k}))\conv{k\rightarrow \infty} \vec{\Psi}_{j}\circ \pi^{-1}(y)\quad \text{ in }C^l_{\mathrm{loc}}(\C \setminus \big\{a_1^{j},\cdots,a^{j}_{N_{j}}\big\}), \quad \forall  l\in \N,
        \end{equation*}    
             where $\pi$ denotes the stereographic projection from $S^{2}$ into $\C$, and  $\big\{a_1^{j},\cdots,a^{j}_{N_{j}}\big\}$ is a finite set of points in the complex plane.
   \item[\rm{(5)}] The Euclidean Willmore bubbles $\vec{\eta}_s, \vec{\zeta}_t: S^{2}\to \R^{m}$ are obtained by the following blow up procedure: there exist a sequence of unit area and constant curvature metrics $h_{k}$ on $\Sigma$ conformally equivalent to $\vec{\xi}_{k}^{*}h$ and strongly converging in $C^{l}(\Sigma)$ such that for any $s\in \{1,\ldots, p\}$ (resp. for any $t\in \{1,\ldots, q\}$), there exists a point $\bar{x}^{s}\in M$ (resp.  $\bar{x}^{t}\in M$), there exists a sequence of points $x^{s}_{k}\in \Sigma$ (resp. $x^{t}_{k}\in \Sigma$) converging to one of $\ens{a_1,\cdots,a_N}$, a sequence of radii $\rho^{s}_{k}$ (resp.  $\rho^{t}_{k}$) converging to zero, a sequence of rescalings $\lambda^{s}_{k}\to \infty$ (resp. $\lambda^{t}_{k}\to \infty$) and inversions $\Xi^{t}_{k}$ of $\R^{m}$ such that (in converging $h_{k}$ conformal coordinates $\varphi_{k}$ around the given point in $\ens{a_1,\cdots,a_N}$):
     \begin{equation*}
       \lambda^{s}_{k} \cdot {\rm Exp}_{\bar{x}^{s}}^{-1}\circ \vec{\xi}_k \circ \varphi_{k} (\rho^{s}_{k} y+ \varphi_{k}^{-1}(x^{s}_{k}))\conv{k\rightarrow \infty} \vec{\eta}_{s}\circ \pi^{-1}(y)\quad \text{ in }C^l_{\mathrm{loc}}(\C \setminus \big\{a_1^{s},\cdots,a^{s}_{N_{s}}\big\}), \quad \forall  l\in \N,
        \end{equation*}    
        and, respectively,
          \begin{equation*}
  \Xi^{t}_{k}\circ     \lambda^{t}_{k} \cdot {\rm Exp}_{\bar{x}^{t}}^{-1}\circ \vec{\xi}_k \circ \varphi_{k} (\rho^{t}_{k} y+ \varphi_{k}^{-1}(x^{t}_{k}))\conv{k\rightarrow \infty} \vec{\zeta}_{t}\circ \pi^{-1}(y)\quad \text{ in }C^l_{\mathrm{loc}}(\C \setminus \ens{a_1^{t},\cdots,a^{t}_{N_{t}}}), \quad \forall  l\in \N,
        \end{equation*}   
        where $\pi$ denotes the stereographic projection from $S^{2}$ into $\C$, and  $\big\{a_1^{s},\cdots,a^{s}_{N_{s}}\big\}, \big\{a_1^{t},\cdots,a^{t}_{N_{t}}\big\}$ are finite sets of points in the complex plane.
             \end{itemize}
     \end{theorem}

    \section{Appendix: Generalised Lorentz Spaces}\label{appendix}
    
   \subsection{General discussion and first example}

    Let $(X,\mu)$ be a measured space. For any $\mu$-measurable function $f:X\rightarrow \R^n$, we have 
    \begin{align*}
    	\int_{X}|f|^pd\mu=p\int_{0}^{\infty}t^{p-1}\lambda_f(t)dt\,,
    \end{align*}
    where, for all $t>0$, 
    we denote
    \begin{align*}
    	\lambda_f(t)=\mu\left(X\cap\ens{x:|f(x)|>t}\right)\, . 
    \end{align*}
    Define  the decreasing rearrangement $f_{\ast}:\R_+\rightarrow \R_+\cup\ens{\infty}$ of $f$ by 
    \begin{align*}
    	f_{\ast}(t)=\inf\left(\R_+\cap\ens{s: \lambda_f(s)\leq t}\right)\,.
    \end{align*}
 It is clear from the definitions that for all $t>0$
    \begin{align*}
    	\leb^1\left(\R_+\cap\ens{s:f_{\ast}(s)>t}\right)=\lambda_f(t)\,, \quad \text{ for all $t>0$ .}
    \end{align*}
    Applying twice the slicing formula, we deduce that 
    \begin{align*}
    	\int_{X}|f|^pd\mu=p\int_{0}^{\infty}t^{p-1}\lambda_f(t)dt=p\int_{0}^{\infty}t^{p-1}\leb^1\left(\R_+\cap\ens{s:f_{\ast}(s)>t}\right)dt=\int_{0}^{\infty}f_{\ast}^p(s)ds. 
    \end{align*}
    More generally, for all real-valued differentiable functions $\varphi,\psi:\R_+\rightarrow \R_+$, we have the integration by parts formula
    \begin{align}\label{ipp_gen}
    	\int_{0}^{\infty}\varphi(\lambda_f(t))\psi'(t)dt=\int_{0}^{\infty}\varphi'(t)\psi\left(f_{\ast}(t)\right)dt.
    \end{align}
    Since this formula is not completely standard, we give a proof of it. Let $f$ be a non-negative step function, then there exists $0<a_1<a_2<\cdots <a_n<\infty$ and pair-wise disjoint measurable sets $A_1, \cdots,A_n$ such that 
    \begin{align*}
    	f=\sum_{i=1}^{n}a_i\mathbf{1}_{A_i}.
    \end{align*}
    Following \cite{lorentz_general}, defining $\displaystyle B_i=\bigcup_{j=i}^nA_j$, we have
    \begin{align*}
    	f=\sum_{i=1}^{n}(a_i-a_{i-1})\mathbf{1}_{B_i},
    \end{align*}
    where $a_0=0$. Then it holds
    \begin{align*}
    	\lambda_f=\sum_{i=1}^{n}\mu(B_i)\mathbf{1}_{[a_{i-1},a_i]} \quad \text{and} \quad f_{\ast}=\sum_{i=1}^{n}a_i\mathbf{1}_{[\mu(B_{i+1}),\mu(B_i)]}. 
    \end{align*}
    Thanks to a discrete integration by parts with $B_{n+1}=\emptyset$, we deduce that
    \begin{align*}
    	&\int_{0}^{\infty}\varphi(\lambda_f(t))\psi'(t)dt=\sum_{i=1}^{n}\int_{a_{i-1}}^{a_i}\varphi(\mu(B_i))\psi'(t)=\sum_{i=1}^{n}\varphi(\mu(B_i))\left(\psi(a_i)-\psi(a_{i-1})\right) \\
    	&=\sum_{i=1}^{n}\psi(a_i)\left(\varphi(\mu(B_i))-\varphi(\mu(B_i+1))\right)
    	=\sum_{i=1}^n\psi(a_i)\int_{\mu(B_{i+1})}^{\mu(B_i)}\varphi'(t)dt=\sum_{i=1}^{n}\int_{\mu(B_{i+1})}^{\mu(B_i)}\varphi'(t)\psi(f_{\ast}(t))dt\\
    	&=\int_{0}^{\infty}\varphi'(t)\psi(f_{\ast}(t))dt. 
    \end{align*}
    The rest of the proof is the same as \cite{lorentz_general} if $\varphi$ is unbounded or if $(X,\mu)$ is $\sigma$-finite, and given in \cite{lorentz_mary} in the general case. Notice that the proof would hold unchanged only assuming that     $\varphi$ and $\psi$ are absolutely continuous. 
    
    We will now define a class of Lorentz spaces (which can also be seen as generalisation of Orlicz spaces \cite{orlicz}) of interest in this paper (see \cite{lorentz_general} and \cite{lorentz_mary}). Let $\mathscr{C}$ be the set of non-negative concave functions $\varphi:\R_+\rightarrow \R_+$ such that $\varphi$ is continuous at $0$, 
    \begin{align*}
    	\varphi(0)=\lim\limits_{t\rightarrow 0}\varphi(t)=0
    \end{align*}
    and $\varphi(t)>0$ for all $t>0$. For all measurable $f:X\rightarrow \R^n$ (of $f:X\rightarrow \R\cup\ens{\pm\infty}$), define the norm 
    \begin{align*}
    	\genorm{f}{N(\varphi)}=\int_{0}^{\infty}\varphi(\lambda_f(t))dt.
    \end{align*}
    Now fix some integer $n\geq 1$ and let $\mathscr{M}(X)$ be the class of measurable $\R^n$-valued function on $X$. Define
    \begin{align*}
    	N(\varphi)=\mathscr{M}(X)\cap\ens{f:\genorm{f}{N(\varphi)}<\infty}.
    \end{align*}
    Then we have the following result:
    \begin{theorem}[Steigerwalt-White \cite{lorentz_general}, Steigerwalt \cite{lorentz_mary}]\label{nspace}
    	The functional $\Vert\,\cdot\,\Vert_{N(\varphi)}$ is a norm and $N(\varphi)$ and $(N(\varphi),\Vert\,\cdot\,\Vert_{N(\varphi)})$ is a Banach space. 
    \end{theorem}
    By the integration by parts formula \eqref{ipp_gen}, we deduce that
    \begin{align*}
        \genorm{f}{N(\varphi)}=\int_{0}^{\infty}\varphi'(t)f_{\ast}(t)\, dt\, .
    \end{align*}
    Now let $1<p<\infty$ and $1\leq q< \infty$.   Define for all $t>0$
    \begin{align*}
    	f_{\ast\ast}(t)=\dashint{0}^tf_{\ast}(s)ds=\frac{1}{t}\int_{0}^tf_{\ast}(s)ds.
    \end{align*}
     Then, the Lorentz space $\mathrm{L}^{p,q}(X)$ is defined by 
    \begin{align*}
    	\mathrm{L}^{p,q}(X)\cap\ens{f:\np{f}{p,q}{X}<\infty}, \quad \text{where} \quad \np{f}{p,q}{X}=\left(\int_{0}^{\infty}t^{\frac{q}{p}-1}f_{\ast\ast}^q(t)dt\right)^{\frac{1}{q}}.
    \end{align*}
    It is a Banach space and the following semi-norm $|\,\cdot\,|_{\mathrm{L}^{p,q}(X)}$
    \begin{align*}
    	| f|_{\mathrm{L}^{p,q}(X)}=\left(\int_{0}^{\infty}t^{\frac{q}{p}-1}f_{\ast}^q(t)dt\right)^{\frac{1}{q}}
    \end{align*}
     is equivalent to $\Vert\,\cdot\,\Vert_ {\mathrm{L}^{p,q}(X)}$. In the case $q=1$, we have by Fubini's theorem
    \begin{align*}
     	\np{f}{p,1}{X}=\frac{p}{p-1}| f|_{\mathrm{L}^{p,1}(X)}.
    \end{align*}
    Using the integration by parts formula \eqref{ipp_gen},   with $\varphi(t)=\frac{p}{q}t^{\frac{p}{q}}$ and  $\psi(t)=t^q$ we deduce that 
    \begin{align*}
    	\int_{0}^{\infty}t^{\frac{p}{q}-1}f_{\ast}^q(t)dt=\int_{0}^{\infty}qt^{q-1}\frac{p}{q}\lambda_f(t)^{\frac{q}{p}}dt=p\int_{0}^{\infty}t^{q-1}\lambda_f(t)^{\frac{q}{p}}dt.
    \end{align*}
    This gives the well known fact that $\mathrm{L}^{p,p}(X)=\mathrm{L}^p(X)$ with equivalent norms. Taking instead $\varphi(t)=t^{\frac{1}{p}}$, we get that
    \begin{align*}
    	\np{f}{p,1}{X}=\frac{p}{p-1}|f|_{\mathrm{L}^{p,1}(X)}=\frac{p}{p-1}\int_{0}^{\infty}t^{\frac{1}{p}-1}f_{\ast}(t)dt=\frac{p^2}{p-1}\genorm{f}{N(\varphi)}. 
    \end{align*}
    Therefore, the spaces $N(\varphi)$ are generalisations of $\mathrm{L}^{p,1}$ spaces, but $\mathrm{L}^{p,q}$ spaces with 
    $1< q<\infty$     are not $N(\varphi)$-spaces. 
        
    Now, we will define generalisations of the weak $\mathrm{L}^p$ spaces or Marcinkiewicz spaces. Fix a $\sigma$-algebra $\mathscr{A}\subset X(\mu)$ and assume the following property: for all $A\in \mathscr{A}$, if $\mu(A)=\infty$, there exists $B\subset A$ such    that $0<\mu(B)<\infty$.  For $\varphi\in \mathscr{C}$, set
    \begin{align}\label{show1}
    	\genorm{f}{M(\varphi)}=\sup_{t>0}\ens{\frac{1}{\varphi(t)}\int_{0}^{t}f_{\ast}(s)ds}, \quad \text{for all $f\in \mathscr{M}(X)$.}
    \end{align}
    If $1\leq p<\infty$, define $\mathrm{L}^{p,\infty}(X)=M(t^{1-\frac{1}{p}})$. It is known that $\mathrm{L}^{p,\infty}(X)$ is a Banach space equipped with this norm for $1<p<\infty$ (and such a  norm is denoted by $\np{\,\cdot\,}{p,\infty}{X}$). Furthermore, the following result holds.
    \begin{theorem}[Steigerwalt-White \cite{lorentz_general}, Steigerwalt \cite{lorentz_mary}]\label{mspace}
    	Assume that $\varphi(t)=o(t)$ as $t\rightarrow \infty$. Then $M(\varphi)$ is a norm and $(M(\varphi),\Vert\,\cdot\,\Vert)$ is a Banach space.
    \end{theorem}
    \begin{rem}
    	In \cite{lorentz_general}, the authors first define the norm
    	\begin{align}\label{show2}
    		\genorm{f}{M(\varphi)}=\sup\ens{\frac{1}{\varphi(\mu(A))}\int_{A}|f|d\mu: A\in \mathscr{A}_1}
    	\end{align}
        where 
        \begin{align*}
        	\mathscr{A}_1=\mathscr{A}\cap\ens{A:0<\mu(A)<\infty}. 
        \end{align*}
        Then Theorem \ref{mspace} holds with this norm without any restrictions on $\varphi$, and the authors show (Theorem $3.3$) that  \eqref{show1} and \eqref{show2} coincide if either $(X,\mathscr{A},\mu)$ is $\sigma$-finite or if $\varphi(t)=o(t)$ as $t\rightarrow \infty$. Then they quote Mary Steingerwalt's PhD thesis \cite{lorentz_mary} where the result is proven without any hypothesis on $X$ or $\varphi$. Notice that this result does not contradict the fact that $\mathrm{L}^{1,\infty}$ is not a Banach space (even with $X=\R$ and $\mu=\leb^1$). Indeed, since $\mathrm{L}^{p,\infty}(X)=M(t^{1-\frac{1}{p}})$, we would have $\mathrm{L}^{1,\infty}(X)=M(1)$ but the function $\varphi(t)=1$ is not admissible (does not belong to $\mathscr{C}$) since it does not satisfy $\varphi(0)=0$. However, taking $p=\infty$, we formally get $\mathrm{L}^{\infty,\infty}(X)=M(t)=N(t)^{\ast}=\mathrm{L}^{1}(X)^{\ast}=\mathrm{L}^{\infty}(X)$, so there is no new Lorentz space corresponding to $p=q=\infty$.
    \end{rem}

    Now, it is known that for all $1<p<\infty$ and $1\leq q\leq \infty$, the dual space of $\mathrm{L}^{p,q}(X)$ is $\mathrm{L}^{p',q'}(X)$ with 
    \begin{align*}
   \frac{1}{p}+\frac{1}{p'}=1\, ,\qquad \frac{1}{q}+\frac{1}{q'}=1\, .
    \end{align*}
    Moreover,  for all measurable $f,g:X\rightarrow \R\cup\ens{\pm\infty}$ such that $f\in \mathrm{L}^{p,q}(X)$ and $g\in \mathrm{L}^{p',q'}(X)$, it holds
    \begin{align*}
    	\left|\int_{X}fg\, d\mu\right|\leq \np{f}{p,q}{X}\np{g}{p',q'}{X}.
    \end{align*}
    For the generalised Lorentz space $N(\varphi)$, we have the following duality result. 
    \begin{theorem}[Steigerwalt-White \cite{lorentz_general}, Steigerwalt \cite{lorentz_mary}]\label{thm:dualityNM}
    	For all $(f,g)\in N(\varphi)\times M(\varphi)$, it holds $fg\in \mathrm{L}^1(X,\mu)$ and
    	\begin{align*}
    		\left|\int_{X}fg\, d\mu\right|\leq \genorm{f}{N(\varphi)}\genorm{g}{M(\varphi)}. 
    	\end{align*}
	In particular, $N(\varphi)^{\ast}=M(\varphi)$.
    \end{theorem}
    Notice that those results are consistent with the $\mathrm{L}^{p,1}-\mathrm{L}^{p',\infty}$ duality: indeed $\mathrm{L}^{q,\infty}(X)=M(t^{\frac{1}{q}})$,  so that $\mathrm{L}^{p,1}(X)^{\ast}=N(t^{\frac{1}{p}})^{\ast}=M(t^{\frac{1}{p}})=\mathrm{L}^{p',\infty}(X)$. 
    
    Finally, generalising both the Lorentz spaces and the Orlicz spaces, one can add a positive weight in the definition of $N(\varphi)$ which gives a new norm $N(\varphi,\psi)$ defined by 
    \begin{align*}
    	\genorm{f}{N(\varphi,\psi)}=\int_{0}^{\infty}\varphi(\lambda_f(t))\psi(t)\, dt=\int_{0}^{\infty}\varphi'(t)\Psi(f_{\ast}(t))\, dt\, ,\quad \text{where } \Psi(t)=\int_{0}^t\psi(s)\, ds\, .
    \end{align*}
   where the second identity follows from \eqref{ipp_gen}. Here, to make sure that $f$ is a normal, we must assume that $\psi(t)>0$ for all $t>0$ and $\varphi\in \mathscr{C}$. However, even this generalisation does not permit to recover the Lorentz spaces $\mathrm{L}^{p,q}(X)$ for $1<q<\infty$.

    \subsection{Another Generalised Lorentz Space}\label{app2}  
    This appendix is linked to the content of the paper but not strictly needed; we include it for future reference.
    Let $d>0$, and  define the function $f=f_{d}:B(0,R)\rightarrow \R_+$ by
    \begin{align*}
    	f(r)=\frac{1}{r}\left(1+\frac{1}{d}\left(\left(\frac{R}{r}\right)^{d}-1\right)\right).
    \end{align*}
    We want to study a Lorentz space based on this function and obtain estimates reminiscent of Lemma \ref{lemme_holomorphe1} and uniform as $d\rightarrow 0$. Notice that when $d\rightarrow 0$, we have
    \begin{align*}
    f_{d}(r)=\frac{1}{r}\left(1+\frac{1}{d}\left(\exp\left(d\log\left(\frac{R}{r}\right)\right)-1\right)\right)\conv{d\rightarrow 0}\frac{1}{r}\left(1+\log\left(\frac{R}{r}\right)\right),
    \end{align*}
    which will correspond to the standard function in $M(\Lambda_1)$, where $\Lambda_1$ is defined in \eqref{lambda}.
    We have
    \begin{align*}
    	f'(r)=-\frac{1}{r^2}\left(1+\frac{1}{d}\left(\left(\frac{R}{r}\right)^{d}-1\right)\right)-\frac{1}{r^2}\left(\frac{R}{r}\right)^{d}<0\, ,
    \end{align*}
    which shows that $f$ is strictly decreasing. Therefore, we have for all $t>0$
    \begin{align*}
    	&\leb^2\left(B(0,R)\cap\ens{x:f(|x|)}>t\right)
    	=\leb^2\left(B(0,R)\cap\ens{x:|x|<f^{-1}(t)}\right)
    	=\pi \left(f^{-1}(t)\right)^2.
    \end{align*}
    If $u=u_d=f_d(|\,\cdot\,|)$, we deduce that
    \begin{align*}
    	\leb^2(B(0,R)\cap\ens{x:|u(x)|>s})\leq t\Longleftrightarrow \pi (f^{-1}(s))^2\leq t\Longleftrightarrow f^{-1}(s)\leq \sqrt{\frac{t}{\pi}}\Longleftrightarrow s\geq f\left(\sqrt{\frac{t}{\pi}}\right). 
    \end{align*}
   For $t\leq \pi R^2$, we infer that
    \begin{align*}
    	u_{\ast}(t)=u_{d,\ast}=f\left(\sqrt{\frac{t}{\pi}}\right)=\sqrt{\frac{\pi}{t}}\left(1+\frac{1}{d}\left(\left(\frac{\pi R^2}{t}\right)^{\frac{d}{2}}-1\right)\right),
    \end{align*}
    while $u_{\ast}(t)=0$ for $t\geq \pi R^2$.
    Therefore, we have for all $t\leq \pi R^2$
    \begin{align}\label{new_lorentz1}
    	\int_{0}^tu_{\ast}(s)ds&=\left[2\sqrt{\pi s}\left(1+\frac{1}{d}\left(\left(\frac{\pi R^2}{s}\right)^{\frac{d}{2}}-1\right)\right)\right]_0^t +\int_{0}^{t}\sqrt{\frac{\pi}{s}}\left(\frac{\pi R^2}{s}\right)^{\frac{d}{2}}ds\nonumber\\
    	&=2\sqrt{\pi t}\left(1+\frac{1}{\beta}\left(\left(\frac{\pi R^2}{t}\right)^{\frac{d}{2}}-1\right)\right)+\frac{2}{1-d}\sqrt{\pi}(\pi R^2)^{\frac{d}{2}}t^{\frac{1-d}{2}},
    \end{align}
    while for all $t\geq \pi R^2$ we have
    \begin{align}\label{new_lorentz2}
    	\int_{0}^tu_{\ast}(s)ds=2\pi R+\frac{2\pi R}{1-d}\conv{d\rightarrow 0}4\pi R<\infty\, .
    \end{align}
    Now, let $\alpha<\dfrac{1}{2}$ and define $\varphi:[0,1]\rightarrow \R_+$ by
    \begin{align*}
    	\varphi(t)=\sqrt{t}\left(\frac{1}{t^{\alpha}}-1\right)=t^{\frac{1}{2}-\alpha}-t^{\frac{1}{2}}.
    \end{align*}
    Then we have
    \begin{align*}
    	\varphi''(t)=\left(-\frac{1}{4}+\alpha^2\right)t^{-\frac{3}{2}-\alpha}+\frac{1}{4}t^{-\frac{3}{2}}<0\Longleftrightarrow t^{\alpha}<(1-(2\alpha)^2)\Longleftrightarrow t<\left(1-(2\alpha)^2\right)^{\frac{1}{\alpha}},
    \end{align*}
    which shows that $\varphi$ is concave in $[0,(1-(2\alpha)^2)^{\frac{1}{\alpha}}]$. Notice that as $\alpha\rightarrow 0$, we have
    \begin{align*}
    	(1-(2\alpha)^2)^{\frac{1}{\alpha}}=e^{\frac{1}{\alpha}\log\left(1-(2\alpha)^2\right)}=e^{-4\alpha+O(\alpha^2)}\conv{\alpha\rightarrow 0}1. 
    \end{align*}
    Therefore, for all $0<\beta<\dfrac{1}{4}$, define the concave function $K_{\beta}:\R_+\rightarrow \R_+$ by 
    \begin{align*}
    	K_{\beta}(t)&=\sqrt{t}\left(1+\frac{1}{2\beta}\left(\left(\frac{\pi R^2}{t}\right)^{\beta}-1\right)\mathbf{1}_{[0,(1-(2\beta)^2)^{\frac{1}{\beta}}\pi R^2]}(t)\right)+\frac{(\pi R^2 t)^{\frac{1}{2}-\beta}}{1-2\beta }\,.
    \end{align*}
    It  is continuous at $0$, satisfies $K_{\beta}(0)=0$, and $K_{\beta}(t)>0$ for all $t>0$. It is continuous everywhere except in $(1-(2\beta)^2)^{\frac{1}{\beta}}\pi R^2$, but the function used in the definition of Lorentz spaces only needs to be continuous at $0$. Furthermore, the function $K_{\beta}$ could be replaced by its continuous counter-part 
    \begin{align*}
    	&K_{\beta}^{\ast}(t)=\sqrt{t}\left(1+\frac{1}{2\beta}\left(\left(\frac{\pi R^2}{t}\right)^{\beta}-1\right)\mathbf{1}_{[0,(1-(2\beta)^2)^{\frac{1}{\beta}}\pi R^2]}(t)+\frac{1}{2\beta}\left(\left(\frac{1}{1-(2\beta)^2}-1\right)\mathbf{1}_{\ens{t>(1-(2\beta)^2)^{\frac{1}{\beta}}}}\right)\right)\\
    	&+\frac{(\pi R^2 t)^{\frac{1}{2}-\beta}}{1-2\beta }\\
    	&=\sqrt{t}\left(1+\frac{1}{2\beta}\left(\left(\frac{\pi R^2}{t}\right)^{\beta}-1\right)\mathbf{1}_{[0,(1-(2\beta)^2)^{\frac{1}{\beta}}\pi R^2]}(t)+\left(\frac{2\beta}{1-(2\beta)^2}\mathbf{1}_{\ens{t>(1-(2\beta)^2)^{\frac{1}{\beta}}\pi R^2}}\right)\right)\\
    	&+\frac{(\pi R^2 t)^{\frac{1}{2}-\beta}}{1-2\beta }
    \end{align*}
    and the results stated below would be unchanged. 
    Furthermore, by \eqref{new_lorentz1}, we deduce that for all $0<t\leq (1-d^2)^{\frac{2}{d}}\pi R^2$
    \begin{align*}
    	\frac{1}{K_{\frac{d}{2}}(t)}\int_{0}^tu_{\ast}(s)ds=2\sqrt{\pi},
    \end{align*}
    and for all $(1-d^2)^{\frac{2}{d}}\pi R^2<t\leq \pi R^2$, we have 
    \begin{align*}
    	&\frac{1}{K_{\frac{d}{2}}(t)}\int_{0}^{t}u_{\ast}(s)ds=\frac{1}{\sqrt{t}+\frac{(\pi R^2t)^{\frac{1-d}{2}}}{1-d}}\left(2\sqrt{\pi t}\left(1+\frac{1}{d}\left(\left(\frac{\pi R^2}{t}\right)^{\frac{d}{2}}-1\right)\right)+\frac{2\sqrt{\pi}}{1-d}(\pi R^2)^{\frac{d}{2}}t^{\frac{1-d}{2}}\right)\\
    	&\leq \frac{1}{(1-d^2)^{\frac{1}{d}}\sqrt{\pi}R+\frac{(\pi R^2)^{1-d}}{(1-d)^2(1+d)}(1-d^2)^{\frac{1}{d}}}\left(2\pi R\left(1+\frac{1}{d}\left(\frac{1}{(1-d^2)}-1\right)\right)+\frac{2\pi R}{1-d}\right).
    \end{align*}
    We  have
    \begin{align*}
    	\frac{1}{d}\left(\frac{1}{(1-d^2)}-1\right)=d+O(d^3)\conv{d\rightarrow 0}0\,.
    \end{align*}
    As above we have
    \begin{align*}
    	(1-d^2)^{d}=e^{d\log(1-d^2)}=1-d^3+O(d^5)\conv{d\rightarrow 0}1,
    \end{align*}
    which implies that 
    \begin{align*}
    	\frac{1}{(1-d^2)^{\frac{1}{d}}\sqrt{\pi}R+\frac{(\pi R^2)^{1-d}}{(1-d)^2(1+d)}(1-d^2)^{\frac{1}{d}}}\left(2\pi R\left(1+\frac{1}{d}\left(\frac{1}{(1-d^2)}-1\right)\right)+\frac{2\pi R}{1-d}\right)&\conv{d\rightarrow 0}\frac{4\pi R}{\sqrt{\pi}R+\pi R^2}\\
    	&=\frac{4\sqrt{\pi}}{1+\sqrt{\pi}R}<\infty. 
    \end{align*}
    Thus, for all $t>\pi R^2$, we have
    \begin{align*}
    	\frac{1}{K_{\frac{d}{2}}(t)}\int_{0}^tu_{\ast}(s)ds&=\frac{2\pi R}{\sqrt{t}+\frac{(\pi R^2t)^{\frac{1-d}{2}}}{1-d}}2\pi R\left(1+\frac{1}{1-d}\right)\\
    	&\leq \frac{2\pi R}{\sqrt{\pi}R+\frac{(\pi R^2)^{{1-d}}}{1-d}}\left(1+\frac{1}{1-d}\right)
    	\conv{d\rightarrow 0}\frac{4\pi R}{\sqrt{\pi}R+\pi R^2}=\frac{4\sqrt{\pi}}{1+\sqrt{\pi}R}<\infty.
    \end{align*}   
    Thanks to \eqref{new_lorentz2}, we have
    \begin{align*}
    	\lim\limits_{t\rightarrow \infty}\frac{1}{K_{\frac{d}{2}}(t)}\int_{0}^tu_{\ast}(s)ds=0. 
    \end{align*}
    We deduce that 
    \begin{align*}
    	\limsup_{d\rightarrow 0}\sup_{t>0}\frac{1}{K_{\frac{d}{2}}(t)}\int_{0}^{t}u_{\ast}(s)ds\leq 4\sqrt{\pi}<\infty,
    \end{align*}
    showing that $u=u_d$ is uniformly bounded in the Lorentz space $M(K_{\frac{d}{2}})$ as $d\rightarrow 0$,
    where for all $0<\beta<\frac{1}{4}$, we have 
    \begin{align*}
    	M(K_{\beta})=\mathrm{L}^{1}_{\mathrm{loc}}(B(0,R))\cap\ens{f:\Vert f\Vert_{M(K_{\beta})}=\sup_{t>0}\left(\frac{1}{K_{\beta}(t)}\int_{0}^tf_{\ast}(s)ds\right)<\infty}. 
    \end{align*}
    Notice that when $\beta\rightarrow 0$, for all $t>0$, we have
    \begin{align*}
    	K_{\beta}(t)&=\sqrt{t}\left(1+\frac{1}{2\beta}\left(e^{\beta\log\left(\frac{\pi R^2}{t}\right)}-1\right)\mathbf{1}_{[0,(1-(2\beta)^2)^{\frac{1}{\beta}}\pi R^2]}(t)\right)+\frac{(\pi R^2t)^{\frac{1}{2}-\beta}}{1-2\beta}\\
    	&=\sqrt{t}\left(1+\frac{1}{2}\log\left(\frac{\pi R^2}{t}\right)\mathbf{1}_{[0,(1-(2\beta)^2)^{\frac{1}{\beta}}\pi R^2]}(t)+O\left(\beta \log^2\left(\frac{\pi R^2}{t}\right)\right)\right)+\frac{(\pi R^2t)^{\frac{1}{2}-\beta}}{1-2\beta}\\
    	&\conv{\beta\rightarrow 0}\sqrt{t}\left(1+\sqrt{\pi}R+\frac{1}{2}\log\left(\frac{\pi R^2}{t}\right)_+\right)=\sqrt{t}\left(1+\log_+\left(R\sqrt{\frac{\pi}{t}}\right)\right)+\sqrt{\pi} R\sqrt{t}=\Lambda_1(t)+\sqrt{\pi}R\sqrt{t},
    \end{align*}
    showing that $M(K_{0})=M(\Lambda_{1})$, where $\Lambda_1$ is defined in \eqref{lambda}. Therefore, we introduce the notation
    \begin{align*}
    	\mathrm{L}^{2,\infty}_{\log,\beta}(B(0,R))=M(K_{\beta}).
    \end{align*}
    Likewise, we define the space $\mathrm{L}^{2,1}_{\log,\beta}(B(0,R))=N(K_{\beta})$ by 
    \begin{align*}
    	N(K_{\beta})=\mathrm{L}^{1}_{\mathrm{loc}}(B(0,R))\cap\ens{f:\Vert f\Vert_{N(K_{\beta})}=\int_{0}^{\infty}K_{\beta}\left(\leb^2\left(B(0,R)\cap\ens{x:|f(x)|>t}
    		\right)\right)dt<\infty
    	}. 
    \end{align*}
    Now, we will show a result similar to Lemma \ref{lemme_holomorphe1} (with a uniform estimate as $\beta\rightarrow 0$).
    Fix $0<\beta<\dfrac{1}{4}$, let $x\in B(0,\frac{R}{2})$ and $0<r<\frac{R}{4}$. By the co-area formula, we have
    \begin{align*}
    	\int_{B_{2r}\setminus\bar{B}_r(x)}|u(x)|dx=\int_{r}^{2r}\left(\rho\int_{\partial B_{\rho}(x)}|u|d\mathscr{H}^1\right)\frac{d\rho}{\rho}\leq \log(2)\inf_{r<\rho<2r}\left(\rho\int_{\partial B_{\rho}(x)}|u|d\mathscr{H}^1\right).
    \end{align*}
    Therefore, there exists $\rho\in (r,2r)$ such that 
    \begin{align*}
    	\rho\int_{\partial B_{\rho}(x)}|u|d\mathscr{H}^1&\leq \frac{1}{\log(2)}\int_{B_{2r}\setminus\bar{B}_r(x)}|u(x)|dx.
    \end{align*}  
    We will show that if $u$ is a holomorphic function on $B(0,R)$, then a $\mathrm{L}^{2,\infty}_{\log^{\beta}}$ control implies a $\mathrm{W}^{1,1}$ control on $B(0,\alpha R)$ for all $\alpha<1$. Since it seems not standard to us, we give a full proof of this claim (in fact, we know not whether such spaces where ever used in the past). 
    
    Using the $\mathrm{L}^{2,1}_{\log,\beta}/\mathrm{L}^{2,\infty}_{\log,\beta}$ duality (Theorem $4.4$ \cite{lorentz_general}), we deduce that 
    \begin{align*}
    	\int_{B_{2r}\setminus \bar{B}_{r}(x)}|u(x)|dx\leq \Vert 1\Vert_{\mathrm{L}^{2,1}_{\mathrm{log},{\beta}}(B_{2r}\setminus \bar{B}_r(0))}\Vert u\Vert_{\mathrm{L}^{2,\infty}_{\mathrm{log},{\beta}}(B_{2r}\setminus\bar{B}_r(0))}.
    \end{align*}
    Notice that we have
    \begin{align*}
    	\lambda_{1}(t)=\leb^2\left(B_{2r}\setminus\bar{B}_r(x)\cap\ens{x:1>t}\right)=\left\{\begin{alignedat}{2}
    		&3\pi r^2\qquad&&\text{if}\;\,t<1\\
    		&0\qquad&&\text{if}\;\,t\geq 1.
    	\end{alignedat}\right.
    \end{align*}
    We have by definition
    \begin{align*}
    	\Vert 1\Vert_{\mathrm{L}^{2,1}_{\mathrm{log}}(B_{2r}\setminus \bar{B}_r(x))}&=\int_{0}^{\infty
    	}\left((\lambda_1(t))^{\frac{1}{2}}\left(1+\frac{1}{2\beta}\left(\left(\frac{\pi R^2}{\lambda_1(t)}\right)^{\beta}-1\right)\mathbf{1}_{[0,(1-(2\beta)^2)^{\frac{1}{\beta}}\pi R^2]}(\lambda_1(t))\right)\right.\\
    	&\left.+\frac{(\pi R^2 \lambda_1(t))^{\frac{1}{2}-\beta}}{1-2\beta }\right)dt\\
    	&=\sqrt{3\pi}r\left(1+\frac{1}{2\beta}\left(\left(\frac{1}{\sqrt{3}}\frac{R}{r}\right)^{2\beta}\right)-1\right)\mathbf{1}_{\ens{\sqrt{3}r<(1-(2\beta)^2)^{\frac{1}{2\beta}}R}}+\frac{\pi (\sqrt{3}rR)^{1-2\beta}}{1-2 \beta}. 
    \end{align*}
    Finally, we deduce that 
    \begin{align*}
    	\rho\int_{\partial B_{\rho}(x)}|u|d\mathscr{H}^1&\leq \bigg(\frac{\sqrt{3\pi}}{\log(2)}r\left(1+\frac{1}{2\beta}\left(\left(\frac{1}{\sqrt{3}}\frac{R}{r}\right)^{2\beta}-1\right)\mathbf{1}_{\ens{\sqrt{3}r<(1-(2\beta)^2)^{\frac{1}{2\beta}}R}}\right)\\
    	&+\frac{\pi}{\log(2)}\frac{ (\sqrt{3}rR)^{1-2\beta}}{1-2 \beta}\bigg)\znp{u}{2,\infty}{\log,\beta}{B_{2r}\setminus\bar{B}_r(x)},
    \end{align*}
    and \emph{a fortiori} that for all $x\in B(0,\frac{R}{2})$ and $0<r\leq \frac{R}{2}$ such that $B(x,2r)\subset B(0,R)$, there exists $\rho\in[r,2r]$ such that 
    \begin{align}\label{estimate_lebesgue3}
    	\int_{\partial B_{\rho}(x)}|u|d\mathscr{H}^1&\leq \bigg(\frac{\sqrt{3\pi}}{\log(2)}\left(1+\frac{1}{2\beta}\left(\left(\frac{1}{\sqrt{3}}\frac{R}{r}\right)^{2\beta}-1\right)\mathbf{1}_{\ens{\sqrt{3}r<(1-(2\beta)^2)^{\frac{1}{2\beta}}R}}\right)\nonumber\\
    	&+\frac{\pi}{\log(2)}\frac{1}{r^{2\beta}}\frac{ (\sqrt{3}R)^{1-2\beta}}{1-2 \beta}\bigg)\znp{u}{2,\infty}{\log,\beta}{B_{2r}\setminus\bar{B}_r(x)}.
    \end{align}
    Thanks to this result, we will now be able to show a variant of a lemma presented first in \cite{quantamoduli}. 
    \begin{lemme}\label{lemme_holomorphe12}
    	Let $u:B(0,R)\rightarrow \C$ be a holomorphic function and fix some $0\leq \alpha<1$ and $0\leq \beta<\dfrac{1}{4}$. Assume that $u\in \mathrm{L}^{2,\infty}_{\log,{\beta}}(B(0,R))$. Then $u\in \mathrm{L}^2(B(0,\alpha R))$ and there exists a universal constant $\Gamma_0$ {\rm(}independent of $\alpha$ and $\beta${\rm)} such that 
    	\begin{align*}
    		\np{u}{2}{B(0,\alpha R)}\leq \Gamma_1\frac{\alpha}{(1-\alpha)^{1+2\beta}}\left(1+R^{1-4\beta}+\frac{1}{2\beta}\left(\left(\frac{2}{\sqrt{3}}\frac{1}{1-\alpha}\right)^{2\beta}-1\right)\right)\znp{u}{2,\infty}{\log,{\beta}}{B(0,R)}. 
    	\end{align*}
    \end{lemme}
    \begin{proof}
    	Write 
    	\begin{align*}
    		u(z)=\sum_{n=0}^{\infty}a_nz^n.
    	\end{align*}
    	Notice that $u=\p{z}v$, where 
    	\begin{align}
    		v(z)=\sum_{n=1}^{\infty}b_nz^n=\sum_{n=1}^{\infty}\frac{a_{n-1}}{n}z^n.
    	\end{align}
    	First, using 
    	the estimate \eqref{estimate_lebesgue2} applied to $\D v$ at a point $z\in \partial B(0,\alpha R)$ with $r=\frac{1}{2}(1-\alpha)R$, we deduce by the mean-value formula that for some $\rho\in[\frac{1}{2}(1-\alpha)R,(1-\alpha)R]$, we have 
    	\begin{align}\label{neue_id02}
    		&|\D v(z)|=|2\,\p{z}v(z)|=\left|\frac{1}{\pi \rho}\int_{\partial B(z,\rho)}\p{\zeta}v(\zeta)d\zeta\right|\nonumber\\
    		&\leq \left(\frac{1}{\log(2)}\sqrt{\frac{3}{\pi}}\frac{1}{\rho}\left(1+\frac{1}{2\beta}\left(\left(\frac{2}{\sqrt{3}}\frac{1}{1-\alpha}\right)^{2\beta}-1\right)\right)+\frac{1}{\log(2)}\frac{1}{\rho\, r^{2\beta}}\frac{ (\sqrt{3}R)^{1-2\beta}}{1-2 \beta}\right)\znp{\D v}{2,\infty}{\log,\beta}{B_{2r}\setminus\bar{B}_r(x)}\nonumber\\
    		&\leq 
    		\frac{8}{\log(2)}\left(\sqrt{\frac{3}{\pi}}\frac{1}{(1-\alpha)R}\left(1+\frac{1}{2\beta}\left(\left(\frac{2}{\sqrt{3}}\frac{1}{1-\alpha}\right)^{2\beta}-1\right)\right)+\frac{1}{((1-\alpha)R)^{1+2\beta}}\frac{ (\sqrt{3}R)^{1-2\beta}}{1-2 \beta}\right)\znp{u}{2,\infty}{\log,\beta}{B_{2r}\setminus\bar{B}_r(x)},
    	\end{align}
    	where we used $|\D v|^2=4|\p{z}v|^2$ by the holomorphy of $v$. Therefore, by integrating by parts, we deduce since $u$ and $\bar{u}$ are harmonic that 
    	\begin{align}\label{neue_id12}
    		\int_{B(0,\alpha R)}|u(z)|^2|dz|^2&\leq \frac{1}{2}\int_{B(0, \alpha R)}|\D v|^2dx=\frac{1}{2}\int_{B(0, \alpha R)}\dive\left(\bar{v}\D v\right)=\frac{1}{2}\int_{\partial B(0, \alpha R)}\bar{v}\,{\partial_{\nu}{v}}\,d\mathscr{H}^1\nonumber\\
    		&=\frac{1}{2}\int_{\partial B(0,\alpha R)}\bar{(v-v_{\alpha R})}\partial_{\nu} v\,d\mathscr{H}^1,
    	\end{align}
    	where for all $0<\rho<R$, we have
    	\begin{align*}
    		v_{\rho}=\dashint{\partial B_{\rho}(0)}v\,d\mathscr{H}^1.
    	\end{align*}
    	Thanks to the $\mathrm{L}^{\infty}$ bound \eqref{neue_id0} and the Sobolev embedding $\mathrm{H}^{\frac{1}{2}}(S^1)\hookrightarrow \mathrm{L}^1(B(0,1))$, there exists a uniform constant $C_0>0$ such that 
    	\begin{align*}
    		&\np{\D v}{2}{B(0,\alpha R)}^2=\left|\int_{\alpha R}\bar{(v-{v}_{\partial B_{\alpha R}})}\partial_{\nu}v\,d\mathscr{H}^1\right|\leq \np{v-v_{\alpha R}}{1}{\partial B_{\alpha R}(0)}\np{\D v}{\infty}{\partial B_{\alpha R(0)}}\\
    		&\leq C_0\alpha R\hs{v}{\frac{1}{2}}{\partial B_{\alpha R}(0)}\times \frac{8}{\log(2)}\bigg(\sqrt{\frac{3}{\pi}}\frac{1}{(1-\alpha)R}\left(1+\frac{1}{2\beta}\left(\left(\frac{2}{\sqrt{3}}\frac{1}{1-\alpha}\right)^{2\beta}-1\right)\right)\\
    		&+\frac{1}{((1-\alpha)R)^{1+2\beta}}\frac{ (\sqrt{3}R)^{1-2\beta}}{1-2 \beta}\bigg)\znp{\D v}{2,\infty}{\log,\beta}{B_{2r}\setminus\bar{B}_r(x)}\\
    		&\leq \frac{8C_0}{\log(2)}\bigg(\sqrt{\frac{3}{\pi}}\frac{\alpha}{(1-\alpha)}\left(1+\frac{1}{2\beta}\left(\left(\frac{2}{\sqrt{3}}\frac{1}{1-\alpha}\right)^{2\beta}-1\right)\right)
    		+\frac{\alpha}{(1-\alpha)^{1+2\beta}}\frac{ 3^{\frac{1}{2}-\beta}R^{1-4\beta}}{1-2 \beta}\bigg)\\
    		&\times \np{\D v}{2}{B(0,\alpha R)}\znp{\D v}{2,\infty}{\log,\beta}{B_{2r}\setminus\bar{B}_r(x)}.
    	\end{align*}
    	Therefore, we get
    	\begin{align*}
    		\np{u}{2}{B(0,\alpha R)}&\leq \frac{8C_0}{\log(2)}\bigg(\sqrt{\frac{3}{\pi}}\frac{\alpha}{(1-\alpha)}\left(1+\frac{1}{2\beta}\left(\left(\frac{2}{\sqrt{3}}\frac{1}{1-\alpha}\right)^{2\beta}-1\right)\right)\\
    		&
    		+\frac{\alpha}{(1-\alpha)^{1+2\beta}}\frac{ 3^{\frac{1}{2}-\beta}R^{1-4\beta}}{1-2 \beta}\bigg)\znp{\D v}{2,\infty}{\log,\beta}{B_{2r}\setminus\bar{B}_r(x)}
    	\end{align*}
    	which concludes the proof of the lemma since $\beta<\dfrac{1}{4}$.    
    \end{proof}
    
    \begin{cor}\label{second_lorentz_space}
    	Let $\ens{\beta_k}_{k\in \N}\subset (0,\frac{1}{4})$ such that $\beta_k\conv{k\rightarrow \infty}0$. 
    	Assume that $\ens{u_k}_{k\in \N}:B(0,R)\rightarrow \C$ is a sequence of holomorphic functions bounded in $\mathrm{L}^{2,\infty}_{\log,\beta_k}(B(0,R))$, \emph{i.e.} such that
    	\begin{align*}
    		\Lambda=\limsup_{k\in \N}\znp{u_k}{2,\infty}{\log,\beta_k}{B(0,R)}<\infty.
    	\end{align*}
    	Then $\ens{u_k}_{k\in \N}$ is bounded in $\mathrm{L}^2(B(0,\alpha R))$ for all $0<\alpha<1$ and for all $0<\alpha<1$, we have
    	\begin{align}\label{unif_est_lorentz2}
    		\limsup_{k\rightarrow \infty}\np{u_k}{2}{B(0,\alpha R)}\leq \Gamma_1\frac{\alpha}{1-\alpha}\left(1+\log(2)+R+\log\left(\frac{1}{1-\alpha}\right)\right)\Lambda<\infty. 
    	\end{align}
    \end{cor}
    \begin{proof}
    	For all fixed $0<\alpha<1$, as $\beta\rightarrow 0$, we have
    	\begin{align*}
    		\frac{1}{2\beta}\left(\left(\frac{2}{\sqrt{3}}\frac{1}{1-\alpha}\right)^{2\beta}-1\right)=\frac{1}{2\beta}\left(e^{2\beta\log\left(\frac{2}{\sqrt{3}}\frac{1}{1-\alpha}\right)}-1\right)\conv{\beta\rightarrow 0}\log\left(\frac{2}{\sqrt{3}}\frac{1}{1-\alpha}\right)\leq \log(2)+\log\left(\frac{1}{1-\alpha}\right)
    	\end{align*}
    	which concludes the proof using Lemma \ref{lemme_holomorphe12}. 
    \end{proof}

    \nocite{}
    \bibliographystyle{plain}
    \bibliography{biblio_full}

\begin{thebibliography}{10}

\bibitem{alexakismazzeo1}
Spyridon Alexakis and Rafe Mazzeo.
\newblock Renormalized area and properly embedded minimal surfaces in
  hyperbolic 3-manifolds.
\newblock {\em Comm. Math. Phys. 297, no. 3, 621--651.}, 2010.

\bibitem{alexakismazzeo2}
Spyros Alexakis and Rafe Mazzeo.
\newblock Complete {W}illmore surfaces in $\mathbb{H}^3$ with bounded energy:
  boundary regularity and bubbling.
\newblock {\em J. Differential Geom. 101, no. 3, 369--422.}, 2015.

\bibitem{bangert_kuwert}
Victor Bangert and Ernst Kuwert.
\newblock An area bound for surfaces in {R}iemannian manifolds.
\newblock {\em Preprint arXiv:1907.03457}, 2019.

\bibitem{quanta}
Yann Bernard and Tristan Rivi\`{e}re.
\newblock Energy quantization for {W}illmore surfaces and applications.
\newblock {\em Annals of Math.}, 180:87--136, 2014.

\bibitem{blaschke}
Wilhelm J.~E. Blaschke.
\newblock {\em Vorlesungen \"{U}ber {D}ifferentialgeometrie {III}:
  {D}ifferentialgeometrie der {K}reise und {K}ugeln}.
\newblock Springer-{V}erlag, coll. {G}rundlehren der mathematischen
  {W}issenschaften, 1929.

\bibitem{calabi}
Eugenio Calabi.
\newblock Minimal immersions of surfaces in {E}uclidean spheres.
\newblock {\em J. of Differential Geom., 1, 111-125}, 1967.

\bibitem{chenli}
Jingyi Chen and Yuxiang Li.
\newblock Bubble tree of branched conformal immersions and applications to the
  {W}illmore functional.
\newblock {\em Amer. J. Math.}, 136(4):1107--1154, 2014.

\bibitem{meyer}
Ronald Coifman, Pierre-Louis Lions, Yves Meyer, and Stephen~W. Semmes.
\newblock Compensated compactness and hardy spaces.
\newblock {\em J. Math. Pures Appl. (9) 72 (1993), no. 3, 247--286.}, 1993.

\bibitem{dingtian}
Weiyue Ding and Gang Tian.
\newblock Energy identity for a class of approximate harmonic maps from
  surfaces.
\newblock {\em Comm. Anal. Geom.}, 3(3-4):543--554, 1995.

\bibitem{eichmairkorber}
Michael Eichmair and Thomas Koerber.
\newblock Large area-constrained {W}illmore surfaces in asymptotically
  {S}chwarzschild 3-manifolds.
\newblock {\em Preprint arXiv:2101.12665}, 2021.

\bibitem{germain3}
Sophie Germain.
\newblock M{\'e}moire sur cette question propos{\'e}e par la premiere classe de
  l'institut : Donner la th{\'e}orie math{\'e}matique des vibrations des
  surfaces {\'e}lastiques, et la comparer \`{a} l'exp{\'e}rience.
\newblock {\em Acad{\'e}mie des Sciences}, 30 septembre 1815.

\bibitem{gilbarg}
David Gilbarg and Neil~S. Trudinger.
\newblock {\em {E}lliptic {P}artial {D}ifferential {E}quations of {S}econd
  {O}rder}.
\newblock Springer-{V}erlag {B}erlin {H}eidelberg {G}mbH, 1977.

\bibitem{hawking}
Stephen Hawking.
\newblock Gravitational radiation in an expanding universe.
\newblock {\em J. Mathematical Phys.}, 9:598--604, 1968.

\bibitem{helein}
Frederick H{\'e}lein.
\newblock {\em Applications harmoniques, lois de conservation, et rep{\`e}res
  mobiles}.
\newblock Diderot {\'e}diteur, Sciences et Arts, 1996.

\bibitem{hoffman_gauss}
{David A.} Hoffman and Robert Osserman.
\newblock The geometry of the generalized {G}auss map.
\newblock {\em Memoirs or the American Math. Soc.}, 28(236), 1980.

\bibitem{IMM1}
Norihisa Ikoma, Andrea Malchiodi, and Andrea Mondino.
\newblock Embedded area-constrained {W}illmore tori of small area in
  {R}iemannian three-manifolds, {I}: minimization.
\newblock {\em Proc. Lond. Math. Soc.}, 115(3):502--544, 2017.

\bibitem{IMM2}
Norihisa Ikoma, Andrea Malchiodi, and Andrea Mondino.
\newblock Embedded area-constrained {W}illmore tori of small area in
  {R}iemannian three-manifolds, {II}: {M}orse theory.
\newblock {\em Amer. J. Math.}, 139(5):1315--1378, 2017.

\bibitem{Jost91}
J\"urgen Jost.
\newblock {\em Two-Dimensional Geometric Variational Problems}.
\newblock Pure Appl. Math. (New York). John Wiley \& Sons, Ltd., A
  Wiley-Interscience Publication, Chichester, 1991.

\bibitem{JostRS}
J\"urgen Jost.
\newblock {\em Compact {R}iemann {S}urfaces. {A}n {I}ntroduction to
  {C}ontemporary {M}athematics}.
\newblock Universitext. Springer-Verlag, $3^{rd}$ edition, 2006.

\bibitem{KMS}
Ernst Kuwert, Andrea Mondino, and Johannes Schygulla.
\newblock Existence of immersed spheres minimizing curvature functionals in
  compact 3-manifolds.
\newblock {\em Math. Ann.}, 359(1-2):379--425, 2014.

\bibitem{lammmetzger}
Tobias Lamm and Jan Metzger.
\newblock Small surfaces of {W}illmore type in {R}iemannian manifolds.
\newblock {\em Int. Math. Res. Not. IMRN}, 19:3786--3813, 2010.

\bibitem{lammmetzger2}
Tobias Lamm and Jan Metzger.
\newblock Minimizers of the {W}illmore functional with a small area constraint.
\newblock {\em Ann. Inst. H. Poincar{\'e} Anal. Non Lin{\'e}aire},
  30(3):497--518, 2013.

\bibitem{lammmetzgerschulze}
Tobias Lamm, Jan Metzger, and Felix Schulze.
\newblock Foliations of asymptotically flat manifolds by surfaces of {W}illmore
  type.
\newblock {\em Math. Ann.}, 350(1):1--78, 2011.

\bibitem{laurainmondino}
Paul Laurain and Andrea Mondino.
\newblock Concentration of small {W}illmore spheres in {R}iemannian
  $3$-manifolds.
\newblock {\em Anal. PDE}, 8:1901--1921, 2014.

\bibitem{angular}
Paul Laurain and Tristan Rivi\`{e}re.
\newblock Angular energy quantization for linear elliptic systems with
  antisymmetric potentials and applications.
\newblock {\em Anal. PDE 7, no. 1, 1--41.}, 2014.

\bibitem{quantamoduli}
Paul Laurain and Tristan Rivi{\`e}re.
\newblock Energy quantization of {W}illmore surfaces at the boundary of the
  moduli space.
\newblock {\em Duke Math. J.}, 167(11):2073--2124, 2018.

\bibitem{kuwertli}
Ernst Ernst~Yuxiang Li.
\newblock ${W}^{2,2}$-conformal immersions of a closed {R}iemann surface into
  $\mathbb{R}^n$.
\newblock {\em Comm. Anal. Geom. 20, no. 2, 313–340}, 2010.

\bibitem{linriv_GL2}
Fang-Hua Lin and Tristan Rivi{\`e}re.
\newblock A quantization property for moving line vortices.
\newblock {\em Comm. Pure Appl. Math.}, 54(7):826--850, 2001.

\bibitem{linriv_GL1}
Fang-Hua Lin and Tristan Rivi{\`e}re.
\newblock A quantization property for static {G}inzburg-{L}andau vortices.
\newblock {\em Comm. Pure Appl. Math.}, 54(2):206--228, 2001.

\bibitem{linriv}
Fang-Hua Lin and Tristan Rivi{\`e}re.
\newblock Energy quantization for harmonic maps.
\newblock {\em Duke Math. J.}, 111(1):177--193, 2002.

\bibitem{half_harmonic}
Francesca~Da Lio, Paul Laurain, and Tristan Rivi{\`e}re.
\newblock Pohozaev-type formula and quantization of horizontal half-harmonic
  maps.
\newblock {\em Preprint arXiv:1607.05504}, 2016.

\bibitem{thesis}
Alexis Michelat.
\newblock {\em Morse-{T}heoretic {A}spects of the {W}illmore {E}nergy}.
\newblock PhD thesis, ETH Z\"{u}rich, 2019.

\bibitem{pointwise}
Alexis Michelat and Tristan Rivi{\`e}re.
\newblock Pointwise expansion of degenerating immersions of finite total
  curvature.
\newblock {\em Chapter 3 of the PhD thesis \cite{thesis}}, 2019.

\bibitem{mon1}
Andrea Mondino.
\newblock Some results about the existence of critical points for the
  {W}illmore functional.
\newblock {\em Math. Z.}, 266(3):583--622, 2010.

\bibitem{mon2}
Andrea Mondino.
\newblock The conformal {W}illmore functional: a perturbative approach.
\newblock {\em J. Geom. Anal.}, 23(2):764--811, 2013.

\bibitem{mondiminus}
Andrea Mondino.
\newblock Existence of integral m-varifolds minimizing $\int {|A|}^p$ and $\int
  {|H|}^p$, $p>m$, in {R}iemannian manifolds.
\newblock {\em Calc. Var. Partial Differential Equations}, 49(1-2):431--470,
  2014.

\bibitem{mondinonguyen}
Andrea Mondino and {Huy The} Nguyen.
\newblock Global conformal invariants of submanifolds.
\newblock {\em Ann. Inst. Fourier (Grenoble)}, 68(6):2663--2695, 2018.

\bibitem{mondinoriviere}
Andrea Mondino and Tristan Rivi{\`e}re.
\newblock Willmore spheres in compact {R}iemannian manifolds.
\newblock {\em Advances in {M}athematics, 232, 608-676}, 2013.

\bibitem{mondinoriviereACV}
Andrea Mondino and Tristan Rivi{\`e}re.
\newblock Immersed spheres of finite total curvature into manifolds.
\newblock {\em Adv. Calc. Var.}, 7(4):493--538, 2014.

\bibitem{brownemondino}
Andrea Mondino and Aidan Templeton-Browne.
\newblock Some rigidity results for the {H}awking mass and a lower bound for
  the {B}artnik capacity.
\newblock {\em Preprint arXiv:2107.08110}, 2021.

\bibitem{muller}
Stefan M\"{u}ller and Vladim\'{i}r \v{S}ver\'{a}k.
\newblock On surfaces of finite total curvature.
\newblock {\em J. Differential Geom.}, 42(2):229--258, 1995.

\bibitem{random}
Thomas Murphy and Frederick Wilhelm.
\newblock Random manifolds have no totally geodesic submanifolds.
\newblock {\em Michigan Math. J.}, 68(2):323--335, 2019.

\bibitem{orlicz}
W.~Orlicz.
\newblock {\"U}ber eine gewisse {K}lasse von {R}{\"a}umen vom {T}ypus ${B}$.
\newblock {\em Bull. Intern. Acad. Pol. Ser. A , 8/9, pp. 207--220}, 1932.

\bibitem{parker}
Thomas~H. Parker.
\newblock Bubble tree convergence for harmonic maps.
\newblock {\em J. {D}ifferential {G}eom.}, 44(3):595--633, 1996.

\bibitem{poisson}
Sim\'{e}on~Denis Poisson.
\newblock M\'{e}moire sur les surfaces \'{e}lastiques.
\newblock {\em M\'{e}moire de l'Institut}, 1814.

\bibitem{polyakov}
Alexander Polyakov.
\newblock Fine structure of strings.
\newblock {\em Nuclear Phys. B}, 268(2):406--412, 1986.

\bibitem{RivICM}
Tristan Rivi{\`e}re.
\newblock Bubbling and regularity issues in geometric non-linear analysis.
\newblock In {\em Proceedings of the International Congress of Mathematicians
  (Beijing, 2002)}, volume~3, pages 197--208, Beijing, 2002. Higher Ed. Press.

\bibitem{riviere1}
Tristan Rivi\`{e}re.
\newblock Analysis aspects of {W}illmore surfaces.
\newblock {\em Invent. Math.}, 174:1--45, 2008.

\bibitem{rivnotes}
Tristan Rivi\`{e}re.
\newblock {\em Conformally Invariant Variational Problems}.
\newblock arXiv:1206.2116, 2012.

\bibitem{rivierecrelle}
Tristan Rivi{\`e}re.
\newblock Variational principles for immersed surfaces with ${L}^2$-bounded
  second fundamental form.
\newblock {\em J. Reine Angew. Math.}, 695:41--98, 2014.

\bibitem{eversion}
Tristan Rivi{\`e}re.
\newblock Willmore {M}inmax {S}urfaces and the {C}ost of the {S}phere
  {E}version.
\newblock {\em J. Eur. Math. Soc., Volume 23, Issue 2, pp. 349--423}, 2021.

\bibitem{sacks}
Jonathan Sacks and Karen Uhlenbeck.
\newblock The existence of minimal immersions of $2$-spheres.
\newblock {\em Annals of Math.}, 113:1--24, 1981.

\bibitem{simon}
Leon Simon.
\newblock Existence of surfaces minimizing the {W}illmore functional.
\newblock {\em Comm. Anal. Geom.}, 1(2):281--326, 1993.

\bibitem{lorentz_mary}
{Mary S.} Steigerwalt.
\newblock {\em Some Banach function spaces related to Orlicz spaces}.
\newblock PhD thesis, University of Aberdeen Thesis, 1967.

\bibitem{lorentz_general}
{Mary S.} Steigerwalt and {Alan J.} White.
\newblock Some function spaces related to {$L_p$} spaces.
\newblock {\em Proc. London Math. Soc.}, 22:137--163, 1971.

\bibitem{Struwe85}
Michael Struwe.
\newblock On the evolution of harmonic mappings of {R}iemannian surfaces.
\newblock {\em Comment. Math. Helv.}, 60:558--581, 1985.

\bibitem{toro}
Tatiana Toro.
\newblock Geometric conditions and existence of bi-{L}ipschitz
  parameterizations.
\newblock {\em Duke Math. J.}, 77(1):193--227, 1995.

\bibitem{willmore1}
{Thomas J.} Willmore.
\newblock Note on embedded surfaces.
\newblock {\em An. st. Univ. Iaso, s.I.a., Mathematica, 11B, 493-496}, 1965.

\end{thebibliography}

\end{document}